\documentclass[11pt, reqno, a4paper]{article}
\usepackage[left=1.5cm,right=1.5cm,bottom=2cm,top=1.5cm]{geometry}
\usepackage[colorlinks=true,breaklinks=true,linkcolor=lightblue,citecolor=lightblue,urlcolor=lightblue]{hyperref}

\usepackage{amssymb,mathrsfs,amsmath,epsfig}
\usepackage{graphics,psfrag,graphicx,color}
\usepackage{booktabs} 
\usepackage[table]{xcolor} 
\definecolor{lightgray}{gray}{0.9}
\definecolor{lightgreen}{rgb}{0.22,0.50,0.25}
\definecolor{lightblue}{rgb}{0.22,0.45,0.70}

\usepackage{float}
\usepackage{siunitx}
\usepackage{algorithm} 
\usepackage{algpseudocode}

\usepackage{cleveref}
\numberwithin{equation}{section}
\numberwithin{figure}{section}
\numberwithin{table}{section}

\newcommand{\bta}{{\boldsymbol\tau}}
\newcommand\bsi{{\boldsymbol \sigma}}

\newcommand{\bn}{{\boldsymbol n}}

\newcommand\bv{{\boldsymbol v}}
\newcommand\bw{{\boldsymbol w}}
\newcommand\bu{{\boldsymbol u}}

\newcommand\br{{\boldsymbol r}}
\newcommand\bs{{\boldsymbol s}}
\newcommand\bi{{\mathbf{i}}}
\newcommand\bt{{\mathbf{t}}}
\newcommand\tr{{\mathrm{tr}}}
\newcommand\bdiv{{\mathbf{div}}}
\renewcommand\div{{\mathrm{div}}}
\newcommand\qan{{\quad\hbox{and}\quad}}
\newcommand\qin{{\quad\hbox{in}\quad}}
\newcommand\qon{{\quad\hbox{on}\quad}}
\newcommand\disp{\displaystyle}
\newcommand{\oL}{{\mathscr{L}}}
\newcommand{\Thb}{{\mathcal{T}_h^{\rm b}}}

\newcommand{\curle}{\mathrm{curl}}
\newcommand{\curlv}{\underline{\mathrm{curl}}}
\newcommand{\curlt}{\underline{\mathbf{curl}}}
\newcommand{\tncomp}{\boldsymbol{\gamma_{\ast}}}

\newtheorem{thm}{Theorem}[section]

\newtheorem{rem}{Remark}[section]
\newtheorem{lem}[thm]{Lemma}

\newenvironment{proof}{\noindent{\it Proof.}}{\hfill$\square$}
\newenvironment{hip}{\noindent{\textbf{Regularity Hipothesis.}}}

\def\R{\mathcal{R}}

\def\A{{\mathscr{A}}}
\def\B{{\mathscr{B}}}

\def\C{{\mathscr{C}}}

\def\F{{\mathscr{F}}}

\def\jump#1{\text{ $\hspace{-0.1cm}[\![#1]\!]$}}

\allowdisplaybreaks

\title{Convergence analysis and adaptive computation of a Banach-space mixed finite element method for generalized bioconvective flows\thanks{This work was
partially supported by ANID-Chile through the project 11190241 and  Centro de Modelamiento Matem\'atico - BASAL Project FB210005, by the Australian Research Council through the Future
Fellowship grant FT220100496 and Discovery Project grant DP22010316, and by 
the Swedish Research Council under grant no. 2021-06594 while RRB was in residence at Institut Mittag--Leffler in Djursholm, Sweden during the fall semester of 2025.}}

\author{{\sc Eligio Colmenares}\thanks{Centro de Ciencias Exactas, Grupo de Investigaci\'on en M\'etodos Num\'ericos y Aplicaciones (GIMNAP), Departamento de Ciencias B\'asicas, Facultad de Ciencias, Universidad del
B\'io-B\'io, Campus Fernando May, Chill\'an, Chile, email: {\tt ecolmenares@ubiobio.cl}.}
\quad
{\sc Ricardo Ruiz-Baier}\thanks{School of Mathematics, Monash University, 9 Rainforest Walk, Melbourne, VIC 3800, Australia; and Universidad Adventista de Chile, Casilla 7-D, Chill\'an, Chile, 
	email: {\tt ricardo.ruizbaier@monash.edu}.}
\quad 
{\sc Dalidet Sanhueza}\thanks{Departamento de Matem\' atica,
Universidad del B\'io-B\'io, Casilla 5-C, Concepci\' on, Chile, 
email: {\tt dalidet.sanhueza1701@alumos.ubiobio.cl}.}}

\date{ }

\begin{document}

\maketitle

	
\begin{abstract}
\noindent
	We develop and analyse an adaptive fully mixed finite element method for stationary generalized bioconvective flows, in which the Navier--Stokes equations with concentration-dependent viscosity are coupled to a conservation law for swimming microorganisms. The method introduces the trace-free velocity gradient, a symmetric pseudo-stress tensor, the concentration gradient and a semi-advective microorganism flux, the latter also enabling a consistent treatment of Robin-type boundary conditions for the concentration. The variational formulation is developed within a Banach space framework that includes these auxiliary variables, in addition to the fluid's velocity, pressure, and the microorganisms' concentration. The analysis progresses by examining the fixed-point operator, which reformulates the continuous problem's variational formulation equivalently. The existence of solutions is obtained by using Schauder's theorem, while uniqueness relies on particular data constraints. In the discrete setting, we utilize Raviart-Thomas spaces and piecewise polynomials defined on macroelement-structured meshes. The existence of solutions in this context is established with Brouwer's theorem, and the uniqueness is guaranteed by the Banach fixed point theorem in the case in which the viscosity is constant. An a priori error analysis yields optimal convergence estimates. Additionally, we derive a residual-based a posteriori error estimator whose reliability is demonstrated using global inf-sup conditions, appropriate Helmholtz decompositions, and properties of Raviart-Thomas and Cl\'ement projectors. The efficiency of the estimator is ensured through localization techniques and classical bubble functions. A set of numerical experiments in two and three dimensions confirms the predicted convergence rates, demonstrates the effectiveness of adaptive refinement for singular solutions and complex geometries with inclusions, and illustrates the robustness of the proposed formulation when applied to a time-dependent bioconvective benchmark exhibiting plume formation, based on an Einstein--Batchelor-type viscosity law.
	
\end{abstract}



\maketitle

\section{Introduction}\label{section1}

Bioconvection, or biological convection, encompasses the transport of substances or particles within a biological medium, driven by the active movement of microorganisms in response to external stimuli such as gravity, light, oxygen, nutrient or temperature gradients, or some  combination of these \cite{Pedley.Kessler}. This motility of microorganisms generates concentration gradients, leading to directed transport of substances through the medium. Bioconvection phenomena are observed in various natural settings, including nutrient transport in plants, cellular transport systems, blood circulation, and aquatic environments inhabited by bacteria, algae, and other cellular organisms. Their active movement induces water flow and creates concentration gradients essential for nutrient and vital substance transport. Furthermore, bioconvection has been recognized for its practical applications in biotechnology, medicine, and engineering, such as in controlled environments like laboratories or cell culture facilities, biocombustible production processes, and wastewater treatment, where understanding and controlling bioconvection mechanisms can enhance efficiency and effectiveness \cite{bees-2014,kuznetsov-2005, lauga-2009,Libro-bioconvection}.  

\noindent Based on the principles of hydrodynamics and mass transfer, a mathematical model for describing the interaction between the fluid flow and microorganisms in a bioconvection phenomenon (see \cite{levandowsky-1975} and \cite{moribe-1973})  involves  the fluid velocity $\bu$, the fluid pressure $p$, and the concentration $\varphi_\alpha$ of microorganisms within a  culture fluid $\Omega\subset \mathbb{R}^d$ ($d = 2, 3$),  satisfying
\begin{equation}\label{eqn:bio}
	\begin{array}{c}    
		-2\,\mathbf{div}\left(\mu(\varphi_{\alpha})\mathbf{e}(\bu)\right)+\left(\bu \cdot \nabla \right)\bu+\nabla \,p=\,\boldsymbol{f}-g\left[1+\gamma\varphi_{\alpha}\right]\widehat{\mathbf{e}}_d\,, \quad 
		\mathrm{div}\,\bu\,=\,  0\quad \qin\quad\Omega,\\[2ex]
		\displaystyle{ -\kappa\Delta\varphi_\alpha+\bu\cdot\nabla\varphi_\alpha+U\frac{\partial\varphi_\alpha}{\partial x_{d}}}= 0 \quad \qin\quad\Omega\quad \qan\quad \frac{1}{|\Omega|}\int_{\Omega}\varphi_\alpha=\alpha\,.
	\end{array}
\end{equation}

\noindent The first equation  of \eqref{eqn:bio} represents the momentum balance for the fluid flow within the domain, incorporating the effects of a concentration-dependent viscosity $\mu(\,\cdot\,)$ and the strain rate tensor $\mathbf{e}(\bu)$, convective transport of momentum $\left(\bu \cdot \nabla\right)\bu$, and the hydrostatic pressure gradient $\nabla p$. The term $\boldsymbol{f}$ refers to a volume-distributed external force or source term, and $-g\left[1+\gamma\varphi_{\alpha}\right]\widehat{\mathbf{e}}_d$ models the buoyancy effects due to the presence of microorganisms in  the $d$--axis direction, represented by the unit vector $\widehat{\mathbf{e}}_d$. The constant  $g$ is the acceleration due to gravity and the  parameter $\gamma:=\rho_0/\rho_m-1$ stands for the relative density difference, with $\rho_0$ being the density of the microorganisms and $\rho_m$  being the density of the culture fluid.  In turn, the relation $\div\, \bu =0$, given by the second equation in \eqref{eqn:bio} represents the fluid incompressibility constraint.

\noindent The third equation in   \eqref{eqn:bio} describes the mass transfer  of microorganisms within the fluid, reflecting the balance between diffusion, characterized by the diffusivity constant $\kappa$, and advection by the fluid flow and swimming motion of the microorganisms. The term $\bu\cdot\nabla\varphi_\alpha$ accounts for the advection of microorganisms by the fluid, while $U\frac{\partial\varphi_\alpha}{\partial x_{d}}$ represents the directed swimming of microorganisms in the $d$--axis direction, with $U$ being the swimming speed. This equation states the conservation of microorganisms in the domain, aligning with the principle that the average concentration of microorganisms, denoted by $\alpha$, is preserved throughout the fluid culture, according to the last equation  in \eqref{eqn:bio}. Precise assumptions on the data and parameters as well as the  boundary conditions to be taken into consideration are presented in the next section.

\noindent Formal mathematical analyses of the model  \eqref{eqn:bio} have been conducted in references \cite{aguiar-2017,Bold-bioconvection,Climent-2013,Teramoto-bioconvection}. In \cite{Teramoto-bioconvection}, the authors focus on a scenario with constant viscosity,  demonstrating the existence of solutions for the stationary problem, the positivity of concentration, and deriving conditions necessary for the existence of a global weak solution in the non-stationary case. Building upon this, \cite{aguiar-2017} extends these findings by determining convergence rates for the error associated with spectral Galerkin approximations. In contrast, \cite{Bold-bioconvection} approaches viscosity as a concentration--dependent variable, establishing the existence of both weak and strong solutions. Finally, \cite{Climent-2013} explores the latter case under periodic conditions.

\noindent From a computational perspective, the system \eqref{eqn:bio} gives rise to strongly coupled nonlinear dynamics and may exhibit localized flow structures and sharp gradients in the microorganism concentration. These features pose significant challenges for numerical simulation, since standard discretizations and uniform mesh refinement may become inefficient or fail to accurately resolve the multiscale nature of bioconvective phenomena. A variety of numerical techniques have been proposed for the simulation of bioconvection models, including early finite-difference and spectral approaches focusing mainly on two-dimensional configurations and constant viscosity assumptions \cite{childress-1976,ghorai-2000,ghorai-1999,ghorai-2000-2,harashima-1988,hopkins-2002,lee-2015,tb-2007}.

\noindent
Finite element discretizations of \eqref{eqn:bio} have been considered in \cite{cao-2013,CGMir-1}. In particular, \cite{cao-2013} establishes existence and uniqueness of weak solutions and proposes standard finite element approximations, validated through numerical experiments incorporating laboratory data. A fully mixed finite element formulation within a Hilbert space setting is introduced in \cite{CGMir-1}, where auxiliary variables such as the fluid strain tensor, vorticity and a pseudo--concentration gradient are employed. By augmenting the weak formulation with Galerkin penalization terms and reformulating the coupled system as a fixed--point problem, the authors prove well--posedness and derive optimal a priori error estimates, supported by numerical results.

\noindent
More recently, a growing body of work has developed mixed finite element methods for nonlinear problems within a Banach space framework (see \cite{bcgh-2020,bcgh-2022,ccgi-2023,CGMor,CGMR-2020,cgr-2022,CN-16,hw-M2AN-2013}, among others). A key advantage of this approach lies in the natural incorporation of variables into their intrinsic functional settings, achieved through tailored testing strategies and integration by parts, thus avoiding the additional stabilization or penalization mechanisms required in Hilbert space formulations such as \cite{CGMir-1}. Moreover, this framework allows for the direct computation or postprocessing of physically relevant quantities. Complementarily, residual-based a posteriori error estimation techniques for mixed methods in Banach spaces have been investigated in \cite{COV-2022,CCOV-2022,GIRS-2022}, where reliability and efficiency are established. Nevertheless, adaptive mixed finite element methods of this type have not yet been developed for the bioconvection model \eqref{eqn:bio}.

Inspired by the above discussion and related developments in mixed formulations for nonlinear problems in Banach spaces \cite{CGMor,CN-16,hw-M2AN-2013,GIRS-2022}, we develop an adaptive fully mixed finite element method for the bioconvection model \eqref{eqn:bio}. The proposed approach is formulated entirely within a Banach space setting and is complemented by a residual-based a posteriori error analysis, enabling adaptive mesh refinement strategies for the efficient numerical simulation of bioconvective flows. The main contributions of this work can be summarized as follows:

\begin{itemize}
	\item[(a)]	The trace--free velocity gradient and the concentration gradient are introduced as primary unknowns, which allows the computation of physically meaningful quantities  such as vorticity, shear stress tensors and microorganism fluxes through standard postprocessing of discrete solutions.
	\item[(b)] By suitably defining associated function spaces, the pressure variable is eliminated from the primary computation and can subsequently be recalculated through a straightforward postprocessing step.
	\item[(c)] 	The intrinsic skew--symmetry of the convective forms is preserved at both continuous and discrete levels, leading to simplified a priori estimates and well--posedness results without the need for additional stabilization or symmetry--preserving modifications.
	\item[(d)] Unlike \cite{CGMir-1,CGMor,CGMR-2020}, we reformulate both the continuous and discrete problems as a fixed-point problem involving a single operator. In those references, the authors address the coupled problem by defining an operator for each subproblem and composing them, which requires additional data restrictions and conditions. In contrast, we first derive a priori estimates to determine the solution set, then use a single operator to handle the entire coupled problem, thus avoiding the need for restrictive data conditions.
	
	\item[(e)] The method supports high-order approximations and provides optimal-order a priori error estimates for both the primary variables and those obtained via postprocessing.
\end{itemize}

\subsection*{Outline}
This paper is structured as follows: The rest of this introductory section sets the stage by defining standard notations,  describing the functional spaces, and specifying the assumptions about the data and the boundary conditions under study. Section \ref{section2} is dedicated to presenting the fully mixed formulation of our problem. Following that, in Section \ref{section3}  we present  the Galerkin scheme associated with our formulation along with the corresponding  Cea's estimate and proving optimal a priori error estimates. In Section \ref{section4}, we carry out an a posteriori error analysis for our fully-mixed method.  Section \ref{section6} concludes the paper with numerical examples that demonstrate the efficacy of the fully mixed method and verify the theoretical results.

\subsection*{Preliminary notations and definitions}

\textbf{Domain.} Let $\Omega\subset \mathbb{R}^d$ ($d = 2, 3$) be a bounded domain with polygonal/polyhedral boundary $\Gamma := \partial\Omega$, and $\bn$ denote the outward unit normal vector.

\noindent\textbf{Vector and Tensor Operators.} The notations $\mathrm{A}$, $\mathbf{A}$, and $\mathbb{A}$ are used to represent scalar, vector, and tensor field spaces, respectively. For vector fields $\bv = (v_i)_{1\leq i\leq d}$ and $\bw = (w_i)_{1\leq i\leq d}$, their gradient, divergence, and dyadic product are defined as $\nabla \bv := (\partial_{x_j} v_i)_{1\leq i,j\leq d}$, $\div\, \bv := \sum_{i=1}^d \partial_{x_i}v_i$, and $\bv \otimes \bw := (v_i w_j)_{1\leq i,j\leq d}$, respectively. For tensor fields $\boldsymbol{\tau} = (\tau_{ij})_{1\leq i,j\leq d}$ and $\boldsymbol{\zeta} = (\zeta_{ij})_{1\leq i,j\leq d}$, $\bdiv \,\boldsymbol{\tau}$ denotes the row-wise divergence, and we define the transpose $\boldsymbol{\tau}^{\text{t}} := (\tau_{ji})_{1\leq i,j\leq d}$, trace $\text{tr}(\boldsymbol{\tau}) := \sum_{i=1}^d \tau_{ii}$, tensor inner product $\boldsymbol{\tau} : \boldsymbol{\zeta} := \sum_{i,j=1}^d \tau_{ij}\zeta_{ij}$, and deviatoric part $\boldsymbol{\tau}^{\text{d}} := \boldsymbol{\tau} - \frac{1}{d}\text{tr}(\boldsymbol{\tau})\mathbb{I}$, with $\mathbb{I}$ being the $d\times d$ identity tensor.

\noindent\textbf{Function Spaces.} For any  $r\in[1,+\infty)$ and  $s\geq 0$,  we  denote the conventional Sobolev space  by  $\mathrm{W}^{s,r}(\Omega)$  where both the functions and all their distributional derivatives up to order $s$ are elements of the Lebesgue space $\mathrm{L}^r(\Omega)$.  The norm and semi-norm in this space are denoted by  $\|\cdot\|_{s, r,\Omega}$ and $|\cdot|_{s,r,\Omega}$, respectively. In particular, when $r=2$, we simply write  $\mathrm{H}^{s}(\Omega):=\mathrm{W}^{s,2}(\Omega), $  $\|\cdot\|_{s, \Omega}:=\|\cdot\|_{s, 2,\Omega}$ and $|\cdot|_{s,\Omega}:=|\cdot|_{s,2,\Omega}$. The space of traces of functions in $\mathrm H^1(\Omega)$ is denoted by  ${\mathrm H}^{1/2}(\Gamma)$, while ${\mathrm H}^{-1/2}(\Gamma)$ represents its dual space, with duality pairing denoted as $\langle\cdot,\cdot\rangle,$ and $\mathrm{H}^{1}_{0}(\Omega)$ stands for the set of functions  in $\mathrm H^1(\Omega)$ with trace zero on the boundary $\Gamma$. 

\noindent For functions with zero mean and involving first-order Sobolev spaces we set
\begin{equation}\label{eqn:L2H1Spaces}
	\mathrm{L}^r_0(\Omega) := \Bigg\{q \in \mathrm{L}^r(\Omega) : \quad \int_{\Omega}q = 0\Bigg\}\quad \qan \quad \widetilde{\mathrm{H}}^1(\Omega) := \mathrm{H}^1(\Omega) \cap \mathrm{L}^4_0(\Omega)\,,
	\end{equation}
respectively. From the Friedrichs-Poincaré inequality, we recall the existence of a  constant $C_{\mathrm{FP}} > 0$ ensuring that
\begin{equation}\label{eqn:fp}
	\|\psi\|_{1,\Omega} \leq C_{\mathrm{FP}} |\psi|_{1,\Omega}, \quad \forall\, \psi  \in \mathrm{H}^1_0(\Omega)  \text{ or }\widetilde{\mathrm{H}}^1(\Omega)\,.
\end{equation}
Similarly, the Korn inequality is expressed as
\begin{equation}\label{eqn:korn}
	\|\mathbf{e}(\bw)\|^2_{0,\Omega} \geq \frac{1}{2} |\bw|^2_{1,\Omega}, \quad \forall \, \bw \in \mathbf{H}^1_0(\Omega),
\end{equation}
with $\mathbf{e}(\bw) = \frac{1}{2}(\nabla \bw + (\nabla \bw)^{\mathrm{t}})$ denoting the symmetric gradient or strain tensor.

\noindent To handle trace-free tensors we consider the space 
\begin{equation}\label{eqn:espL2tr}
	\mathbb{L}_{\tr}^2(\Omega) := \Big\{\mathbf{r} \in \mathbb{L}^2(\Omega) : \text{tr}\,\mathbf{r} = 0\Big\}\,.
	\end{equation}
We will frequently utilize the following spaces
\begin{equation}\label{eq:spacesH&Ht} \mathbf{H}:= \mathbb{L}^2_{\tr}(\Omega)\times \mathbf{L}^4(\Omega) \quad \qan \quad  \widetilde{\mathbf{H}}:= \mathbf{L}^2(\Omega)\times\mathrm{L}_0^{4}(\Omega)\,,
\end{equation}
with the corresponding natural norms given by
\begin{subequations}
\begin{equation}\label{eq:norm-u}
	\|\vec{\bv}\|_\mathbf{H}^2=\|(\br,\bv)\|_\mathbf{H}^2=\|\br\|^2_{0,\Omega}+\|\bv\|^2_{0,4,\Omega} \qquad \forall\,\vec{\bv}\in\,\mathbf{H}:= \mathbb{L}^2_{\tr}(\Omega)\times \mathbf{L}^4(\Omega)\,,
\end{equation}
and 
\begin{equation}\label{eq:norm-varphi}
	\|\vec{\psi}\|_{\widetilde{\mathbf{H}}}^2=\|(\widetilde{\br},\psi)\|_{\widetilde{\mathbf{H}}}^2=\|\widetilde{\br}\|^2_{0,\Omega}+\|\psi\|^2_{0,4,\Omega} \qquad \forall\,\vec{\psi}\in\,\widetilde{\mathbf{H}}:= \mathbf{L}^2(\Omega)\times\mathrm{L}_0^{4}(\Omega)\,.
\end{equation}\end{subequations}

\noindent For specific divergence and normal component conditions, we consider 
\begin{subequations}
\begin{equation}\label{eq:spaceHgamma}
	\mathbf{H}_\Gamma(\div_r;\Omega) := \bigg\{\widetilde{\bta}\in \mathbf{H}(\div_r;\Omega):\quad \widetilde{\bta}\cdot\bn=0\quad \qon\quad\Gamma\bigg\}\,,
\end{equation}
\begin{equation}\label{eq:spaceH0div}
	\mathbb{H}_0(\mathbf{div}_r;\Omega)\ :=\disp\ \left\{\bta\in \mathbb{H}(\bdiv_r;\Omega):\quad\int_{\Omega}\tr\,\bta=0\right\}.
\end{equation}
\end{subequations}
Here, the set $	\mathbf{H}(\div_{r};\Omega)\ :=\ \Big\{\widetilde{\bta}\in\mathbf{L}^2(\Omega):\quad{\div}\,\widetilde{\bta}\in\mathrm{L}^r(\Omega)\Big\}$  (analogously for its tensorial version $\mathbb{H}(\mathbf{div}_r; \Omega)$)
is  a Banach space equipped with the   norm 
\begin{equation}\label{eq:normHdiv}
	\|\widetilde{\bta}\|^2_{\div_{r},\Omega}:=\|\widetilde{\bta}\|^2_{0,\Omega}+\|\div\,\widetilde{\bta}\,\|^2_{0,r,\Omega}\,,\quad \forall\,\widetilde{\bta}\in	\mathbf{H}(\div_{r};\Omega)\,.
	\end{equation}

\noindent Moreover, for $ r\in [6/5,+\infty)$ the following integration-by-parts formula holds (see \cite[Section 4.1]{CMO-2018} and \cite[Section 3.1]{CGMor})
\begin{equation}\label{eq:ipp-1}
	\langle \widetilde{\bta} \cdot\bn,v\rangle=\int_{\Omega} \left\{\widetilde{\bta}\cdot\nabla v +v\, \div\,\widetilde{\bta}\,\right\}\quad\forall\, (\widetilde{\bta},v)\in\, \mathbf{H}(\div_{r};\Omega)\times \mathrm{H}^1(\Omega). 
\end{equation}

\noindent We  further recall that  the injections  $i:  \mathrm{L}^q(\Omega)\rightarrow \mathrm{L}^r(\Omega)$ ($q>r$) and $i_r:\mathrm{H}^1(\Omega)\rightarrow \mathrm{L}^r(\Omega)$ are both continuous and  satisfy 
 \begin{equation}\label{eq:inclusion-Lebesgue}
 	\|\psi\|_{0,r,\Omega}\leq |\Omega|^{1/r-1/q}\|\psi\|_{0,q,\Omega}\quad \mbox{and}\quad 	\|\psi\|_{0,r,\Omega}\,\leq\,\|i_{r}\|\|\psi\|_{1,\Omega}  \quad \mbox{for}\quad 
 	\begin{cases}
 		\mbox{ $r \geq 1 $ \quad if $d=2$\,, }\\
 		\mbox{ $r \in [1,6] $ \quad if $d=3$.}
 	\end{cases}    
 \end{equation}

\noindent \textbf{Data assumptions and boundary conditions.} In the model  \eqref{eqn:bio}, we assume that  the source term $\boldsymbol{f}$ belongs to $\mathbf{L}^{4/3}(\Omega)$ and that the parameters $g$, $\kappa,$ $U$, $\gamma$ and $\alpha$ are given positive constants. Furthermore, $\mu(\,\cdot\,)$ is a concentration dependent function  assumed to be a Lipschitz continuous and bounded from above and below; that is, for some constants $L_{\mu}>0$  and $\mu_1, \mu_2 > 0$, there hold
\begin{subequations}
\begin{equation}\label{eqn:mu-lips}
	|\mu(s)-\mu(t)|\,\leq\, L_{\mu}\,|s-t|\,,\qquad \forall\, s,t\geq 0,
\end{equation}
\begin{equation}\label{eqn:mu-bound}
	\mu_1\,\leq \,\mu(s)\,\leq\, \mu_2,\qquad \forall s\,\geq 0.
\end{equation}\end{subequations}
Additionally, the system (\ref{eqn:bio}) is supplemented with a non-slip boundary condition for the velocity and a zero flux Robin-type condition for the micro-organisms on the boundary, that is

\begin{equation}\label{eqn:conditions}
	\bu=\textbf{0}\quad\qon \,\Gamma\qquad \qan \quad \kappa\dfrac{\partial\varphi_\alpha}{\partial \bn}-n_d U \varphi_\alpha=0\quad\qon \,\Gamma\,.
\end{equation}
The last condition given in \eqref{eqn:conditions} says  that micro-organisms are not allowed to leave or enter the physical domain; that is, the  total mass  of microorganisms remains constant and equals to $\alpha$. 

\section{The fully-mixed formulation}\label{section2}

In this section, we carry out the variational formulation of the problem of our interest  \eqref{eqn:bio} and \eqref{eqn:conditions}. In Section \ref{section:first-order-system},  we introduce auxiliary variables to set the original model into a first-order partial differential system. Subsequently, Section \ref{section22} is dedicated to deriving a fully--mixed formulation within Banach spaces, along with a discussion on the properties of the forms involved. Moving forward to Section \ref{section23}, the focus shifts to establishing the well-posedness of the problem. This involves exploiting the structure of the setting and the properties of the forms to reformulate the weak formulation as a fixed--point problem.

\subsection{The equivalent first-order system}\label{section:first-order-system}
The initial step prior to reformulating the model \eqref{eqn:bio} in the context of first--order partial differential equations involves a translation process employing the total mass condition.  In fact, note that
\begin{equation*}
\dfrac{1}{|\Omega|}	\int_{\Omega}\varphi_\alpha=\alpha\qquad\Longleftrightarrow\qquad\int_{\Omega}(\varphi_\alpha-\alpha)=\,0\,,
\end{equation*}
and thus, by considering the auxiliary concentration $\varphi:=\varphi_\alpha-\alpha$, the system \eqref{eqn:bio} and \eqref{eqn:conditions} reads
\begin{equation}\label{eqn:bio-translated}
	\left.
	\begin{array}{c}    
		-2\,\mathbf{div}\left(\mu(\varphi +\alpha)\mathbf{e}(\bu)\right)+\left(\bu\cdot \nabla\right)\bu+\nabla \,p=\,\boldsymbol{f}-g\left[1+\gamma(\varphi+\alpha)\right]\widehat{\mathbf{e}}_d\,, \quad 
		\mathrm{div}\,\bu\,=\,  0\,\\[2ex]
		\disp
		-\kappa\Delta\varphi+\bu\cdot\nabla\varphi+U\frac{\partial\varphi}{\partial x_{d}}= 0
	\end{array}\right\} \qin \Omega
\end{equation}
and 
\begin{equation}\label{eqn:bc-translated}
	\bu=\textbf{0}\quad\qon \,\Gamma\,, \quad \kappa\dfrac{\partial\varphi}{\partial \bn}-n_d U \varphi=n_d U \alpha\quad\qon \,\Gamma\,\quad \qan\quad \int_{\Omega}\varphi=0\,.
\end{equation}
Next we  incorporate some auxiliary  variables. We start with the fluid equations by defining the velocity gradient and the symmetric pseudo-stress tensors given by  
\begin{equation}\label{eqn:aux.NS}
	\bt :=\nabla\bu \qin \Omega\,, \qan \bsi\,:= 2 \mu(\varphi+\alpha)\bt_{sym}-\dfrac{1}{2}(\bu\otimes\bu)-(p+\mathrm{c}_\bu)\mathbb I\qin \Omega\,,
\end{equation}
where $\bt_{sym}:=\dfrac{1}{2}\{\bt+\bt^{\rm t}\}$ is the symmetric part of $\bt$ and satisfies $\tr \,\bt_{sym}=0$ by the incompressibility condition (second equation of \eqref{eqn:bio-translated}). In turn, the constant $\disp\mathrm{c}_\bu$ is defined as
 \begin{equation}\label{eq:constant-cu}
\disp\mathrm{c}_\bu:=-\dfrac{1}{2d|\Omega|}\int_{\Omega}\tr(\bu\otimes\bu)\,.
\end{equation}
Given this context, and since a unique pressure solution $p$ for the system \eqref{eqn:bio-translated} is required to be in $\mathrm{L}^2_0(\Omega)$ (see \eqref{eqn:L2H1Spaces}). The definition of $\bsi$ in \eqref{eqn:aux.NS}, with $c_{\bu}$ as in \eqref{eq:constant-cu},  then translates the zero mean value condition on $p$ into the imposition on the trace of $\bsi$ as 
 \begin{equation*}
\disp \int_\Omega \tr\, \bsi =0\,.
\end{equation*}

\noindent The second equation of \eqref{eqn:aux.NS} will be referred to as the constitutive law governing the behavior of the fluid. To get the respective equilibrium relation, we take divergence there,  and after using the first equation of \eqref{eqn:bio-translated} we find that
 \begin{equation*}
 -\bdiv\,\bsi+\dfrac{1}{2}\bt\bu=\boldsymbol{f}-g[1+\gamma(\varphi+\alpha)]\widehat{\mathbf{e}}_d \qquad\qin \Omega\,,
\end{equation*}
where we have utilized that $\bdiv(\bu\otimes\bu)=(\bu \cdot \nabla)\bu$ when $\div \,\bu=0$.  The latter also implies that $\tr \,\bt=\tr \,\bt^{\tt t }=0$ and so $\bt_{sym}^{\tt d }=\bt_{sym}$. Thus, after taking deviatoric part to $\bsi$ in \eqref{eqn:aux.NS}, we find
\begin{equation*}
	\bsi^{\tt d}= 2\mu(\varphi+\alpha)\bt_{sym}-\dfrac{1}{2}\left(\bu\otimes\bu\right)^{\tt d} \qquad\qin \Omega\,,
\end{equation*}
and then the pressure can be  removed from the original system \eqref{eqn:bio-translated}, but it is possible to retrieve it  using the post-processing formula
\begin{equation}\label{eqn:p}
	p=-\dfrac{1}{2d}\tr(\,2\bsi\,+\,\bu\otimes\bu\,)-\mathrm{c}_\bu \qin \Omega\,, 
\end{equation}
which is obtained after taking trace to $\bsi$ in \eqref{eqn:aux.NS} and using the incompressibility condition once again. 

\noindent As for the equation modeling the micro-organisms concentration, we introduce as new variables the concentration gradient and the semi-advective flux given by  
\begin{equation}\label{eqn:aux.C}
	\widetilde{\bt}=\nabla \varphi \qin \Omega\,, \qan \widetilde{\bsi}=\kappa\,\widetilde{\bt}-\dfrac{1}{2}\varphi \bu-U(\varphi+\alpha)\widehat{\mathbf{e}}_d \qin \Omega\,,
\end{equation}
and upon applying the divergence operator to $\widetilde{\bsi}$ in \eqref{eqn:aux.C}  and using that $\bu$ is divergence-free in $\Omega$,  the last equation of the system  \eqref{eqn:bio-translated} transforms into
\begin{equation}\label{eqn:aux.C2}
	-\div\,\widetilde{\bsi}+\dfrac{1}{2}\widetilde{\bt}\cdot\bu=0\qquad\qin \Omega.
\end{equation}
Observe that the second equation of  \eqref{eqn:aux.C} and \eqref{eqn:aux.C2} represent the constitutive and equilibrium relationships associated with the concentration equation, respectively. Also note from the boundary condition for $\bu$ and $\varphi$ in \eqref{eqn:bc-translated} that $\widetilde{\bsi}$ satisfies
\begin{equation}\label{eqn:sig.n}
	\widetilde{\bsi}\cdot\bn=0\qquad\qon \Gamma\,.
\end{equation}
As a result, by combining \eqref{eqn:aux.NS}-\eqref{eqn:sig.n}, we restate our model problem (\ref{eqn:bio}) as a first-order system of PDEs. The task is to find the tuple $((\bt,\bu),\bsi,(\widetilde{\bt},\varphi),\widetilde{\bsi})$ within appropriately defined spaces (see Section \ref{section22} below), satisfying
\begin{subequations}
\begin{equation}\label{eqn:auxcomp}
	\left.
	\begin{array}{ccc}
		\bt=\nabla\bu \,, & \bsi^{\tt d}=2\mu(\varphi+\alpha)\bt_{sym}- \dfrac{1}{2}\left(\bu\otimes\bu\right)^{\tt d} \,,&-\bdiv\,\bsi+\dfrac{1}{2}\bt\bu=\boldsymbol{f}-g[1+\gamma(\varphi+\alpha)]\widehat{\mathbf{e}}_d\\
		\widetilde{\bt}=\nabla \varphi\,, & \widetilde{\bsi}=\kappa\,\widetilde{\bt}-\dfrac{1}{2}\varphi\bu-U(\varphi+\alpha)\widehat{\mathbf{e}}_d\,,&	-\div\,\widetilde{\bsi}+\dfrac{1}{2}\widetilde{\bt}\cdot\bu=0		
	\end{array}\right\} \qin \Omega
\end{equation}
along with
\begin{equation}\label{eq:bc-full}
	\bu=\textbf{0} \,\qon \Gamma\,,\quad  \widetilde{\bsi}\cdot\bn =0\,\qon \Gamma\,,\quad\int_{\Omega}\tr \,\bsi=0\quad\qan \quad \int_{\Omega}\varphi=0.
\end{equation}
\end{subequations}

\subsection{The fully mixed formulation}\label{section22}
Before carrying out any testing, we realize from the constitutive and equilibrium equations in  \eqref{eqn:auxcomp} that the  variables $\bsi$ and $\widetilde{\bsi}$ must be, at least, square-integrable with divergence in appropriate $\mathrm{L}^r$--Lebesgue spaces. Moreover, from the second and third equations of \eqref{eq:bc-full}, it is then clear that $\bsi\in \mathbb{H}_{0}(\bdiv_r;\Omega)$ (cf. \eqref{eq:spaceH0div}) and $\widetilde{\bsi}\in\mathbf{H}_{\Gamma}(\div_r;\Omega)$ (cf. \eqref{eq:spaceHgamma}), for some $r$ to be specified below.  With this at hand, and considering the condition \eqref{eq:bc-full} on $\bu$ and $\varphi$, let us   initially search  for $\bu\in\mathbf{H}^1_{0}(\Omega)$ and $\varphi\in\widetilde{\mathrm{H}}^1(\Omega)$. Thus, after multiplying the opening equations  in the first and second rows of \eqref{eqn:auxcomp} by test functions $\bta\in \mathbb{H}_{0}(\bdiv_r;\Omega)$ and $\widetilde{\bta}\in\mathbf{H}_{\Gamma}(\div_r;\Omega)$,  and using the integration-by-parts formula \eqref{eq:ipp-1} and its tensorial version, with $ r\in [6/5,+\infty),$ we find that 
\begin{subequations}
\begin{equation}\label{eqn:1NS}
	\int_{\Omega}\bt:\bta +\int_{\Omega}\bu\cdot\bdiv\,\bta=0\quad\forall\,\bta\in \mathbb{H}_0(\div_{r};\Omega)\,,
\end{equation}
and
\begin{equation}\label{eqn:1Con}
	\int_{\Omega}\widetilde{\bt}\cdot\widetilde{\bta}+\int_{\Omega}\varphi\,\div\,\widetilde{\bta}=0\quad \forall\, \widetilde{\bta}\in \mathbf{H}_{\Gamma}(\div_{r};\,\Omega)\,,
\end{equation}
\end{subequations}
respectively, where we have used the homogeneous Dirichlet boundary condition for $\bu$ and the fact that $\widetilde{\bta}\cdot \bn =0$ on $\Gamma$. Also, note from the Cauchy-Schwarz inequality that the first terms of \eqref{eqn:1NS} and \eqref{eqn:1Con} are well-defined for $\bt\in\mathbb{L}^2_{\tr}(\Omega)$ (cf. \eqref{eqn:espL2tr} and since $\tr\,\bt=\div \,\bu =0$) and for $\widetilde{\bt}\in\mathbf{L}^2(\Omega)$. Let us then consider respective test functions $\br\in \mathbb{L}^2_{\tr}(\Omega)$ and  $\widetilde{\br}\in \mathbf{L}^{2}(\Omega)$, to weakly rewrite the constituve relations (intermediate equations from both the first and second row of \eqref{eqn:auxcomp}) as
\begin{equation}\label{eqn:2NS}
	\int_{\Omega}\bsi:\br=2\int_{\Omega}\mu(\varphi+\alpha)\bt_{sym}:\br-\dfrac{1}{2}\int_{\Omega}(\bu\otimes\bu):\br \qquad\forall\,\br\in \mathbb{L}^2_{\tr}(\Omega)\,,
\end{equation}
where we have used that $\disp \int_{\Omega} \bsi^{\tt d}:\br =\int_{\Omega} \bsi:\br$ and  $\disp \int_{\Omega} (\bu\otimes\bu)^{\tt d}:\br =\int_{\Omega}  (\bu\otimes\bu):\br$ due to $ \tr\, \br=0$,  and
\begin{equation}\label{eqn:2Con}
	\int_{\Omega}\widetilde{\bsi}\cdot\widetilde{\br}=\kappa\int_{\Omega}\widetilde{\bt}\cdot\widetilde{\br}-\dfrac{1}{2}\int_{\Omega}\varphi\bu\cdot\widetilde{\br}-U\int_\Omega\varphi\widehat{\mathbf{e}}_d\cdot\widetilde{\br}- \alpha U\int_\Omega\widehat{\mathbf{e}}_d\cdot\widetilde{\br} \quad \forall\, \widetilde{\br}\in \mathbf{L}^{2}(\Omega)\,.
\end{equation}
From the H\"{o}lder inequality, note  that for  the second terms at the right-hand side of \eqref{eqn:2NS} and  \eqref{eqn:2Con}, involving the convective terms,  to be well-defined, it is suffices to consider $\bu\in\mathbf{L}^4(\Omega)$ and  $\varphi\in\mathrm{L}_0^4(\Omega)$. Therefore,  we now take  $\bv \in \mathbf{L}^4(\Omega)$ and $\psi \in \mathrm{L}_0^{4}(\Omega)$ to test the equilibrium relations (last expressions of \eqref{eqn:auxcomp}) yielding 
\begin{subequations}
\begin{equation}\label{eqn:3NS}
	-\int_{\Omega}\bdiv\, \bsi\cdot\bv+\dfrac{1}{2}\int_{\Omega}\bt \bu\cdot\bv=\int_{\Omega}\left\{\boldsymbol{f}-g[1+\gamma(\varphi+\alpha)]\widehat{\mathbf{e}}_d\right\}\cdot\bv\quad\forall\, \bv\in\mathbf{L}^4(\Omega)\,,
\end{equation}
and
\begin{equation}\label{eqn:3Con}
	-\int_{\Omega}\div \,\widetilde{\bsi}\,\psi +\dfrac{1}{2}\int_{\Omega}\bu\cdot\widetilde{\bt}\psi=0  \quad \forall\, \psi \in \mathrm{L}^{4}(\Omega)\,.
\end{equation}
\end{subequations}
Therefore, applying the H\"{o}lder inequality, the first terms of \eqref{eqn:3NS} and  \eqref{eqn:3Con} are well-defined for $r=4/3$, aligning with the valid range for $r$ specified in \eqref{eq:ipp-1}. Consequently,  we end up finding $\bsi\in \mathbb{H}_{0}(\bdiv_{4/3};\Omega)$  and $\widetilde{\bsi}\in\mathbf{H}_{\Gamma}(\div_{4/3};\Omega)$.

 Now, to simplify the notation, and according to \eqref{eq:spacesH&Ht}, we set
\[\vec{\bu}:=(\bt,\bu)\,,\quad \vec{\bv}:=(\br,\bv)\;\in\, \mathbf{H}\,,\quad \qan \quad  \vec{\varphi}:=(\widetilde{\bt},\varphi)\,,\quad \vec{\psi}:=(\widetilde{\br},\psi)\;\in\,\widetilde{\mathbf{H}}\]

\noindent In this way, within this framework, from \eqref{eqn:2NS}+\eqref{eqn:3NS}, \eqref{eqn:1NS}, \eqref{eqn:2Con}+\eqref{eqn:3Con} and \eqref{eqn:1Con}, we arrive at the following
fully-mixed variational formulation for the generalized bioconvective flows problem: Find $(\vec{\bu},\bsi,\vec{\varphi},\widetilde{\bsi}) \in \mathbf{H}\times \mathbb{H}_0(\bdiv_{4/3};\Omega)\times \widetilde{\mathbf{H}}\times \mathbf{H}_\Gamma(\div_{4/3};\Omega)$ such that
%
\begin{equation}\label{eq:FV}
	\begin{array}{rl}
	\A_\varphi(\vec{\bu},\vec{\bv})+\C(\bu;\vec{\bu},\vec{\bv})- \B(\vec{\bv},\bsi)&=\F_\varphi(\vec{\bv})  \\[2ex]
		\B(\vec{\bu},\bta)&=0  \\[2ex]
		\widetilde{\A}(\vec{\varphi},\vec{\psi}) +\widetilde{\C}(\bu;\vec{\varphi},\vec{\psi}) - \widetilde{\B}(\vec{\psi},\widetilde{\bsi})&=\widetilde{\F}(\vec{\psi})  \\[2ex]
		\widetilde{\B}(\vec{\varphi},\widetilde{\bta})&=0,
	\end{array}
\end{equation}
for all $(\vec{\bv},\bta,\vec{\psi},\widetilde{\bta}) \in \mathbf{H}\times \mathbb{H}_0(\bdiv_{4/3};\Omega)\times\widetilde{\mathbf{H}}\times \mathbf{H}_\Gamma(\div_{4/3};\Omega)$.
Here, $\A_\phi$, for a  $\phi\in \mathrm{L}_0^4(\Omega)$ given, and $\widetilde{\A}$ are the bilinear forms
\begin{subequations}
\begin{equation}\label{A^S}
	\A_\phi(\vec{\bu},\vec{\bv})=2\int_{\Omega}\mu(\phi+\alpha)\bt_{sym}:\br \quad \forall\,\vec{\bu},\vec{\bv}\in \mathbf{H}\,,
\end{equation}

 \begin{equation}\label{A^C}
 	\widetilde{\A}(\vec{\varphi},\vec{\psi})=\kappa\int_{\Omega}\widetilde{\bt}\cdot\widetilde{\br}-U\int_{\Omega}\varphi \widehat{\mathbf{e}}_d\cdot\widetilde{\br}\quad \forall\,\vec{\varphi},\vec{\psi}\in \widetilde{\mathbf{H}}\,.
 \end{equation}
 \end{subequations}
In turn, $\B$ and $\widetilde{\B}$ are the bilinear forms  defined as 
\begin{subequations}
\begin{equation}\label{B^S}
	\B(\vec{\bv},\bta)=\int_{\Omega}\br:\bta +\int_{\Omega}\bv\cdot\bdiv\,\bta\quad  \forall\,\vec{\bv}\in \mathbf{H}\,,\quad  \forall\, \bta \in\mathbb{H}_{0}(\bdiv_{4/3};\Omega)\,,
\end{equation}
 \begin{equation}\label{B^C}
 \disp	\widetilde{\B}(\vec{\psi},\widetilde{\bta})=\int_{\Omega}\widetilde{\br}\cdot\widetilde{\bta} +\int_{\Omega}\psi\,\div\,\widetilde{\bta}  \quad \forall\,\vec{\psi}\in \widetilde{\mathbf{H}}\,,\quad \forall\,\widetilde{\bta}\in \mathbf{H}_{\Gamma}(\div_{4/3};\Omega)\,.
 \end{equation}
 \end{subequations}
 In turn, the kernels $\mathbf{V}=\mathrm{ker}(\B)$ and $\widetilde{\mathbf{V}}=\mathrm{ker}(\widetilde{\B})$ of the bilinear forms $\B$ and $\widetilde{\B},$ are given by
 \begin{subequations}
 	\begin{equation}\label{def:kernel-BS}
 	\mathbf{V}:=\left\{ \vec{\bv}\in \mathbb{L}^2_{\tr}(\Omega)\times \mathbf{L}^4(\Omega)\,: \quad  \B(\vec{\bv},\bta)=0\,\quad \forall\,\bta\in\mathbb{H}_{0}(\bdiv_{4/3};\Omega)\,\right\}\,,
 \end{equation}
	\begin{equation}\label{def:kernel-BC}
	\widetilde{\mathbf{V}}:=\left\{ \vec{\psi}\in \mathbf{L}^2(\Omega)\times \mathrm{L}^4_0(\Omega)\,: \quad  \widetilde{\B}(\vec{\psi},\widetilde{\bta})=0\,\quad \forall\,\widetilde{\bta}\in\mathbf{H}_{\Gamma}(\div_{4/3};\Omega)\,\right\}\,.
\end{equation}
 \end{subequations}
 
\noindent On the other hand, for a given $\bw\in\mathbf{L}^4(\Omega)$, the  forms $\C(\bw;\cdot,\cdot)$ and $\widetilde{\C}(\bw;\cdot,\cdot)$ associated to the convective nonlinear terms, are defined as
\begin{subequations}
\begin{equation}\label{C^S}
	\C(\bw;\vec{\bu},\vec{\bv})=\dfrac{1}{2}\left[\int_{\Omega}\bt \bw\cdot\bv-\int_{\Omega}\br \bw\cdot \bu\right] \quad  \forall\,\vec{\bu},\vec{\bv}\in \mathbf{H}\,,
\end{equation}
\begin{equation}\label{C^C}
	\widetilde{\C}(\bw;\vec{\varphi},\vec{\psi})=\dfrac{1}{2}\left[\int_{\Omega}(\widetilde{\bt}\cdot\bw)\psi -\int_{\Omega}\left(\widetilde{\br}\cdot\bw\right)\varphi\right]\quad \forall\,\vec{\varphi},\vec{\psi}\in \widetilde{\mathbf{H}}\,,
\end{equation}
\end{subequations}
 where we used that $\disp\int_{\Omega} \br\bw\cdot\bu=\int_{\Omega}\br :(\bw\otimes\bu)$ to rewrite the last term defining $\C$, coming from \eqref{eqn:2NS}.
 
\noindent Finally, $\F_\phi$ (for a given $\phi\in\mathrm{L}_0^4(\Omega)$) and $\widetilde{\F}$ are the linear functionals  defined by
\begin{subequations}
\begin{equation}\label{FS}
	\F_\phi(\vec{\bv})= \int_{\Omega}\left\{\boldsymbol{f}-g[1+\gamma(\phi+\alpha)]\widehat{\mathbf{e}}_d\right\}\cdot\bv\quad \forall\,\vec{\bv}\in \mathbf{H}\,,
\end{equation}
\begin{equation}\label{FC}
	\widetilde{\F}(\vec{\psi})= \alpha U \int_\Omega \widehat{\mathbf{e}}_d\cdot \widetilde{\br}\quad \forall\,\vec{\psi}\in \widetilde{\mathbf{H}}\,.
\end{equation}\end{subequations}
Moreover, from the application of the triangle, H\"{o}lder, and Cauchy-Schwarz inequalities, in conjunction with the norm definitions for $\mathbf{H}$ and $\widetilde{\mathbf{H}}$ (cf.  \eqref{eq:norm-u}  and \eqref{eq:norm-varphi}),  the functional $\F_\phi$ (for each $\phi\in\mathrm{L}^4_0(\Omega)$) and $\widetilde{\F}$ are bounded. Specifically, 
%
\begin{subequations}			\begin{equation}\label{eq:continuous-FS}
		|\F_\phi(\vec{\bv})| \leq \left\{ \|\boldsymbol{f}\|_{0,4/3,\Omega} + g(1+\gamma\alpha) |\Omega|^{4/3} + g\gamma |\Omega|^{1/2} \|\phi\|_{0,4,\Omega}  \right\} \|\vec{\bv}\|_\mathbf{H} \quad \forall\, \vec{\bv} \in \mathbf{H},
	\end{equation}
		and
		\begin{equation}\label{eq:continuous-FC}
			|\widetilde{\F}(\vec{\psi})| \leq \alpha U |\Omega|^{\frac{1}{2}}\|\vec{\psi}\|_{ \widetilde{\mathbf{H}}} \quad \forall \,\vec{\psi} \in \widetilde{\mathbf{H}}.
		\end{equation}
\end{subequations}	
Below, we summarize the properties of the forms involved with the model \eqref{eq:FV}. We begin with the following result regarding the bilinear forms $\B$ and  $\widetilde{\B}$. 
\begin{lem}\label{lem:properties-form-B} The forms $\B:\mathbf{H}\times\mathbb{H}_{0}(\bdiv_{4/3};\Omega) \rightarrow \mathrm{R}$ and $\widetilde{\B}:\widetilde{\mathbf{H}}\times\mathbf{H}_{\Gamma}(\div_{4/3};\Omega)\rightarrow \mathrm{R}$ defined in \eqref{B^S} and \eqref{B^C}, possess the following properties
	\begin{itemize}
		\item[(a)]  \textbf{Continuity:} $\B$ and $\widetilde{\B}$  are bounded, that is 
		\begin{equation*}
			 |\B(\vec{\bv},\bta)|\leq  \|\vec{\bv}\|_\mathbf{H}\,\|\bta\|_{\bdiv_{4/3},\Omega} \qquad \forall\, \vec{\bv}\in \mathbf{H}\,,\quad  \forall\, \bta \in\mathbb{H}_{0}(\bdiv_{4/3};\Omega)\,,
		\end{equation*}
		\begin{equation*}
			|\widetilde{\B}(\vec{\psi},\widetilde{\bta})|\leq  \|\vec{\psi}\|_{\widetilde{\mathbf{H}}}\,\|\widetilde{\bta}\|_{\div_{4/3},\Omega}\qquad \forall\,\vec{\psi}\in\widetilde{\mathbf{H}}\,,\quad \forall\,\widetilde{\bta}\in \mathbf{H}_{\Gamma}(\div_{4/3};\Omega)\,.
		\end{equation*}
		%
		\item[(b)] \textbf{Inf--sup conditions:} There exist positive constants $\beta$ and $\widetilde{\beta}$ such that 
        \begin{subequations}
		\begin{equation}\label{eq:inf-supBS}
			\sup_{\substack{\vec{\bv}\,\in\,\mathbf{H} \\ \vec{\bv} \neq \vec{\mathbf{0}}}}
			\frac{\B(\vec{\bv},\bta)}{\| \vec{\bv}\|_{\mathbf{H}}}\,\geq\,\beta\,\|\bta\|_{\bdiv_{4/3},\Omega}\,\qquad  \forall\, \bta \in\mathbb{H}_{0}(\bdiv_{4/3};\Omega)\,,
		\end{equation}
		\begin{equation}\label{eq:inf-supBC}
			\sup_{\substack{\vec{\psi}\,\in\,\widetilde{\mathbf{H}}\\ \vec{\psi} \neq \vec{0} } }
			\frac{\widetilde{\B}(\vec{\psi},\widetilde{\bta})}{\| \vec{\psi}\|_{\widetilde{\mathbf{H}}}}\,\geq\,\widetilde{\beta}\,\|\widetilde{\bta}\|_{\div_{4/3} ,\Omega}\,\qquad  \forall\,\widetilde{\bta}\in \mathbf{H}_{\Gamma}(\div_{4/3};\Omega)\,.
		\end{equation}
        \end{subequations}
		\item[(c)] The kernels $\mathbf{V}$ and $\widetilde{\mathbf{V}}$ of the forms $\B$ and $\widetilde{\B}$  (cf. \eqref{def:kernel-BS} and \eqref{def:kernel-BC}, respectively) are characterized by the relations
        \begin{subequations}
		\begin{equation}\label{kernel-BS}
 \vec{\bv}=(\br,\bv)\in \mathbf{V} \subset \mathbb{L}^2_{\tr}(\Omega)\times \mathbf{L}^4(\Omega)\qquad\Longleftrightarrow\qquad 	\bv \in \mathbf{H}^1_0(\Omega)  \qan	 \br=\nabla \bv\,,
		\end{equation}
			\begin{equation}\label{kernel-BC}
		\vec{\psi}=(\widetilde{\br},\psi)\in \widetilde{\mathbf{V}} \subset \mathbf{L}^2(\Omega)\times \mathrm{L}^4_0(\Omega)\qquad\Longleftrightarrow\qquad 	\psi \in \widetilde{\mathrm{H}}^1(\Omega)  \qan	 \widetilde{\br}=\nabla \psi\,.
	\end{equation}
    \end{subequations}
		\end{itemize}	
	\end{lem}
	\begin{proof} The continuity of the bilinear forms $\B$ and $\widetilde{\B}$ is verified by applying the triangle, H$\ddot{\mathrm{o}}$lder, and Cauchy-Schwarz inequalities, together with the norm definitions in spaces $\mathbf{H}$, $\widetilde{\mathbf{H}}$, $\mathbb{H}(\bdiv_{4/3};\Omega)$, and $\mathbf{H}(\div_{4/3};\Omega)$ (see \eqref{eq:norm-u}, \eqref{eq:norm-varphi}, and \eqref{eq:normHdiv}). The inf-sup conditions, \eqref{eq:inf-supBS} and \eqref{eq:inf-supBC}, along with the kernel properties, \eqref{kernel-BS} and \eqref{kernel-BC}, have been thoroughly proven in  \cite{hw-M2AN-2013} and \cite{CGMor}, albeit with slight variations in the specified spaces concerning the bilinear form $\widetilde{\B}$.
		
	\noindent Given that $\mathbf{H}_\Gamma(\div_{4/3},\Omega) \subset \mathbf{H}(\div_{4/3},\Omega)$, the inf-sup condition \eqref{eq:inf-supBC} is naturally satisfied due to the subspace relationship. Concerning the kernel equivalence \eqref{kernel-BC}, the condition that $\vec{\psi}=(\widetilde{\mathbf{r}},\psi) \in \widetilde{\mathbf{V}} \subset \mathbf{L}^2(\Omega) \times \mathrm{L}^4_0(\Omega)$ holds if and only if the following identity is satisfied for all $\widetilde{\bta} \in \mathbf{H}_\Gamma(\div_{4/3};\Omega)$:
		\begin{equation}\label{eq:proofkerb}
			\widetilde{\B}(\vec{\psi},\widetilde{\bta}) = \int_\Omega\widetilde{\mathbf{r}}\cdot\widetilde{\bta} + \int_{\Omega}\psi\,\div\,\widetilde{\bta} = 0.
		\end{equation}
		Choosing $\widetilde{\bta} \in \mathbf{C}^{\infty}_{0}(\Omega) \subset \mathbf{H}_\Gamma(\div_{4/3};\Omega)$ in \eqref{eq:proofkerb} implies that the term $\int_\Omega\psi\,\div\,\widetilde{\bta}$ corresponds to the action of the vectorial distribution $-\nabla\psi$ on the test function $\widetilde{\bta}$. Consequently, \eqref{eq:proofkerb} says that $\nabla\psi = \mathbf{r} \in \mathbf{L}^2(\Omega)$, leading to $\psi \in \widetilde{\mathrm{H}}^1(\Omega) = \mathrm{H}^1(\Omega) \cap \mathrm{L}^4_0(\Omega)$ (see \eqref{eqn:L2H1Spaces}). Unlike the approach in \cite{CGMor}, the validity of \eqref{eq:proofkerb} exclusively in $\mathbf{H}_\Gamma(\div_{4/3};\Omega)$ does not allow any conclusions about $\psi$ on  $\Gamma$.
	\end{proof}
	
	Next, we establish the following result concerning the bilinear forms $\A_\phi$ and $\widetilde{\A}$.

	\begin{lem}\label{lem:properties-form-A}
		The bilinear forms $\A_\phi : \mathbf{H} \times \mathbf{H} \to \mathbb{R}$ (for a given $\phi \in \mathrm{L}^4_0(\Omega)$) and $\widetilde{\A} : \widetilde{\mathbf{H}} \times \widetilde{\mathbf{H}} \to \mathbb{R}$, as defined in \eqref{A^S} and \eqref{A^C} respectively, have the following properties:
		\begin{itemize}
			\item[(a)] \textbf{Continuity:} Both $\A_\phi$ and $\widetilde{\A}$ are continuous, satisfying
			\begin{align*}
				|\A_\phi(\vec{\bu},\vec{\bv})| &\leq  \|\A\|\, \|\vec{\bu}\|_\mathbf{H}\, \|\vec{\bv}\|_\mathbf{H} \quad \forall \, \vec{\bu}, \vec{\bv} \in \mathbf{H}, \\
				|\widetilde{\A}(\vec{\varphi}, \vec{\psi})| &\leq \|\widetilde{\A}\|\, \|\vec{\varphi}\|_{\widetilde{\mathbf{H}}}\, \|\vec{\psi}\|_{\widetilde{\mathbf{H}}} \quad \forall \, \vec{\varphi}, \vec{\psi} \in \widetilde{\mathbf{H}},
			\end{align*}
			where $\|\A\| :=2\mu_2$,  with $\mu_2$ from \eqref{eqn:mu-bound}, and $\|\widetilde{\A}\| :=\sqrt{2}\max\{\kappa, U|\Omega|^{1/4}\} $.
			
			\item[(b)] \textbf{Coercivity of $\A_\phi$:} The form $\A_\phi$ is coercive on the kernel $\mathbf{V}$ of the bilinear form $\B$ (cf. \eqref{def:kernel-BS}), for any $\phi \in \mathrm{L}^4_0(\Omega)$. There exists a positive constant $\alpha_{\A} := \frac{1}{2}\mu_1 \min\{1, C_{\mathrm{FP}}^{-2}\|\bi_4\|^{-2}\}$, ensuring
			\begin{equation}\label{eq:ellipAS}
			\A_\phi(\vec{\bv},\vec{\bv}) \geq \alpha_{\A} \|\vec{\bv}\|_{ \mathbf{H}}^2 \quad \forall \, \vec{\bv} \in \mathbf{V},
			\end{equation}
			with $\mu_1$, $C_{\mathrm{FP}}$, and $\|\bi_4\|$ from \eqref{eqn:mu-bound}, \eqref{eqn:fp}, and \eqref{eq:inclusion-Lebesgue}, respectively.
			
			\item[(c)] \textbf{Coercivity of $\widetilde{\A}$:} Assuming the diffusive constant $\kappa$ and the mean velocity constant $U$ satisfy
			\begin{equation}\label{eq:rest-conc}
			\frac{U}{\kappa}|\Omega|^{1/4} < \min\{1, C_{\mathrm{FP}}^{-2}\|\mathrm{i}_4\|^{-2}\},
			\end{equation}
			the form $\widetilde{\A}$ is coercive on the kernel $\widetilde{\mathbf{V}}$ of $\widetilde{\B}$ (cf. \eqref{def:kernel-BC}). Specifically, there exists a positive constant $\alpha_{\widetilde{\A}} := \frac{\kappa}{2} \left(\min \{1, C_{\mathrm{FP}}^{-2} \|\mathrm{i}_4\|^{-2}\} - \frac{U|\Omega|^{1/4}}{\kappa}\right)$ such that
			\begin{equation}\label{eq:ellipAC}
			\widetilde{\A}(\vec{\psi},\vec{\psi}) \geq \alpha_{\widetilde{\A}}\|\vec{\psi}\|_{\widetilde{\mathbf{H}}}^2 \quad \forall\, \vec{\psi} \in \widetilde{\mathbf{V}}\,,
		\end{equation}
				with $C_{\mathrm{FP}}$ and $\|\mathrm{i}_4\|$ from \eqref{eqn:fp}, and \eqref{eq:inclusion-Lebesgue}, respectively.
		\end{itemize}
	\end{lem}
	
	\begin{proof} The continuity of the forms $\A$ and $\widetilde{\A}$ can be established through the application of  the H\"{o}lder inequality, combined with the boundedness of $\mu$ (cf. \eqref{eqn:mu-bound}), and the Cauchy-Schwarz inequality. Additionally, the definitions of norms on the spaces $\mathbf{H}$ and $\widetilde{\mathbf{H}}$, as specified in \eqref{eq:norm-u} and \eqref{eq:norm-varphi} respectively, are employed. The  $\mathbf{V}$--coercivity of $\A$  is shown in \cite[Lemma 3.2]{CGMor} by using  the characterization of $\mathbf{V}$ stated in part (c) of Lemma \eqref{lem:properties-form-B},  the lower bound on $\mu$ (refer to \eqref{eqn:mu-bound}), the Korn inequality (see \eqref{eqn:korn}), and the embedding provided by the continuous injection $\mathbf{i}_{4}:\mathbf{H}^1(\Omega) \rightarrow \mathbf{L}^4(\Omega)$ (see \eqref{eq:inclusion-Lebesgue}) along with  the Friedrichs-Poincaré inequality (cf. \eqref{eqn:fp}). 
	
	\noindent We now proceed to show the coercivity of $\widetilde{\A}$ over the kernel $\widetilde{\mathbf{V}}$ in a similar way. Given $\vec{\psi}=(\widetilde{\br},\psi)\in\widetilde{\mathbf{V}}$, in light of \eqref{kernel-BC} we have that  $\widetilde{\br}=\nabla \psi$ and $\psi \in \widetilde{\mathrm{H}}^1(\Omega)$.  Using this characterization, in combination with the H$\ddot{\mathrm{o}}$lder, Young and Friedrichs-Poincar\'e inequalities, the continuous injection $\mathrm{i}_{4}:\mathrm{H}^1(\Omega) \rightarrow \mathrm{L}^4(\Omega)$ and the norm definition on $\widetilde{\mathbf{H}}$, we obtain
	\begin{equation*}
		\begin{array}{cl}
			\widetilde{\A}(\vec{\psi},\vec{\psi})&\disp =\kappa\int_{\Omega}\widetilde{\br}\cdot\widetilde{\br}-U\int_{\Omega}\psi \widehat{\mathbf{e}}_d\cdot\widetilde{\br} 
			\geq \kappa \|\widetilde{\br}\|^2_{0,\Omega}-U\|\psi\|_{0,4,\Omega}\|\widehat{\mathbf{e}}_4\|_{0,4,\Omega}\|\widetilde{\br}\|_{0,\Omega}\\[2ex]
			&\geq \dfrac{\kappa}{2}\Big\{ \|\widetilde{\br}\|^2_{0,\Omega}+\|\widetilde{\br}\|^2_{0,\Omega}\Big\}- \dfrac{U}{2}|\Omega|^{1/4}\left\{\|\psi\|_{0,4,\Omega}^2+\|\widetilde{\br}\|_{0,\Omega}^2\right\} \\[2ex]
			&=\dfrac{\kappa}{2}\, |\psi|^2_{1,\Omega} + \dfrac{\kappa}{2}      \|\widetilde{\br}\|^2_{0,\Omega}-\dfrac{U}{2}|\Omega|^{1/4} \|\vec{\psi}\|_{\widetilde{\mathbf{H}}}^2 \\[2ex]
			&\geq\dfrac{\kappa}{2}C_{\mathrm{FP}}^{-2}\|\mathrm{i}_{4}\|^{-2}\, \|\psi\|^2_{0,4,\Omega} + \dfrac{\kappa}{2}      \|\widetilde{\br}\|^2_{0,\Omega}-\dfrac{U}{2}|\Omega|^{1/4} \|\vec{\psi}\|_{\widetilde{\mathbf{H}}}^2 \\[2ex]
			&\geq\dfrac{\kappa}{2}\min\Big\{1,C_{\mathrm{FP}}^{-2}\|\mathrm{i}_{4}\|^{-2} \Big\}\|\vec{\psi}\|^2_{\widetilde{\mathbf{H}}}-\dfrac{U}{2}|\Omega|^{1/4} \|\vec{\psi}\|_{\widetilde{\mathbf{H}}}^2 \geq \alpha_{\widetilde{\A}}\|\vec{\psi}\|_{\widetilde{\mathbf{H}}}^2\,.
		\end{array}	
	\end{equation*}
	As a result, the hypothesis \eqref{eq:rest-conc} guarantees that $\alpha_{\widetilde{\A}} := \frac{\kappa}{2} \left(\min \{1, C_{\mathrm{FP}}^{-2} \|\mathrm{i}_4\|^{-2}\} - \frac{U|\Omega|^{1/4}}{\kappa}\right)>0$ and so the $\widetilde{\mathbf{V}}$--coercivity of $\widetilde{\A}$ follows with constant $\alpha_{\widetilde{\A}}$. 
	\end{proof}
	
	\begin{rem}\label{rem:data-constraint-concentration} It is noteworthy that, similar to the analyses presented in previous works \cite{Bold-bioconvection,cao-2013,CGMir-1,Teramoto-bioconvection}, a comparable constraint to \eqref{eq:rest-conc} is considered for ensuring the well--posedness of their respective models. In our case, this condition is required for the coercivity of the bilinear form $\widetilde{\A}$, corresponding to the concentration equation. In all the cases, the restriction necessitates a sufficiently high diffusion rate $\kappa$, while requiring both the average upward swimming velocity $U$ and the physical domain $\Omega$ to remain comparatively low.
	\end{rem}	

	Regarding the forms related to the convective terms, we have the following result. 
	
	\begin{lem}\label{lem:properties-form-C-F}  For each $\bw\in\mathbf{L}^4(\Omega),$ the bilinear forms $\C(\bw:\cdot,\cdot) : \mathbf{H}\times\mathbf{H}\rightarrow \mathrm{R}$ and  $\widetilde{\C}(\bw;\cdot,\cdot):\widetilde{\mathbf{H}}\times\widetilde{\mathbf{H}} \rightarrow \mathrm{R}$ (cf. \eqref{C^S} and \eqref{C^C}, respectively), satisfy the following properties:
	\begin{itemize}
		\item[(a)] \textbf{Continuity:} Both $\C$ and $\widetilde{\C}$ are bounded. Moreover,
			\begin{align*}
			|\C(\bw;\vec{\bu},\vec{\bv})|&\leq \dfrac{1}{2}\|\bw\|_{0,4,\Omega}\, \|\vec{\bu}\|_{\mathbf{H}} \, \|\vec{\bv}\|_{\mathbf{H}}\quad\forall \, \vec{\bu},\vec{\bv}\in \mathbf{H}\\
			|\widetilde{\C}(\bw;\vec{\varphi},\vec{\psi})|&\leq 
			\dfrac{1}{2}\|\bw\|_{0,4,\Omega}\, \|\vec{\varphi}\|_{\widetilde{\mathbf{H}}}\,\|\vec{\psi}\|_{\widetilde{\mathbf{H}}}
			\quad\forall\,\vec{\varphi},\vec{\psi}\in \widetilde{\mathbf{H}}\,.
		\end{align*}
		\item[(b)] \label{lem:properties-form-C-F-Partb} \textbf{Skew-symmetry:}  the bilinear forms satisfy skew-symmetry properties
			\begin{align*}
			\C(\bw;\vec{\bu},\vec{\bv})\,&=-\C(\bw;\vec{\bv},\vec{\bu})\quad \forall\,\vec{\bu},\vec{\bv}\in \mathbf{H}\\
			\widetilde{\C}(\bw;\vec{\varphi},\vec{\psi})\,&=-\widetilde{\C}(\bw;\vec{\psi},\vec{\varphi})\quad\forall\,\vec{\varphi},\vec{\psi}\in\widetilde{\mathbf{H}}\,.
		\end{align*}
		In particular, 
		\begin{equation}\label{eq:skewproperties}
		\C(\bw;\vec{\bv},\vec{\bv})\,=0 \quad \forall\,\vec{\bv}\in \mathbf{H}\qan
		\widetilde{\C}(\bw;\vec{\psi},\vec{\psi})\,=0\quad\forall\,\vec{\psi}\in\widetilde{\mathbf{H}}\,.
		\end{equation}
\item[(c)] \textbf{Boundedness properties:} In addition,
\begin{equation*}\label{weakconti-c}
	\big| \C(\bw_1;\vec{\bu},\vec{\bv})\,-\, \C(\bw_2;\vec{\bu},\vec{\bv})\big| \,\leq\,
	\|\bw_1 - \bw_2\|_{0,4,\Omega} \, \|\vec{\bu}\|_\mathbf{H} \, \|\vec{\bv}\|_\mathbf{H} \qquad \forall \, \bw_1,\bw_2, \in \mathbf L^4(\Omega)\,,
	\quad \forall \, \vec{\bu}, \, \vec{\bv} \in \mathbf{H}\,,
\end{equation*} 
%
\begin{equation*}\label{weakconti-ct2}
	\big| \widetilde{\C}(\bw_1;\vec{\varphi},\vec{\psi})\,-\, \widetilde{\C}(\bw_2;\vec{\varphi},\vec{\psi})\big| \,\leq\,
\|\bw_1 - \bw_2\|_{0,4,\Omega} \, \|\vec{\varphi}\|_{\widetilde{\mathbf{H}}} \, \|\vec{\psi}\|_{\widetilde{\mathbf{H}}} \qquad \forall \, \bw_1,\bw_2, \in \mathbf L^4(\Omega)\,,
\quad \forall \, \vec{\bu}, \, \vec{\bv} \in \widetilde{\mathbf{H}}\,.
\end{equation*}
	\end{itemize}	
\end{lem}
\begin{proof} The continuity of the bilinear forms $\C(\bw;\cdot,\cdot)$ and $\widetilde{\C}(\bw;\cdot,\cdot)$ (for each $\bw\in\mathbf{L}^4(\Omega)$) is ensured by the application of the H\"{o}lder and Cauchy-Schwarz inequalities. The skew-symmetric nature of both $\C$ and $\widetilde{\C}$ is inherent to their definitions. Detailed demonstrations of the estimates presented in part (c) are found in \cite[Lemma 3.4]{CGMor}.
\end{proof}

\begin{rem} Similar to the works found in \cite{CN-16,hw-M2AN-2013}, our fully mixed  formulation \eqref{eq:FV} adopts a Navier-Stokes--type structure. Moreover, the skew--symmetry inherent in the forms $\C$ and $\widetilde{\C}$, related to the convective terms, simplifies the mathematical analysis considerably. As will be detailed in Section \ref{section23}, this skew--symmetry, combined with the inf--sup condition met by $\B$ and $\widetilde{\B}$, facilitates transforming the problem into one suitable for a fixed--point approach. This strategy proves crucial for deriving a priori estimates, which are essential for establishing existence and uniqueness.
\end{rem}

\subsection{Well-posedness of the continuous problem}\label{section23}
	Note that the eventual solutions $\vec{\bu}$ and $\vec{\varphi}$ to problem \eqref{eq:FV}  belong to  the kernels  $\mathbf{V}$ and $\widetilde{\mathbf{V}}$  of the forms  $\B$ and $\widetilde{\B}$ (cf. \eqref{def:kernel-BS} and \eqref{def:kernel-BC}, respectively).   From the inf-sup conditions established  in Lemma \ref{lem:properties-form-B}, it can be inferred that the problem  \eqref{eq:FV} is  equivalent to the kernel--constrained variant. The modified problem seeks $( \vec{\bu},\vec{\varphi}) \in \mathbf{V}\times\widetilde{\mathbf{V}}$ that satisfies
	\begin{equation}\label{eq:FV-ker}
		\begin{array}{rl}
			\A_\varphi(\vec{\bu},\vec{\bv})+\C(\bu;\vec{\bu},\vec{\bv})&=\F_\varphi(\vec{\bv})\,,  \\[2ex]
			\widetilde{\A}(\vec{\varphi},\vec{\psi}) 	+\widetilde{\C}(\bu;\vec{\varphi},\vec{\psi})&=\widetilde{\F}(\vec{\psi})\,,
		\end{array}
	\end{equation}
	for all $(\vec{\bv},\vec{\psi}) \in  \mathbf{V}\times\widetilde{\mathbf{V}}$.  This formulation enables us to derive the upcoming result.

	\begin{lem}\label{lem:a-priori-estimates}
		Assuming condition \eqref{eq:rest-conc} holds, any solution $(\vec{\bu}, \vec{\varphi})$ to problem \eqref{eq:FV-ker} satisfies the a priori estimates
		\begin{equation*}\label{eq:est-a-priori}
			\|\vec{\bu}\|_{\mathbf{H}}\leq C_1(\mu,\gamma,\boldsymbol{f},g,\alpha,\kappa,U,\Omega)\quad\qan\quad\|\vec{\varphi}\|_{\widetilde{\mathbf{H}}}\leq C_2 (\alpha,\kappa,U,\Omega)\,,
		\end{equation*}
		where
        \begin{subequations}
		\begin{equation}\label{eq:C1}
			C_1(\mu,\gamma,\boldsymbol{f},g,\alpha,\kappa,U,\Omega):=\alpha_{\A}^{-1} \left\{ \|\boldsymbol{f}\|_{0,4/3,\Omega} + g(1+\gamma\alpha) |\Omega|^{4/3} + g\gamma |\Omega|^{1/2} C_2 (\alpha,\kappa,U,\Omega) \right\}
		\end{equation}
		and
		\begin{equation}\label{eq:C2}
			C_2 (\alpha,\kappa,U,\Omega):=\alpha_{\widetilde{\A}}^{-1} \alpha U|\Omega|^{\frac{1}{2}}\,.
		\end{equation}\end{subequations}
	\end{lem}
	\begin{proof} Let $(\vec{\bu},\vec{\varphi})$ be a solution of problem \eqref{eq:FV-ker}. By taking $\vec{\bv}=\vec{\bu}$ and $\vec{\psi}=\vec{\varphi}$ in \eqref{eq:FV-ker} and applying the skew--symmetry property of  $\C$ y $\widetilde{\C}$  (refers to  \eqref{eq:skewproperties} in Lemma \ref{lem:properties-form-C-F-Partb}), we obtain
		\begin{equation}\label{eq:int-apri}
			\A_\varphi(\vec{\bu},\vec{\bu})=\F_\varphi(\vec{\bu})\quad\qan\quad \widetilde{\A}(\vec{\varphi},\vec{\varphi})=\widetilde{\F}(\vec{\varphi})\,.
		\end{equation}
		Using the $\widetilde{\mathbf{V}}$--coercivity of  $\widetilde{\A}$   and the continuity of   $\widetilde{\F}$ (as shown in \eqref{eq:ellipAC} and   \eqref{eq:continuous-FC}, respectively), we derive
		\[\alpha_{\widetilde{\A}}\|\vec{\varphi}\|_{\widetilde{ \mathbf{H}}}^2\leq\widetilde{\A}(\vec{\varphi},\vec{\varphi})\leq |\widetilde{\F}(\vec{\varphi})|\leq \alpha U |\Omega|^{\frac{1}{2}}\|\vec{\varphi}\|_{\widetilde{ \mathbf{H}}}\,. \]
	leading to the estimate
		\begin{equation}\label{eq:estimatevarphi}
		\|\vec{\varphi}\|_{\widetilde{ \mathbf{H}}}\leq \alpha_{\widetilde{\A}}^{-1}\alpha U |\Omega|^{\frac{1}{2}}:=  C_2 (\alpha,\kappa,U,\Omega)\,,
		\end{equation}
		where de dependence of $C_2(\cdot)$ on $\kappa$ is implicit in the coercivity constant $\alpha_{\widetilde{\A}}$ (cf. Lemma \ref{lem:properties-form-A}, part (c)). Regarding  $\vec{\bu},$  from the first equation of  \eqref{eq:int-apri}, we apply the coercivity of $\A$ on $\mathbf{V}$   and the continuity of   $\F_{\phi}$ (cf.    \eqref{eq:ellipAS} and \eqref{eq:continuous-FS}, respectively), with $\varphi$  instead of $\phi,$ to get 
		\[\alpha_{\A}\|\vec{\bu}\|_{ \mathbf{H}}^2\leq |\F_\varphi(\vec{\bu})|\leq\left\{ \|\boldsymbol{f}\|_{0,4/3,\Omega} + g(1+\gamma\alpha) |\Omega|^{4/3} + g\gamma |\Omega|^{1/2} \|\varphi\|_{0,4,\Omega}  \right\} \|\vec{\bu}\|_{ \mathbf{H}}\,.\]
		\noindent Which simplifies to  
		\[\|\vec{\bu}\|_{ \mathbf{H}}\leq\alpha_{\A}^{-1} \left\{ \|\boldsymbol{f}\|_{0,4/3,\Omega} + g(1+\gamma\alpha) |\Omega|^{4/3} + g\gamma |\Omega|^{1/2} C_2 (\alpha,U,\kappa,\Omega) \right\}:=	C_1(\mu,\gamma,\boldsymbol{f},g,\alpha,U,\kappa, \Omega)\]
		after using that $\|\varphi\|_{0,4,\Omega} \leq\|\vec{\varphi}\|_{\widetilde{\mathbf{H}}},$ (cf. \eqref{eq:norm-varphi}), and the a priori bound \eqref{eq:estimatevarphi} already obtained for $\vec{\varphi}.$ Note that the dependence of the constant $C_1(\cdot)$ on the viscosity stems from the $\mu-$dependence of $\alpha_{\A}$ as detailed in part (a) of Lemma \ref{lem:properties-form-A}.
	\end{proof}

Having established that solutions to the problem \eqref{eq:FV-ker} adhere to specific a priori estimates, our next objective is to prove the existence of solutions. This is achived by translating the kernel--reduced problem \eqref{eq:FV-ker} into a fixed--point framework, for which we employ Schauder's Fixed--Point Theorem \cite[Thm. 9.12-1(b)]{ciarlet-AF-2013}, stated as follows.

	 \begin{thm}\label{thm:Schauder}
	 	Let $B$  be a closed convex subset of a Banach space $X$ and let $\oL:B\to B$ be a continuous operator such that $\overline{\oL(B)}$ is compact. Then, $\oL$ has at least one fixed point.
	 \end{thm}

 In preparation to apply Theorem \ref{thm:Schauder}, we define a suitable  subset $\mathbf{B}$ and  an operator $\oL$. This operator $\oL$ will reformulate  the solution to the kernel-reduced problem \eqref{eq:FV-ker} as a fixed point of this operator.  In turn, thanks to  Lemma \eqref{lem:a-priori-estimates},  we specify the closed convex subset $\mathbf{B}$ of $\textbf{V}\times\widetilde{\mathbf{V}}$ by
		\begin{equation}\label{eq:B}
			\mathbf{B}=\left\{ (\vec{\bw},\vec{\phi})\in\textbf{V}\times\widetilde{\mathbf{V}}\;/\quad \|\vec{\bw}\|_{\mathbf{H}}\leq  C_1(\mu,\gamma,\boldsymbol{f},g,\alpha,\kappa,U,\Omega)\;\qan\;
			\|\vec{\phi}\|_{\widetilde{\mathbf{H}}}\leq C_2 (\alpha,\kappa,U,\Omega) \right\}\,,
		\end{equation}
		where $C_1(\mu,\gamma,\boldsymbol{f},g,\alpha,\kappa,U,\Omega)$ and $C_2 (\alpha,\kappa,U,\Omega)$ defined in \eqref{eq:C1} and \eqref{eq:C2}, respectively, come from the a priori estimates. 
		
		\noindent On the other hand, for each pair $(\vec{\bw},\vec{\phi})\in \mathbf{V}\times\widetilde{\mathbf{V}}$, we first address the uncoupled and linearized variant of the problem \eqref{eq:FV-ker}:  find  $(\vec{\bu},\vec{\varphi}) \in \mathbf{V}\times\widetilde{\mathbf{V}}$, such that

		\begin{equation}\label{eq:FV-des-lin}
			\begin{array}{rl}
				\A_\phi(\vec{\bu},\vec{\bv})+\C(\bw;\vec{\bu},\vec{\bv})&=\F_\phi(\vec{\bv})\,,  \\[2ex]
				\widetilde{\A}(\vec{\varphi},\vec{\psi}) 	+\widetilde{\C}(\bw;\vec{\varphi},\vec{\psi})&=\widetilde{\F}(\vec{\psi})\,,
			\end{array}
		\end{equation}
		for all pairs $(\vec{\bv},\vec{\psi}) \in  \mathbf{V}\times\widetilde{\mathbf{V}}$. 	Subsequently, 	we introduce the operator $\oL:\mathbf{V}\times\widetilde{\mathbf{V}}\to\mathbf{V}\times\widetilde{\mathbf{V}}$ as the solution mapping for $\eqref{eq:FV-des-lin}, $ namely 
		\begin{equation}\label{eq:L}
			\oL(\vec{\bw},\vec{\phi})=(\vec{\bu},\vec{\varphi})\quad\mbox{for a given }\quad \, (\vec{\bw},\vec{\phi})\in \mathbf{V}\times\widetilde{\mathbf{V}}
		\end{equation}
		where $(\vec{\bu},\vec{\varphi})$ is not but the solution of the uncoupled and linearized kernel--reduced problem \eqref{eq:FV-des-lin}. It then becomes clear that
		\begin{equation}\label{eq:equvalencia-L}
			(\vec{\bu},\vec{\varphi}) \text{ is  a solution of \eqref{eq:FV-des-lin} } \quad\Longleftrightarrow\quad \oL(\vec{\bu},\vec{\varphi})=(\vec{\bu},\vec{\varphi})\,.
		\end{equation}

\noindent The ensuing result confirms the well--definedness of the operator $\oL$ and its property of mapping $\mathbf{B}$ onto itself.

		\begin{lem}\label{lem:Lweldef}
	Under the assumption specified in \eqref{eq:rest-conc}, consider $\mathbf{B}$ to be the ball given  in \eqref{eq:B}. The operator $\oL:\mathbf{B} \subseteq \mathbf{V} \times \widetilde{\mathbf{V}} \rightarrow \mathbf{V} \times \widetilde{\mathbf{V}}$, as detailed through \eqref{eq:FV-des-lin}--\eqref{eq:L}, is well-defined. Furthermore, it holds that $\oL(\mathbf{B}) \subseteq \mathbf{B}$.
	
		\end{lem}
		\begin{proof} Given a pair  $(\vec{\bw},\vec{\phi})$ within the ball $\mathbf{B}\subset\mathbf{V}\times\widetilde{\mathbf{V}}$,  the linear and uncoupled nature of the equations in \eqref{eq:FV-des-lin} permits a separate analysis. On the one hand,  we focus on the problem of finding $\vec{\bu}\in\textbf{V}$ such that 
			\begin{equation}\label{eq:fluid-uncoupled}
				\A_\phi(\vec{\bu},\vec{\bv})+\C(\bw;\vec{\bu},\vec{\bv})=\F_\phi(\vec{\bv})\quad\forall\, \vec{\bv}\in \textbf{V} \,.
				\end{equation}
			\noindent Drawing on the continuity of  $\A_\phi$ and $\C$, as established in part (a) of Lemma \ref{lem:properties-form-A} and Lemma \ref{lem:properties-form-C-F}, respectively, and the relation $\|\bw\|_{0,4,\Omega}\leq\|\vec{\bw}\|_{\mathbf{H}}$ (see \eqref{eq:norm-u}), combined with the definition \eqref{eq:B} of $\mathbf{B}$, we arrive at
			\begin{equation*}
			\begin{array}{c}
				 \Big|\A_\phi(\vec{\bu},\vec{\bv})+\C(\bw;\vec{\bu},\vec{\bv})\Big|\leq \left\{2\mu_2+\dfrac{1}{2}\|\bw\|_{0,4,\Omega}\right\}\|\vec{\bu}\|_{\mathbf{H}}\,\|\vec{\bv}\|_{\mathbf{H}}\\[2ex]
		\qquad 	\qquad	 \leq \left\{2\mu_2+\dfrac{1}{2}C_1(\mu,\gamma,\boldsymbol{f},g,\alpha,\kappa,U,\Omega)\right\}\|\vec{\bu}\|_{\mathbf{H}}\|\,\vec{\bv}\|_{\mathbf{H}}\quad\forall\, \vec{\bu},\vec{\bv}\in\textbf{V}\,,
				 \end{array}
			 \end{equation*}
			confirming that the bilinear form $\A_\phi(\cdot, \cdot )+\C(\bw;\cdot,\cdot)$ is bounded. Moreover, considering   the skew--symmetry  of $\C$  and the coercivity of $\A$, as indicated in  \eqref{eq:skewproperties} and \eqref{eq:ellipAS}, respectively, we have
			\[\A_\phi(\vec{\bv},\vec{\bv})+\C(\bw;\vec{\bv},\vec{\bv}) = \A_\phi(\vec{\bv},\vec{\bv})\geq \alpha_{\A}\|\vec{\bv}\|_{\mathbf{H}}^2 \quad\forall\, \vec{\bv}\in\textbf{V}\,.\]
			The latter says that the form $\A_\phi(\cdot, \cdot )+\C(\bw;\cdot,\cdot)$ is $\mathbf{V}$-- coercive with constant $\alpha_{\A}.$ In turn,  in accordance with \eqref{eq:continuous-FS}, the fact that $\|\phi\|_{0,4,\Omega}\leq\|\vec{\phi}\|_{\widetilde{ \mathbf{H}}}$ and $\vec{\phi}$ satisfies  \eqref{eq:B}, it follows that 
			\[ \Big|\F_\phi(\vec{\bv})\Big|\leq \left\{\|\boldsymbol{f}\|_{0,4/3,\Omega}+g(1+\gamma\alpha) |\Omega|^{4/3}+ g\gamma|\Omega|^{1/2}C_2(\alpha,\kappa,U,\Omega)\right\}\|\vec{\bv}\|_{\mathbf{H}}\quad\forall\,\vec{\bv}\in\mathbf{V}\,.\]
			Then, from  the Banach–Ne\v{c}as–Babu\v{s}ka Theorem
			 (see e.g. \cite[Theorem 1.1]{Gatica-Mixtos}, \cite[Lemma 2.8]{ErnG}), there exists a unique solution $\vec{\bu}\in\mathbf{V}$ to \eqref{eq:fluid-uncoupled}  satisfying
			\begin{equation}\label{eq:dep.cont.dat.u}
				\|\vec{\bu}\|_{\mathbf{H}}\leq \alpha_{\A}^{-1} \left\{\|\boldsymbol{f}\|_{0,4/3,\Omega}+g(1+\gamma\alpha) |\Omega|^{4/3}+ g\gamma|\Omega|^{1/2}C_2(\alpha,\kappa,U,\Omega)\right\}\,.
			\end{equation}
			
			\noindent For the concentration equation, the  approach is similar to that of the fluid equation. We use  the continuity of the forms $\widetilde{\A}$ and $\widetilde{\C}$, taking into account that $\bw$ lies within $\mathbf{B}$ to get
				\[\Big|\widetilde{\A}(\vec{\varphi},\vec{\psi})+\widetilde{\C}(\bw;\vec{\varphi},\vec{\psi})\Big|\leq \left\{\sqrt{2}\max\left\{\kappa,U\,|\Omega|^{\frac{1}{4}}\right\}+\dfrac{1}{2}C_1(\mu,\gamma,\boldsymbol{f},g,\alpha,\kappa,U,\Omega)\right\} \|\vec{\varphi}\|_{\widetilde{ \mathbf{H}}}\,\|\vec{\psi}\|_{\widetilde{ \mathbf{H}}}\quad\forall\,\vec{\varphi},\vec{\psi}\in\widetilde{\mathbf{V}}\,.\]
			 Utilizing the skew--symmetry of $\widetilde{\C}$ as outlined in \eqref{eq:skewproperties}, together with the coercivity of $\widetilde{\A}$ on the kernel $\widetilde{\mathbf{V}}$, as detailed in part (c) of Lemma \ref{lem:properties-form-A} (which requires condition \eqref{eq:rest-conc}), we find that
			\[\widetilde{\A}(\vec{\psi},\vec{\psi})+ \widetilde{\C}(\bw;\vec{\psi},\vec{\psi}) =\widetilde{\A}(\vec{\psi},\vec{\psi})\geq \alpha_{\widetilde{\A}} \|\vec{\psi}\|_{\widetilde{ \mathbf{H}}}^2\quad\forall\, \vec{\psi}\in\widetilde{\mathbf{V}}\,.\]
			In this way, $\widetilde{\A}(\cdot,\cdot)+\widetilde{\C}(\bw;\cdot,\cdot)$ is continuous on $\widetilde{\mathbf{V}}\times\widetilde{\mathbf{V}}$ and $\widetilde{\mathbf{V}}$--coercive. Furthermore, the linear functional  is bounded with $\|\F\|_{\widetilde{ \mathbf{H}}'}\leq \alpha U|\Omega|^{1/2}$, according to  \eqref{FC}. As a result, applying the Banach–Ne\v{c}as–Babu\v{s}ka Theorem
			 again, we conclude that there exists a unique $\vec{\varphi}\in\widetilde{\mathbf{V}}$, such that
			\[\widetilde{\A}(\vec{\varphi},\vec{\psi}) 	+\widetilde{\C}(\bw;\vec{\varphi},\vec{\psi})=\widetilde{\F}(\vec{\psi})\quad\forall\,\vec{\psi}\in\widetilde{\mathbf{V}}\,,\]
			satisfying the continuous dependence estimate 
			\begin{equation}\label{eq:dep.cont.dat.varphi}
				\|\vec{\varphi}\|_{\widetilde{ \mathbf{H}}}\leq \alpha_{\widetilde{\A}}^{-1}\alpha U|\Omega|^{1/2}\,.
			\end{equation}

			\noindent Thus, we deduce that for any given  $(\vec{\bw},\vec{\phi})\in \mathbf{V}\times\widetilde{\mathbf{V}}$, there exists a unique corresponding  $(\vec{\bu},\vec{\varphi}) \in \mathbf{V}\times\widetilde{\mathbf{V}}$, such that
			$\oL(\vec{\bw},\vec{\phi})=(\vec{\bu},\vec{\varphi}) \,.$ Moreover, by the definition of $C_1(\cdot)$ and $C_2(\cdot)$ in \eqref{eq:C1} and \eqref{eq:C2}, along  with the  estimates provided by \eqref{eq:dep.cont.dat.u} and \eqref{eq:dep.cont.dat.varphi} and the characterization of the ball  $\mathbf{B}$ in \eqref{eq:B}, we can readly infer that $(\vec{\bu},\vec{\varphi})$ resides in $\mathbf{B}$. This establishes that $\oL(\mathbf{B})$ is a subset of $\mathbf{B}.$
		\end{proof}

	To further our analysis, and in accordance with previous works  \cite{CGMir-1,CGMor}, we introduce an additional regularity hypothesis. This hypothesis is essential for managing non-linear terms associated with the concentration dependence of the viscosity in the bilinear form $\mathbf{\A}_{\phi}$ (see estimation \eqref{eq:est-u-e3-1} below). The assumption is presented as follows.

		\begin{hip}
	Given $(\vec{\bw},\vec{\phi})\in \mathbf{B}$, we assume that $\oL_1(\vec{\bw},\vec{\phi})=\vec{\bu}=(\bt,\bu)$, satisfies that $\bu\in \mathbf{W}^{\varepsilon,4}(\Omega)$ and $\bt\in \mathbb{L}^2_{\tr}(\Omega)\cap \mathbb{H}^{\varepsilon}(\Omega)$, with $\varepsilon \in [1/2,1)$ (resp. $\varepsilon \in [3/4,1)$) when $d=2$ (resp. $d=3$). We further assume the  existence of a constant $C_\varepsilon>0$, independent of the given pair $(\bw,\phi)$,  such that
			\begin{equation}\label{eq:HRA}
				\|\bu\|_{\epsilon,4,\Omega}+\|\bt\|_{\epsilon,\Omega}\leq C_{\epsilon} \left\{ \|\boldsymbol{f}\|_{0,4/3,\Omega} + g(1+\gamma\alpha) |\Omega|^{4/3} + g\gamma |\Omega|^{1/2} \|\phi\|_{0,4,\Omega}  \right\}:=C_{1,\epsilon}\,.
			\end{equation}
		
			\end{hip}
We examine the Lipschitz continuity of the operator $\oL$, which is cornerstone for our analysis.
		
		\begin{lem}\label{lem:L-Lipshitz-continuo}
		Under the assumption specified in \eqref{eq:rest-conc} and \eqref{eq:HRA},	$\oL$ exhibits Lipschitz continuity. Specifically, there exists a constant $C_{\rm{LIP}} > 0$ (refer to \eqref{eq:C-LIP}, below) ensuring that 
			\begin{equation}\label{eq:Lipshitz-continuo}
				\|\oL(\vec{\bw},\vec{\phi})-\oL(\vec{\bw}_0,\vec{\phi}_0)\|_{\mathbf{H}\times\widetilde{\mathbf{H}}}\leq C_{\rm LIP} \|(\vec{\bw},\vec{\phi})-(\vec{\bw}_0,\vec{\phi}_0)\|_{\mathbf{H}\times\widetilde{\mathbf{H}}}\quad\forall\, (\vec{\bw},\vec{\phi}),\,(\vec{\bw}_0,\vec{\phi}_0)\in \mathbf{B}\subset\mathbf{V}\times\widetilde{\mathbf{V}}\,,
			\end{equation}
		where $C_{\rm{LIP}}$ is dependent on  data but remains independent of the pairs $(\vec{\bw},\vec{\phi})$ and $(\vec{\bw}_0,\vec{\phi}_0)$.
		\end{lem}
		\begin{proof}
			Given $(\vec{\bw},\vec{\phi})$ and $(\vec{\bw}_0,\vec{\phi}_0)$ arbitrary pairs in  the ball  $\mathbf{B}\subset\mathbf{V}\times\widetilde{\mathbf{V}}$, from Lemma \ref{lem:Lweldef}, there exist corresponding pairs  $(\vec{\bu},\vec{\varphi})=\oL(\vec{\bw},\vec{\phi})$ and $(\vec{\bu}_0,\vec{\varphi}_0)=\oL(\vec{\bw}_0,\vec{\phi}_0)$ in $\mathbf{B}\subset \mathbf{V}\times\widetilde{\mathbf{V}}$, satisfying 
				\begin{equation}\label{eq:est-in-B}
				\|\vec{\bu}\|_{\mathbf{H}}\,, 	\|\vec{\bu}_0\|_{\mathbf{H}}\leq  C_1(\mu,\gamma,\boldsymbol{f},g,\alpha,\kappa,U,\Omega)\qan
				\|\vec{\varphi}\|_{\widetilde{ \mathbf{H}}}\,,\|\vec{\varphi}_0\|_{\widetilde{ \mathbf{H}}}\leq C_2 (\alpha,\kappa,U,\Omega)\,.
			\end{equation}
Moreover, according to the definition of the operator $\oL$ (see \eqref{eq:FV-des-lin}-\eqref{eq:L}), there holds
            \begin{subequations}
			\begin{align}\label{eq:u}
				\A_\phi(\vec{\bu},\vec{\bv})+\C(\bw;\vec{\bu},\vec{\bv})=\F_\phi(\vec{\bv})\qquad \forall\, \vec{\bv}\in \mathbf{V} \,, \\\label{eq:varphi}
				\widetilde{\A}(\vec{\varphi},\vec{\psi}) 	+\widetilde{\C}(\bw;\vec{\varphi},\vec{\psi})=\widetilde{\F}(\vec{\psi})\qquad \forall\, \vec{\psi}\in \widetilde{\mathbf{V}}\,,
			\end{align}
			and
			\begin{align}\label{eq:u0}
				\mathbf{\A}_{\phi_0}(\vec{\bu}_0,\vec{\bv})+\mathbf{\C}(\bw_0;\vec{\bu}_0,\vec{\bv})=\F_{\phi_0}(\vec{\bv})\qquad \forall\, \vec{\bv}\in \mathbf{V} \,, \\\label{eq:varphi0}
				\widetilde{\A}(\vec{\varphi}_0,\vec{\psi}) 	+\widetilde{\C}(\bw_0;\vec{\varphi}_0,\vec{\psi})=\widetilde{\F}(\vec{\psi})\qquad \forall\, \vec{\psi}\in \widetilde{\mathbf{V}}\,.
			\end{align}\end{subequations}
		
%
	\noindent	Since  $\vec{\bu}-\vec{\bu}_0\in\mathbf{V}$, after applying  the $\mathbf{V}$--coercivity of $\A_\phi$ (cf. Lemma \ref{lem:properties-form-A}, part (b)) and the linearity in the first component we get
			\[
			\begin{array}{l}
				\alpha_{\A}\|\vec{\bu}-\vec{\bu}_0\|_{\mathbf{H}}^2\leq\, \A_\phi (\vec{\bu}-\vec{\bu}_0,\vec{\bu}-\vec{\bu}_0)=\A_\phi (\vec{\bu},\vec{\bu}-\vec{\bu}_0)-\A_\phi(\vec{\bu}_0,\vec{\bu}-\vec{\bu}_0) \\[2ex]
			\qquad =	\F_\phi(\vec{\bu}-\vec{\bu}_0)-\C(\bw;\vec{\bu},\vec{\bu}-\vec{\bu}_0)-\A_\phi(\vec{\bu}_0,\vec{\bu}-\vec{\bu}_0)\,,
				\end{array} 
			\]
	where we have utilized the equation \eqref{eq:u} with $\vec{\bv} = \vec{\bu} - \vec{\bu}_0$. Next,  by strategically adding and subtracting $\A_{\phi_0}(\vec{\bu}_0, \vec{\bu} - \vec{\bu}_0)$ on the right--hand side of the preceding inequality, and then using equation \eqref{eq:u0} with $\vec{\bu} - \vec{\bu}_0$ instead of $\vec{\bv}$, and  carefully grouping the terms, we deduce that
			\begin{equation}\label{eq:est-u-lip}
			\begin{array}{l}
		\alpha_{\A}	\|\vec{\bu}-\vec{\bu}_0\|_{\mathbf{H}}^2	\leq
			\Big\{\F_\phi(\vec{\bu}-\vec{\bu}_0)-\F_{\phi_0}(\vec{\bu}-\vec{\bu}_0)\Big\}\,+\,\Big\{\C(\bw_0;\vec{\bu}_0,\vec{\bu}-\vec{\bu}_0)-\C(\bw;\vec{\bu},\vec{\bu}-\vec{\bu}_0)\Big\}
				 \\[2ex]
				\qquad \qquad	\qquad \qquad +\Big\{\A_{\phi_0}(\vec{\bu}_0,\vec{\bu}-\vec{\bu}_0)-\A_\phi(\vec{\bu}_0,\vec{\bu}-\vec{\bu}_0)\Big\}\\[2ex]
		\qquad \qquad 	\qquad 	=:E_1\, +\,E_2\, +\, E_3\,.
			\end{array}
		\end{equation}

		\noindent	Next, we bound each of the expressions $E_i$, $(i=1,2,3)$.	To estimate $E_1,$  the definition of   $\F_{\phi}$ (refer to \eqref{FS}) and a straightforward application of the H\"{o}lder inequality yield
			\begin{equation}\label{eq:est-u-e1}
					E_1
= \disp \int_{\Omega}g\gamma(\phi-\phi_0)\widehat{\mathbf{e}}_d \cdot\bu-\bu_0 \leq  g\gamma|\Omega|^{1/2}\|\phi-\phi_0\|_{0,4,\Omega}\|\vec{\bu}-\vec{\bu}_0\|_{\mathbf{H}}\,.
			\end{equation}
		Concerning the expresion $E_2$,  we conveniently add and subtract the form  $\C(\bw_0;\vec{\bu},\vec{\bu}-\vec{\bu}_0)$.  Using the linearity,  the skew--symmetry property and the continuity of $\C$ (cf. \eqref{eq:skewproperties} and part (a) of Lemma \ref{lem:properties-form-C-F}, respectively), we ascertain that
		\begin{equation}\label{eq:est-u-e2}
			\begin{array}{cl}
				E_2=\C(\bw_0-\bw;\vec{\bu},\vec{\bu}-\vec{\bu}_0)\leq \dfrac{1}{2}\|\vec{\bu}\|_{\mathbf{H}}\, \|\bw-\bw_0\|_{0,4,\Omega}\, \|\vec{\bu}-\vec{\bu}_0\|_{\mathbf{H}}\,.
			\end{array}
		\end{equation}
%
	
\noindent Regarding $E_3$, after using the definitions of $\A_{\phi}$ (cf. \eqref{A^S}), the Lipschitz-continuity property assumed for $\mu$ (cf.\eqref{eqn:mu-lips}) and the H\"{o}lder inequality
			\begin{equation}\label{eq:est-u-e3-1}
				\begin{array}{cl}
					E_3&=
					\disp2\int_{\Omega}\big\{\mu(\phi_0+\alpha)-\mu(\phi+\alpha)\big\}\,\bt_{0,{sym}}:\bt-\bt_0 \\[2ex]
					&\disp\leq2L_{\mu}\int_{\Omega}|\phi-\phi_0|\,\bt_{0,{sym}}:\bt-\bt_0 \\[2ex]
					&\disp\leq2L_{\mu}\|\phi-\phi_0\|_{0,2q,\Omega}\|\bt_0\|_{0,2p,\Omega} \|\bt-\bt_0\|_{0,\Omega} \,.
				\end{array}
			\end{equation}
			Here, $p, q \in [1, +\infty)$ must satisfy $\frac{1}{2p} + \frac{1}{2q} = \frac{1}{2}$. To fulfill this condition, we can select $2p = \varepsilon^*$, where $\varepsilon^* = \frac{2}{1 - \epsilon}$ for $d = 2$, and $\varepsilon^* = \frac{6}{3 - 2\varepsilon}$ for $d = 3$. This choice ensures the validity of the continuous embedding $i_\varepsilon: \mathbb{H}^{\varepsilon}(\Omega) \to \mathbb{L}^{\varepsilon^*}(\Omega)$ (cf. \cite[Theorem 4.12]{Adams}, \cite[Corollary B.43]{ErnG}, \cite[Theorem 1.3.4]{QV}). Consequently, with $\mathbf{t}_{0} \in \mathbb{L}^{2p}(\Omega)$, we have $\|\mathbf{t}_0\|_{0,2p,\Omega} \leq \|i_{\varepsilon}\| \|\mathbf{t}_0\|_{\varepsilon,\Omega}$. This also implies that $2q = \frac{d}{\varepsilon}$, and thus the Lebesgue embedding $i: \mathrm{L}^4(\Omega) \to \mathrm{L}^{d/\varepsilon}(\Omega)$ is applicable (cf. \eqref{eq:inclusion-Lebesgue}), given the epsilon ranges specified in the regularity hypothesis  \eqref{eq:HRA}. As a result,
			\begin{equation}\label{eq:est-u-e3}
				E_3\leq2L_{\mu}\|\phi-\phi_0\|_{0,\frac{d}{\varepsilon},\Omega}\|\bt_0\|_{0,\varepsilon^*,\Omega} \|\bt-\bt_0\|_{0,\Omega} \leq2L_{\mu}\|i_\varepsilon\||\Omega|^{\frac{\varepsilon}{d}-\frac{1}{2}}\|\phi-\phi_0\|_{0,4,\Omega}\|\bt_0\|_{\varepsilon,\Omega} \|\bt-\bt_0\|_{0,\Omega}\,.
			\end{equation}
			
			
		\noindent 	Replacing back \eqref{eq:est-u-e1}, \eqref{eq:est-u-e2} and \eqref{eq:est-u-e3} in \eqref{eq:est-u-lip} we get
			\begin{equation*}
				\begin{array}{l}
					\alpha_{\A}\|\vec{\bu}-\vec{\bu}_0\|_{\mathbf{H}}^2\leq  \Big\{
					g\gamma|\Omega|^{\frac{1}{2}}\|\phi-\phi_0\|_{0,4,\Omega} +\dfrac{1}{2}\|\vec{\bu}\|_{\mathbf{H}}\, \|\bw-\bw_0\|_{0,4,\Omega} \Big\}\|\vec{\bu}-\vec{\bu}_0\|_{\mathbf{H}} \\[2ex]
			\qquad	\qquad 	\qquad\qquad+\,2L_{\mu}\|i_\varepsilon\||\Omega|^{\frac{\varepsilon}{d}-\frac{1}{2}}\|\phi-\phi_0\|_{0,4,\Omega}\|\bt_0\|_{\varepsilon,\Omega} \|\bt-\bt_0\|_{0,\Omega}\,.
				\end{array}
			\end{equation*}
			Finally, using that  $ \|\bt-\bt_0\|_{0,\Omega} \leq \|\vec{\bu}-\vec{\bu}_0\|_{\mathbf{H}} $ and simplifying terms,  
			\begin{equation}\label{eq:u-u0}
				\begin{array}{l}
					\|\vec{\bu}-\vec{\bu}_0\|_{\mathbf{H}}
					\leq \dfrac{1}{\alpha_{\A}}\left\{g\gamma|\Omega|^{\frac{1}{2}}\|\phi-\phi_0\|_{0,4,\Omega}+\dfrac{1}{2}\|\vec{\bu}\|_{\mathbf{H}}\,\|\bw-\bw_0\|_{0,4,\Omega} \right.\\[2ex] \left. \hspace{3cm}+2L_{\mu}\|i_\varepsilon\||\Omega|^{\frac{\varepsilon}{d}-\frac{1}{2}}\|\phi-\phi_0\|_{0,4,\Omega}\|\bt_0\|_{\varepsilon,\Omega} \right\}\,.
				\end{array}
			\end{equation}
				Our next step is to estimate the difference $\vec{\varphi} - \vec{\varphi}_0$. In this case, we simply  use the $\widetilde{\mathbf{V}}$--ellipticity of the form $\widetilde{\A}$ (see \eqref{eq:ellipAC}) and setting $\vec{\psi} = \vec{\varphi} - \vec{\varphi}_0 \in \widetilde{\mathbf{V}}$, along with employing the linearity property and applying \eqref{eq:varphi} and \eqref{eq:varphi0}, with $\vec{\varphi} - \vec{\varphi}_0$ in place of $\vec{\psi},$ we derive
				\begin{equation*}
				\begin{array}{cl}
				\alpha_{\widetilde{\A}}	\|\vec{\varphi}-\vec{\varphi}_0\|_{\widetilde{ \mathbf{H}}}^2&\leq  \widetilde{\A}(\vec{\varphi}-\vec{\varphi}_0,\vec{\varphi}-\vec{\varphi}_0)= \widetilde{\A}(\vec{\varphi},\vec{\varphi}-\vec{\varphi}_0)-\widetilde{\A}(\vec{\varphi}_0,\vec{\varphi}-\vec{\varphi}_0)\\[2ex]
					&=\widetilde{\F}(\vec{\varphi}-\vec{\varphi}_0)-\widetilde{\C}(\bw;\vec{\varphi},\vec{\varphi}-\vec{\varphi}_0)+\widetilde{\C}(\bw_0;\vec{\varphi}_0,\vec{\varphi}-\vec{\varphi}_0)-\widetilde{\F}(\vec{\varphi}-\vec{\varphi}_0)\\[2ex]
					&=\widetilde{\C}(\bw_0;\vec{\varphi}_0,\vec{\varphi}-\vec{\varphi}_0)-\widetilde{\C}(\bw;\vec{\varphi},\vec{\varphi}-\vec{\varphi}_0)\\[2ex]
					&=\widetilde{\C}(\bw_0-\bw;\vec{\varphi},\vec{\varphi}-\vec{\varphi}_0) \,,
				\end{array}
			\end{equation*}
		where the last line results from adding and subtracting $\widetilde{\C}(\bw_0; \vec{\varphi}, \vec{\varphi} - \vec{\varphi}_0)$ in the previous step and utilizing the skew-symmetric property of $\widetilde{\C}$ (cf. \eqref{eq:skewproperties}).
%
			Then, from the continuity of $\widetilde{\C}$ (cf. Lemma \eqref{lem:properties-form-C-F}, part (a))
			and simplifying we get
			\begin{equation}\label{eq:var-var0}
				\|\vec{\varphi}-\vec{\varphi}_0\|_{\widetilde{ \mathbf{H}}}\leq \dfrac{1}{2\alpha_{\widetilde{\A}}}\,\|\bw-\bw_0\|_{0,4,\Omega}\|\vec{\varphi}\|_{\widetilde{ \mathbf{H}}}\,.
			\end{equation}
			Next, combining  \eqref{eq:u-u0} and \eqref{eq:var-var0}, and the fact that $\|\phi-\phi_0\|_{0,4,\Omega},\|\bw-\bw_0\|_{0,4,\Omega}\leq\|(\vec{\bw}-\vec{\bw}_0,\vec{\phi}-\vec{\phi}_0)\|_{\widetilde{ \mathbf{H}}}$, we deduce that
			\begin{equation*}
				\begin{array}{l}
					\|\oL(\vec{\bw},\vec{\phi})-\oL(\vec{\bw}_0,\vec{\phi}_0)\|_{\mathbf{H}\times\widetilde{\mathbf{H}}}=\|(\vec{\bu}-\vec{\bu}_0,\vec{\varphi}-\vec{\varphi_0})\|_{\mathbf{H}\times\widetilde{\mathbf{H}}}=\|\vec{\bu}-\vec{\bu}_0\|_{\mathbf{H}}\,+\,\|\vec{\varphi}-\vec{\varphi_0}\|_{\widetilde{ \mathbf{H}}}\\[2ex]
					\leq  \left\{\dfrac{1}{\alpha_{\A}}\left[g\gamma|\Omega|^{\frac{1}{2}}+\dfrac{1}{2}\|\vec{\bu}\|_{\mathbf{H}}\,+2L_{\mu}\|i_\varepsilon\||\Omega|^{\frac{\varepsilon}{d}-\frac{1}{2}}\|\bt_0\|_{\varepsilon,\Omega} \right] + \dfrac{1}{2\alpha_{\widetilde{\A}}}\|\vec{\varphi}\|_{\widetilde{ \mathbf{H}}}   \right\}\|(\vec{\bw},\vec{\phi})-(\vec{\bw}_0,\vec{\phi}_0)\|_{\mathbf{H}\times\widetilde{ \mathbf{H}}}\,.
				\end{array}
			\end{equation*}
		Drawing upon the additional regularity hypothesis \eqref{eq:HRA}, we are able to bound the term $\|\mathbf{t}_0\|_{\varepsilon,\Omega}$. In turn, by applying the estimates \eqref{eq:est-in-B} to bound $\vec{\bu}$ and $\vec{\varphi}$ in the preceding expression, we conclude that the operator $\oL$ fulfills the Lipschitz condition \eqref{eq:Lipshitz-continuo}, with the  constant
		{\small{
			\begin{equation}\label{eq:C-LIP}
				C_{\rm LIP}:=\dfrac{1}{\alpha_{\A}}\left\{g\gamma|\Omega|^{\frac{1}{2}}+\dfrac{1}{2}C_1(\mu,\gamma,\boldsymbol{f},g,\alpha,\kappa,U,\Omega)\|\,+2L_{\mu}\|i_\varepsilon\||\Omega|^{\frac{\varepsilon}{d}-\frac{1}{2}}C_{1,\varepsilon} \right\} + \dfrac{1}{2\alpha_{\widetilde{\A}}} C_2 (\alpha,\kappa,U,\Omega)\,.
			\end{equation}}}
		\end{proof}
		
		Now, we state and prove the compactness of $\overline{\oL(B)}$.
		\begin{lem}\label{lem:L-compac}
			Under the assumption specified in \eqref{eq:rest-conc} and \eqref{eq:HRA},	$\overline{\oL(B)}$ is compact.
		\end{lem}
		\begin{proof}
			Considering \eqref{eq:B}, any sequence $\{(\vec{\bw}_n,\vec{\varphi}_n)\}_{n\geq 1}\subset \mathbf{B}\subset\mathbf{V}\times\widetilde{\mathbf{V}}$  is necessarily bounded.  Consequently, it admits a weakly convergent subsequence $\{(\vec{\bw}^{(1)}_n,\vec{\varphi}^{(1)}_n)\}_{n\geq 1}$  
			 converging to $(\vec{\bw}_0,\vec{\varphi}_0)\in\mathbf{B}$. In particular, this means that $\bw_n^{(1)}\overset{w}{\longrightarrow}\bw_0$ in $\mathbf{H}^1_{0}(\Omega)$ and $\phi_n^{(1)}\overset{w}{\longrightarrow}\phi_0$ in $\widetilde{\mathrm{H}}^1(\Omega),$ as detailed  in the characterization of  $\mathbf{V}$ and $\mathbf{\widetilde{V}}$ (see  \eqref{kernel-BS} and \eqref{kernel-BC}). 
			  \noindent Moreover, since the inclusion $i:\mathrm{H}^1(\Omega)\longrightarrow\mathrm{L}^4(\Omega)$ is compact, we ensure the strong convergence 
			  \begin{equation}\label{eq:strong-converge}
			  	\|\bw_n^{(1)}-\bw_0\|_{0,4,\Omega}\overset{n\to\infty}{\longrightarrow}\quad 0 \qan 	\|\phi_n^{(1)}-\phi_0\|_{0,4,\Omega}\quad \overset{n\to\infty}{\longrightarrow}\quad 0\,.
			  	\end{equation}
		     \noindent On the other hand,  setting
			 \[ (\vec{\bu}^{(1)}_n,\vec{\varphi}^{(1)}_n)=\oL(\vec{\bw}^{(1)}_n,\vec{\varphi}^{(1)}_n)\qan  (\vec{\bu}_0,\vec{\varphi}_0)=\oL(\vec{\bw}_0,\vec{\phi}_0)\,,\]
			 and proceeding exactly as in Lema \ref{lem:L-Lipshitz-continuo} with  $(\vec{\bw}^{(1)}_n,\vec{\varphi}^{(1)}_n)$ instead of $(\vec{\bw},\vec{\varphi}),$  we deduce  from \eqref{eq:u-u0} and \eqref{eq:var-var0} that
			  			\begin{equation*}
			  	\begin{array}{l}
			  		\|\oL(\vec{\bw}^{(1)}_n,\vec{\phi}^{(1)}_n)-\oL(\vec{\bw}_0,\vec{\phi}_0)\|_{\mathbf{H}\times\widetilde{\mathbf{H}}}
			  		\leq \dfrac{1}{2\alpha_{\widetilde{\A}}}\|\vec{\varphi}^{(1)}_{n}\|_{\widetilde{ \mathbf{H}}}\,\|\bw_n-\bw_0\|_{0,4,\Omega}\\[2ex]
			  		+\dfrac{1}{\alpha_{\A}}\left\{ \big(g\gamma|\Omega|^{\frac{1}{2}} + 2L_{\mu}\|i_\epsilon\||\Omega|^{\frac{\epsilon}{d}-\frac{1}{2}}\|\bt_0\|_{\varepsilon,\Omega}\big)\|\phi^{(1)}_n-\phi_0\|_{0,4,\Omega}+\dfrac{1}{2}\|\vec{\bu}_n^{(1)}\|_{\mathbf{H}}\, \|\bw^{(1)}_n-\bw^{(1)}_0\|_{0,4,\Omega} \right\}\,.
			  	\end{array}
			  \end{equation*}	
Note that  $\|\vec{\varphi}_{n}^{(1)}\|_{\widetilde{\mathbf{H}}}$, $\|\vec{\bu}_{n}^{(1)}\|_{\mathbf{H}}$ and  $\|\bt_0\|_{\varepsilon,\Omega}$ are bounded by data in accordance with \eqref{eq:est-u-lip} and \eqref{eq:HRA}. In virtue of \eqref{eq:strong-converge}, we then conclude that $\oL(\vec{\bw}^{(1)}_n,\vec{\phi}^{(1)}_n)\longrightarrow\oL(\vec{\bw}_0,\vec{\phi}_0)$ as $n\to+\infty,$ and so $\overline{\oL(B)}$ is compact.
%
%
%
		\end{proof}

At this point, we are in position to state the main result of this section. 
		
		\begin{thm}\label{thm:existence}
				Under the assumption specified in \eqref{eq:rest-conc} and \eqref{eq:HRA}, there exists at least one solution $(\vec{\bu},\vec{\varphi})$ for \eqref{eq:FV-des-lin}, satisfying  the a priori estimates 
					\begin{equation}\label{eq:a-priori-estimates-continuous}
					\|\vec{\bu}\|_{\mathbf{H}}\leq C_1(\mu,\gamma,\boldsymbol{f},g,\alpha,\kappa,U,\Omega)\quad\qan\quad\|\vec{\varphi}\|_{\widetilde{ \mathbf{H}}}\leq C_2 (\alpha,\kappa,U,\Omega)\,.
				\end{equation}
		\end{thm}
		\begin{proof} From Lemmas \ref{lem:Lweldef}, \ref{lem:L-Lipshitz-continuo}, and \ref{lem:L-compac}, it follows that the operator $\oL$ (cf. equations \eqref{eq:FV-des-lin}--\eqref{eq:L}) satisfies all the conditions required by the Schauder Fixed Point Theorem (cf. Theorem~\ref{thm:Schauder}), thereby guaranteeing the existence of at least one fixed point for $\oL$. Furthermore, according to the relation in equation~\eqref{eq:equvalencia-L}, this fixed point corresponds to a solution of the problem described by equation~\eqref{eq:FV-des-lin}. This solution must also fulfill the a priori estimates presented in Lemma~\ref{lem:a-priori-estimates}.			
%
		\end{proof}

		Observe that  if $(\vec{\bu},\vec{\varphi}), (\vec{\bu}_0,\vec{\varphi}_0)$ are two solutions of problem \eqref{eq:FV-des-lin}, and therefore fixed points of operator $\oL$, then from the  Lipschitz continuity (cf. Lemma \ref{lem:L-Lipshitz-continuo}), it follows that 
		\[\|(\vec{\bu},\vec{\varphi})-  (\vec{\bu}_0,\vec{\varphi}_0)\|_{\mathbf{H}\times\widetilde{\mathbf{H}}}=\|\oL(\vec{\bu},\vec{\varphi})-\oL  (\vec{\bu}_0,\vec{\varphi}_0)\|_{\mathbf{H}\times\widetilde{\mathbf{H}}}\leq C_{\rm{LIP}}\|(\vec{\bu},\vec{\varphi})-  (\vec{\bu}_0,\vec{\varphi}_0)\|_{\mathbf{H}\times\widetilde{\mathbf{H}}}\,,\]
		and so
		\[(1-C_{\rm{LIP}})\,\|(\vec{\bu},\vec{\varphi})-  (\vec{\bu}_0,\vec{\varphi}_0)\|_{\mathbf{H}\times\widetilde{\mathbf{H}}}\leq 0\,.\]
		Then $(\vec{\bu},\vec{\varphi}) =(\vec{\bu}_0,\vec{\varphi}_0)$ whenever $C_{\mathrm{LIP}}<1,$ as defined in \eqref{eq:C-LIP}.   The following uniqueness result has been then demonstrated.

		\begin{thm}\label{thm:uniqueness}
			Under the hypothesis of Theorem \ref{thm:existence}, and assuming  that the data are sufficiently small such that the Lipschitz continuity constant (cf. \eqref{eq:C-LIP}) satisfies $C_{\rm{LIP}}<1$, there exists a unique solution $(\vec{\bu},\vec{\varphi})$ of problem  \eqref{eq:FV-des-lin}.
		\end{thm}

	As we conclude this section, it is pertinent to highlight some key observations that underpin the framework of our analysis.

	\begin{rem}\label{rem:sigma-existence}
		 \begin{itemize}
		\item[(a)] The existence of the tensor $\bsi$ and the semi-advective flux vector $\widetilde{\bsi}$ follows from the inf-sup compatibility conditions satisfied by the bilinear forms $\B$ and $\widetilde{\B}$, as stated in part (b) of Lemma \ref{lem:properties-form-B}. Moreover, utilizing \eqref{eq:FV}, the continuity of the forms $\A_{\phi}(\cdot, \cdot)$, $\C(\bw;\cdot, \cdot)$, and $\F_{\phi}(\cdot)$, with $\bw = \bu$ and $\phi = \varphi$ (cf. Lemma \ref{lem:properties-form-A}, Lemma \ref{lem:properties-form-C-F}, and estimate \eqref{eq:continuous-FS}), as well as the a priori bounds \eqref{eq:a-priori-estimates-continuous}
		\begin{equation*} 
			\begin{array}{l}
	\disp	\,\|\bsi\|_{\bdiv_{4/3},\Omega} \leq \dfrac{1}{\beta}	\sup_{\substack{\vec{\bv}\,\in\,\mathbf{H} \\ \vec{\bv} \neq \vec{\mathbf{0}}}}
			\frac{\B(\vec{\bv},\bsi)}{\| \vec{\bv}\|_{\mathbf{H}}} 
			= \dfrac{1}{\beta} \sup_{\substack{\vec{\bv}\,\in\,\mathbf{H} \\ \vec{\bv} \neq \vec{\mathbf{0}}}}
			\frac{	\A_\varphi(\vec{\bu},\vec{\bv})+\C(\bu;\vec{\bu},\vec{\bv})- \F_\varphi(\vec{\bv})}{\| \vec{\bv}\|_{\mathbf{H}}}\\[2ex]
			\quad 
	\qquad 	\leq \dfrac{1}{\beta} \Big\{	\|\A_\varphi\|\|\vec{\bu}\|_{\mathbf{H}}+\dfrac{1}{2}\|\bu\|_{0,4,\Omega}^2+ \|\F_\varphi\|_{\mathbf{H}'}\Big\}	\leq C_3(\mu,\gamma,\boldsymbol{f},g,\alpha,\kappa,U,\Omega)	
			\end{array}
		\end{equation*}
which corresponds to the following a priori estimate for $\bsi$	with	
\begin{equation*}
	\begin{array}{l}
		C_3(\mu,\gamma,\boldsymbol{f},g,\alpha,\kappa,U,\Omega) := \dfrac{1}{\beta}\, \Big\{ 2\mu_2 \, C_1(\mu,\gamma,\boldsymbol{f},g,\alpha,\kappa,U,\Omega) + \dfrac{1}{2} \, C_1(\mu,\gamma,\boldsymbol{f},g,\alpha,\kappa,U,\Omega)^2 \\[2ex]
		\qquad + \|\boldsymbol{f}\|_{0,4/3,\Omega} + g(1+\gamma\alpha) |\Omega|^{4/3} + g\gamma|\Omega|^{1/2} C_2(\alpha,\kappa,U,\Omega) \Big\}\,.
	\end{array}
\end{equation*}
Similarly, we obtain the a priori bound for $\widetilde{\bsi},$
	\begin{equation*} 
	\begin{array}{l}
		\disp	\,\|\widetilde{\bsi}\|_{\div_{4/3},\Omega} \leq\dfrac{1}{\widetilde{\beta}} \Big\{	\|\widetilde{\A}\|\|\vec{\varphi}\|_{\widetilde{\mathbf{H}}}+\dfrac{1}{2}\|\bu\|_{0,4,\Omega}\|\vec{\varphi}\|_{\widetilde{\mathbf{H}}}+ \|\widetilde{\F}\|_{\widetilde{\mathbf{H}}'}\Big\} \leq C_4(\mu,\gamma,\boldsymbol{f},g,\alpha,\kappa,U,\Omega)	
	\end{array}
\end{equation*}
with
\begin{equation*}
	\begin{array}{l}
		C_4(\mu,\gamma,\boldsymbol{f},g,\alpha,\kappa,U,\Omega) := \dfrac{1}{\widetilde{\beta}}\, \Big\{ \sqrt{2}\max\{\kappa, U|\Omega|^{1/4}\}  \, C_2(\alpha,\kappa,U,\Omega) \\[2ex]
		\qquad  + \dfrac{1}{2} \, C_1(\mu,\gamma,\boldsymbol{f},g,\alpha,\kappa,U,\Omega) \,C_2(\alpha,\kappa,U,\Omega)+ \alpha U |\Omega|^{1/2} \Big\}\,.
	\end{array}
\end{equation*}

		\item[(b)] The   additional regularity   \eqref{eq:HRA} serves exclusively to bound the term $E_3$, as indicated in \eqref{eq:est-u-e3-1}, when proving the Lipschitz condition for $\oL$. This is due to the viscosity $\mu$ being concentration--dependent. However, in analyses that consider constant viscosity -- such as in
		\cite{aguiar-2017,childress-1976,harashima-1988,Teramoto-bioconvection,tb-2007} -- that assumption is unnecessary. Under these circumstances, the only prerequisite for the existence of solutions (cf. Theorem \ref{thm:existence}) is  condition \eqref{eq:rest-conc}, and the Lipschitz constant \eqref{eq:C-LIP} becomes
				\begin{equation*}
			C_{\rm LIP}:=\dfrac{1}{\alpha_{\A}}\left\{g\gamma|\Omega|^{\frac{1}{2}}+\dfrac{1}{2}C_1(\mu,\gamma,\boldsymbol{f},g,\alpha,\kappa,U,\Omega)\, \right\} + \dfrac{1}{2\alpha_{\widetilde{\A}}} C_2 (\alpha,\kappa,U,\Omega)\,.
	\end{equation*}

		\end{itemize}
		
		\end{rem}

\section{The Galerkin scheme}\label{section3}

We now describe the discretization of the variational formulation   \eqref{eq:FV}. We start in Section \ref{sec:disc-form} by introducing the finite element spaces that serve as the basis for our discrete problem formulation  along with the properties of the forms involved. Next, in Section \ref{section32}, we analyze the well--posedness of the discrete problem, applying a fixed--point strategy similar to the one used in the continuous case. In Section \ref{section33} we derive the corresponding Cea's estimate and prove optimal order a priori error estimates

	\subsection{Discretization and Finite Element Spaces}\label{sec:disc-form}
Consider a regular triangulation $\mathcal{T}_h$ over $\overline{\Omega}$, consisting of simplices $T$, specifically, triangles for $d=2$ and tetrahedra for $d=3$. We use $\mathcal{T}_h^{\rm b}$ to represent the barycentric refinement of $\mathcal{T}_h$. Denote by $h_T$ the diameter of each simplex $T$ in $\mathcal{T}^{\rm {b}}_h$, and define $h$ as the maximum diameter, $h:=\max\{h_T:T \in \mathcal{T}^b_h\}$, corresponding to the  mesh size for $\mathcal{T}^b_h$.

\noindent  For a given positive integer $\ell$, we define $\mathrm{P}_{\ell}(\mathcal{T}_h^{\rm b})$ as the set of scalar piecewise polynomial functions of degree less than or equal to $\ell$ on $\mathcal{T}_h^{\rm b}$, 
that is
\begin{equation*}
	\mathrm{P}_{\ell}(\mathcal{T}_h^{\rm b}) := \left\{ p_h:\quad p_h|_{T} \in \mathrm{P}_{\ell}(T)\quad   \forall \,T\, \in \mathcal{T}_h^{\rm b} \right\}.
\end{equation*}
%
Consistent with the notations introduced in Section \ref{section1}, we denote the spaces of vector--valued and tensor--valued polynomials on $\Thb$ by $\mathbf{P}_{\ell}(\mathcal{T}_h^{\rm b}) $ and $\mathbb{P}_{\ell}(\mathcal{T}_h^{\rm b})$, respectively. We also recall the local Raviart--Thomas  space of order $\ell$ defined as $\disp \mathbf{RT}_\ell(T):=\mathbf{P}_\ell(T)\oplus \mathrm{P}_{\tilde{\ell}}(T)\textbf{x}$, where $\textbf{x}$ is a generic vector in $\mathbf{R}$ and $\mathrm{P}_{\tilde{\ell}}(T)$ denotes the space of polynomials of degree $\ell$ on $T$
. Consequently, the global Raviart--Thomas space of order $\ell$ is characterized by
\[
	\mathbb{RT}_{\ell}(\Thb):= \left\{\bta_h \in \mathbb{H}(\bdiv;\Omega)\,:\quad\mathbf{c}^t\bta_h|_T\in \mathbf{RT}_\ell(T)\,,\quad \forall \, \mathbf{c}\in \mathbf{R}\,,\quad \,\forall\, T\in\mathcal{T}_h^b \right\}\,.
	\]

\noindent The finite element spaces  for approximating the unknowns $\bt$, $\bu$, $\bsi$, $\widetilde{\bt}$, $\varphi$, and $\widetilde{\bsi}$ of problem \eqref{eq:FV} are then given as
\begin{equation}\label{eq:FEspaces}
\begin{aligned}
	\mathbb{H}^\bt_h &:= \mathbb{L}^2_{\text{tr}}(\Omega) \cap \mathbb{P}_{\ell}(\mathcal{T}_h^{\rm b}), & 
	\mathbf{H}^\bu_h &:= \mathbf{L}^4(\Omega) \cap \mathbf{P}_{\ell}(\mathcal{T}_h^{\rm b}), & 
	\mathbb{H}^\bsi_h &:= \mathbb{H}_0(\bdiv_{4/3}; \Omega) \cap \mathbb{RT}_\ell(\mathcal{T}_h^{\rm b}), \\
	\mathbf{H}^{\widetilde{\bt}}_h &:= \mathbf{L}^2(\Omega) \cap \mathbf{P}_{\ell}(\mathcal{T}_h^{\rm b}), & 
	\mathrm{H}^\varphi_h &:= \mathrm{L}_0^4(\Omega) \cap \mathrm{P}_{\ell}(\mathcal{T}_h^{\rm b}), & 
	\mathbf{H}^{\widetilde{\bsi}}_h &:= \mathbf{H}_\Gamma(\mathrm{div}_{4/3}; \Omega) \cap \mathbf{RT}_\ell(\mathcal{T}_h^{\rm b}).
\end{aligned}
\end{equation}

	Following the approach used in the continuous case, we simplify the notation by setting
	\[\vec{\bu}_h:=(\bt_h,\bu_h)\,,\quad \vec{\bv}_h:=(\br_h,\bv_h)\,\in \mathbf{H}_h:=\mathbf{H}^\bu_h\times\mathbb{H}^\bt_h\,,  \] 
	\[\vec{\varphi}_h:=(\widetilde{\bt}_h,\varphi_h)\,,\quad\vec{\psi}_h:=(\widetilde{\br}_h,\psi_h)\,\in \widetilde{\mathbf{H}}_h:=\mathrm{H}^\varphi_h\times\mathbf{H}^{\widetilde{\bt}}_h\,.\]
	In turn, for each $\bt_h\in\mathbb{H}^\bt_h$ we identify $\bt_{h,sym}$ and $\bt_{h,skw}$ as the symmetric and skew--symmetric parts, respectively.  The Galerkin scheme associated with problem \eqref{eq:FV} seeks to find $(\vec{\bu}_h, \bsi_h, \vec{\varphi}_h, \widetilde{\bsi}_h) \in \mathbf{H}_h \times \mathbb{H}^\bsi_h \times \widetilde{\mathbf{H}}_h \times \mathbf{H}^{\widetilde{\bsi}}_h$ satisfying 
	\begin{equation}\label{eq:FV_h}
		\begin{array}{rl}
			\A_{\varphi_h}(\vec{\bu}_h,\vec{\bv}_h)+\C(\bu_h;\vec{\bu}_h,\vec{\bv}_h)- \B(\vec{\bv}_h,\bsi_h)&=\F_{\varphi_h}(\vec{\bv}_h)  \\[2ex]
			\B(\vec{\bu}_h,\bta_h)&=0  \\[2ex]
			\widetilde{\A}(\vec{\varphi}_h,\vec{\psi}_h) +\widetilde{\C}(\bu_h;\vec{\varphi}_h,\vec{\psi}_h) - \widetilde{\B}(\vec{\psi}_h,\widetilde{\bsi}_h)&=\widetilde{\F}(\vec{\psi}_h)  \\[2ex]
			\widetilde{\B}(\vec{\varphi}_h,\widetilde{\bta}_h)&=0,
		\end{array}
	\end{equation}
	for all $(\vec{\bv}_h,\bta_h,\vec{\psi}_h,\widetilde{\bta}_h) \in \mathbf{H}_h\times \mathbb{H}^\bsi_h\times \widetilde{\mathbf{H}}_h\times\mathbf{H}^{\widetilde{\bsi}}_h$. Here,  $\A_{\phi_{h}}(\cdot,\cdot),\C(\bw_h;\cdot,\cdot):\mathbf{H}_h\times\mathbf{H}_h,\rightarrow \mathrm{R}$ (with $\phi_h$ and $\bw_h$ in place of $\phi$ and $\bw$, respectively), $\widetilde{\A}(\cdot,\cdot),\widetilde{\C}(\bw_h;\cdot,\cdot):\widetilde{\mathbf{H}}_h\times\widetilde{\mathbf{H}}_h\rightarrow\mathrm{R},$ (with $\bw_h$ in place of $\bw$), $\B:\mathbf{H}_h\times\mathbb{H}_h^\bsi\rightarrow\mathrm{R}$ and $\widetilde{\B}:\widetilde{\mathbf{H}}_h\times\mathbf{H}_h^{\widetilde{\bsi}}\rightarrow\mathrm{R}$  are the bilinear forms defined in \eqref{A^S}-\eqref{C^C} constrained to operate within the respective finite--dimensional spaces.
	
	\noindent In turn, $\F_{\phi_h}$ (with $\phi_h$ instead of $\phi$) and $\widetilde{\F}$ are the linear functionals defined in \eqref{FS} and \eqref{FC}, respectively, and satisfy
    \begin{subequations}
				\begin{equation}\label{eq:continuous-FS-h}
		\big|\F_{\phi_h}(\vec{\bv_h})\big| \leq \left\{ \|\boldsymbol{f}\|_{0,4/3,\Omega} + g(1+\gamma\alpha) |\Omega|^{4/3} + g\gamma |\Omega|^{1/2} \|\phi_h\|_{0,4,\Omega}  \right\} \|\vec{\bv}_h\|_{\mathbf{H}} \quad \forall\, \vec{\bv}_h \in \mathbf{H}_h,
	\end{equation}
	and
	\begin{equation}\label{eq:continuous-FC-h}
		\big|\widetilde{\F}(\vec{\psi}_h)\big| \leq \alpha U |\Omega|^{\frac{1}{2}}\|\vec{\psi}_h\|_{\widetilde{\mathbf{H}}} \quad \forall \,\vec{\psi}_h \in \widetilde{\mathbf{H}}_h.
	\end{equation}\end{subequations}

In the following, we outline the properties of the forms at the discrete level, starting with $\B$ and $\widetilde{\B}$. We emphasize that the extensive development and analysis of the finite element set $(\mathbf{H}_h, \mathbb{H}_h^{\bsi})$ are detailed in \cite{hw-M2AN-2013} for a dual-mixed formulation of the Navier--Stokes equations. That work states that if the discrete spaces are constructed on meshes with a macroelement structure (such as $\Thb$) and if the polynomial degree $\ell$ meets the condition $\ell \geq d-1$, then these spaces are inf--sup compatible and satisfy a discrete Korn's inequality (cf. \eqref{eq:inf-supBS_h} and \eqref{eq:ineq-t}). These conditions are vital for ensuring the well-posedness of the discrete problem, particularly with fluid equations. For similar properties of $\widetilde{\B}$ (refer to equations \eqref{eq:inf-supBC_h} and \eqref{eq:ineq-psi}), please see \cite{CGMor}. Consequently, we omit the proofs here.

\begin{lem}\label{lem:properties-form-B-h} For $\ell \geq d-1,$ the forms $\B:\mathbf{H}_h\times\mathbb{H}_h^{\bsi} \rightarrow \mathrm{R}$ and $\widetilde{\B}:\widetilde{\mathbf{H}}_h\times\mathbf{H}_h^{\widetilde{\bsi}}\rightarrow \mathrm{R}$ defined in \eqref{B^S} and \eqref{B^C}, exhibit the following properties.
	\begin{itemize}
		\item[(a)]  \textbf{Continuity:} $\B$ and $\widetilde{\B}$  are bounded, that is 
		\begin{equation*}\label{eq:continuous-bilinear-BS-h}
			|\B(\vec{\bv}_h,\bta_h)|\leq  \|\vec{\bv}_h\|_{\mathbf{H}}\,\|\bta_h\|_{\bdiv_{4/3},\Omega} \qquad \forall\, \vec{\bv}_h\in \mathbf{H}_h\,,\quad  \forall\, \bta_h\in\mathbb{H}_{h}^{\bsi}\,,
		\end{equation*}
		\begin{equation*}\label{eq:continuous-bilinear-BC-h}
			|\widetilde{\B}(\vec{\psi}_h,\widetilde{\bta}_h)|\leq  \|\vec{\psi}_h\|_{\widetilde{\mathbf{H}}}\,\|\widetilde{\bta}_h\|_{\div_{4/3},\Omega}\qquad \forall\,\vec{\psi}_h\in\widetilde{\mathbf{H}}_h\,,\quad \forall\,\widetilde{\bta}_h\in \mathbf{H}_{h}^{\widetilde{\bsi}}\,.
		\end{equation*}
		%
		\item[(b)] \textbf{Discrete inf--sup conditions:} There exist positive constants $\beta_\mathtt{d}$ and $\widetilde{\beta}_\mathtt{d}$, independent of $h$, such that
        \begin{subequations}
		\begin{equation}\label{eq:inf-supBS_h}
			\sup_{\substack{\vec{\bv}_h\,\in\,\mathbf{H}_h \\ \vec{\bv}_h \neq \vec{\mathbf{0}}}}
			\frac{\B(\vec{\bv}_h,\bta_h)}{\| \vec{\bv}_h\|_{\mathbf{H}}}\,\geq\,\beta_\mathtt{d}\,\|\bta_h\|_{\bdiv_{4/3};\Omega}\,\qquad  \forall\, \bta_h \in\mathbb{H}^\bsi_h\,,
		\end{equation}
		\begin{equation}\label{eq:inf-supBC_h}
	\sup_{\substack{\vec{\psi}_h\,\in\,\widetilde{\mathbf{H}}_h\\ \vec{\psi}_h \neq \vec{0} } }
			\frac{\widetilde{\B}(\vec{\psi}_h,\widetilde{\bta}_h)}{\| \vec{\psi}_h\|_{\widetilde{\mathbf{H}}}}\, \geq\,\widetilde{\beta}_\mathtt{d}\,\|\widetilde{\bta}_h\|_{\div_{4/3};\Omega}\,\qquad  \forall\,\widetilde{\bta}_h\in \mathbf{H}^{\widetilde{\bsi}}_h\,.
		\end{equation}\end{subequations}
		\item[(c)] There exist  positive constants $C_\mathtt{d}$ and $\widetilde{C}_\mathtt{d}$, independents of $h$, such that
        \begin{subequations}
		\begin{equation}\label{eq:ineq-t}
			\| \br_{h,sym}\|_{0,\Omega}\geq C_\mathtt{d} \|(\br_{h,skw},\bv_h)\|\,,\qquad\forall\, \vec{\bv}_h=(\br_h,\bv_h)\,\in \mathbf{V}_h\,,
		\end{equation}
		and 
		\begin{equation}\label{eq:ineq-psi}
			\|\widetilde{\br}_h\|_{0,\Omega}\geq \widetilde{C}_\mathtt{d}\,\|\psi_h\|_{0,4;\Omega}\,,\qquad\forall\, \vec{\psi}_h=(\widetilde{\br}_h,\psi_h)\in \widetilde{\mathbf{V}}_h\,,
		\end{equation} \end{subequations}
		where $\mathbf{V}_h$ and $\widetilde{\mathbf{V}}_h$ are the discrete kernels of the forms $\B$ and $\widetilde{\B}$, that is,
        \begin{subequations}
		\begin{equation}\label{kernel-BS_h}
			\mathbf{V}_h:=\left\{ \vec{\bv}_h\in  \mathbf{H}_h\quad : \quad \int_{\Omega}\br_h:\bta_h +\int_{\Omega}\bv_h\cdot\bdiv\,\bta_h=0 \quad, \,\forall\, \bta_h \in\mathbb{H}^\bsi_h\right\}\,,
		\end{equation}
		and
		\begin{equation}\label{kernel-BC_h}
			\widetilde{\mathbf{V}}_h:=\left\{ \vec{\psi}_h\in \widetilde{\mathbf{H}}_h\quad : \quad \int_{\Omega}\widetilde{\br}_h\cdot\widetilde{\bta}_h +\int_{\Omega}\psi_h\,\div\,\widetilde{\bta}_h=0 \quad, \, \forall\,\widetilde{\bta}_h\in \mathbf{H}^{\widetilde{\bsi}}_h\right\}\,.
		\end{equation}\end{subequations}
	\end{itemize}	
\end{lem}
   \begin{proof}
   	See \cite[Lemma 4.1]{hw-M2AN-2013} \cite[Section 5]{CGMor}
   \end{proof}

   The discrete version of the Lemma \eqref{lem:properties-form-A} concerning the bilinear forms $\A$ and $\widetilde{\A}$ is presented as follows
   
   \begin{lem}\label{lem:properties-form-A-h}
   	The bilinear forms $\A_{\phi_h} : \mathbf{H}_h \times \mathbf{H}_h \to \mathbb{R}$ (for a given $\phi_h \in \mathrm{H}_h^{\varphi}$) and $\widetilde{\A} : \widetilde{\mathbf{H}}_h \times \widetilde{\mathbf{H}}_h \to \mathbb{R}$, as defined in \eqref{A^S} and \eqref{A^C} respectively, have the following properties:
   	\begin{itemize}
   		\item[(a)] \textbf{Continuity:} Both $\A_{\phi_h}$ and $\widetilde{\A}$ are continuous, satisfying
   		\begin{align*}
   			|\A_{\phi_h}(\vec{\bu}_h,\vec{\bv}_h)| &\leq \|\A\|\, \|\vec{\bu}_h\|_{\mathbf{H}}\, \|\vec{\bv}_h\|_{\mathbf{H}} \quad \forall \, \vec{\bu}_h, \vec{\bv}_h \in \mathbf{H}_h, \\
   			|\widetilde{\A}(\vec{\varphi}_h, \vec{\psi}_h)| &\leq \|\widetilde{\A}\|\, \|\vec{\varphi}_h\|_{\widetilde{\mathbf{H}}}\, \|\vec{\psi}_h\|_{\widetilde{\mathbf{H}}} \quad \forall \, \vec{\varphi}_h, \vec{\psi}_h \in \widetilde{\mathbf{H}}_h,
   		\end{align*}
   		with the same constants $\|\A\|$ and $\|\widetilde{\A}\|$ from Lemma \ref{lem:properties-form-A}, part (a).
   		
   		\item[(b)] \textbf{Coercivity of $\A_{\phi_h}$:} The form $\A_{\phi_{h}}$ is coercive on the kernel $\mathbf{V}_h$ of the bilinear form $\B$ (cf. \eqref{kernel-BS_h}), for any $\phi_h \in \mathrm{H}_h^{\varphi}$. That is, there exists a positive constant $\alpha_{\A}^\ast:=\mu_1\, \min\left\{1,C_\mathtt{d}^2\right\}$, independent of $h$, such that
   		\begin{equation}\label{eq:ellipAS-h}
   			\A_{\phi_h}(\vec{\bv}_h,\vec{\bv}_h) \geq \alpha_{\A}^\ast \|\vec{\bv}_h\|_{\mathbf{H}}^2 \quad \forall \, \vec{\bv}_h \in \mathbf{V}_h\,,
   		\end{equation}
   		where the constant $C_{\rm d}$ comes from \eqref{eq:ineq-t}.
   		\item[(c)] \textbf{Coercivity of $\widetilde{\A}$:} Assume the diffusive constant $\kappa$ and the mean velocity constant $U$ satisfy
   		\begin{equation}\label{eq:rest-conc-h}
   			\frac{U}{\kappa}|\Omega|^{1/4} < \min\{1,\widetilde{C}_{\rm{d}}^2\}
   		\end{equation}
   		where the constant $\widetilde{C}_{\rm d}$ comes from \eqref{eq:ineq-psi}. Then, the form $\widetilde{\A}$ is coercive on the kernel $\widetilde{\mathbf{V}}_h$ of $\widetilde{\B}$ (cf. \eqref{kernel-BC_h}). Specifically, there exists a positive constant $\alpha_{\widetilde{\A}}^\ast := \frac{\kappa}{2}\min\left(1-\frac{U}{\kappa},\widetilde{C}_{\rm d}^2-\frac{U}{\kappa}|\Omega|^{1/4}\right)$, independent of $h$, such that
   		\begin{equation}\label{eq:ellipAC-h}
   			\widetilde{\A}(\vec{\psi}_h,\vec{\psi}_h) \geq \alpha_{\widetilde{\A}}^\ast\|\vec{\psi}_h\|_{\widetilde{\mathbf{H}}}^2 \quad\, \forall\, \vec{\psi}_h\in \widetilde{\mathbf{V}}_h\,.
   		\end{equation}
   		%
   	\end{itemize}
   \end{lem}

	\begin{proof}
	The boundedness of the forms $\A_{\phi_h}$ and $\widetilde{\A}$ is a direct consequence of the inclusions $\mathbf{H}_h \subset \mathbf{H}$ and $\widetilde{\mathbf{H}}_h \subset \widetilde{\mathbf{H}}$, respectively. The coercivity of the form $\A_{\phi_h}$ with respect to $\mathbf{V}_h$ has been established in \cite{hw-M2AN-2013} and is further discussed in \cite[Lemma 4.1]{CGMor}. As for the $\widetilde{\mathbf{V}}_h$-coercivity of the form $\widetilde{\A}$,  it is suffices to utilize the definition of $\widetilde{\A}$ along with the application of H\"{o}lder's and Young's inequalities. Then,  the property \eqref{eq:ineq-psi} and the norm defined on $\widetilde{\mathbf{H}}_h$ lead to the desired result, as follows
		\begin{equation*}
		\begin{array}{cl}
			\widetilde{\A}(\vec{\psi}_h,\vec{\psi}_h)&
			\geq \kappa \|\widetilde{\br}_h\|^2_{0,\Omega}-U\|\psi_h\|_{0,4,\Omega}\|\widehat{\mathbf{e}}_d\|_{0,4,\Omega}\|\widetilde{\br}_h\|_{0,\Omega}\\[2ex]
			&\geq \dfrac{\kappa}{2}\Big\{ \|\widetilde{\br}_h\|^2_{0,\Omega}+\|\widetilde{\br}_h\|^2_{0,\Omega}\Big\}- \dfrac{U}{2}|\Omega|^{1/4}\left\{\|\psi_h\|_{0,4,\Omega}^2+\|\widetilde{\br}_h\|_{0,\Omega}^2\right\} \\[2ex]
			&\geq\dfrac{\kappa}{2}\, \|\widetilde{\br}_h\|^2_{0,\Omega} + \dfrac{\kappa}{2}  \widetilde{C}_{\rm d}^2    \|\psi_h\|^2_{0,4,\Omega}-\dfrac{U}{2}|\Omega|^{1/4} \|\vec{\psi}_h\|_{\widetilde{\mathbf{H}}}^2 \\[2ex]
			&\geq\dfrac{\kappa}{2}\min\Big\{1,\widetilde{C}_{\mathrm{d}}^{2} \Big\}\|\vec{\psi}_h\|_{\widetilde{\mathbf{H}}}^2-\dfrac{U}{2}|\Omega|^{1/4} \|\vec{\psi}_h\|_{\widetilde{\mathbf{H}}}^2 \geq \alpha_{\widetilde{\A}}^\ast\|\vec{\psi}_h\|_{\widetilde{\mathbf{H}}}^2\,.
		\end{array}	
	\end{equation*}
		where the constant $\alpha_{\widetilde{\A}}^\ast$ is clearly positive, thanks to the assumption \eqref{eq:rest-conc-h}, and independent of $h$.
	\end{proof}

\noindent The following properties directly follows from the definitions of $\C_h$ and $\widetilde{\C}_h$, paralleling the proof presented for their continuous counterparts in Lemma~\ref{lem:properties-form-C-F}, adjusted for the discrete spaces.  
	
\begin{lem}\label{lem:discrete-properties-form-C-F-h}
For each $\bw_h \in \mathbf{H}_h^{\bu}$, the bilinear forms $\C(\bw_h;\cdot,\cdot) : \mathbf{H}_h \times \mathbf{H}_h \rightarrow \mathrm{R}$ and $\widetilde{\C}(\bw_h;\cdot,\cdot) : \widetilde{\mathbf{H}}_h \times \widetilde{\mathbf{H}}_h \rightarrow \mathrm{R}$  are endowed with the properties of continuity, skew-symmetry, and boundedness as  in Lemma \ref{lem:properties-form-C-F}. In particular, 
	\begin{equation}\label{eq:skewproperties-h}
	\C(\bw_h;\vec{\bv}_h,\vec{\bv}_h)\,=0 \quad \forall\,\vec{\bv}_h\in \mathbf{H}_h\qan
	\widetilde{\C}(\bw_h;\vec{\psi}_h,\vec{\psi}_h)\,=0\quad\forall\,\vec{\psi}_h\in\widetilde{\mathbf{H}}_h\,.
\end{equation}

\end{lem}

	\begin{rem} We conclude this section by highlightinhg that   our fully mixed finite element formulation \eqref{eq:FV_h} naturally endows the forms $\C$ and $\widetilde{\C}$, corresponding to the convective terms, with inherent skew--symmetry at discrete level. Consequently, the customary requisite for post--discretization mo\-di\-fi\-ca\-tions to maintain such mathematical properties is unnecesarry in our framework. This inherent skew--symmetry, a direct advantage of our formulation mathematical structure, plays a pivotal role in preserving the conservation of  energy and numerical stability. 
\end{rem}

	\subsection{Well-posedness  of the discrete problem}\label{section32} 
	
	Following the approach from Section \ref{section23}, we find that the problem \eqref{eq:FV_h} is equivalent to a problem reduced to the kernel of $\B_h$ and $\widetilde{\B}_h$, as defined in \eqref{kernel-BS_h} and \eqref{kernel-BC_h}, respectively. The task is to find $(\vec{\bu}_h, \vec{\varphi}_h) \in \mathbf{V}_h \times \widetilde{\mathbf{V}}_h$ such that
	\begin{equation}\label{eq:FV-ker_h}
		\begin{array}{rl}
			\A_{\varphi_h}(\vec{\bu}_h,\vec{\bv}_h)+\C(\bu_h;\vec{\bu}_h,\vec{\bv}_h)&=\F_{\varphi_h}(\vec{\bv}_h)\,,  \\[2ex]
			\widetilde{\A}(\vec{\varphi}_h,\vec{\psi}_h) 	+\widetilde{\C}(\bu_h;\vec{\varphi}_h,\vec{\psi}_h)&=\widetilde{\F}(\vec{\psi}_h)\,,
		\end{array}
	\end{equation}
	for all $(\vec{\bv}_h,\vec{\psi}_h) \in  \mathbf{V}_h\times\widetilde{\mathbf{V}}_h$.

The discrete counterpart of Lemma \ref{lem:a-priori-estimates} is presented below.
\begin{lem}\label{lem:a-priori-estimates-h}
	Assuming the data satisfy \eqref{eq:rest-conc-h}, any solution $(\vec{\bu}_h, \vec{\varphi}_h)$ to problem \eqref{eq:FV-ker_h} satisfies the following a priori estimates
	\begin{equation}\label{eq:est-a-priori_h}
		\|\vec{\bu}_h\|_{\mathbf{H}} \leq C_1^\ast(\mu, \gamma, \boldsymbol{f}, g, \alpha, \kappa, U, \Omega) \quad \text{and} \quad \|\vec{\varphi}_h\|_{\widetilde{\mathbf{H}}} \leq C_2^\ast(\alpha, \kappa, U, \Omega),
	\end{equation}
	where
\begin{subequations}
    \begin{align}\label{eq:C1*}
		C_1^\ast(\mu, \gamma, \boldsymbol{f}, g, \alpha, \kappa, U, \Omega) &:= \frac{1}{\alpha_{\A}^\ast} \left\{ \|\boldsymbol{f}\|_{0,4/3,\Omega} + g(1+\gamma\alpha) |\Omega|^{4/3} + g\gamma |\Omega|^{1/2} C_2^\ast(\alpha, \kappa, U, \Omega) \right\}\,,\\
	\label{eq:C2*}
		C_2^\ast(\alpha, \kappa, U, \Omega) & := (\alpha^\ast_{\widetilde{\A}})^{-1}\alpha U|\Omega|^{\frac{1}{2}}.
	\end{align}\end{subequations}
\end{lem}
	\begin{proof}
		Proceeding similarly to the a priori estimates for the continuous problem, let $(\vec{\bu}_h, \vec{\varphi}_h)$ be a solution to problem \eqref{eq:FV-ker_h}. Taking $\vec{\bv}_h = \vec{\bu}_h$ and $\vec{\psi}_h = \vec{\varphi}_h$, and utilizing the skew--symmetry property of the forms $\C$ and $\widetilde{\C}$ (refer to \eqref{eq:skewproperties-h}), we find that
		\begin{equation}\label{eq:int-apri_h}
			\A_{\varphi_h}(\vec{\bu}_h,\vec{\bu}_h)=\F_{\varphi_h}(\vec{\bu}_h)\quad\qan\quad \widetilde{\A}(\vec{\varphi}_h,\vec{\varphi}_h)=\widetilde{\F}(\vec{\varphi}_h)\,.
		\end{equation}
		In particular, we use the $\widetilde{\mathbf{V}}_h$--coercivity of $\widetilde{\A}$ and the continuity of   $\widetilde{\F}$ (cf. part (c) of Lemma \ref{lem:properties-form-A-h}  and  \eqref{eq:continuous-FC-h},  respectively), to get 
		\[\alpha_{\widetilde{\A}}^\ast\|\vec{\varphi}_h\|_{\widetilde{\mathbf{H}}}^2\leq\widetilde{\A}(\vec{\varphi}_h,\vec{\varphi}_h)\leq \big|\widetilde{\F}(\vec{\varphi}_h)\big|\leq \alpha U |\Omega|^{\frac{1}{2}}\|\vec{\varphi}_h\|_{\widetilde{\mathbf{H}}}\,. \]
	Thus, after simplification, we readily obtain the a priori bound \eqref{eq:est-a-priori_h} for $\vec{\varphi}_h$ with the constant $C_2^\ast (\alpha, \kappa, U, \Omega)$ as defined in \eqref{eq:C2*}.  Similarly, from the first equation in \eqref{eq:int-apri_h}, we derive the corresponding a priori bound for $\vec{\bu}_h$. By using the coercivity of $\A_{\varphi_h}$ (refer to Lemma \ref{lem:properties-form-A-h}, part (b)) and the continuity bound of the functional $\F_{\varphi_h}$ (see \eqref{eq:continuous-FS-h}), we deduce that
	
		\[\alpha_{\A}^\ast\|\vec{\bu}_h\|_{\mathbf{H}}^2\leq |\F_{\varphi_h}(\vec{\bu}_h)|\leq \left\{ \|\boldsymbol{f}\|_{0,4/3,\Omega} + g(1+\gamma\alpha) |\Omega|^{4/3} + g\gamma |\Omega|^{1/2} \|\varphi_h\|_{0,4,\Omega}  \right\} \|\vec{\bu}_h\|_{\mathbf{H}}\,.\]
This immediately leads to the desired result upon simplifying, given that $\|\varphi_h\|_{0,4,\Omega} \leq \|\vec{\varphi}_h\|_{\widetilde{\mathbf{H}}}$, and considering the estimate previously derived for $\vec{\varphi}_h$.
\end{proof}

The next step involves transforming \eqref{eq:FV-ker_h} into a fixed--point problem. Following a methodology inspired by the continuous case, we first address a linearized and decoupled version of the problem. Thus, given $(\vec{\bw}_h, \vec{\phi}_h) \in \mathbf{V}_h \times \widetilde{\mathbf{V}}_h$, we seek $(\vec{\bu}_h, \vec{\varphi}_h) \in \mathbf{V}_h \times \widetilde{\mathbf{V}}_h$ that satisfies:
\begin{equation}\label{eq:FV-des-lin_h}
	\begin{array}{rl}
		\A_{\phi_h}(\vec{\bu}_h, \vec{\bv}_h) + \C(\bw_h; \vec{\bu}_h, \vec{\bv}_h) &= \F_{\phi_h}(\vec{\bv}_h), \\[2ex]
		\widetilde{\A}(\vec{\varphi}_h, \vec{\psi}_h) + \widetilde{\C}(\bw_h; \vec{\varphi}_h, \vec{\psi}_h) &= \widetilde{\F}(\vec{\psi}_h),
	\end{array}
\end{equation}
for all $(\vec{\bv}_h, \vec{\psi}_h) \in \mathbf{V}_h \times \widetilde{\mathbf{V}}_h$.

\noindent With the help of the a priori estimates for discrete solutions derived in Lemma \ref{lem:a-priori-estimates-h}, we define the closed convex subset $\mathbf{B}_h$ of $\mathbf{V}_h \times \widetilde{\mathbf{V}}_h$ given by
\begin{equation}\label{eq:B_h}
	\mathbf{B}_h = \left\{ (\vec{\bw}_h, \vec{\phi}_h) \in \mathbf{V}_h \times \widetilde{\mathbf{V}}_h : \|\vec{\bw}_h\|_{\mathbf{H}} \leq C_1^\ast(\mu, \gamma, \boldsymbol{f}, g, \alpha, \kappa, U, \Omega) \qan \|\vec{\phi}_h\|_{\widetilde{\mathbf{H}}} \leq C_2^\ast(\alpha,\kappa,U,\Omega) \right\},
\end{equation}
where $C_1^\ast(\mu, \gamma, \boldsymbol{f}, g, \alpha, \kappa, U, \Omega) $ and $C_2^\ast(\alpha,\kappa,U,\Omega) $ are as defined in \eqref{eq:C1*} and \eqref{eq:C2*}, respectively.

We then introduce the operator $\oL_h: \mathbf{B}_h \to \mathbf{V}_h \times \widetilde{\mathbf{V}}_h$, defined by
\begin{equation}\label{eq:L_h}
	\oL_h(\vec{\bw}_h, \vec{\phi}_h) = (\vec{\bu}_h, \vec{\varphi}_h) \quad \forall (\vec{\bw}_h, \vec{\phi}_h) \in \mathbf{B}_h,
\end{equation}
where $(\vec{\bu}_h, \vec{\varphi}_h)$ is the solution to the problem \eqref{eq:FV-des-lin_h}. It is evident that any solution of \eqref{eq:FV-ker_h} corresponds to a fixed--point of the operator $\oL_h$, i.e.,
\begin{equation}\label{eq:equivalencia-L_h}
	(\vec{\bu}_h, \vec{\varphi}_h) \text{ solves \eqref{eq:FV-ker_h} } \quad \Longleftrightarrow \quad \oL_h(\vec{\bu}_h, \vec{\varphi}_h) = (\vec{\bu}_h, \vec{\varphi}_h).
\end{equation}

Certainly, the viability of this approach hinges on the wel--defined nature of $\oL_h$. This is  addressed in the ensuing discussion.

	\begin{lem}\label{lem:Lweldef-h}
	Under the assumption specified in \eqref{eq:rest-conc-h}, consider $\mathbf{B}_h$ to be the ball given  in \eqref{eq:B_h}. The operator $\oL_h:\mathbf{B}_h \rightarrow \mathbf{V}_h \times \widetilde{\mathbf{V}}_h$, as detailed through \eqref{eq:FV-des-lin_h}--\eqref{eq:L_h}, is well-defined. Furthermore, it holds that $\oL_h(\mathbf{B}_h) \subseteq \mathbf{B}_h$.
\end{lem}
	\begin{proof}
		Adapting the proof from Lemma \ref{lem:Lweldef} to the discrete setting of \eqref{eq:FV-des-lin_h}, consider any pair $(\vec{\bw}_h, \vec{\phi}_h) \in \mathbf{B}_h$. We begin with the fluid problem, seeking $\vec{\bu}_h \in \mathbf{V}_h$ that satisfies
		\begin{equation}\label{eq:FV-des-lin_h-1}
			\A_{\phi_h}(\vec{\bu}_h, \vec{\bv}_h) + \C(\bw_h; \vec{\bu}_h, \vec{\bv}_h) = \F_{\phi_h}(\vec{\bv}_h) \quad \forall \vec{\bv}_h \in \mathbf{V}_h.
		\end{equation}
		Given the continuity of $\A_{\phi_h}$ and $\C$, and considering the coercivity of $\A$ and the skew--symmetry of $\C$ (cf. \eqref{eq:ellipAS-h} and \eqref{eq:skewproperties-h}), the form $\A_{\phi_h}(\cdot, \cdot) + \C(\bw_h; \cdot, \cdot)$ is shown to be uniformly coercive on $\mathbf{V}_h$, independent of $(\phi_h, \bw_h)$. Moreover, with $\F_{\phi_h} \in \mathbf{V}_h^\prime$, we find
		\[
		\big\|\F_{\phi_h}\big\|_{\mathbf{V}_h^\prime} \leq \|\boldsymbol{f}\|_{0,4/3,\Omega} + g(1+\gamma\alpha) |\Omega|^{4/3} + g\gamma |\Omega|^{1/2} C_2^\ast(\alpha,\kappa,U,\Omega),
		\]
		where we have used that $\|\phi_h\|_{0,4,\Omega} \leq C_2^\ast(\alpha,\kappa,U,\Omega)$ due to $(\vec{\bw}_h, \vec{\phi}_h) \in \mathbf{B}_h$. The Banach–Ne\v{c}as–Babu\v{s}ka Theorem
	 assures the existence and uniqueness of $\vec{\bu}_h$ solving \eqref{eq:FV-des-lin_h-1}, with
		\begin{equation}\label{eq:dep.cont.dat.u_h}
			\|\vec{\bu}_h\|_{\mathbf{H}} \leq \frac{1}{\alpha_{\A}^\ast} \left\{ \|\boldsymbol{f}\|_{0,4/3,\Omega} + g(1+\gamma\alpha) |\Omega|^{4/3} + g\gamma |\Omega|^{1/2} C_2^\ast(\alpha, \kappa, U, \Omega) \right\}.
		\end{equation}

	\noindent 	For the concentration equation, the problem of finding $\vec{\varphi}_h \in \widetilde{\mathbf{V}}_h$ that satisfies
		\begin{equation}\label{eq:FV-des-lin_h-2}
			\widetilde{\A}(\vec{\varphi}_h, \vec{\psi}_h) + \widetilde{\C}(\bw_h; \vec{\varphi}_h, \vec{\psi}_h) = \widetilde{\F}(\vec{\psi}_h) \quad \forall \vec{\psi}_h \in \widetilde{\mathbf{V}}_h,
		\end{equation}
		follows a similar approach. The continuity and coercivity of $\widetilde{\A}$ (cf. \eqref{eq:ellipAC-h}), along with the skew-symmetry of $\widetilde{\C}$ (cf. \eqref{eq:skewproperties-h}), and the  boundedness of $\widetilde{\F}$, allow to deduce that
		\[\widetilde{\A}(\vec{\psi}_h,\vec{\psi}_h)+ \widetilde{\C}(\bw_h;\vec{\psi}_h,\vec{\psi}_h) \geq \alpha_{\widetilde{\A}}^\ast \|\vec{\psi}_h\|_{\widetilde{\mathbf{H}}}^2\quad\forall\, \vec{\psi}_h\in\widetilde{\mathbf{V}}_h\,,\qan \big\|\widetilde{\F}\big\|_{\widetilde{\mathbf{V}}^{\prime}_h}\leq  \alpha U|\Omega|^{1/2}\,.\]
		\noindent Again, Banach–Ne\v{c}as–Babu\v{s}ka Theorem
		 gives the existence and uniqueness of $\vec{\varphi}_h$ solving \eqref{eq:FV-des-lin_h-2}, satisfying
		\begin{equation}\label{eq:dep.cont.dat.varphi_h}
			\|\vec{\varphi}_h\|_{\widetilde{ \mathbf{H}}} \leq (\alpha^\ast_{\widetilde{\A}})^{-1}\alpha U|\Omega|^{1/2}\,.
		\end{equation}

		\noindent Consequently, $\oL_h$ is well--defined. Note further from \eqref{eq:dep.cont.dat.u_h} and \eqref{eq:dep.cont.dat.varphi_h}, the definition of $C_1^\ast(\cdot)$ and $C_2^\ast(\ast)$ in \eqref{eq:C1*} and \eqref{eq:C2*}, and the definition of $\mathbf{B}_h$ in \eqref{eq:B_h} that $(\vec{\bu}_h,\vec{\varphi}_h)\in\mathbf{B}_h$ and therefore $\oL_h(\mathbf{B}_h) \subseteq \mathbf{B}_h$.
	\end{proof}

We now turn to the Lipschitz continuity of $\oL_h$. We caution in advance that due to the characteristics of finite element spaces, applying a continuous regularity hypothesis such as \eqref{eq:HRA} directly to the discrete context is impractical, as highlighted in references \cite{CGMir-1,CGMor}. The main reasons are that finite element spaces inherently restrict the level of regularity that can be achieved, and the discretization process introduces mesh--size--dependent estimates. Therefore, our analysis of $\oL_h$ moves forward without relying on this type of regularity assumption.

\begin{lem}\label{lem:L-Lipshitz-continuo-h}
Under the hypotheses from Lemma \ref{lem:properties-form-B-h} and the condition  \eqref{eq:rest-conc-h},	$\oL_h$ exhibits Lipschitz continuity. Specifically, there exists a constant $C^{\ast}_{\rm{LIP}} > 0$ (refer to \eqref{eq:C-LIP-h}, below) ensuring that 
	\begin{equation}\label{eq:Lipshitz-continuo-h}
		\|\oL_h(\vec{\bw}_h,\vec{\phi}_h)-\oL_h(\vec{\bw}_{0,h},\vec{\phi}_{0,h})\|_{\mathbf{H}\times\widetilde{ \mathbf{H}}}\leq C^\ast_{\rm LIP} \|(\vec{\bw}_h,\vec{\phi}_h)-(\vec{\bw}_{0,h},\vec{\phi}_{0,h})\|_{\mathbf{H}\times\widetilde{ \mathbf{H}}}
	\end{equation}
	for all $(\vec{\bw}_h,\vec{\phi}_h),\,(\vec{\bw}_{0,h},\vec{\phi}_{0,h})\in \mathbf{B}_h. $
\end{lem}
\begin{proof}
	We adapt the proof of Lemma \ref{lem:L-Lipshitz-continuo-h} and consider $(\vec{\bw}_h,\vec{\phi}_h)$ and $(\vec{\bw}_{0,h},\vec{\phi}_{0,h})$ arbitrary pairs in  the ball  $\mathbf{B}_h$, we denote by   $(\vec{\bu}_h,\vec{\varphi}_h)=\oL_h(\vec{\bw}_h,\vec{\phi}_h)$ and $(\vec{\bu}_{0,h},\vec{\varphi}_{0,h})=\oL_h(\vec{\bw}_{0,h},\vec{\phi}_{0,h})$ in $\mathbf{B}_h$, satisfying 
	\begin{equation}\label{eq:est-in-B-h}
		\|\vec{\bu}_h\|_{\mathbf{H}}\,, 	\|\vec{\bu}_{0,h}\|_{\mathbf{H}}\leq  C^{\ast}_1(\mu,\gamma,\boldsymbol{f},g,\alpha,\kappa,U,\Omega)\qan
		\|\vec{\varphi}_h\|_{\widetilde{ \mathbf{H}}}\,,\|\vec{\varphi}_{0,h}\|_{\widetilde{ \mathbf{H}}}\leq C^{\ast}_2 (\alpha,\kappa,U,\Omega)\,.
	\end{equation}
	From the definition of the operator $\oL_h$ (see \eqref{eq:FV-des-lin_h}-\eqref{eq:L_h}), it follows that
	\begin{align*}
	\A_{\phi_h}(\vec{\bu}_h,\vec{\bv}_h)+\C(\bw_h;\vec{\bu}_h,\vec{\bv}_h)=\F_{\phi_h}(\vec{\bv}_h)\qquad \forall\, \vec{\bv}_h\in \textbf{V}_h \,, \\
	\widetilde{\A}(\vec{\varphi}_h,\vec{\psi}_h)+\widetilde{\C}(\bw_h;\vec{\varphi}_h,\vec{\psi}_h)=\widetilde{\F}(\vec{\psi}_h)\qquad \forall\, \vec{\psi}_h\in \widetilde{\mathbf{V}}_h\,,
\end{align*}
and
\begin{align*}
	\A_{\phi_{0,h}}(\vec{\bu}_{0,h},\vec{\bv}_{0,h})+\C(\bw_{0,h};\vec{\bu}_{0,h},\vec{\bv}_{0,h})=\F_{\phi_{0,h}}(\vec{\bv}_{0,h})\qquad \forall\, \vec{\bv}_{0,h}\in \textbf{V}_h \,, \\
	\widetilde{\A}(\vec{\varphi}_{0,h},\vec{\psi}_{0,h})+\widetilde{\C}(\bw_{0,h};\vec{\varphi}_{0,h},\vec{\psi}_{0,h})=\widetilde{\F}(\vec{\psi}_{0,h})\qquad \forall\, \vec{\psi}_{0,h}\in \widetilde{\mathbf{V}}_h\,.
\end{align*}

	\noindent	The analogous estimatation to the continuous one in \eqref{eq:est-u-lip} becomes

	\begin{equation}\label{eq:est-u-lip-h}
		\begin{array}{l}
			\alpha^{\ast}_{\A}	\|\vec{\bu}_h-\vec{\bu}_{0,h}\|_{\mathbf{H}}^2	\leq
			\Big\{\F_\phi(\vec{\bu}-\vec{\bu}_0)-\F_{\phi_0}(\vec{\bu}-\vec{\bu}_{0,h})\Big\}\,\\[2ex]
			\qquad +\,\Big\{\C(\bw_{0,h};\vec{\bu}_{0,h},\vec{\bu}-\vec{\bu}_{0,h})-\C(\bw_h;\vec{\bu}_h,\vec{\bu}_h-\vec{\bu}_{0,h})\Big\}
			\\[2ex]
		\qquad +\Big\{\A_{\phi_{0,h}}(\vec{\bu}_{0,h},\vec{\bu}-\vec{\bu}_{0,h})-\A_\phi(\vec{\bu}_{0,h},\vec{\bu}-\vec{\bu}_{0,h})\Big\}\\[2ex]
		\qquad 	=:E_1^{\ast}\, +\,E_2^{\ast}\, +\, E_3^{\ast}\,.
		\end{array}
	\end{equation}

	\noindent The expresions $E_1^{\ast}$ and $E_2^{\ast}$ can be estimated straightforwardly from the respective counterparts $E_1$ and $E_2$ (see \eqref{eq:est-u-e1} and \eqref{eq:est-u-e2}), respectively, which leads to 
	\begin{equation}\label{eq:est-u-e1-h}
			E_1^{\ast}
			 \leq  g\gamma|\Omega|^{1/2}\|\phi_h-\phi_{0,h}\|_{0,4,\Omega}\|\vec{\bu}_h-\vec{\bu}_{0,h}\|_{\mathbf{H}}\,.
	\end{equation}
and
	\begin{equation}\label{eq:est-u-e2-h}
		\begin{array}{l}
				E_2^{\ast}\leq \dfrac{1}{2}\|\vec{\bu}_h\|_{\mathbf{H}}\, \|\bw_h-\bw_{0,h}\|_{0,4,\Omega}\, \|\vec{\bu}_h-\vec{\bu}_{0,h}\|_{\mathbf{H}},.
		\end{array}
	\end{equation}
	In turn, as anticipated,  for bounding $E_3^\ast$ we note that   regularity assumption such as \eqref{eq:HRA} is not available in the present setting.  Therefore,  we will utilize a $\mathrm{L}^4-\mathbb{L}^4-\mathbb{L}^2$ argument based on the H$\ddot{\mathrm{o}}$lder inequality in the estimation  \eqref{eq:est-u-e3-1}, adapted to the discrete setting, to obtain that
	\begin{equation}\label{eq:est-u-e3-h}
	E_3^\ast\,\leq\,2L_{\mu}\|\phi_h-\phi_{0,h}\|_{0,4,\Omega}\|\bt_{0,h}\|_{0,4,\Omega} \|\bt_h-\bt_{0,h}\|_{0,\Omega} \,.
	\end{equation}
		\noindent 	Replacing back \eqref{eq:est-u-e1-h}, \eqref{eq:est-u-e2-h} and \eqref{eq:est-u-e3-h} in \eqref{eq:est-u-lip-h}, and symplifying, we get
	\begin{equation}\label{eq:u-u0_h}
		\begin{array}{l}
			\|\vec{\bu}_h-\vec{\bu}_{0,h}\|
			\leq \dfrac{1}{\alpha_{\A}^\ast}\left\{\big(g\gamma|\Omega|^{\frac{1}{2}}+2L_{\mu}\|\bt_{0,h}\|_{0,4,\Omega}\big)\|\phi_h-\phi_{0,h}\|_{0,4,\Omega} +\dfrac{1}{2}\|\vec{\bu}_h\|_{\mathbf{H}}\, \|\bw_{0,h}-\bw_h\|_{0,4,\Omega}\right\}.
		\end{array}
	\end{equation}
To estimate the difference $\vec{\varphi}_h - \vec{\varphi}_{0,h}$, we follow the same procedure to get \eqref{eq:var-var0} and easily deduce that
	\begin{equation}\label{eq:var-var0-h}
		\|\vec{\varphi}_h-\vec{\varphi}_{0,h}\|_{\widetilde{ \mathbf{H}}}\leq \dfrac{1}{2\alpha^{\ast}_{\widetilde{\A}}}\,\|\bw_h-\bw_{0,h}\|_{0,4,\Omega}\|\vec{\varphi}_h\|_{\widetilde{ \mathbf{H}}}\,.
	\end{equation}
	Next, combining  \eqref{eq:u-u0_h} and \eqref{eq:var-var0-h}, and the fact that $\|\phi_h-\phi_{0,h}\|_{0,4,\Omega}$ and $\|\bw_h-\bw_{0,h}\|_{0,4,\Omega}$ are both bounded by $\|(\vec{\bw}_h,\vec{\phi}_h)-(\vec{\bw}_{0,h},\vec{\phi}_{0,h})\|$, we deduce that
	\begin{equation*} 
		\begin{array}{l}
			\|\oL_h(\vec{\bw}_h,\vec{\phi}_h)-\oL_h(\vec{\bw}_{0,h},\vec{\phi}_{0,h})\|_{\mathbf{H}\times\widetilde{ \mathbf{H}}}
			\\[2ex]
			\leq  \left\{\dfrac{1}{\alpha_{\A}^\ast}\left[\big(g\gamma|\Omega|^{\frac{1}{2}}+2L_{\mu}\|\bt_{0,h}\|_{0,4,\Omega}\big) +\dfrac{1}{2}\|\vec{\bu}_h\|_{\mathbf{H}}\,\right] + \dfrac{1}{2\alpha^\ast_{\widetilde{\A}}}\|\vec{\varphi}_h\|_{\widetilde{ \mathbf{H}}}   \right\}\|(\vec{\bw}_h,\vec{\phi}_h)-(\vec{\bw}_{0,h},\vec{\phi}_{0,h})\|_{\mathbf{H}\times\widetilde{ \mathbf{H}}}\,.
		\end{array}
	\end{equation*}
	Finally, by applying the estimates \eqref{eq:est-in-B-h} to bound $\vec{\bu}_h$ and $\vec{\varphi}_h$ in the preceding expression, we conclude that the operator $\oL_h$ satisfies the Lipschitz condition \eqref{eq:Lipshitz-continuo-h}, with the  constant

			\begin{equation}\label{eq:C-LIP-h}
				C^{\ast}_{\rm LIP}:=\dfrac{1}{\alpha_{\A}^\ast}\left\{\big(g\gamma|\Omega|^{\frac{1}{2}}+2L_{\mu}\|\bt_{0,h}\|_{0,4,\Omega}\big) +\dfrac{1}{2}C^\ast_1(\mu,\gamma,\boldsymbol{f},g,\alpha,\kappa,U,\Omega)\,\right\} + \dfrac{1}{2\alpha^\ast_{\widetilde{\A}}} C^\ast_2 (\alpha,\kappa,U,\Omega)\,.
	\end{equation}
\end{proof}

It is important to highlight a few key points here.

\begin{rem}\label{rem:Clip-h}
	\begin{itemize}
	\item[(a)] The determination of the constant $C_{\mathrm{LIP}}^\ast$ is influenced by the term $\bt_{0,h}$, which re\-pre\-sents the first component of the pair $(\bt_{0,h}, \bu_{0,h}) = \vec{\bu}_{0,h} = \oL_{h,1}(\vec{\bw}_{0,h}, \vec{\phi}_{0,h})$ within $\mathbf{B}_h$. Similar to the findings in previous studies \cite{CGMir-1,CGMor}, given that elements of $\mathbb{H}_h^{\bt}$ are  piecewise polynomial by components, we can affirm that $\|\bt_{0,h}\|_{0,4,\Omega}$ is finite. However, we cannot assert that this finiteness is independent of the discretization parameter $h$.
	\item[(b)] Significantly, in scenarios where the viscosity is constant as analyzed in 
		\cite{aguiar-2017,childress-1976,harashima-1988,Teramoto-bioconvection,tb-2007}, the term $E^\ast_3$ (see  \eqref{eq:est-u-lip-h}), does not appear. This means that, with constant viscosity, the Lipschitz continuity constant for the discrete operator $\oL_h$ depends only on given data and is independent of the mesh size $h$. More precisely, the constant $C^{\ast}_{\mathrm{LIP}}$ is defined as
	\begin{equation}\label{eq:C-LIP-h-nuconstant}
		C^{\ast}_{\mathrm{LIP}} := \frac{1}{\alpha_{\A}^\ast}\left\{g\gamma|\Omega|^{\frac{1}{2}} +\frac{1}{2}C^\ast_1(\mu,\gamma,\boldsymbol{f},g,\alpha,\kappa,U,\Omega)\right\} + \frac{1}{2\alpha^\ast_{\widetilde{\A}}} C^\ast_2 (\alpha,\kappa,U,\Omega).
	\end{equation}
	
	\end{itemize}
\end{rem}

We close the section with the main result establishing the well--posedness of the discrete problem \eqref{eq:FV-ker_h}. 

	\begin{thm}\label{thm:existence-h}
	Under the hypotheses from Lemma \ref{lem:properties-form-B-h} and the condition  \eqref{eq:rest-conc-h}, there exists at least one solution $(\vec{\bu}_h,\vec{\varphi}_h)$ for \eqref{eq:FV-ker_h}, satisfying  the a priori estimates 
	\begin{equation*}
		\|\vec{\bu}_h\|_{\mathbf{H}}\leq C^\ast_1(\mu,\gamma,\boldsymbol{f},g,\alpha,\kappa,U,\Omega)\quad\qan\quad\|\vec{\varphi}_h\|_{\widetilde{ \mathbf{H}}}\leq C^\ast_2 (\alpha,\kappa,U,\Omega)\,.
	\end{equation*}
\end{thm}
\begin{proof} From Lemmas \ref{lem:Lweldef-h} and \ref{lem:L-Lipshitz-continuo-h}, the operator $\oL_h$ (refer to equations \eqref{eq:FV-des-lin_h}--\eqref{eq:L_h}) satisfies the Brouwer Fixed Point Theorem criteria, ensuring at least one fixed point for $\oL_h$. Moreover, from \eqref{eq:equivalencia-L_h}, this fixed point aligns with a solution for the problem outlined in equation \eqref{eq:FV-ker_h}, which also adheres to the a priori estimates from Lemma \ref{lem:a-priori-estimates-h}.		
\end{proof}

%
%

As we conclude this section, it is pertinent to highlight some key observations that underpin the framework of our analysis.

\begin{rem}
	\begin{itemize}
		\item[(a)] Similar to the continuous scenario, the existence of   $\bsi_h$ and  $\widetilde{\bsi}_h$ as well as the corresponding bounds follow from the inf--sup compatibility of $\B$ and $\widetilde{\B}$, as outlined in part (b) of Lemma \ref{lem:properties-form-B-h} (see Remark \ref{rem:sigma-existence}, part (a)).
		\item[(b)] Due to the Lipschitz continuity constant \eqref{eq:C-LIP-h}'s dependence on $\bt_{0,h}$, establishing a uniqueness result for the discrete problem \eqref{eq:FV_h} is not straightforward. Nonetheless, as mentioned in Remark \eqref{rem:Clip-h}, for   constant viscosity the Lipschitz constant \eqref{eq:C-LIP-h-nuconstant} of $\oL_h$ only depends on given data. Thus, the result \eqref{thm:existence-h} can be improved to assure both the existence and uniqueness of the discrete solution, from the Banach Fixed-Point Theorem, by requiring $\oL_h$ to be a contraction. This implies the data must be sufficiently small so that the Lipschitz constant \eqref{eq:C-LIP-h-nuconstant} ensures $C^{\ast}_{\rm LIP}<1.$
	\end{itemize}
	
\end{rem}

\subsection{A priori error analysis}\label{section33}
The objective of this section is to estimate the approximation error associated with the Galerkin scheme presented by the fully-mixed finite element method \eqref{eq:FV_h}, utilizing the discrete spaces specified in \eqref{eq:FEspaces} and under the conditions given in Theorems \ref{thm:existence}, \ref{thm:uniqueness}, and \ref{thm:existence-h}. We aim to derive theoretical convergence rates in terms of the discretization parameter $h$, providing a error estimate of the form
\begin{equation*}
	\|(\vec{\bu}, \bsi) - (\vec{\bu}_h, \bsi_h)\|_{\mathbf{H} \times \mathbb{H}_0(\bdiv_{4/3}; \Omega)} + \|(\vec{\varphi}, \widetilde{\bsi}) - (\vec{\varphi}_h, \widetilde{\bsi}_h)\|_{\widetilde{\mathbf{H}} \times \mathbf{H}_\Gamma(\div_{4/3}; \Omega)} \leq C \, h^s, 
\end{equation*}
where $(\vec{\bu}, \bsi, \vec{\varphi}, \widetilde{\bsi}) \in \mathbf{H} \times \mathbb{H}_0(\bdiv_{4/3}; \Omega) \times \widetilde{\mathbf{H}} \times \mathbf{H}_\Gamma(\div_{4/3}; \Omega)$ is the unique solution of the coupled problem \eqref{eq:FV}, $(\vec{\bu}_h, \bsi_h, \vec{\varphi}_h, \widetilde{\bsi}_h) \in \mathbf{H}_h \times \mathbb{H}^\bsi_h \times \widetilde{\mathbf{H}}_h \times \mathbf{H}^{\widetilde{\bsi}}_h$ is a solution of the discrete coupled problem \eqref{eq:FV_h},  $C$ is a positive constant independent of $h$, and $s$ denotes the theoretical convergence rate.

	 According to the structure inherent in the respective variational formulations, we employ \cite[Lemma 6.1]{CGMor}. This lemma provides a Strang--type estimate for a problem with a similar structure of ours. The lemma is presented as follows. 
	\begin{lem}\label{strang}
		Let $\mathrm H$ and $\mathrm Q$ be reflexive Banach spaces,  and let $a:\mathrm H \times \mathrm H \longrightarrow \mathrm{R}$
		and $b : \mathrm H\times \mathrm Q \longrightarrow \mathrm{R}$ be bounded bilinear forms  such that
		$a$ and $b$ satisfy the hypotheses of \cite[Theorem 2.34]{ErnG}. Furthermore, let $\big\{H_h \big\}_{h>0}$ and 
		$\big\{Q_h \big\}_{h>0}$ be sequences of finite dimensional subspaces of $H$ and $Q$, respectively, and for each 
		$h>0$ consider a bounded bilinear form $a_h: H \times H \longrightarrow R$, such that $a_h |_{H_h\times H_h}$ and $b|_{H_h\times Q_h}$
		satisfy the hypotheses of  \cite[Theorem 2.34]{ErnG} as well, with discrete  coercivity constant $\alpha^\ast$
		and discrete inf-sup condition constant $\beta^\ast$, both independent of $h$. In turn, given $F\in H'$, $G\in Q'$, and a sequence
		of functionals $\big\{F_h \big\}_{h>0}$, with $F_h \in H'_h$ for each $h > 0$, we let $(u,\sigma)\in H\times Q$
		and $(u_h,\sigma_h) \in H_h \times Q_h$ be the unique solutions, respectively, to the problems
		\begin{equation*}
			\begin{array}{rcll}
				a(u,v) \,+\, b(v,\sigma)&=& F(v) & \qquad \forall \, v\in \mathrm H \,,\\[1ex]
				b(v,\tau)&=&G(\tau) & \qquad \forall\, \tau \in \mathrm Q\,,
			\end{array}
		\end{equation*}
		and
		\begin{equation*}
			\begin{array}{rcll}
				a_h(u_h,v_h) \,+\, b(v_h,\sigma_h)&=& F_h(v_h) & \qquad \forall \, v_h\in \mathrm H_h \,,\\[1ex]
				b(v_h,\tau_h)&=&G(\tau_h) & \qquad \forall\, \tau_h \in \mathrm Q_h\,.
			\end{array}
		\end{equation*}
		Then, there holds
		\begin{equation}\label{eq-Strang}
			\begin{array}{c}
				\|u - u_h\| + \|\sigma - \sigma_h\| \,\le\, C_{S,1}\,\mathrm{dist}\big(u,H_h\big) \,+\, C_{S,2}\,\mathrm{dist}\big(\sigma,Q_h\big)\\[2ex]
				\disp
				+\,\, C_{S,3}\,\Big\{ \|F - F_h\|_{H'_h} \,+\, \|a(u,\cdot) - a_h(u,\cdot)\|_{H'_h} \Big\}\,,
			\end{array}
		\end{equation}
		where $C_{S,i}$, $i\in \{1,2,3\}$, are positive constants depending only on $\alpha^\ast$,
		$ \beta^\ast$, and other constants, all of which are independent of $h$.
	\end{lem}

With this at hand, we now separately address the fluid and concentration equations and estimate the individual errors $\|(\vec{\bu},\bsi)-(\vec{\bu}_h,\bsi_h)\|_{\mathbf{H}\times\mathbb{H}_{0}(\bdiv_{4/3};\Omega)}$ and $\|(\vec{\varphi},\widetilde{\bsi})-(\vec{\varphi}_h,\widetilde{\bsi}_h)\|_{\widetilde{\mathbf{H}}\times\mathbf{H}_{\Gamma}(\div_{4/3};\Omega)}.$

\subsubsection*{Estimation for $\|(\vec{\bu},\bsi)-(\vec{\bu}_h,\bsi_h)\|_{\mathbf{H}\times\mathbb{H}_{0}(\bdiv_{4/3};\Omega)}$.} 
The fluid equations, as expressed in the first two rows of \eqref{eq:FV} and \eqref{eq:FV_h}, can be equivalently rewritten as
	\begin{equation}\label{eq:velocity}
		\begin{array}{rll}
			\mathcal{A}_{\bu,\varphi}(\vec{\bu},\vec{\bv})- \B(\vec{\bv},\bsi)&=\F_\varphi(\vec{\bv}) &\forall\,\vec{\bv}\in \mathbf{H}\,, \\
			\B(\vec{\bu},\bta)&=0 &\forall\, \bta\in \mathbb{H}_0(\bdiv_{4/3};\Omega)\,,
		\end{array}
	\end{equation}
and 
	\begin{equation}\label{eq:velocity_h}
	\begin{array}{rll}
		\mathcal{A}_{\bu_h,\varphi_h}(\vec{\bu}_h,\vec{\bv}_h)-\B(\vec{\bv}_h,\bsi_h)&=\F_{\varphi_h}(\vec{\bv}_h)& \forall\varphi_h\in  \mathbf{H} _h\,, \\
		\B(\vec{\bu}_h,\bta_h)&=0 &\forall\, \bta_h\in \mathbb{H}^\bsi_h\,,
	\end{array}
\end{equation}
where  $\mathcal{A}_{\bu,\varphi}$ and $\mathcal{A}_{\bu_h,\varphi_h}$ are the bilinear forms
\begin{equation}\label{eq:A-u,var}
	\begin{array}{c}
	\A_{\bu,\varphi}(\vec{\bw},\vec{\bv}):=\A_\varphi(\vec{\bw},\vec{\bv})+\C(\bu;\vec{\bw},\vec{\bv}) \qquad \forall \,\vec{\bw},\vec{\bv}\in \mathbf{H}\\[2ex]
	\A_{\bu_h,\varphi_h}(\vec{\bw}_h,\vec{\bv}_h):=\A_{\varphi_h}(\vec{\bw}_h,\vec{\bv}_h)+\C(\bu_h;\vec{\bw}_h,\vec{\bv}_h) \qquad \forall\,\vec{\bw}_h,\vec{\bv}_h\in\, \mathbf{H}_h\,.
	\end{array}
\end{equation}
By proceding similarly as in \eqref{eq:est-u-e3-1}-\eqref{eq:est-u-e3} and using the regularity assumption \eqref{eq:HRA}, we find that
\begin{equation}\label{eq:AS-ASh}
	\begin{array}{l}
		\big|\A_{\varphi}(\vec{\bu},\vec{\bv}_h)-\A_{\varphi_h}(\vec{\bu},\vec{\bv}_h )\big|\leq 2L_{\mu}\|i_\varepsilon\||\Omega|^{\frac{\epsilon}{d}-\frac{1}{2}}\|\varphi-\varphi_h\|_{0,4,\Omega}\|\bt_0\|_{\varepsilon,\Omega} \|\vec{\bv}_h\|_{\mathbf{H}_h}\\[2ex]
		\quad \leq  2L_{\mu}\|i_\varepsilon\||\Omega|^{\frac{\epsilon}{d}-\frac{1}{2}}\,C_{1,\varepsilon}\|\varphi-\varphi_h\|_{0,4,\Omega}\|\vec{\bv}_h\|_{\mathbf{H}_h}\,.
	\end{array}
\end{equation}
%
\noindent 	In turn, from the boundedness properties of $\C$ and in part (c) of Lemmas \ref{lem:properties-form-C-F} and \ref{lem:discrete-properties-form-C-F-h} along with the a priori \eqref{eq:a-priori-estimates-continuous} for $\vec{\bu}$, we find that 
\begin{equation}\label{eq:CS-CSh}
	\begin{array}{l}
		\big|\C(\bu;\vec{\bu},\vec{\bv}_h)-\C(\bu_h;\vec{\bu},\vec{\bv}_h)\big|\leq \|\bu-\bu_h\|_{0,4,\Omega} \|\vec{\bu}\|_{\mathbf{H}}\, \|\vec{\bv}_h\|_{\mathbf{H}}\\[2ex]
		\leq	C_1(\mu,\gamma,\boldsymbol{f},g,\alpha,\kappa,U,\Omega)\, \|\bu-\bu_h\|_{0,4,\Omega}\|\vec{\bv}_h\|_{\mathbf{H}} \,.
	\end{array}
\end{equation}
\noindent 	Combining   \eqref{eq:AS-ASh} and \eqref{eq:CS-CSh}, and using the norm definition on $\mathbf{H}^{\prime}_{h}$
%
%
\begin{equation}\label{eq:estimation-A_u-A_uh}
	\begin{array}{l}
		\|\mathcal{A}_{\bu,\varphi}(\vec{\bu},\cdot)-\mathcal{A}_{\bu_h,\varphi_h}(\vec{\bu},\cdot)\|_{\mathbf{H}_h'}\\[2ex]
		\quad	\disp \leq 2L_{\mu}\|i_\varepsilon\||\Omega|^{\frac{\epsilon}{d}-\frac{1}{2}}\,C_{1,\varepsilon}\|\varphi-\varphi_h\|_{0,4,\Omega}+C_1(\mu,\gamma,\boldsymbol{f},g,\alpha,\kappa,U,\Omega) \, \|\bu-\bu_h\|_{0,4,\Omega}\,.
	\end{array}
\end{equation}

\noindent In turn, similarly to the estimation $E_1$ in Lemma \ref{lem:L-Lipshitz-continuo}, we find that
\begin{equation}\label{eq:estimation-F_var-F_varh}
	\| \F_\varphi-\F_{\varphi_h}\|_{\mathbf{H}_h'}\leq  g\gamma|\Omega|^{\frac{1}{2}}\|\varphi-\varphi_h\|_{0,4,\Omega}\,.
\end{equation}
Note that the hypotheses of  Lemma \ref{strang} are satisfied for the bilinear forms $\mathcal{A}_{\bu,\varphi}$, $\mathcal{A}_{\bu_h,\varphi_h}$ and $\B$.  Thus, considering the  estimates \eqref{eq:estimation-A_u-A_uh} and \eqref{eq:estimation-F_var-F_varh}, then the Strang estimate \eqref{eq-Strang} for problems  \eqref{eq:velocity} and \eqref{eq:velocity_h}, is given by 
\begin{equation}\label{eq:u-uh}
	\begin{aligned}
		&\|(\vec{\bu},\bsi)-(\vec{\bu}_h,\bsi_h)\|_{\mathbf{H}\times\mathbb{H}_{0}(\bdiv_{4/3};\Omega)}\leq C_{S,1} \,\mathrm{dist}(\vec{\bu},\widetilde{\mathbf{H}}_h)+C_{S,2}\, \mathrm{dist}(\bsi,\mathbb{H}^\bsi_h)\\[2ex]
		&\quad  +C_{S,3} \boldsymbol{C}_{1,\varepsilon}(\mu,\gamma,\boldsymbol{f},g,\alpha,\Omega)\|\varphi-\varphi_h\|_{0,4,\Omega} \\[2ex]
		&\quad	+C_{S,3} C_1(\mu,\gamma,\boldsymbol{f},g,\alpha,\kappa,U,\Omega) \, \|\bu-\bu_h\|_{0,4,\Omega}\,.
	\end{aligned}
\end{equation}
where $C_{S,i}$, $i\in \{1,2,3\}$, are positive constants  independent of $h$ and 
	\begin{equation}\label{eq:constantc1-apr}
		\boldsymbol{C}_{1,\varepsilon}(\mu,\gamma,\boldsymbol{f},g,\alpha,\Omega):=g\gamma|\Omega|^{\frac{1}{2}}+	2L_{\mu}\|i_\varepsilon\||\Omega|^{\frac{\epsilon}{d}-\frac{1}{2}}\,C_{1,\varepsilon}\,.
	\end{equation}

\subsubsection*{Estimation for $\|(\vec{\varphi},\widetilde{\bsi})-(\vec{\varphi}_h,\widetilde{\bsi}_h)\|_{\widetilde{\mathbf{H}}\times\mathbf{H}_{\Gamma}(\div_{4/3};\Omega)} $.}
Regarding the concentration equations, given as the two last rows of \eqref{eq:FV} and \eqref{eq:FV_h}, they can be expressed as
	\begin{equation}\label{eq:concentration}
		\begin{array}{rll}
			\widetilde{\mathcal{A}}_{\bu}(\vec{\varphi},\vec{\psi})-\widetilde{\B}(\vec{\psi},\widetilde{\bsi})&=\widetilde{\F}(\vec{\psi}) & \forall\,\vec{\psi}\in \widetilde{\mathbf{H}}\,, \\
			\widetilde{\B}(\vec{\varphi},\widetilde{\bta})&=0 &\forall\,\widetilde{\bta}\in \mathbf{H}_\Gamma(\div_{4/3};\Omega)\,,
		\end{array}
	\end{equation}
	and
	\begin{equation}\label{eq:concentration_h}
		\begin{array}{rll}
			\widetilde{\mathcal{A}}_{\bu_h}(\vec{\varphi}_h,\vec{\psi}_h)-\widetilde{\B}(\vec{\psi}_h,\widetilde{\bsi}_h)&=\widetilde{\F}(\vec{\psi}_h) &\forall\, \vec{\psi_h}\in  \widetilde{\mathbf{H}}_h \,,\\
			\widetilde{\B}(\vec{\varphi}_h,\widetilde{\bta}_h)&=0&\forall\, \widetilde{\bta_h}\in \mathbf{H}^{\widetilde{\bsi}}_h\,,
		\end{array}
	\end{equation}
where $\widetilde{\mathcal{A}}_{\bu}$ and $\widetilde{\mathcal{A}}_{\bu_h,}$ are the bilinear forms
	\begin{equation}\label{eq:Atilde-uh}
		\begin{array}{c}
			\widetilde{\mathcal{A}}_{\bu}(\vec{\phi},\vec{\psi}):=\widetilde{\A}(\vec{\phi},\vec{\psi}) +\widetilde{\C}(\bu;\vec{\phi},\vec{\psi}) \qquad \forall \, \vec{\phi},\vec{\psi}\in \widetilde{\mathbf{H}}\,,\\[2ex]
		\widetilde{\mathcal{A}}_{\bu_h}(\vec{\phi}_h,\vec{\psi}_h):=\widetilde{\A}(\vec{\phi}_h,\vec{\psi}_h) +\widetilde{\C}(\bu_h;\vec{\phi}_h,\vec{\psi}_h) \qquad \forall \,\vec{\phi}_h,\vec{\psi}_h\in \widetilde{\mathbf{H}}_h\,.
		\end{array}
	\end{equation}

\noindent From  the boundedness properties of $\widetilde{\C}$ and the a priori estimate \eqref{eq:a-priori-estimates-continuous} for $\vec{\varphi}$, we find that 
\begin{equation*}
	\begin{array}{l}
		\big|\widetilde{\mathcal{A}}_{\bu}(\vec{\varphi},\vec{\psi}_h)-\widetilde{\mathcal{A}}_{\bu_h}(\vec{\varphi},\vec{\psi}_h)\big|=\big|\widetilde{\C}(\bu;\vec{\varphi},\vec{\psi}_h)-\widetilde{\C}(\bu_h;\vec{\varphi},\vec{\psi}_h)\big|\\[2ex]
		\quad \leq\,\|\bu-\bu_h\|_{0,4,\Omega}  \|\vec{\varphi}\|_{\widetilde{ \mathbf{H}}} \|\vec{\psi}\|_{\widetilde{ \mathbf{H}}}\leq C_2 (\alpha,\kappa,U,\Omega)\,\|\bu-\bu_h\|_{0,4,\Omega} \|\vec{\psi}_h\|_{\widetilde{ \mathbf{H}}}\,,
	\end{array}
\end{equation*}
 and so 
	\begin{equation}\label{eq:estimation Atilde_u-Atilde_uh}
		\|\widetilde{\mathcal{A}}_{\bu}(\vec{\varphi},\cdot)-\widetilde{\mathcal{A}}_{\bu_h}(\vec{\varphi},\cdot)\|_{\widetilde{\mathbf{H}}'_h}\leq C_2 (\alpha,\kappa,U,\Omega)\,\|\bu-\bu_h\|_{0,4,\Omega} \,.
	\end{equation}

The bilinear forms $(\widetilde{\mathcal{A}}_{\bu,\varphi},\widetilde{\B})$ and $(\widetilde{\mathcal{A}}_{\bu_h,\varphi_h}, \widetilde{\B})$ satisfy the Lemma \ref{strang}'s criteria. Then, the Strang estimate  \eqref{eq-Strang}  applied to problems \eqref{eq:concentration} and \eqref{eq:concentration_h}, together with the bound \eqref{eq:estimation Atilde_u-Atilde_uh}, is given by
	\begin{equation}\label{eq:var-varh}
		\begin{array}{l}
	\|(\vec{\varphi},\widetilde{\bsi})-(\vec{\varphi}_h,\widetilde{\bsi}_h)\|_{\widetilde{\mathbf{H}}\times\mathbf{H}_{\Gamma}(\div_{4/3};\Omega)}\leq \widetilde{C}_{S,1} \,\mathrm{dist}(\vec{\varphi}, \widetilde{\mathbf{H}}_h)+\widetilde{C}_{S,2} \,\mathrm{dist}(\widetilde{\bsi},\mathbf{H}^{\widetilde{\bsi}}_h) \\[2ex]
	\qquad +\widetilde{C}_{S,3}C_2 (\alpha,\kappa,U,\Omega)\,\|\bu-\bu_h\|_{0,4,\Omega}
		\end{array}
	\end{equation}
	where $\widetilde{C}_{S,i}$, $i\in \{1,2,3\}$, are positive constants  independent of $h$. 
\subsubsection*{Estimation for the total error.}

	Finally, we bound $\|\varphi-\varphi_h\|_{0,4,\Omega}$ in \eqref{eq:u-uh}  by using equation \eqref{eq:var-varh} to get
	\begin{equation}\label{eq:||u-uh||}
		\begin{array}{l}
		\|(\vec{\bu},\bsi)-(\vec{\bu}_h,\bsi_h)\|_{\mathbf{H}\times\mathbb{H}_{0}(\bdiv_{4/3};\Omega)}+\|(\vec{\varphi},\widetilde{\bsi})-(\vec{\varphi}_h,\widetilde{\bsi}_h)\|_{\widetilde{\mathbf{H}}\times\mathbf{H}_{\Gamma}(\div_{4/3};\Omega)} \\[2ex]
		\qquad \leq \,C_{S,1} \,\mathrm{dist}(\vec{\bu},\widetilde{\mathbf{H}}_h)+C_{S,2}\, \mathrm{dist}(\bsi,\mathbb{H}^\bsi_h)\\[2ex]
			\qquad+\,C_{S,3}\,\widetilde{C}_{S,1} \boldsymbol{C}_{1,\varepsilon}(\mu,\gamma,\boldsymbol{f},g,\alpha,\Omega) \mathrm{dist}(\vec{\varphi}, \widetilde{\mathbf{H}}_h)\\[2ex]
			\qquad+\,C_{S,3}\,\widetilde{C}_{S,2} \boldsymbol{C}_{1,\varepsilon}(\mu,\gamma,\boldsymbol{f},g,\alpha,\Omega) \mathrm{dist}(\widetilde{\bsi},\mathbf{H}^{\widetilde{\bsi}}_h)\\[2ex]
			\qquad+	\,	\boldsymbol{C}_{2,\varepsilon}(\mu,\gamma,\boldsymbol{f},g,\alpha,\kappa,U,\Omega)\, \|\bu-\bu_h\|_{0,4,\Omega} \\[2ex]
		\end{array}
	\end{equation}
where $\boldsymbol{C}_{1,\varepsilon}$ is given in \eqref{eq:constantc1-apr} and
\begin{equation}\label{eq:constantapc2}
	\begin{aligned}
		\boldsymbol{C}_{2,\varepsilon}(\mu,\gamma,\boldsymbol{f},g,\alpha,\kappa,U,\Omega) &:= C_{S,3}\Big\{\widetilde{C}_{S,3}C_2 (\alpha,\kappa,U,\Omega)\boldsymbol{C}_{1,\varepsilon}(\mu,\gamma,\boldsymbol{f},g,\alpha,\Omega) \\
		&\quad + C_1(\mu,\gamma,\boldsymbol{f},g,\alpha,\kappa,U,\Omega)\Big\}\,.
	\end{aligned}
\end{equation}
%

%
%
%

We are now in a position to establish the main result of this section, which provides the theoretical convergence rates for the numerical approximation of our fully-mixed formulation.
	\begin{thm}\label{thm:convergencerates}
	Assume that the hypotheses of Theorems \ref{thm:existence}, \ref{thm:uniqueness}, and \ref{thm:existence-h} hold, and the  data is sufficiently small so that 
	\begin{equation}\label{eq:datares-apriori}	\boldsymbol{C}_{2,\varepsilon}(\mu,\gamma,\boldsymbol{f},g,\alpha,\kappa,U,\Omega)\leq \dfrac{1}{2}\,,
		\end{equation}
	where $\boldsymbol{C}_{2,\varepsilon}(\,\cdot\,)$ is defined in \eqref{eq:constantapc2}. Suppose further that the solution satisfies $\bu\in \mathrm{\mathbf{W}}^{s,4}(\Omega),$ $ \bt \in \mathbb{H}^s(\Omega)\cap\mathbb{L}_{\mathrm{tr}}^2(\Omega),$ $\bsi \in \mathbb{H}^s(\Omega)\cap\mathbb{H}_0(\bdiv_{4/3};\Omega)$, $\bdiv\, \bsi \in \mathrm{\mathbf{W}}^{s,4/3}(\Omega),$ $\,\varphi\in \mathrm{W}^{s,4}(\Omega)\cap\mathrm{L}^4_{0}(\Omega),$ $\widetilde{\bt}\in \mathrm{\mathbf{H}}^s(\Omega)$, $\widetilde{\bsi} \in \mathrm{\mathbf{H}}^s(\Omega)\cap\mathbf{H}_\Gamma(\div_{4/3};\Omega)$ and $\mathrm{div}\,\widetilde{\bsi}\in\mathrm{\mathbf{W}}^{s,4/3}(\Omega)$,  for some $s\in [0,\ell+1].$ Then the errors satisfy 
		\begin{equation}\label{eq:rates-of-convergence}
			\begin{array}{c}
				\|(\vec{\bu},\bsi)-(\vec{\bu}_h,\bsi_h)\|_{\mathbf{H}\times\mathbb{H}_{0}(\bdiv_{4/3};\Omega)}+\|(\vec{\varphi},\widetilde{\bsi})-(\vec{\varphi}_h,\widetilde{\bsi}_h)\|_{\widetilde{\mathbf{H}}\times\mathbf{H}_{\Gamma}(\div_{4/3};\Omega)}\leq C_{\rm rate}\,h^s
			\end{array}
		\end{equation}
		where the constant  $C_{\rm rate}>0$, independent of $h$, depends on the data and high-order norms of the solution (cf. \eqref{eq:rates-of-convergence-2}), but is independent of $h$.
	\end{thm}
	\begin{proof} The hypothesis  \eqref{eq:datares-apriori} applied to \eqref{eq:||u-uh||} gives the Cea estimate
			\begin{equation*}\label{eq:cea}
			\begin{array}{l}
				\|(\vec{\bu},\bsi)-(\vec{\bu}_h,\bsi_h)\|_{\mathbf{H}\times\mathbb{H}_{0}(\bdiv_{4/3};\Omega)}+\|(\vec{\varphi},\widetilde{\bsi})-(\vec{\varphi}_h,\widetilde{\bsi}_h)\|_{\widetilde{\mathbf{H}}\times\mathbf{H}_{\Gamma}(\div_{4/3};\Omega)} \\[2ex]
				\qquad\leq C \left\{ \mathrm{dist}(\vec{\bu},\widetilde{\mathbf{H}}_h)+\mathrm{dist}(\bsi,\mathbb{H}^\bsi_h)+\mathrm{dist}(\vec{\varphi}, \widetilde{\mathbf{H}}_h)+\mathrm{dist}(\widetilde{\bsi},\mathbf{H}^{\widetilde{\bsi}}_h)  \right\}\,.
			\end{array}
		\end{equation*}
Then, from the regularity of the solution and the approximation properties of the finite dimensional subspaces \cite[Sections 5.2 and 5.5]{CGMor}, we obtain
	\begin{equation}\label{eq:rates-of-convergence-2}
	\begin{array}{c}
		\|(\vec{\bu},\bsi)-(\vec{\bu}_h,\bsi_h)\|_{\mathbf{H}\times\mathbb{H}_{0}(\bdiv_{4/3};\Omega)}+\|(\vec{\varphi},\widetilde{\bsi})-(\vec{\varphi}_h,\widetilde{\bsi}_h)\|_{\widetilde{\mathbf{H}}\times\mathbf{H}_{\Gamma}(\div_{4/3};\Omega)}\leq C\,h^s\,\biggl\{\biggr. \|\bu\|_{l,4;\Omega}+\|\bt\|_{l,\Omega} \\[2ex]
			+\,\|\bsi\|_{l,\Omega} 	+\|\bdiv\,  \bsi\|_{l,4/3;\Omega} + \|\varphi\|_{l,4;\Omega} + \|\widetilde{\bt}\|_{l,\Omega} + \|\widetilde{\bsi}\|_{l,\Omega} + \|\div\, \widetilde{\bsi}\|_{l,4/3;\Omega} \biggl. \biggr\}\,,
	\end{array}
\end{equation}
which gives the desired result.
\end{proof}
	
\begin{rem}\label{rem:cv-pressure} From the identities \eqref{eqn:p} and \eqref{eq:constant-cu}, we recall that 
	\[p=-\dfrac{1}{2d}\tr(\,2\bsi\,+\,\bu\otimes\bu\,)-\mathrm{c}_\bu \,, \quad\text{with}\quad \mathrm{c}_\bu:=-\dfrac{1}{2d|\Omega|}\int_{\Omega}\tr(\bu\otimes\bu)\,, \]
which suggests to define the discrete pressure as
	\[p_h=-\dfrac{1}{2d}\tr(\,2\bsi_h\,+\,\bu_h\otimes\bu_h\,)-\mathrm{c}_{\bu_h} \,, \quad\text{with}\quad \mathrm{c}_{\bu_h}:=-\dfrac{1}{2d|\Omega|}\int_{\Omega}\tr(\bu_h\otimes\bu_h)\,.\]
	Then it is straightforward to demonstrate the existence of a positive constant $C$, which is independent of $h$, such that
	\[\|p-p_h\|_{0,\Omega}\leq C\left\{ \|\bsi-\bsi_h\|_{\bdiv_{4/3};\Omega}+\|\bu-\bu_h\|_{0,4;\Omega}\right\}\,.\]
	Note that the rate of convergence of $p_h$ is the same of the rest of variables as in  \eqref{eq:rates-of-convergence}.
	\end{rem}
	
\section{A posteriori error analysis}\label{section4}

In contrast to the a priori error analysis (cf. Section \ref{section33}), which provides theoretical convergence rates and estimates the error based on the discretization parameter $h$, this section is dedicated to the a posteriori error analysis for the fully mixed finite element method \eqref{eq:FV_h}. The goal is to establish error bounds that depend on the computed solution and provide practical error estimates for adaptive refinement. Specifically, we aim to show the existence of positive constants $C_{\text{eff}}$ and $C_{\text{rel}}$ such that the following inequality holds
\begin{equation}\label{eq:objective-estimation}
	C_{\text{eff}} \, \mathbf{\Theta} \leq \|(\vec{\bu}, \bsi) - (\vec{\bu}_h, \bsi_h)\|_{\mathbf{H} \times \mathbb{H}_0(\div_{4/3}; \Omega)} + \|(\vec{\varphi}, \widetilde{\bsi}) - (\vec{\varphi}_h, \widetilde{\bsi}_h)\|_{\widetilde{\mathbf{H}} \times \mathbf{H}_\Gamma(\div_{4/3}; \Omega)} \leq C_{\text{rel}} \, \mathbf{\Theta},
\end{equation}
where $\mathbf{\Theta}$ represents the a posteriori error indicator. This approach allows for the assessment of the error based on the actual computed solution, enabling effective adaptive mesh refinement strategies to improve the solution accuracy. In Section \ref{section41}, we establish and review the necessary notations and results to define, derive, and analyze the a posteriori error estimator. The residual-based indicator is then introduced in Section \ref{section42}, where its reliability (upper bound of \eqref{eq:objective-estimation}) is also demonstrated. Finally, the efficiency property (lower bound of \eqref{eq:objective-estimation}) is established in Section \ref{section43}.

\subsection{Preliminary results for the a posteriori error analysis}\label{section41}

\noindent\textbf{Mesh faces, jumps, curl operator, and tangential/normal components.} Consider a barycentric refinement mesh $\Thb$. We denote by $\mathcal{F}_h$ the set of all facets (edges or faces, applicable for $d=2$ or $d=3$), with their respective diameters represented as $h_F$. We categorize the facets into internal and boundary subsets as
\[
\mathcal{F}_h^{\,\rm i} := \{ F \in \mathcal{F}_h : F \subset \Omega \} \quad \text{and} \quad \mathcal{F}_h^{\rm b} := \{ F \in \mathcal{F}_h : F \subset \Gamma \}.
\]
For any element \( T \in \mathcal{T}_h^{\rm b} \), let \(\mathcal{F}_{h,T}\) denote its set of facets. These can be further classified into
\[
\mathcal{F}_{h,T}^{\,\rm i} := \{ F \subset \partial T : F \in \mathcal{F}_h^{\,\rm i} \} \quad \text{and} \quad \mathcal{F}_{h,T}^{\,\rm b} := \{ F \subset \partial T : F \in \mathcal{F}_h^{\rm b} \}.
\]
The unit normal vector $\bn_F$ on each facet and  the tangential vector $\bs_F$  on each edge are defined as
\[
\bn_F := (n_1, \ldots, n_d)^t \quad \forall F \in \mathcal{F}_h\qan  \bs_F := (-n_2, n_1)^t \quad \forall F \in \mathcal{F}_h.
\]
For simplicity, when the context is clear, we will use $\bn$ and $\bs$ instead of $\bn_F$ and $\bs_F$.
\noindent Let  $\psi$ be a sufficiently smooth scalar-valued function to admit on all $F\in\mathcal{F}_{h}^{\,\rm i}$ possibly two-valued trace, we define the jump on $F$ as
\[
\jump{\psi}_{F} = (\psi \big|_{T^+}) \big|_F - (\psi \big|_{T^-}) \big|_F \quad \text{where } F = \partial T^+ \cap \partial T^-\qan T^+,T^-\in\mathcal{T}_h^{\rm b}\,.
\]
 For vector/matrix-valued functions, the above jump operator act component-wise on the function and  whenever no confusion can arise, we write $\jump{\cdot}$ instead of $\jump{\cdot}_{F}.$

\noindent Finally, let $\psi$ be a scalar field, $\bv := (v_1, \ldots, v_d)^{\rm t}$ a vector field, and $\bta := (\bta_1, \ldots, \bta_d)^{\rm t}=(\tau_{ij})_{1 \leq i, j \leq d}$ a tensor field, all of which possess partial distributional derivatives $\partial_{x_i}$. We define for $d=2,$ 
\[
\curle(\psi) := \left( \partial_{x_2} \psi, - \partial_{x_1} \psi \right)^{\rm t},\quad \curlv(\bv) := 
	\partial_{x_1} v_2 - \partial_{x_2} v_1,\quad \curlt(\bta) := 
		\begin{pmatrix} \curlv(\bta_1)^{\rm t} \\ \curlv(\bta_2)^{\rm t} \end{pmatrix},\quad  \tncomp(\bv) := 
			\bv\cdot \bs.
\]
and for $d=3$
\[
\curlv(\bv) := 
	\nabla \times \bv, \quad 
\curlt(\bta) := 
	\begin{pmatrix} \curlv(\bta_1)^{\rm t} \\ \curlv(\bta_1)^{\rm t} \\ \curlv(\bta_3)^{\rm t} \end{pmatrix},\quad \tncomp(\bta):=\begin{pmatrix} \bta_1\times \bn \\ \bta_2\times \bn \\ \bta_3\times \bn \end{pmatrix}\,.
\]

\noindent\textbf{Raviart-Thomas Intepolator.} For each $p\geq \frac{2d}{d+2},$ we set
\[
\mathbf{H}_{p} := \left\{ \bta \in \mathbf{H}(\div_p; \Omega):\quad  \bta |_T \in \mathbf{W}^{1,p}(T) \quad \forall T \in \Thb \right\},
\]
and let 
\[
\Pi_{h}^\ell\,:\mathbf{H}_{p}\longrightarrow \mathbf{H}_{p,h} := \left\{ \bta \in \mathbf{H}(\div_p; \Omega):\quad  \bta |_T \in \mathbf{RT}^{1,p}(T) \quad \forall T \in \Thb \right\},
\]
be the Raviart-Thomas intepolation operator defined by the following properties
\begin{equation}\label{eq:rt-prop1}
\int_{F}(\Pi_h^l(\bta) \cdot \bn)  \xi \, = \int_F (\bta \cdot \bn) \xi \quad \forall\, \xi \in \mathrm{P}_\ell(F) \quad  \forall\, F \in \mathcal{F}_{h}\,, \quad\mbox{when}\quad \ell\geq 0\,,\qan 
\end{equation}
\[
\int_{T}\Pi_h^l(\bta)  \psi \, = \int_K \bta \cdot \psi \quad \forall\, \psi \in \mathbf{P}_{\ell-1}(T) \quad  \forall\, T \in \Thb\,, \quad\mbox{when}\quad \ell\geq 1\,.
\]
Particularly, from \cite[Lemma B.67, Lemma 1.101]{ErnG} and \cite[Lemma 5.3, eq. (5.38)]{CGMor}, there exists a positive and $h-$independent constant $C$ such that  $\Pi_{h}^{\ell}$ satisfies the local approximation property 
\begin{equation}\label{eq:RT-approx-prop}
	\begin{array}{c}
\|\bta - \Pi_h^\ell(\bta)\|_{0,p,T} \leq C h_T^{k+1} |\bta|_{k+1,p,T},\quad \forall\, \bta \in \mathbf{W}^{k+1,p}(T),\quad 0\leq k\leq \ell\,,\qan \forall\, T\in\Thb,\\
\|\bta - \Pi_h^\ell(\bta)\|_{0,4/3,T} \leq C h_T^{1-d/4} |\bta|_{1,4/3,T},\quad \forall\, \bta \in \mathbf{W}^{1,4/3}(T),\quad \forall\, T\in\Thb\,.
\end{array} 
\end{equation}
The tensorial version of $\Pi_h^\ell$ is denoted by $\mathbf{\Pi}_h^\ell$ and operates row-wise as $\Pi_h^\ell$.

\noindent\textbf{Cl\'ement Interpolator.} Consider the space $\mathrm{H}_h^1 = \{ v_h \in \mathrm{C}(\overline{\Omega}) : v_h |_T \in \mathrm{P}_1(T) \; \forall T \in \Thb \}$. Let $\mathrm{I}_h : \mathrm{H}^1(\Omega) \to \mathrm{H}_h^1$ be the well-known Clément interpolation operator. According to \cite{Clement-1975} (see also \cite[Lemma 1.127]{ErnG}), for any $\psi \in \mathrm{H}^1(\Omega)$, there exist positive constants $c_1$ and $c_2$ independent of $h$, such that the following local approximation properties hold
\begin{equation}\label{eq:clement-prop}
	\| \psi - \mathrm{I}_h\psi \|_{0,T} \leq c_1 h_T \|\psi\|_{1,\Delta(T)} \quad \forall\,T \in \Thb \quad \text{and} \quad
	\| \psi -\mathrm{I}_h \psi \|_{0,F} \leq c_2 h_F^{1/2} \| \psi \|_{1,\Delta(F)} \quad \forall\,F \in \mathcal{F}_h,
\end{equation}
where $\Delta(T)$ and $\Delta(F)$ are the sets of elements intersecting $T$ and $F$, respectively. The vector version of $\mathrm{I}_h$, denoted by $\mathbf{I}_h : \mathbf{H}^1(\Omega) \to \mathbf{H}_h^1$, is defined component-wise by $\mathrm{I}_h$.

\noindent\textbf{Hemholtz Decompositions.} Based on \cite[Lemma 4.4]{CCOV-2022} and \cite[Lemma 4.4]{COV-2022}, we recall the following results regarding the existence of stable Helmholtz decompositions applicable to the spaces $\mathbb{H}_0(\bdiv_{p};\Omega)$ and $\mathbf{H}_{\Gamma}(\div_{p};\Omega)$, respectively. These results are established as follows.

\begin{lem}\label{Hdiv-decomp-1}
	Let $p > 1$. Then, for each $\bta \in \mathbb{H}(\bdiv_{p};\Omega)$ there exist
	\begin{itemize}
		\item[$a)$] $\boldsymbol{\eta} \in \mathbb{W}^{1,p}(\Omega)$ and $\boldsymbol{\xi} \in \mathbf{H}^1(\Omega)$ such that $\bta = \boldsymbol{\eta}  + \curlv{(\boldsymbol{\xi})}$ when $d = 2$,
		\item[$a)$] $\boldsymbol{\eta} \in \mathbb{W}^{1,p}(\Omega)$ and $\boldsymbol{\xi} \in \mathbb{H}^1(\Omega)$ such that $\bta = \boldsymbol{\eta}  + \curlt{(\boldsymbol{\xi})}$ when $d = 3$.
	\end{itemize}
	In addition, in both cases,
	\[
	\| \boldsymbol{\eta}\|_{1,p,\Omega} + \| \boldsymbol{\xi} \|_{1,\Omega} \leq C_{Hel} \| \bta \|_{\bdiv_p;\Omega},
	\]
	where $C_{Hel}$ is a positive constant independent of all the foregoing variables.
\end{lem}
\begin{lem}\label{Hdiv-decomp-2}
	Assume that there exists a convex domain $B$ such that $\Omega\subseteq B$ and let $\Gamma_{\mathrm{N}}\subseteq \Gamma$ such that   $\Gamma_\mathrm{N}\subseteq \partial B,$ and  let $p > 1.$ Then, for each $\widetilde{\bta}\in\mathbf{H}(\div_{4/3},\Omega)$ such that $\widetilde{\bta}\cdot \bn=0$ on $\Gamma_\mathrm{N}$ there exist
	\begin{itemize}
		\item[$a)$] $\widetilde{\boldsymbol{\eta}} \in \mathbf{W}^{1,p}(\Omega)$ and $\widetilde{\boldsymbol{\xi}}\in \mathrm{H}_{\Gamma_\mathrm{N}}^1(\Omega)$ such that $\widetilde{\bta} = \widetilde{\boldsymbol{\eta}} + \curle{(\widetilde{\boldsymbol{\xi}})}$ when $d = 2$,
		\item[$a)$] $\widetilde{\boldsymbol{\eta}} \in \mathbf{W}^{1,p}(\Omega)$ and $\widetilde{\boldsymbol{\xi}} \in \mathbf{H}^1_{\Gamma_{\mathrm{N}}}(\Omega)$ such that $\bta = \widetilde{\boldsymbol{\eta}}  + \curlv{(\widetilde{\boldsymbol{\xi}})}$ when $d = 3$, 
	\end{itemize}
	where $\mathrm{H}_{\Gamma_\mathrm{N}}^1(\Omega)=\{\psi\in\mathrm{H}^1(\Omega): \quad \psi|_{\Gamma_{\mathrm{N}}}=0\}\,.$ In addition, in both cases,
	\[
	\| \widetilde{\boldsymbol{\eta}}\|_{1,p,\Omega} + \|\widetilde{ \boldsymbol{\xi}} \|_{1,\Omega} \leq C_{Hel} \| \bta \|_{\bdiv_p;\Omega},
	\]
	where $C_{Hel}$ is a positive constant independent of all the foregoing variables.
\end{lem}

\noindent\textbf{Bubble functions.} Given $T\in\Thb$, we let $\phi_{T}$ be the usual element-bubble function, satisfying (cf. \cite{rv-1996}) $\phi_T \in \mathrm{P}_{3}(T), $ $\mathrm{supp}(\phi_T)\subseteq T$, $ \psi_K =0$ on $ \partial T$, and $0\leq\phi_T\leq 1$ in  $T.$ Additionally, there exists $C>0$, independent of $h,$ such that
\begin{equation}\label{eq:prop-bbf}
	\|\psi_T q\|^2_{0,T}\leq 	\|q\|^2_{0,T} \leq C 	\|\psi_T^{1/2} q\|^0_{0,T} \,.
\end{equation}

\subsection{Residual-based a posteriori error estimator reliability}\label{section42}
%
In this section, we introduce and demonstrate the reliability property (upper bound of \eqref{eq:objective-estimation}) of the a posteriori error indicator $\mathbf{\Theta}$. The global a posteriori error estimator is formulated as 
\begin{equation}\label{eq:global-estimator}
	\mathbf{\Theta} := \Bigg\{ \sum_{T \in \Thb} \overline{\mathbf{\Theta}}^2_{T} \Bigg\}^{1/2} + \Bigg\{ \sum_{T \in \Thb} \widehat{\mathbf{\Theta}}^{4/3}_{T} \Bigg\}^{3/4},
\end{equation}
where, for each $T \in \Thb$, the local error indicators $\overline{\mathbf{\Theta}}$ and $\widehat{\mathbf{\Theta}}$ are defined as
\begin{equation}\label{eq:local-estimator-1}
	\begin{aligned}
		\overline{\mathbf{\Theta}}^{2}_{T} =\; & h_T^{2-d/2} \Big\| \bt_h - \nabla \bu_h \Big\|^2_{0,T} + \Big\| \bsi_h^{\tt d} - 2\mu(\varphi_h + \alpha)\bt_{h,\mathrm{sym}} + \dfrac{1}{2} \left( \bu_h \otimes \bu_h \right)^{\tt d} \Big\|_{0,T}^2 \\[2ex]
		& + h_T^{2-d/2} \Big\| \widetilde{\bt}_h - \nabla \varphi_h \Big\|^2_{0,T} + \Big\| \widetilde{\bsi}_h - \kappa \widetilde{\bt}_h + \dfrac{1}{2} \varphi_h \bu_h + U(\varphi_h + \alpha) \widehat{\mathbf{e}}_d \Big\|_{0,T}^2 \\[2ex]
		& + h_T^2 \Big\| \curlt(\bt_h) \Big\|_{0,T}^2 + \sum_{F \in \mathcal{F}_{h,T}} h_F \Big\| \jump{\tncomp(\bt_h)} \Big\|_{0,F}^2 \\[2ex]
		& + h_T^2 \Big\| \curlv(\widetilde{\bt}_h) \Big\|_{0,T}^2 + \sum_{F \in \mathcal{F}_{h,T}^{\rm i}} h_F \Big\| \jump{\tncomp(\widetilde{\bt}_h)} \Big\|_{0,F}^2
	\end{aligned}
\end{equation}
and
\begin{equation}\label{eq:local-estimator-2}
	\widehat{\mathbf{\Theta}}^{4/3}_{T} = \Big\| \bdiv \, \bsi_h - \dfrac{1}{2} \bt_h \bu_h + \boldsymbol{f} - g[1+\gamma(\varphi_h+\alpha)]\widehat{\mathbf{e}}_d \Big\|^{4/3}_{0,4/3,T} + \Big\| \div \, \widetilde{\bsi}_h - \dfrac{1}{2} \widetilde{\bt}_h \cdot \bu_h \Big\|^{4/3}_{0,4/3,T}.
\end{equation}
\begin{rem}
	\begin{itemize}
		\item[(a)] Note that $\mathbf{\Theta}$ provides a quantitative measure of the discretization error based on the computed solution from our method \eqref{eq:FV_h}. It is evidently a residual-based indicator, as can be recognized by a simple inspection of each term defined in the continuous problem \eqref{eqn:auxcomp}-\eqref{eq:bc-full}.
		\item[(b)] Differently from \cite{GIRS-2022}, our a posteriori error indicator \eqref{eq:global-estimator} involves the residual terms $h_T^{2-d/2} \Big\| \bt_h - \nabla \bu_h \Big\|^2_{0,T}$ and $h_T^{2-d/2} \Big\| \widetilde{\bt}_h - \nabla \varphi_h \Big\|^2_{0,T}$ instead of $h_T^4 \big\| \bt_h - \nabla \bu_h \big\|^4_{0,4,T}$ and $h_T^4 \big\| \widetilde{\bt}_h - \nabla \varphi_h \big\|^4_{0,4,T}$. In this way, $	\mathbf{\Theta}$ offers a more refined control by using the $L^2$ norm and considering the dimension $d$, which can yield sharper error estimates in higher dimensions. Also,  the tensor $\bt$ and the vector $\widetilde{\bt}$ in our formulation are sought in $L^2$, making the use of the $L^2$ norm a theoretically consistent choice that aligns with the function space in which these variables reside.
	\end{itemize}
\end{rem}

Before presenting the main result of this section, it is important to note that,
 for a given $(\bw,\phi)\in\mathbf{L}^4(\Omega)\times\mathrm{L}_0^4(\Omega)$,  thanks to the properties satisfied by the bilinear forms $\A_{\phi}(\cdot,\cdot) + \C(\bw;\cdot,\cdot)$ and $\B(\cdot,\cdot)$ and $\widetilde{\A}(\cdot,\cdot) + \widetilde{\C}(\bw;\cdot,\cdot)$ and $\widetilde{\B}(\cdot,\cdot)$ (cf. Lemmas \ref{lem:properties-form-B}, \ref{lem:properties-form-A} and \ref{lem:properties-form-C-F}), it follows (see \cite[Proposition 2.36]{ErnG}, for instance) that there exist positive constants $\rho,\widetilde{\rho}>0$ such that the  global inf-sup conditions hold 
\begin{subequations}
\begin{equation}\label{eq:inf-sup-glob-1}
	\sup_{(\vec{\bv},\bta) \in \mathbf{H} \times \mathbb{H}_0(\bdiv_{4/3}; \Omega)} \frac{\A_{\phi}(\vec{\mathbf{z}}, \vec{\bv}) + \C(\bw; \vec{\mathbf{z}}, \vec{\bv}) - \B(\vec{\bv}, \mathbf{\zeta}) - \B(\vec{\mathbf{z}}, \bta) }{\|(\vec{\bv},\bta)\|_{\mathbf{H} \times \mathbb{H}_0(\bdiv_{4/3}; \Omega)}} \geq\, \rho\, \|(\vec{\mathbf{z}}, \mathbf{\zeta})\|_{ \mathbf{H} \times \mathbb{H}_0(\bdiv_{4/3}; \Omega)},
\end{equation}
\begin{equation}\label{eq:inf-sup-glob-2}
	\sup_{(\vec{\psi},\widetilde{\bta}) \in\widetilde{\mathbf{H}} \times \mathbf{H}_\Gamma(\div_{4/3}; \Omega)} \frac{\widetilde{\A}(\vec{\eta}, \vec{\psi}) + \widetilde{\C}(\bw; \vec{\eta}, \vec{\psi}) - \widetilde{\B}(\vec{\psi}, \widetilde{\mathbf{\zeta}}) - \widetilde{\B}(\vec{\eta}, \widetilde{\bta}) }{\|(\vec{\psi},\widetilde{\bta})\|_{\widetilde{\mathbf{H}} \times \mathbf{H}_\Gamma(\div_{4/3}; \Omega)}} \geq\,\widetilde{\rho}\, \|(\vec{\eta}, \widetilde{\mathbf{\zeta}})\|_{\widetilde{\mathbf{H}} \times \mathbf{H}_\Gamma(\div_{4/3}; \Omega)}, 
\end{equation}\end{subequations}
for all  $(\vec{\mathbf{z}}, \mathbf{\zeta}) \in \mathbf{H} \times \mathbb{H}_0(\div_{4/3}; \Omega)$	 and for all $(\vec{\eta}, \widetilde{\mathbf{\zeta}}) \in \widetilde{\mathbf{H}} \times \mathbf{H}_\Gamma(\div_{4/3}; \Omega)\,, $ respectively.  These properties are crucial to establish the main result of this section which is stated as follows.  
\begin{thm}\label{thm:reliability}
	Under the hypotheses of Theorems \ref{thm:uniqueness} and \ref{thm:existence-h}, and assuming that the data is sufficiently small so that
	\begin{equation}\label{eq:data-as-reliability}
	2\,	\max\Big\{\widetilde{\rho}^{-1}C^\ast_1(\mu,\gamma,\boldsymbol{f},g,\alpha,\kappa,U,\Omega)+\widetilde{\rho}^{-1}C^\ast_2(\alpha,\kappa,U,\Omega),\rho^{-1}	\boldsymbol{C}_{1,\varepsilon}(\mu,\gamma,\boldsymbol{f},g,\alpha,\Omega) \Big\}\leq 1\,,
	\end{equation}	
where the constants  $ C^\ast_1 (\,\cdot\,)$ $C^\ast_2(\cdot)$, $\boldsymbol{C}_{1,\varepsilon}(\,\cdot\,),$ $\rho$ and $\widetilde{\rho}$ are given by \eqref{eq:C1*}, \eqref{eq:C2*} \eqref{eq:constantc1-apr},  \eqref{eq:inf-sup-glob-1} and \eqref{eq:inf-sup-glob-2}, respectively. 	There exists $C_{\rm rel} > 0$, independent of $h$,  such that $\mathbf{\Theta}$ defined by \eqref{eq:global-estimator} satisfies
	\begin{equation}\label{eq:rel-estimation}
		\|(\vec{\bu},\bsi)-(\vec{\bu}_h,\bsi_{h})\|_{ \mathbf{H} \times \mathbb{H}_0(\div_{4/3}; \Omega)}+\|(\vec{\varphi},\widetilde{\bsi})-(\vec{\varphi}_h,\widetilde{\bsi}_{h})\|_{\widetilde{\mathbf{H}} \times \mathbf{H}_\Gamma(\div_{4/3}}\,\leq \,C_{\rm rel} \,\mathbf{\Theta}\,.
	\end{equation}	
\end{thm}

The proof of Theorem \ref{thm:reliability} is carried out in this subsection through consecutive steps. We begin with the following result, which provides a preliminary upper bound for the total error.

\begin{lem}\label{lem:est1}
	Under the same hypothesis of Theorem \ref{thm:reliability}, there exists $C_1:=2\max\{\rho^{-1}, \widetilde{\rho}^{-1}\}>0$ (cf. \eqref{eq:inf-sup-glob-1} and \eqref{eq:inf-sup-glob-2}), independent of $h$, such that
\begin{equation}\label{eq:est1-apost}
	\begin{array}{l}
	\|(\vec{\bu},\bsi)-(\vec{\bu}_h,\bsi_{h})\|_{ \mathbf{H} \times \mathbb{H}_0(\div_{4/3}; \Omega)}+\|(\vec{\varphi},\widetilde{\bsi})-(\vec{\varphi}_h,\widetilde{\bsi}_{h})\|_{\widetilde{\mathbf{H}} \times \mathbf{H}_\Gamma(\div_{4/3}}\,\\[2ex]
	\qquad \leq \disp \,C_1\Bigg\{  \big\|\R\big\|_{[\mathbf{H}\times \mathbb{H}_0(\bdiv_{4/3}; \Omega)]^{\prime}} + \big\|\widetilde{\R}\big\|_{[\widetilde{\mathbf{H}} \times \mathbf{H}_\Gamma(\div_{4/3}; \Omega)]^{\prime}}\Bigg\},
	\end{array}
\end{equation}	
where $\R:\mathbf{H} \times \mathbb{H}_0(\bdiv_{4/3}; \Omega)\longrightarrow \mathrm{R}$ and $\widetilde{\R}: \widetilde{\mathbf{H}} \times \mathbf{H}_\Gamma(\div_{4/3}; \Omega)\longrightarrow \mathrm{R}$ are the linear functionals given, respectively,  by 
\begin{equation}\label{eq:funcR}
\R(\vec{\bv},\bta)=\F_{\varphi_h}(\vec{\bv})-\A_{\varphi_h}(\vec{\bu}_h,\vec{\bv})-\C(\bu_h,\vec{\bu}_h,\vec{\bv})+\B(\vec{\bv},\bsi_h)+\B(\vec{\bu}_h,\bta)
\end{equation}
and 
\begin{equation}\label{eq:funcRt}
	\widetilde{\R}(\vec{\psi},\widetilde{\bta})=\widetilde{\F}(\vec{\psi})-\widetilde{\A}(\vec{\varphi}_h,\vec{\psi}_h)-\widetilde{\C}(\bu_h,\vec{\varphi}_h,\vec{\psi})+\widetilde{\B}(\vec{\psi},\widetilde{\bsi}_h)+\widetilde{\B}(\vec{\varphi}_h,\widetilde{\bta})\,. 
\end{equation}
\end{lem}

\begin{proof}  Taking $(\bw,\phi)=(\bu,\varphi)$ and $(\vec{\mathbf{z}}, \mathbf{\zeta})=(\vec{\bu},\bsi)-(\vec{\bu}_h,\bsi_h)$ in \eqref{eq:inf-sup-glob-1}, we get 
\begin{equation}\label{eq:inf-sup-glob-1a}
\begin{array}{l}
 \rho\, \|(\vec{\bu},\bsi)-(\vec{\bu}_h,\bsi_h)\|_{ \mathbf{H} \times \mathbb{H}_0(\bdiv_{4/3}; \Omega)} \\[2ex]
\quad   \leq 	\disp \sup_{(\vec{\bv},\bta) \in \mathbf{H} \times \mathbb{H}_0(\bdiv_{4/3}; \Omega)} \frac{\A_{\phi}(\vec{\bu}-\vec{\bu}_h, \vec{\bv}) + \C(\bw; \vec{\bu}-\vec{\bu}_h, \vec{\bv}) - \B(\vec{\bv}, \bsi-\bsi_h) - \B(\vec{\bu}-\vec{\bu}_h, \bta) }{\|(\vec{\bv},\bta)\|_{\mathbf{H} \times \mathbb{H}_0(\bdiv_{4/3}; \Omega)}}\,
%
\end{array}
\end{equation}
From the linearity of the forms, using the first and second equations of \eqref{eq:FV}, and after adding and subtracting $\F_{\varphi_h}(\vec{\bv})$, $\A_{\varphi_h}(\vec{\bu_h},\vec{\bv})$, and $\C(\bu_h; \vec{\bu}_{h}, \vec{\bv})$, we find that
\begin{equation}\label{eq:inf-sup-glob-1c}
	\begin{array}{l}
		\A_{\phi}(\vec{\bu} - \vec{\bu}_h, \vec{\bv}) + \C(\bw; \vec{\bu} - \vec{\bu}_h, \vec{\bv}) - \B(\vec{\bv}, \bsi - \bsi_h) - \B(\vec{\bu} - \vec{\bu}_h, \bta) \\[2ex]
		\quad = \R(\vec{\bv}, \bta) + [\F_{\varphi}(\vec{\bv}) - \F_{\varphi_h}(\vec{\bv})] + [\A_{\varphi_h}(\vec{\bu}_h, \vec{\bv}) - \A_{\varphi}(\vec{\bu}_h, \vec{\bv})] + \C(\bu - \bu_h; \vec{\bu}_{h}, \vec{\bv})\,,
	\end{array}
\end{equation}
where $\R(\,\cdot\,)$ is defined by \eqref{eq:funcR}. By proceeding as in \eqref{eq:est-u-e1}, we have that
\begin{equation}\label{eq:inf-sup-glob-1d}
	\begin{array}{l}
		\F_{\varphi}(\vec{\bv}) - \F_{\varphi_h}(\vec{\bv}) \leq g \gamma |\Omega|^{\frac{1}{2}} \|\varphi - \varphi_h\|_{0,4,\Omega} \|\bv\|_{0,4,\Omega} \\[2ex]
		\qquad \leq g \gamma |\Omega|^{\frac{1}{2}} \|\varphi - \varphi_h\|_{0,4,\Omega} \|(\vec{\bv}, \bta)\|_{\mathbf{H} \times \mathbb{H}_0(\bdiv_{4/3}; \Omega)}\,,
	\end{array}
\end{equation}
and from part (c) of Lemma \ref{lem:properties-form-C-F} and the a priori bound \eqref{eq:est-a-priori_h} for $\vec{\bu}_h$, we have that
\begin{equation}\label{eq:inf-sup-glob-1e}
	\begin{array}{l}
		\C(\bu - \bu_h; \vec{\bu}_{h}, \vec{\bv}) \leq \|\bu - \bu_h\|_{0,4,\Omega} \, \|\vec{\bu}_h\|_{\mathbf{H}} \, \|\vec{\bv}\|_{\mathbf{H}} \\[2ex]
		\qquad \leq C_1^\ast(\mu, \gamma, \boldsymbol{f}, g, \alpha, \kappa, U, \Omega) \|\bu - \bu_h\|_{0,4,\Omega} \|(\vec{\bv}, \bta)\|_{\mathbf{H} \times \mathbb{H}_0(\bdiv_{4/3}; \Omega)}\,.
	\end{array}
\end{equation}
In turn, by adding and subtracting $\A_{\varphi}(\vec{\bu}, \vec{\bv})$ and $\A_{\varphi_h}(\vec{\bu}, \vec{\bv})$, grouping terms conveniently, using the definition of $\A(\cdot, \cdot)$ (cf. \eqref{A^S}), the upper bound for $\mu$ (cf. \eqref{eqn:mu-lips}), and proceeding as in \eqref{eq:est-u-e3-1}--\eqref{eq:est-u-e3} along with the regularity hypothesis \eqref{eq:HRA}, we find that
\begin{equation}\label{eq:inf-sup-glob-1f}
	\begin{array}{l}
		\A_{\varphi_h}(\vec{\bu}_h, \vec{\bv}) - \A_{\varphi}(\vec{\bu}_h, \vec{\bv}) \\[2ex]
		\qquad = [\A_{\varphi_h}(\vec{\bu}_h, \vec{\bv}) - \A_{\varphi_h}(\vec{\bu}, \vec{\bv})] + [\A_{\varphi_h}(\vec{\bu}, \vec{\bv}) - \A_{\varphi}(\vec{\bu}, \vec{\bv})] + [\A_{\varphi}(\vec{\bu}, \vec{\bv}) - \A_{\varphi}(\vec{\bu}_h, \vec{\bv})] \\[2ex]
		\qquad \disp \leq 2 \int_{\Omega} [\mu(\varphi_h + \alpha) - \mu(\varphi + \alpha)] \bt_{sym} : \br \\[2ex]
		\qquad \leq 2 L_{\mu} \|i_\varepsilon\| |\Omega|^{\frac{\varepsilon}{d} - \frac{1}{2}} \|\varphi - \varphi_h\|_{0,4,\Omega} \|\bt\|_{\varepsilon, \Omega} \|\br\|_{0, \Omega} \\[2ex]
		\qquad \leq 2 L_{\mu} \|i_\varepsilon\| |\Omega|^{\frac{\varepsilon}{d} - \frac{1}{2}} C_{1, \varepsilon} \|\varphi - \varphi_h\|_{0,4,\Omega} \|(\vec{\bv}, \bta)\|_{\mathbf{H} \times \mathbb{H}_0(\bdiv_{4/3}; \Omega)}\,.
	\end{array}
\end{equation}
Gathering \eqref{eq:inf-sup-glob-1d}, \eqref{eq:inf-sup-glob-1e}, and \eqref{eq:inf-sup-glob-1f} together and replacing them back in \eqref{eq:inf-sup-glob-1c}, we find from \eqref{eq:inf-sup-glob-1a} after simplifying and using the norm definition in $[\mathbf{H} \times \mathbb{H}_0(\bdiv_{4/3}, \Omega)]'$ that
\begin{equation}\label{eq:inf-sup-glob-1g}
	\begin{array}{c}
		\rho \|(\vec{\bu}, \bsi) - (\vec{\bu}_h, \bsi_h)\|_{\mathbf{H} \times \mathbb{H}_0(\bdiv_{4/3}; \Omega)} \leq \big\| \R \big\|_{[\mathbf{H} \times \mathbb{H}_0(\bdiv_{4/3}, \Omega)]'} \\[2ex]
		\qquad + C_1^\ast(\mu, \gamma, \boldsymbol{f}, g, \alpha, \kappa, U, \Omega) \|\bu - \bu_h\|_{0,4,\Omega} + 	\boldsymbol{C}_{1,\varepsilon}(\mu,\gamma,\boldsymbol{f},g,\alpha,\Omega) \|\varphi - \varphi_h\|_{0,4,\Omega}\,.
	\end{array}
\end{equation}
where $\boldsymbol{C}_{1,\varepsilon}(\,\cdot\,)$ is the constant defined in \eqref{eq:constantc1-apr}. 

Next, to obtain an upper preliminary bound for the error associated with the concentration variables, we proceed similarly.   In \eqref{eq:inf-sup-glob-2}, we now take $\bw=\bu,$ and $(\vec{\eta}, \widetilde{\zeta})=(\vec{\varphi},\widetilde{\bsi})-(\vec{\varphi}_h,\widetilde{\bsi}_h)$, use the last two equations of \eqref{eq:FV}, and after adding ans subtracting $\widetilde{\C}(\bu_h;\vec{\varphi}_h,\vec{\psi})$ and using the norm definition in $[\widetilde{\mathbf{H}} \times \mathbf{H}_\Gamma(\div_{4/3}; \Omega)]'$, it follows that 
	\begin{equation}\label{eq:inf-sup-glob-2a}
		\begin{array}{l}
\widetilde{\rho}\, \|(\vec{\varphi}, \widetilde{\bsi})-(\vec{\varphi}_h,\widetilde{\bsi}_h)\|_{\widetilde{\mathbf{H}} \times \mathbf{H}_\Gamma(\div_{4/3}; \Omega)} 
\\[2ex]
\disp \qquad \leq  	\sup_{(\vec{\psi},\widetilde{\bta}) \in\widetilde{\mathbf{H}} \times \mathbf{H}_\Gamma(\div_{4/3}; \Omega)} \frac{\widetilde{\A}(\vec{\varphi}-\vec{\varphi}_h, \vec{\psi}) + \widetilde{\C}(\bu; \vec{\varphi}-\vec{\varphi}_h, \vec{\psi}) - \widetilde{\B}(\vec{\psi}, \widetilde{\bsi}-\widetilde{\bsi}_h) - \widetilde{\B}(\vec{\varphi}-\vec{\varphi}_h, \widetilde{\bta}) }{\|(\vec{\psi},\widetilde{\bta})\|_{\widetilde{\mathbf{H}} \times \mathbf{H}_\Gamma(\bdiv_{4/3}; \Omega)}}\\[3ex]
\disp \qquad \leq   \big\| \widetilde{\R} \big\|_{[\widetilde{\mathbf{H}} \times \mathbf{H}_\Gamma(\div_{4/3}; \Omega)]'}\,+\, 	\sup_{\vec{\psi} \in\widetilde{\mathbf{H}}} \frac{\widetilde{\C}(\bu_h-\bu;\vec{\varphi}_h,\vec{\psi})}{ \|\vec{\psi}\|_{\widetilde{\mathbf{H}} }}
 \\[3ex]
\disp \qquad \leq   \big\| \widetilde{\R} \big\|_{[\widetilde{\mathbf{H}} \times \mathbf{H}_\Gamma(\div_{4/3}; \Omega)]'}\,+\, C_2^\ast(\alpha, \kappa, U, \Omega) \|\bu-\bu_h\|_{0,4,\Omega}
	\end{array}
\end{equation}
where we have bounded $\widetilde{\C}$ according to part (c) of Lemma \ref{lem:properties-form-C-F} and used the a priori bound \eqref{eq:est-a-priori_h} for $\vec{\varphi}_{h}$ in the last inequality. Finally, by combining \eqref{eq:inf-sup-glob-1g} and  \eqref{eq:inf-sup-glob-2a}, we get
\begin{equation*}
	\begin{array}{l}
	 \|(\vec{\bu}, \bsi) - (\vec{\bu}_h, \bsi_h)\|_{\mathbf{H} \times \mathbb{H}_0(\bdiv_{4/3}; \Omega)} + \|(\vec{\varphi}, \widetilde{\bsi})-(\vec{\varphi}_h,\widetilde{\bsi}_h)\|_{\widetilde{\mathbf{H}} \times \mathbf{H}_\Gamma(\div_{4/3}; \Omega)} 
	 \\[2ex]\qquad  \leq 	\rho^{-1} \big\| \R \big\|_{[\mathbf{H} \times \mathbb{H}_0(\bdiv_{4/3}, \Omega)]'} + \widetilde{\rho}^{-1}\big\| \widetilde{\R} \big\|_{[\widetilde{\mathbf{H}} \times \mathbf{H}_\Gamma(\div_{4/3}; \Omega)]'}  \\[2ex]
		\quad \qquad  + \big(	\rho^{-1} C_1^\ast(\mu, \gamma, \boldsymbol{f}, g, \alpha, \kappa, U, \Omega)\,+\, \widetilde{\rho}^{-1} C_2^\ast(\alpha, \kappa, U, \Omega)\big)\, \|\bu - \bu_h\|_{0,4,\Omega} \\[2ex] 
		\quad \qquad  + 	\rho^{-1}\boldsymbol{C}_{1,\varepsilon}(\mu,\gamma,\boldsymbol{f},g,\alpha,\Omega)\, \|\varphi - \varphi_h\|_{0,4,\Omega}\,.
	\end{array}
\end{equation*}
If the data satisfy \eqref{eq:data-as-reliability}, the terms multiplying $\|\bu - \bu_h\|_{0,4,\Omega}$ and $\|\varphi - \varphi_h\|_{0,4,\Omega}$ can be absorbed into the left-hand side of the previous estimate to finally get \eqref{eq:est1-apost} with the constant $C_1=2\max\{\rho^{-1}, \widetilde{\rho}^{-1}\}>0.$

\end{proof}

We focus now on bounding the functionals $\R$ and $\widetilde{\R}$ defined by \eqref{eq:funcR} and \eqref{eq:funcRt}. Note that we can write them as
\begin{equation*}
	\R(\vec{\bv},\bta)=\R_1(\vec{\bv}) +\R_2(\bta) \qan 	\widetilde{\R}(\vec{\psi},\widetilde{\bta})=	\widetilde{\R}_1(\vec{\psi})+	\widetilde{\R}_2(\widetilde{\bta})\,,
\end{equation*}
where 
\begin{equation}\label{eq:funcR1R2}
	\R_1(\vec{\bv})=\F_{\varphi_h}(\vec{\bv})-\A_{\varphi_h}(\vec{\bu}_h,\vec{\bv})-\C(\bu_h,\vec{\bu}_h,\vec{\bv})+\B(\vec{\bv},\bsi_h)\qan 	\R_2(\bta)=\B(\vec{\bu}_h,\bta)\,,
\end{equation}
for all  $\vec{\bv}\in \mathbf{H} $ and $\bta\in  \mathbb{H}_0(\bdiv_{4/3}; \Omega)$, and 
\begin{equation}\label{eq:funcRt1RT2}
	\widetilde{\R}_1(\vec{\psi})=\widetilde{\F}(\vec{\psi})-\widetilde{\A}(\vec{\varphi}_h,\vec{\psi}_h)-\widetilde{\C}(\bu_h,\vec{\varphi}_h,\vec{\psi})+\widetilde{\B}(\vec{\psi},\widetilde{\bsi}_h)\qan \widetilde{\R}_{2}(\widetilde{\bta})=\widetilde{\B}(\vec{\varphi}_h,\widetilde{\bta})\,. 
\end{equation}
for all $\vec{\psi}\in \widetilde{\mathbf{H}}$ and $\widetilde{\bta}\in  \mathbf{H}_\Gamma(\div_{4/3}; \Omega)$. Therefore, \eqref{eq:est1-apost} becomes
\begin{equation*}
	\begin{array}{l}
		\|(\vec{\bu},\bsi)-(\vec{\bu}_h,\bsi_{h})\|_{ \mathbf{H} \times \mathbb{H}_0(\div_{4/3}; \Omega)}+\|(\vec{\varphi},\widetilde{\bsi})-(\vec{\varphi}_h,\widetilde{\bsi}_{h})\|_{\widetilde{\mathbf{H}} \times \mathbf{H}_\Gamma(\div_{4/3}}\,\\[2ex]
		\qquad \leq \disp \,C_1\Bigg\{  \big\|\R_1\big\|_{\mathbf{H}'}+ \big\|\R_2\big\|_{\mathbb{H}_0(\bdiv_{4/3}; \Omega)'} + \big\|\widetilde{\R}_1\big\|_{\widetilde{\mathbf{H}}'}+\big\|\widetilde{\R}_2\big\|_{\mathbf{H}_\Gamma(\div_{4/3}; \Omega)'}\Bigg\}\,.
	\end{array}
\end{equation*}	
The estimations of $\R_1$ and $\widetilde{\R}_1$ are obtained straightforwardly by applying the definitions of the forms and utilizing the Hölder inequality, which immediately gives the following result. 
\begin{lem}\label{lem:est2}
There exist positive constants $C_2$ and $C_3$, independent of $h$,  such that
	\begin{equation*}
		\begin{array}{c}
	 \big\|\R_1\big\|_{\mathbf{H}'} \,\leq
	 C_2\ \Bigg\{ \Big\|\bdiv\,\bsi_h-\dfrac{1}{2}\bt_h\bu_h+\boldsymbol{f}-g[1+\gamma(\varphi_h+\alpha)]\widehat{\mathbf{e}}_d\Big\|_{0,4/3,\Omega}\\[2ex]
	 \disp \qquad  \,+\,\Big\|\bsi_h^{\tt d}-2\mu(\varphi_h+\alpha)\bt_{h,sym}- \dfrac{1}{2}\left(\bu_h\otimes\bu_h\right)^{\tt d}\Big\|_{0,\Omega} \Bigg\}\,,
		\end{array}
	\end{equation*}	
		\begin{equation*}
		 \big\|\widetilde{\R}_1\big\|_{\widetilde{\mathbf{H}}'}\,\leq
			C_3\Bigg\{ \Big\|\div\,\widetilde{\bsi}_h-\dfrac{1}{2}\widetilde{\bt}_h\cdot\bu_h\Big\|_{0,4/3,\Omega} \,+\,\Big\|\widetilde{\bsi}_h-\kappa\,\widetilde{\bt}_h+\dfrac{1}{2}\varphi_h\bu_h+U(\varphi_h+\alpha)\widehat{\mathbf{e}}_d\Big\|_{0,\Omega} \Bigg\}\,.
	\end{equation*}	
\end{lem}
 It is clear from their definitions of $\R_2$ and $\widetilde{\R}_2$ in \eqref{eq:funcR1R2} and \eqref{eq:funcRt1RT2}, as well as the second and fourth equations of \eqref{eq:FV_h}, that $\R_2(\bta_h) = \widetilde{\R}_{2}(\widetilde{\bta}_h) = 0$ for all $\bta_h \in \mathbb{H}_{h}^{\bsi}$ and $\widetilde{\bta}_h \in \widetilde{\mathbf{H}}_{h}^{\widetilde{\bsi}}$. Therefore,
\begin{equation}\label{eq:dif-helm-gen}
	\R_2(\bta) = \R_2(\bta - \bta_h) \quad \text{and} \quad \widetilde{\R}_{2}(\widetilde{\bta}) = \widetilde{\R}_{2}(\widetilde{\bta} - \widetilde{\bta}_h).
\end{equation}
These properties are used in Lemmas \ref{lem:est3} and \ref{lem:est4} to establish the bounds for the functionals.

\begin{lem}\label{lem:est3}
	There exists a positive constant $C_4>0,$  independent of $h$,  such that
	\begin{equation}\label{eq:est-R2}
			\big\|\R_2\big\|_{\mathbb{H}_{0}(\bdiv_{4/3};\Omega)'} \,\leq \disp
			C_4\ \Bigg\{ \sum_{T\in \Thb}   \overline{\boldsymbol{\Theta}}_{1,T}^{2}\Bigg\}^{1/2}\,,
	\end{equation}	
	where, for all $T\in\Thb$,  
		\begin{equation*}
			\disp   \overline{\boldsymbol{\Theta}}_{1,T}^{2} \,=\, h_T^{2-d/2} \big\|\bt_h - \nabla \bu_{h} \big\|^2_{0,T}\,+\,  h_T^2 \big\| \curlt{(\bt_h)}\big\|^{2}_{0,T} \,+\, \sum_{F\in\mathcal{F}_{h,T}} h_F \big\|\jump{\tncomp{(\bt_h)}} \big\|^2_{0,F}\,.
	\end{equation*}	
\end{lem}
\begin{proof}
  We first address the three-dimensional case. According to Lemma \ref{Hdiv-decomp-1} part (b), for any $\bta\in\mathbb{H}_{0}(\bdiv_p;\Omega)$ there exists $\boldsymbol{\eta} \in \mathbb{W}^{1,p}(\Omega)$ and $\boldsymbol{\xi} \in \mathbb{H}^1(\Omega)$ such that $\bta = \boldsymbol{\eta}  + \curlt{(\boldsymbol{\xi})}$. Let us then define $\bta_h$ as the discrete Helmholtz decomposition   
  \begin{equation}\label{eq:disc-helm-R2}
  	\bta_{h}=\mathbf{\Pi}_{h}^{\ell}(\boldsymbol{\eta})\,+\,\curlt{(\mathbf{I}_h(\boldsymbol{\xi}))}\,+\,c\,\mathbb{I}\in \mathbb{H}_{h}^{\bsi}\,\quad \mbox{where}\quad c = -\dfrac{1}{3|\Omega|}\int_{\Omega} \tr (\mathbf{\Pi}_{h}^{\ell}(\boldsymbol{\eta})\,+\,\curlt{(\mathbf{I}_h(\boldsymbol{\xi}))}\,.
  \end{equation}
  Therefore, since $\R_2(c\,\mathbb{I})=0$ we find that
    \begin{equation}\label{eq:dif-helm-R2}
 \R_2(\bta-\bta_h)=\R_2\big( \boldsymbol{\eta} - \mathbf{\Pi}_{h}^{\ell}(\boldsymbol{\eta})\big) + \R_2\big( \curlt{\big(\boldsymbol{\xi} - \mathbf{I}_h(\boldsymbol{\xi})\big)} \big)\,.
  \end{equation}
  On the one hand, using the definition of $\R_2$ (cf. \eqref{eq:funcR1R2}) and integrating by parts  on each element, we get
  \begin{equation}\label{eq:estR2-1}
  	\begin{array}{l}
  		\disp 	\R_2\big( \boldsymbol{\eta} - \mathbf{\Pi}_{h}^{\ell}(\boldsymbol{\eta})\big)\,=\,\int_\Omega \bt_h:\big( \boldsymbol{\eta} - \mathbf{\Pi}_{h}^{\ell}(\boldsymbol{\eta})\big)\,+\, \sum_{T\in\Thb} \int_T \bu_h \cdot \bdiv\big( \boldsymbol{\eta} - \mathbf{\Pi}_{h}^{\ell}(\boldsymbol{\eta})\big)  \\[2ex]
  		\disp 	\int_\Omega \bt_h:\big( \boldsymbol{\eta} - \mathbf{\Pi}_{h}^{\ell}(\boldsymbol{\eta})\big)\,+\, \displaystyle \sum_{T\in\Thb} \Bigg\{ \sum_{F\in \mathcal{F}_{h,T}}\int_F  \bu_h \cdot \big( \boldsymbol{\eta} - \mathbf{\Pi}_{h}^{\ell}(\boldsymbol{\eta})\big)\bn\,-\, \int_T \nabla \bu_h : \big( \boldsymbol{\eta} - \mathbf{\Pi}_{h}^{\ell}(\boldsymbol{\eta})\big)\Bigg\}\,,
  	\end{array}
  \end{equation}
  However, since $\bu_h \in \mathbf{P}_{\ell}(F)$, the integrals on the facets in \eqref{eq:estR2-1} vanish due to the property of the Raviart-Thomas interpolator (cf. \eqref{eq:rt-prop1}). Combining this with the Cauchy-Schwarz inequality, the local approximation property of the Raviart-Thomas interpolator (cf. \eqref{eq:RT-approx-prop}), the discrete Hölder inequality, and the stability of the Helmholtz decomposition in Lemma \ref{Hdiv-decomp-1}, we obtain
  \begin{equation}\label{eq:estR2-2}
  	\begin{aligned}
  		\mathcal{R}_2\big( \boldsymbol{\eta} - \mathbf{\Pi}_{h}^{\ell}(\boldsymbol{\eta})\big) &= \sum_{T\in\Thb} \int_T \big(\bt_h - \nabla \bu_{h}\big) : \big( \boldsymbol{\eta} - \mathbf{\Pi}_{h}^{\ell}(\boldsymbol{\eta})\big)  \\[2ex]
  		&\leq \sum_{T\in\Thb} \big\|\bt_h - \nabla \bu_{h} \big\|_{0,T} \big\| \boldsymbol{\eta} - \mathbf{\Pi}_{h}^{\ell}(\boldsymbol{\eta}) \big\|_{0,T} \\[2ex]
  		&\leq C \sum_{T\in\Thb} h_T^{1-d/4} \big\|\bt_h - \nabla \bu_{h} \big\|_{0,T} \, | \boldsymbol{\eta} |_{1,4/3,T} \\[2ex]
  		&\leq C\, \Bigg\{ \sum_{T\in\Thb} h_T^{2-d/2} \big\|\bt_h - \nabla \bu_{h} \big\|^2_{0,T} \Bigg\}^{1/2} \big| \bta \big|_{\bdiv;\Omega}
  	\end{aligned}
  \end{equation}
  On the other hand, again using the definition of $\R_2$ (cf. \eqref{eq:funcR1R2}), the fact that $\bdiv(\curlt{\big(\boldsymbol{\xi} - \mathbf{I}_h(\boldsymbol{\xi})\big)})=0$, an element-wise integration by parts formula, the Cauchy-Schwarz inequality in $\mathbb{L}^2(T)$ and $\mathbb{L}^2(F)$, the local approximation properties of the Clément interpolator (cf. \eqref{eq:clement-prop}), the discrete Cauchy-Schwarz inequality, the uniform boundedness of the number of triangles in the macroelement $\Delta(T)$ and $\Delta(F)$, and the stability of the Helmholtz decomposition in Lemma \ref{Hdiv-decomp-1}, we have
  \begin{equation}\label{eq:estR2-3}
  	\begin{array}{l}
  		\disp \R_2\big( \curlt{\big(\boldsymbol{\xi} - \mathbf{I}_h(\boldsymbol{\xi})\big)}\big)\,=\,\int_\Omega \bta_h: \curlt{\big(\boldsymbol{\xi} - \mathbf{I}_h(\boldsymbol{\xi})\big)}  \\[2ex]
  		\quad \disp =\sum_{T\in\Thb}  \int_T \curlt{(\bt_h)} : \big(\boldsymbol{\xi} - \mathbf{I}_h(\boldsymbol{\xi})\big) \,+\,\sum_{F\in\mathcal{F}_{h}} \int_{F}    \jump{\tncomp{(\bt_h)}} : \big(\boldsymbol{\xi} - \mathbf{I}_h(\boldsymbol{\xi})\big)\\[2ex]
  		\quad \disp \leq\sum_{T\in\Thb}  \big\| \curlt{(\bt_h)}\big\|_{0,T} \big\| \boldsymbol{\xi} - \mathbf{I}_h(\boldsymbol{\xi})\big\|_{0,T} \,+\,\sum_{F\in\mathcal{F}_{h}} \big\|\jump{\tncomp{(\bt_h)}} \big\|_{0,F} \big\|\boldsymbol{\xi} - \mathbf{I}_h(\boldsymbol{\xi})\big\|_{0,F} \\[2ex]
  		\quad \disp \leq \sum_{T\in\Thb} c_1 h_T \big\| \curlt{(\bt_h)}\big\|_{0,T} \big\| \boldsymbol{\xi}\big\|_{1,\Delta(T)} \,+\,\sum_{F\in\mathcal{F}_{h}} c_2 h_F^{1/2} \big\|\jump{\tncomp{(\bt_h)}} \big\|_{0,F} \big\|\boldsymbol{\xi} \big\|_{0,\Delta(F)}\\[2ex]
  		\quad \disp \leq C\,\Bigg\{  \sum_{T\in\Thb} h_T^2 \big\| \curlt{(\bt_h)}\big\|^{2}_{0,T} \,+\, \sum_{F\in\mathcal{F}_{h}} h_F \big\|\jump{\tncomp{(\bt_h)}} \big\|^2_{0,F} \Bigg\}^{1/2} \|\bta\|_{\bdiv_{4/3};\Omega}
  	\end{array}
  \end{equation}
 Finally, the estimate \eqref{eq:est-R2} for $\R_2$ follows from the definition of the dual norm in $\mathbb{H}_{0}(\bdiv_{4/3};\Omega)$, utilizing the identity \eqref{eq:dif-helm-gen} with $\bta_h$ as defined in \eqref{eq:disc-helm-R2}, which yields \eqref{eq:dif-helm-R2} and the corresponding bounds \eqref{eq:estR2-2} and \eqref{eq:estR2-3}. For the case $d=2$, it suffices to employ the corresponding Helmholtz decomposition provided by Lemma \ref{Hdiv-decomp-1}, part (a), which involves the curl of the vector-valued function $\boldsymbol{\xi} \in \mathbf{H}^1(\Omega)$ and a respective integration-by-parts formula used in \eqref{eq:estR2-3}.
\end{proof}

The following result gives the bound for $\widetilde{\R}_{2}.$ 

\begin{lem}\label{lem:est4}
	There exists a positive constant $C_5>0,$  independent of $h$,  such that
	\begin{equation*}
			\big\|\widetilde{\R}_2\big\|_{\mathbf{H}_{\Gamma}(\div_{4/3};\Omega)'} \,\leq \disp
			C_5\ \Bigg\{ \sum_{T\in \Thb}   \overline{\boldsymbol{\Theta}}_{2,T}^{2}\Bigg\}^{1/2}\,,
	\end{equation*}	
	where, for all $T\in\Thb$,  
	\begin{equation*}
		\disp \overline{\boldsymbol{\Theta}}_{2,T}^{2} \,=\, h_T^{2-d/2} \big\|\widetilde{\bt}_h - \nabla \varphi_{h} \big\|^2_{0,T} \,+\,  h_T^2 \big\| \curlv{(\widetilde{\bt}_h)} \big\|^{2}_{0,T} \,+\, \sum_{F\in\mathcal{F}^{\, \rm i}_{h,T}} h_F \big\|\jump{\tncomp{(\widetilde{\bt}_h)}} \big\|^2_{0,F}\,.
	\end{equation*}
\end{lem}
\begin{proof} It follows by  the same arguments as in the previous lemma, utilizing the Helmholtz decomposition  for the space $\mathbf{H}_{\Gamma}(\div_{4/3};\Omega)$ (cf. Lemma \ref{Hdiv-decomp-2}), and noting that the respective component $\widetilde{\boldsymbol{\xi}}$ has zero trace on $\Gamma$, and so does its discrete version $\mathrm{I}_h(\widetilde{\boldsymbol{\xi}})$. Consequently, after applying  integration by parts in the analogous estimate \eqref{eq:estR2-3}, the integrals on the boundary faces vanish.
\end{proof}

As a result, the reliability of   $\boldsymbol{\Theta}$ (cf. Theorem \ref{thm:reliability}) is ensured by Lemmas \ref{lem:est1}, \ref{lem:est2}, \ref{lem:est3}, and \ref{lem:est4}.

\subsection{Residual-based a posteriori error estimator efficiency}\label{section43}

The objective of this section is to show the efficiency property of the error indicator $\boldsymbol{\Theta}$ defined by \eqref{eq:global-estimator}-\eqref{eq:local-estimator-1}-\eqref{eq:local-estimator-2}. The result is presented as follows.
\begin{thm}\label{thm:effi}
		There exists $C_{\rm eff} > 0$, independent of $h$,  such that $\mathbf{\Theta}$ defined by \eqref{eq:global-estimator} satisfies
	\begin{equation*}
	C_{\rm eff} \,\mathbf{\Theta}\,\leq\,	\|(\vec{\bu},\bsi)-(\vec{\bu}_h,\bsi_{h})\|_{ \mathbf{H} \times \mathbb{H}_0(\div_{4/3}; \Omega)}+\|(\vec{\varphi},\widetilde{\bsi})-(\vec{\varphi}_h,\widetilde{\bsi}_{h})\|_{\widetilde{\mathbf{H}} \times \mathbf{H}_\Gamma(\div_{4/3}}.
	\end{equation*}	
\end{thm}

 Most of term defining the a posteriori error indicator the error indicator $\boldsymbol{\Theta}$ appear in related works addressing a posteriori error analysis based on Banach spaces-based mixed finite element methods. The following result summarize the estimated of these terms in term of local error approximation.

\begin{lem}
	There exist positive constants $C_i$, $i\in\{1,2,...,10\}$, all of them independent of $h$, such that 
	{\small{
	\begin{itemize}
		\item[(a)]  $\Big\| \bsi_h^{\tt d} - 2\mu(\varphi_h + \alpha)\bt_{h,\mathrm{sym}} - \dfrac{1}{2} \left( \bu_h \otimes \bu_h \right)^{\tt d} \Big\|_{0,T} \leq C_1\,\Big\{ \|\bu-\bu_h\|_{0,4,T}\,+\, \|\bt-\bt_h\|_{0,T} +\|\varphi-\varphi_h\|_{0,4,T}\Big\}   $ 
		\item[(b)]   $\Big\| \widetilde{\bsi}_h - \kappa \widetilde{\bt}_h + \dfrac{1}{2} \varphi_h \bu_h + U(\varphi_h + \alpha) \widehat{\mathbf{e}}_d \Big\|_{0,T}\leq C_2\,\Big\{\| \widetilde{\bsi}- \widetilde{\bsi}_h\|_{\div_{4/3},T}+ \|\bu-\bu_h\|_{0,4,T}\,+\, \| \widetilde{\bt}- \widetilde{\bt}_h\|_{0,T} +\|\varphi-\varphi_h\|_{1,\Omega}\Big\}$ 
			\item[(c)]  $\Big\| \bdiv \, \bsi_h - \dfrac{1}{2} \bt_h \bu_h + \boldsymbol{f} - g[1+\gamma(\varphi_h+\alpha)]\widehat{\mathbf{e}}_d \Big\|_{0,4/3,T}\leq C_3\,\Big\{\| \bsi- \bsi_h\|_{\bdiv_{4/3},T}+ \|\bu-\bu_h\|_{0,4,T}\,+\, \|\bt-\bt_h\|_{0,T} +\|\varphi-\varphi_h\|_{0,4,T}\Big\}$
			\item[(d)]  $\Big\| \div \, \widetilde{\bsi}_h - \dfrac{1}{2} \widetilde{\bt}_h \cdot \bu_h \Big\|_{0,4/3,T}\leq C_4\,\Big\{\| \widetilde{\bsi}- \widetilde{\bsi}_h\|_{\bdiv_{4/3},T}+ \|\bu-\bu_h\|_{0,4,T}\,+\, \|\widetilde{\bt}-\widetilde{\bt}_h\|_{0,T} \Big\}$
			\item[(e)]   $ h_T^2 \Big\| \curlt(\bt_h) \Big\|_{0,T}^2 \leq C_5\, \|\bt-\bt_h\|_{0,T}$ and  $h_T^2 \Big\| \curlv(\widetilde{\bt}_h) \Big\|_{0,T}^2\leq C_6\, \|\bt-\bt_h\|_{0,T},$  for all $T\in\Thb.$
			\item[(g)]   $ h_F \Big\| \jump{\tncomp(\bt_h)} \Big\|_{0,F}^2\leq C_7\|\bt-\bt_h\|_{0,\omega_{F}}$ and  $h_F \Big\| \jump{\tncomp(\widetilde{\bt}_h)}\Big\|_{0,F}\leq C_8\, \|\widetilde{\bt}-\widetilde{\bt}_h\|_0,\omega_{F}, $ for all $F\in\mathcal{F}_h.$  
			\item[(h)] $h_F \Big\| \tncomp(\widetilde{\bt}_h)\Big\|_{0,F}\leq C_9\, \|\widetilde{\bt}-\widetilde{\bt}_h\|_{0,T_F}, $ for all $F\in\mathcal{F}_h^{\rm b},$ where $T_F$ is the element $T$ for which $F\in\mathcal{F}_{h,T}^{\rm b}$.  
	\end{itemize}
}}
	\end{lem}
\begin{proof}
Estimates (a)-(d) are detailed in \cite[Lemma 3.14]{GIRS-2022} for slightly different residual terms, utilizing their continuous counterparts, Hölder's inequality, and a priori bounds. Estimates (e) and (f) are found in \cite[Lemma 3.15]{GIRS-2022} (see also \cite[Lemmas 4.3 and 4.4]{BGGH-2006}). Finally, estimate (h) follows as in \cite[Lemma 3.15]{BGGH-2006} (see also \cite[Lemma 4.15]{GMS-2010}).	
\end{proof}

For the remaining terms, we have the following result.
\begin{lem}
	There exist positive constants $C_{10}$ and $C_{11},$ independent of $h$, such that 
	\begin{itemize}
		\item[(a)]   $ h_T^{1-d/4} \Big\| \bt_h - \nabla \bu_h \Big\|_{0,T} \leq C_{10}\Big\{   \|\bu-\bu_h\|_{0,4,\Omega}\,+\, \| \bt-\bt_h\|_{0,\Omega}\Big\}$
		\item[(b)]   $h_T^{1-d/4} \Big\| \widetilde{\bt}_h - \nabla \varphi_h \Big\|_{0,T}C_{11}\Big\{   \|\varphi-\varphi_h\|_{0,4,\Omega}\,+\, \| \widetilde{\bt}-\widetilde{\bt}_h\|_{0,\Omega}\Big\}$
	\end{itemize}
\end{lem}
\begin{proof} Define $\boldsymbol{\chi}_T = \bt_h - \nabla \bu_h$ in $T$ and recall that $\bt = \nabla \bu$ in $\Omega.$ Thus, using the element bubble function properties as stated in \eqref{eq:prop-bbf}, local integration by parts, and the inverse inequality (cf. \cite[Lemma 1.138]{ErnG}) with $l=1$, $p=4/3$, $m=0$, and $q=2$, we find  
	\begin{align*}
			\|\boldsymbol{\chi}_T\|^2_{0,T}&\leq	C\|\psi_T^{1/2}\boldsymbol{\chi}_T\|^2_{0,T} = \disp C \int_{T}\psi_T\boldsymbol{\chi}_T: \big(\bt_h-\nabla\bu_h\big)\\[1ex]
		&\disp\leq C\Big\{ \int_{T} \psi_T\boldsymbol{\chi}_T:\big(\nabla\bu-\nabla\bu_h\big):+\int_{T} \psi_T\boldsymbol{\chi}_T:\big(\bt_h-\bt\big) \Big\} \\[1ex]
			&\disp\leq  C\Big\{ \int_{T} \bdiv\big(\psi_T\boldsymbol{\chi}_T\big)\cdot \big(\bu-\bu_h\big):+\int_{T}\psi_T\boldsymbol{\chi}_T : \big(\bt_h-\bt\big)\Big\}\\[1ex]
			&\disp\leq  C\Big\{ \|\psi_T\boldsymbol{\chi}_T\|_{1,4/3, \Omega} \|\bu-\bu_h\|_{0,4,\Omega}+ \|\psi_T\boldsymbol{\chi}_T\|_{0,\Omega} \| \bt-\bt_h\|_{0,\Omega}\Big\} \\[1ex]
			& \disp\leq  C\Big\{ h_T^{-1+d/4}\|\psi_T\boldsymbol{\chi}_T\|_{0, \Omega} \|\bu-\bu_h\|_{0,4,\Omega}+ \|\psi_T\boldsymbol{\chi}_T\|_{0,\Omega} \| \bt-\bt_h\|_{0,\Omega}\Big\} \,,
		\end{align*}
	After using the first inequality of \eqref{eq:prop-bbf}, simplifying and multiplying both sides by $h_T^{1-d/4}$, we obtain the desired result. Note that estimate (b) follows by the same arguments.	
\end{proof}

\section{Numerical results}\label{section6}
This section presents a suite of numerical experiments illustrating the performance of the proposed fully mixed finite element formulation  \eqref{eq:FV_h}, corroborating the theoretical convergence rates established in Theorem \ref{thm:convergencerates}. To ensure the stability of the pair $(\mathbf{H}_h, \mathbb{H}_h^{\boldsymbol{\sigma}})$, as established in Theorem \ref{lem:properties-form-B-h}, the computations are performed on barycentric--refined meshes $\mathcal{T}_h^b$, derived from regular triangulations $\mathcal{T}_h$ of the domain $\Omega$. The discrete spaces for approximating $\bu, \bt, \boldsymbol{\sigma}, \varphi, \widetilde{\bt},$ and $\widetilde{\bsi}$ are constructed as specified in \eqref{eq:FEspaces}, suitable for order $\ell \geq d - 1$.

All computations were performed using the \texttt{FEniCS} finite element environment \cite{alnaes15} (legacy version 2019.1.0), with mesh generation and barycentric refinements carried out in \texttt{Gmsh} \cite{geuzaine09} (version 4.10.3). 
We first assess the theoretical convergence rates by means of manufactured smooth solutions in two and three space dimensions. 
We then investigate the performance of the residual-based a posteriori error estimator for non-smooth solutions on an L-shaped domain, comparing uniform and adaptive mesh refinement strategies. 
Next, we consider a steady bioconvective flow in a square cavity with inclusions, illustrating the robustness of the adaptive method in complex geometries. 
Finally, we study a time-dependent bioconvective benchmark with an Einstein--Batchelor-type viscosity law in a two-dimensional configuration, where the method is shown to capture the formation of bioconvective plume patterns.

The nonlinear systems are solved using the Newton--Raphson method with the SNES framework from PETSc and linear systems solved with the MUMPS direct solver. The Newton  initial guess is set to $(\bu_h^{(0)}, \varphi_h^{0}, \bt_h^{(0)}) = (\mathbf{0}, 0, \mathbf{0})$. Successive approximations are generated at each iteration, denoted as
\[\textbf{coeff}^{(m)}:=(\bu_{h}^{(m)},\bt_h^{(m)},\bsi_{h}^{(m)},\varphi_{h}^{(m)}, \widetilde{\bt}_h^{(m)},\widetilde{\bsi_h}^{(m)})\,\quad \forall\,m\geq1\]
where $\mathbf{coeff}^{m} \in \mathbb{R}^N$, with $N$ representing the total number of degrees of freedom across the finite element family $(\mathbf{H}^\bu_h, \, \mathbb{H}^\bt_h, \, \mathbb{H}^\bsi_h, \, \mathrm{H}^\varphi_h, \, \mathbf{H}^{\widetilde{\bt}}_h, \, \mathbf{H}^{\widetilde{\bsi}}_h)$.  Convergence is assessed using the relative error between two successive approximation vectors, $\mathbf{coeff}^{m+1}$ and $\mathbf{coeff}^{m}$, as follows:
\[\dfrac{\|\mathbf{coeff}^{m+1} - \mathbf{coeff}^{m}\|_{l^2}}{\|\mathbf{coeff}^{m+1}\|_{l^2}} < \texttt{tol},\]
where $\texttt{tol}$ denotes a predefined tolerance level set to $10^{-7}$ for both absolute and relative errors and $\|\cdot\|_{l^2}$ is the Euclidean norm in $\mathbb{R}^N$. The  zero-mean condition for $\tr\,\bsi$ and $\varphi$ is enforced via two real Lagrange multipliers. 

Convergence is typically achieved in four iterations for all refinement levels. All volume and surface integrals are evaluated using fourth-order quadrature rules. The a posteriori error indicators are computed locally in piecewise constant spaces and accumulated according to the residual decomposition defined in Section~\ref{section4}. 

For the convergence tests in both uniform and adaptive mesh refinement, barycentric (Alfeld) refinements are generated via the command \texttt{gmsh -barycentric\_refine file.msh} \texttt{-o file-bary.msh -format msh2}, and subsequently converted into \texttt{FEniCS}-compatible \texttt{.xml} meshes using \texttt{dolfin-convert}.

\noindent The individual errors associated with the main unknowns are calculated as
\[e(\bu):= \|\bu-\bu_h\|_{0,4;\Omega},\quad e(\bt):= \|\bt-\bt_h\|_{0;\Omega},\quad e(\bsi):= \|\bsi-\bsi_h\|_{\mathbf{div}_{4/3};\Omega}\]
\[e(\varphi):= \|\varphi-\varphi_h\|_{0,4;\Omega},\quad e(\widetilde{\bt}):= \|\widetilde{\bt}-\widetilde{\bt}_h\|_{0;\Omega},\quad e(\widetilde{\bsi}):= \|\widetilde{\bsi}-\widetilde{\bsi}_h\|_{\mathrm{div}_{4/3};\Omega}\]
and the error associated with the postprocessed pressure as
\[e(p):= \|p-p_h\|_{0,\Omega}\,.\]

Furthermore,  for any $\star \in \{\bu, \bt, \bsi, \varphi, \widetilde{\bt}, \widetilde{\bsi}, p\}$, we let $r(\star)$, we define the experimental convergence rate, $r(\star)$, as 
\[
r(\star):= -\,d\,\frac{\log\bigl(e(\star)/e'(\star)\bigr)}{\log(N/N')},
\]
where $N$ and $N'$ represent the total degrees of freedom of two successive meshes,  and $e(\star)$ and $e'(\star)$ are the corresponding errors associated with $\star$ on these meshes, respectively.

Finally, the effectivity index associated to the global error estimator $	\mathbf{\Theta}$ is defined as
\[
\mathrm{eff}\,=\,\dfrac{ e_\mathrm{tot} }{	\mathbf{\Theta}}\,, \quad \mbox{where} \quad   e_\mathrm{tot}\,=\,\left\{ { e}(\bu)^2\,
  \,+\, e(\bt) \, + { e}(\bsi)^2+\,{ e}(\varphi)^2\,+\, e(\widetilde{\bt}) \, + { e}(\widetilde{\bsi})^2\,\right\}^{1/2}\,.
\]

\begin{table}[t!]
\centering
\caption{Example 1. Convergence history for the 2D manufactured-solution test on barycentrically refined meshes with the finite element approximation $\mathbf{P}_1-\mathbb{P}_1-\mathbb{RT}_1-\mathrm{P}_1-\mathbf{P}_1-\mathbf{RT}_1 $ (with polynomial degree $\ell=1$). Here, $N$ represents the number of degrees of freedom associated with each barycentric-refined mesh $\Thb$.}
\label{tab:conv2D}
\resizebox{\textwidth}{!}{
\rowcolors{2}{white}{lightgray}
\begin{tabular}{rccccccccccccccccc}
\toprule
\rowcolor{white}
$N$ & $h$ & $e(\bu)$ & $r(\bu)$ & $e(\bt)$ & $r(\bt)$ & $e(\bsi)$ & $r(\bsi)$ &
$e(\varphi)$ & $r(\varphi)$ & $e(\widetilde{\bt})$ & $r(\widetilde{\bt})$ & $e(\widetilde{\bsi})$ & $r(\widetilde{\bsi})$ &
$e(p)$ & $r(p)$ & $\mathrm{eff}$ & $\mathrm{it}$\\
\midrule
962    & 1.414 & 4.64e--01 & --    & 3.88e+00 & --    & 5.55e+00 & --    & 4.22e--01 & --    & 2.04e+00 & --    & 1.53e+01 & --    & 6.54e--01 & --    & 0.53 & 5\\
3794   & 0.707 & 1.65e--01 & 1.49  & 1.85e+00 & 1.07  & 2.24e+00 & 1.31  & 1.08e--01 & 1.96  & 4.85e--01 & 2.07  & 2.62e+00 & 2.55  & 2.31e--01 & 1.50  & 0.43 & 4\\
15074  & 0.354 & 4.26e--02 & 1.95  & 9.08e--01 & 1.03  & 6.49e--01 & 1.79  & 2.72e--02 & 1.99  & 1.29e--01 & 1.91  & 6.50e--01 & 2.01  & 5.70e--02 & 2.02  & 0.37 & 4\\
60098  & 0.177 & 1.09e--02 & 1.97  & 3.98e--01 & 1.19  & 2.08e--01 & 1.64  & 6.94e--03 & 1.97  & 3.33e--02 & 1.96  & 1.61e--01 & 2.01  & 1.33e--02 & 2.10  & 0.36 & 4\\
240002 & 0.088 & 2.59e--03 & 2.07  & 1.43e--01 & 1.47  & 6.27e--02 & 1.73  & 1.74e--03 & 1.99  & 8.40e--03 & 1.99  & 4.02e--02 & 2.00  & 3.20e--03 & 2.06  & 0.36 & 4\\
959234 & 0.044 & 6.22e--04 & 2.06  & 4.21e--02 & 1.77  & 1.72e--02 & 1.87  & 4.36e--04 & 2.00  & 2.10e--03 & 2.00  & 1.00e--02 & 2.00  & 7.87e--04 & 2.03  & 0.35 & 4\\
\bottomrule
\end{tabular}}
\end{table}

\subsection*{Example 1: convergence for smooth solutions in 2D  and 3D}
We first assess the theoretical convergence rates on the square domain $\Omega=(-1,1)^2$ for the steady bioconvection system~\eqref{eqn:bio}. The parameters are fixed as
\[
 \widehat{\mathbf{e}}_{2}=(0,1),\quad \mu(\varphi) = e^{-\varphi}, \quad
\kappa = 1, \quad
g = 1, \quad
\gamma = \alpha = 0.5, \quad
U = 0.01.
\]
A manufactured smooth solution with the following primary unknowns 
\[
\bu(x,y) = \begin{pmatrix}\cos(\pi x)\sin(\pi y)\\ -\sin(\pi x)\cos(\pi y)\end{pmatrix}, \quad
p(x,y) = \sin(\pi x)\cos(\pi y), \quad
\varphi(x,y) = 1 + \sin(\pi x)\sin(\pi y),
\]
is used to derive analytic forms of the mixed variables, the forcing terms and boundary data, ensuring exact satisfaction of the governing equations. Non-homogeneous Dirichlet conditions are applied for $\bu$, and natural flux conditions for $\varphi$.

The mixed finite element spaces employ Raviart--Thomas elements $\mathrm{RT}_\ell$ for flux variables and discontinuous Galerkin elements $\mathrm{DG}_\ell$ for all remaining fields, with $\ell \ge d-1$. The meshes are successively refined barycentrically, yielding characteristic sizes $h$ from approximately $1.4$ down to $4.4\times 10^{-2}$. The total number of degrees of freedom ranges from fewer than $10^3$ to almost two million on the finest grid.

Tables \ref{tab:conv2D} and \ref{tab:conv2Dhigher} provide a summary of the errors and associated convergence rates for finite element approximations using  $\mathbf{P}_\ell-\mathbb{P}_\ell-\mathbb{RT}_\ell-\mathrm{P}_\ell-\mathbf{P}_\ell-\mathbf{RT}_\ell$ families for $\ell=1$ and $\ell=2$, respectively. Table \ref{tab:conv2D} particularly says that, when using $\ell=1,$ the error magnitudes exhibit quadratic convergence rates with respect to the mesh size $h$ for all the variables, aligning with the theoretical expectations  in Theorem \ref{thm:convergencerates}. The effectivity index of the residual-based estimator stabilises around $0.35$, and the nonlinear iteration count remains mesh-independent, achieving the prescribed tolerance   in four iterations. Additionally, Table \ref{tab:conv2Dhigher} illustrates that, by elevating the polynomial order to $\ell=2$, the method now achieves almost third-order convergence for all primal variables, while the stress and flux approximations approach the same asymptotic order $\mathcal{O}(h^3)$. This observation not only demonstrates superior convergence performance but also corroborates the theoretical predictions.  The residual-based error indicator remains asymptotically efficient (the effectivity index decreases slightly compared to the lower order case but still reaches a constant value), reflecting the higher resolution of the discrete solution relative to the estimator  for nonlinear coupled problems of this kind.

\begin{table}[t!]
\centering
\caption{Example 1. Convergence history and iteration count for the fully mixed approximation $\mathbf{P}_2-\mathbb{P}_2-\mathbb{RT}_2-\mathrm{P}_2-\mathbf{P}_2-\mathbf{RT}_2$ ($\ell=2$).}
\label{tab:conv2Dhigher}
\resizebox{\textwidth}{!}{
\rowcolors{2}{white}{lightgray}
\begin{tabular}{rccccccccccccccccc}
\toprule
\rowcolor{white}
$N$ & $h$ & $e(\bu)$ & $r(\bu)$ & $e(\bt)$ & $r(\bt)$ & $e(\bsi)$ & $r(\bsi)$ &
$e(\varphi)$ & $r(\varphi)$ & $e(\widetilde{\bt})$ & $r(\widetilde{\bt})$ & $e(\widetilde{\bsi})$ & $r(\widetilde{\bsi})$ &
$e(p)$ & $r(p)$ & $\mathrm{eff}$ & $\mathrm{it}$\\
\midrule
1946    & 1.414 & 1.88e--01 & --    & 1.81e+00 & --    & 1.46e+00 & --    & 1.58e--01 & --    & 8.26e--01 & --    & 1.93e+00 & --    & 3.46e--01 & --    & 0.15 & 4\\
7706    & 0.707 & 2.57e--02 & 2.87  & 5.03e--01 & 1.85  & 2.79e--01 & 2.39  & 1.21e--02 & 3.70  & 6.48e--02 & 3.67  & 1.02e--01 & 4.25  & 6.81e--02 & 2.35  & 0.13 & 4\\
30674   & 0.354 & 3.64e--03 & 2.82  & 1.44e--01 & 1.80  & 6.44e--02 & 2.11  & 1.38e--03 & 3.14  & 8.64e--03 & 2.91  & 1.39e--02 & 2.87  & 1.37e--02 & 2.31  & 0.17 & 4\\
122402  & 0.177 & 3.44e--04 & 3.40  & 2.42e--02 & 2.57  & 9.83e--03 & 2.71  & 1.69e--04 & 3.03  & 1.08e--03 & 2.99  & 1.62e--03 & 3.10  & 1.91e--03 & 2.85  & 0.19 & 4\\
489026  & 0.088 & 3.37e--05 & 3.35  & 3.24e--03 & 2.90  & 1.29e--03 & 2.94  & 2.10e--05 & 3.01  & 1.34e--04 & 3.02  & 1.91e--04 & 3.08  & 2.45e--04 & 2.96  & 0.19 & 4\\
1954946 & 0.044 & 3.84e--06 & 3.14  & 4.13e--04 & 2.97  & 1.63e--04 & 2.98  & 2.62e--06 & 3.00  & 1.66e--05 & 3.00  & 2.33e--05 & 3.03  & 3.08e--05 & 2.99  & 0.19 & 4\\
\bottomrule
\end{tabular}}
\end{table}

\begin{figure}[t!]
\centering
\includegraphics[width=0.244\textwidth]{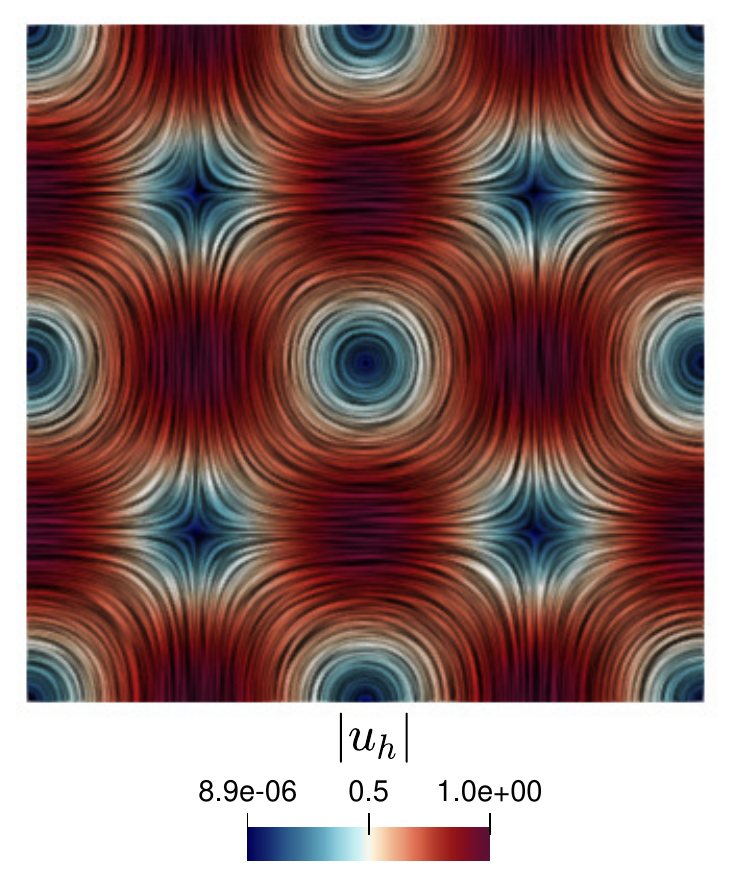}
\includegraphics[width=0.244\textwidth]{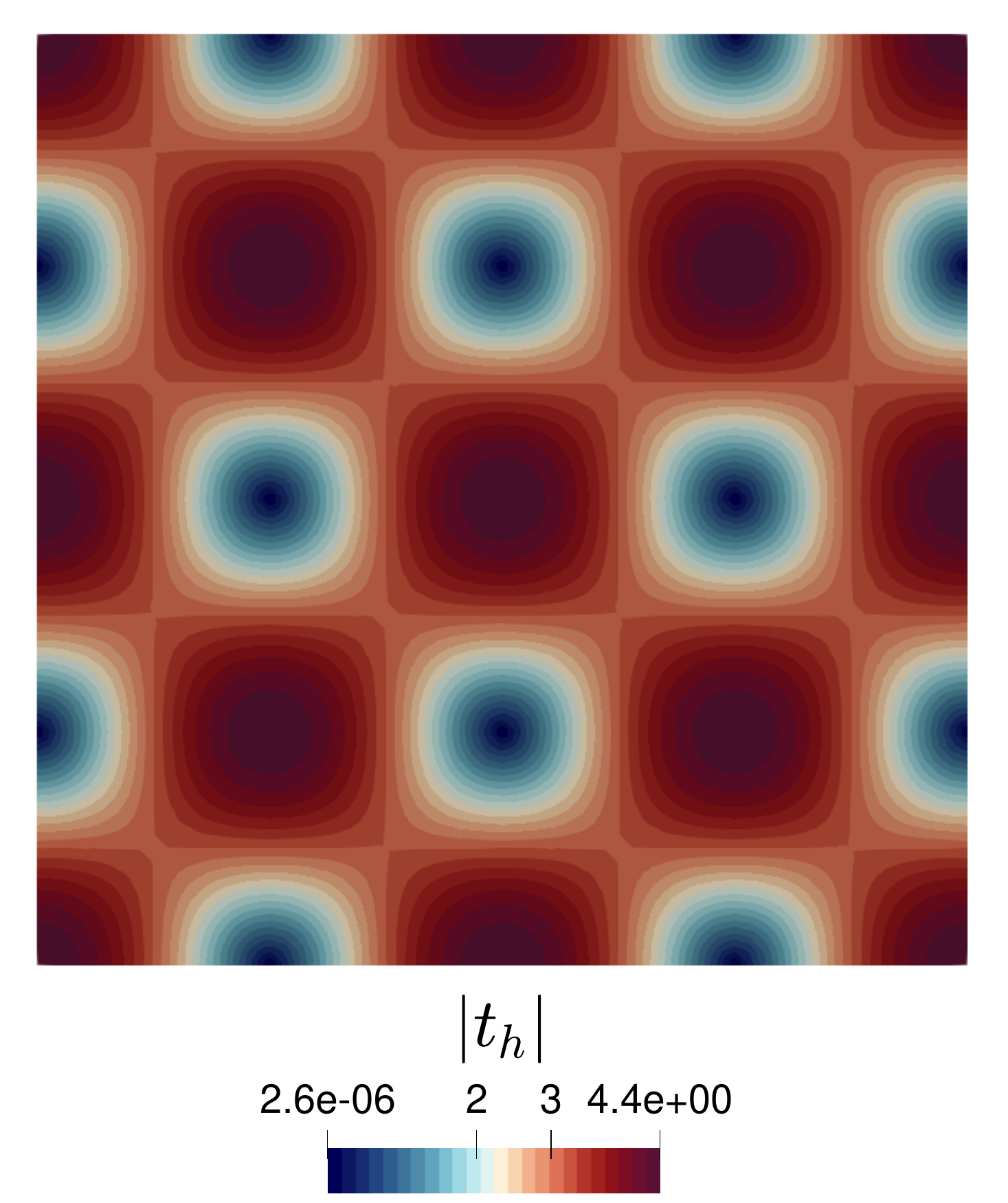}
\includegraphics[width=0.244\textwidth]{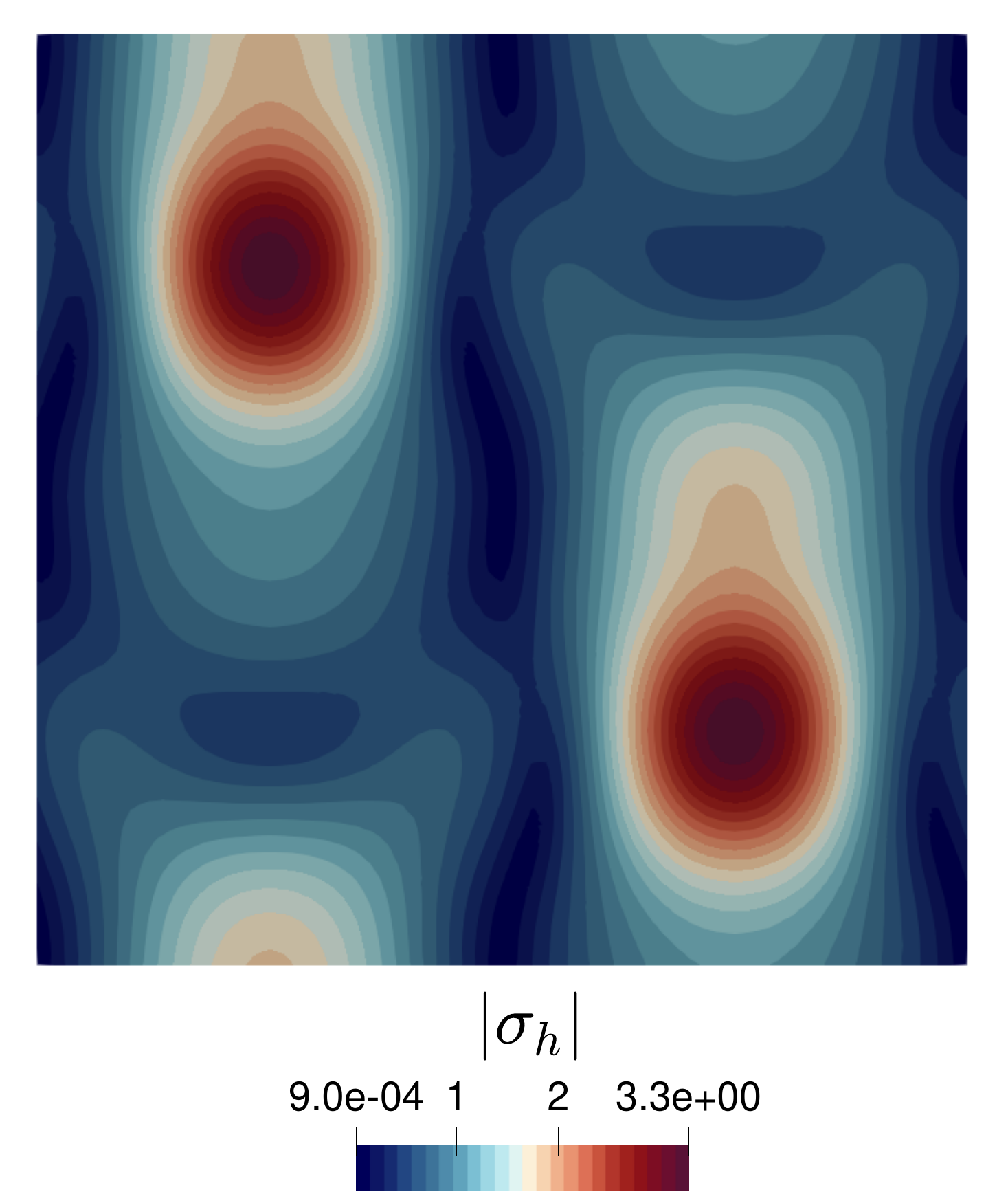}\\
\includegraphics[width=0.244\textwidth]{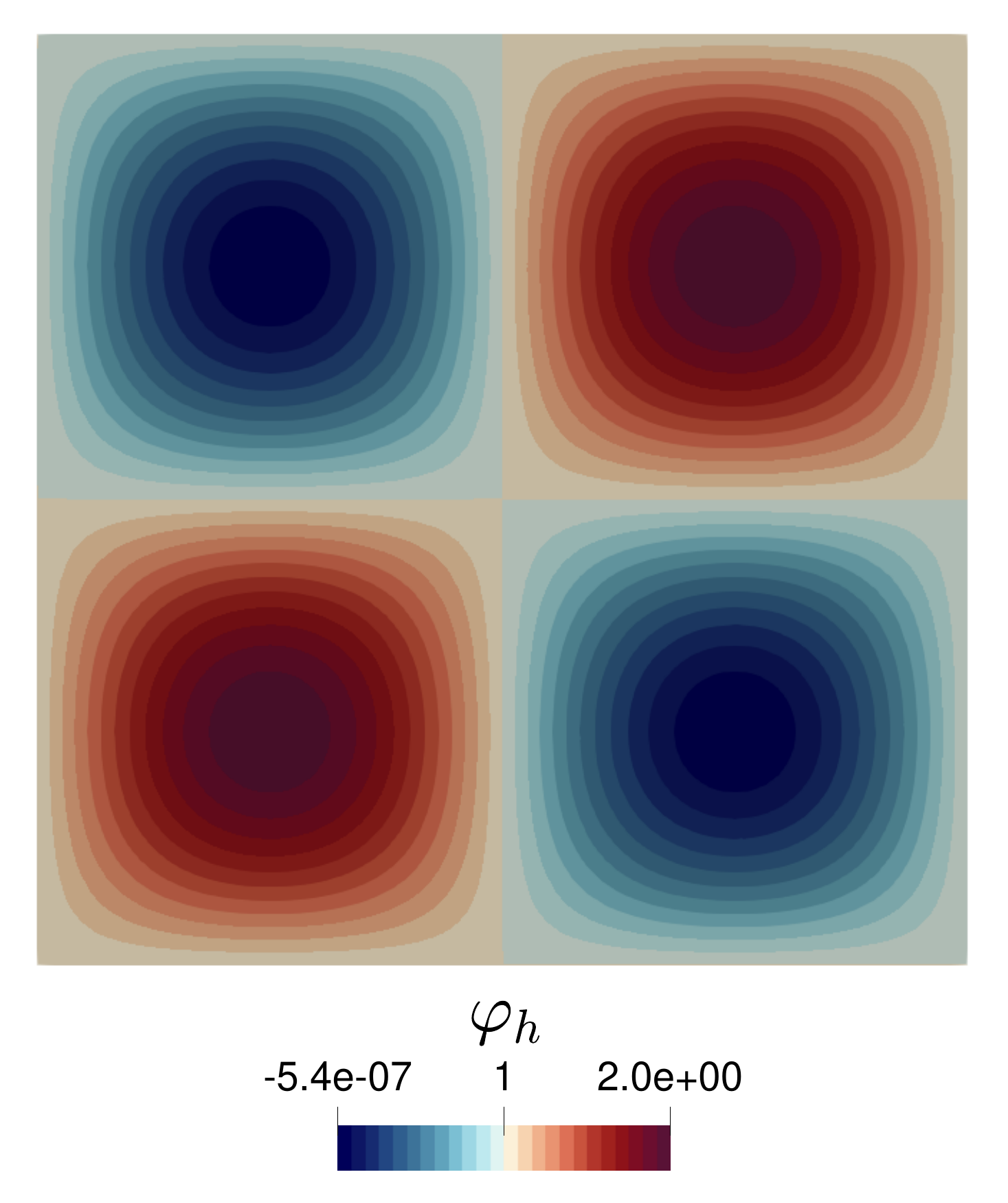}
\includegraphics[width=0.244\textwidth]{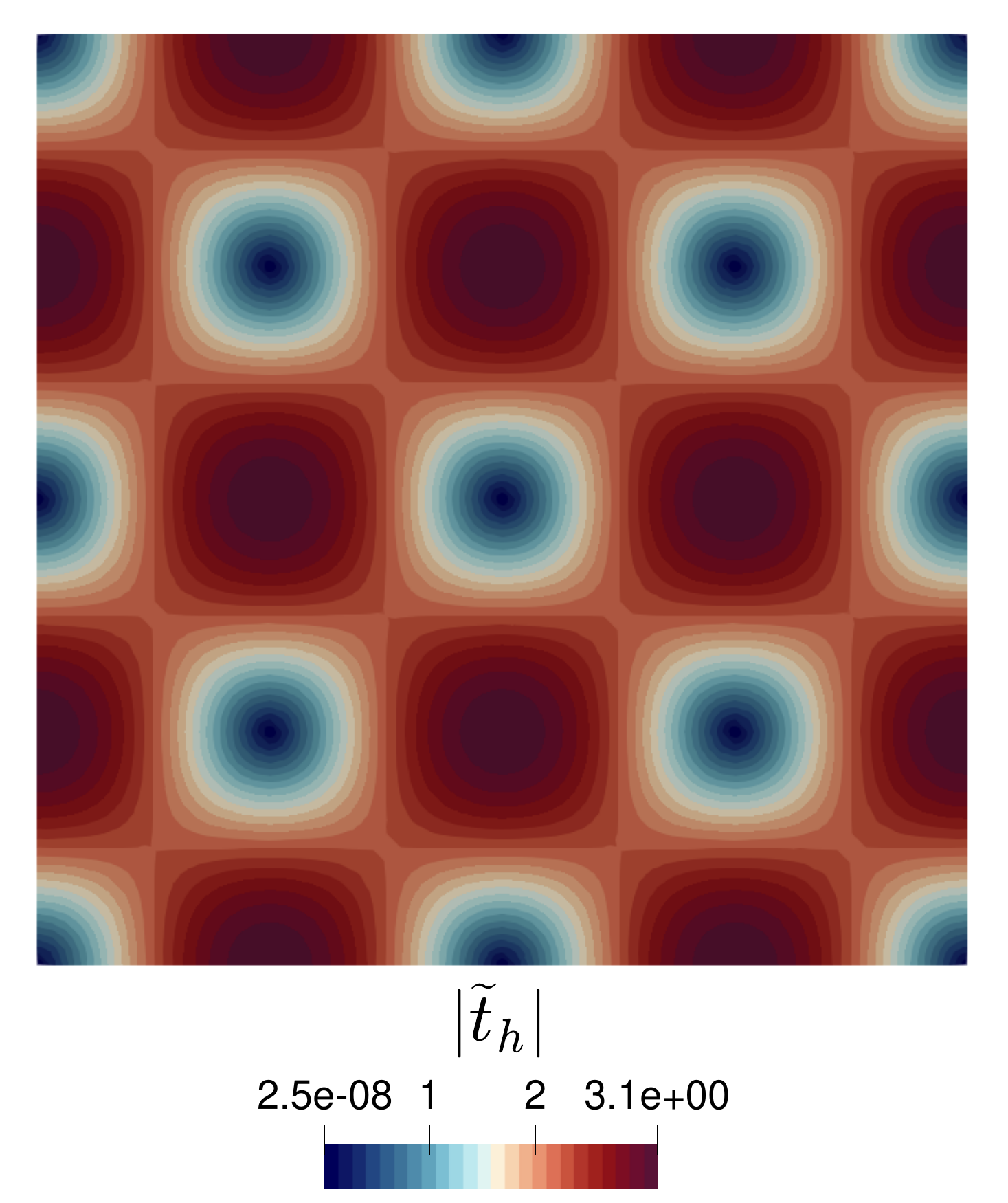}
\includegraphics[width=0.244\textwidth]{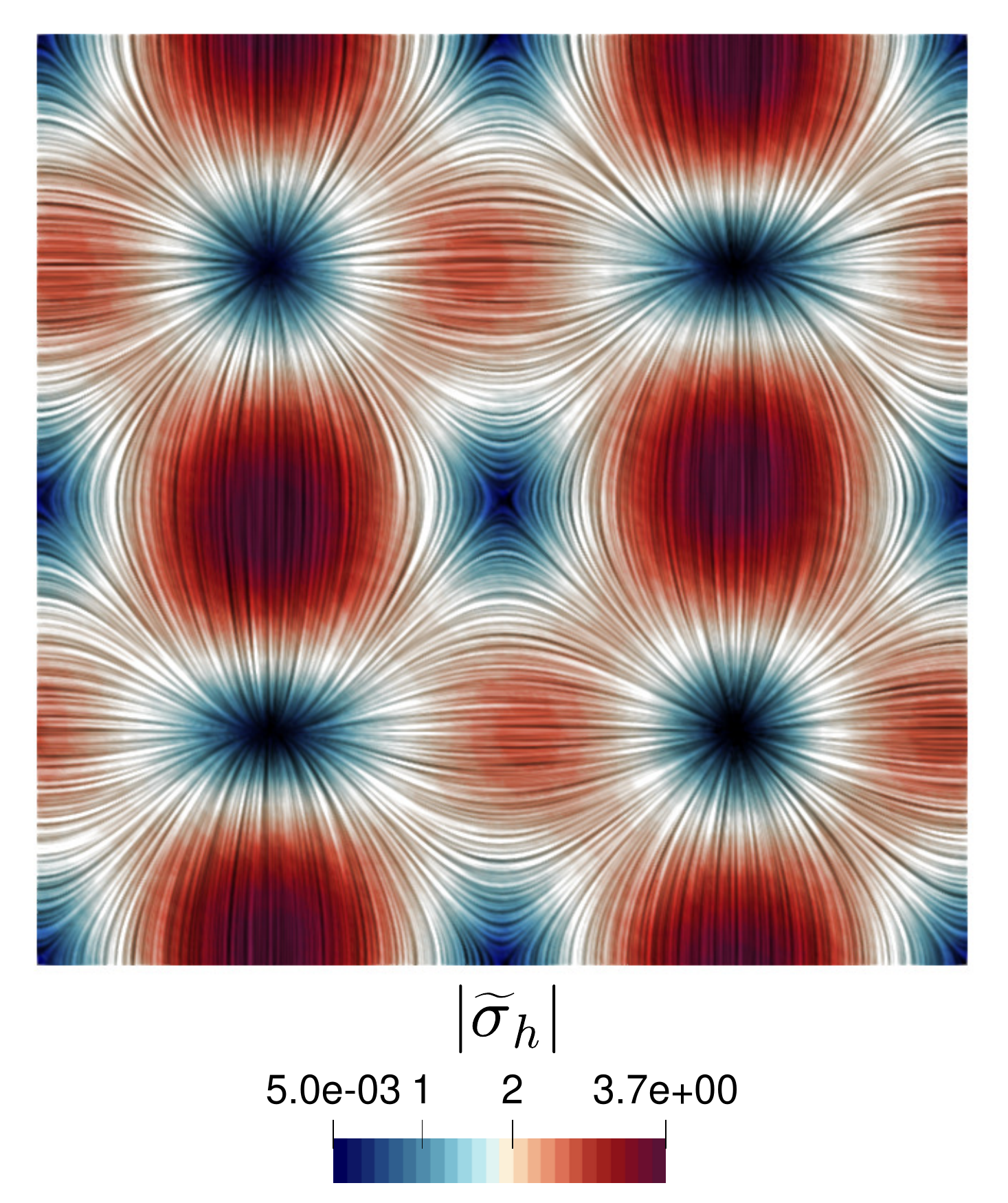}
\includegraphics[width=0.244\textwidth]{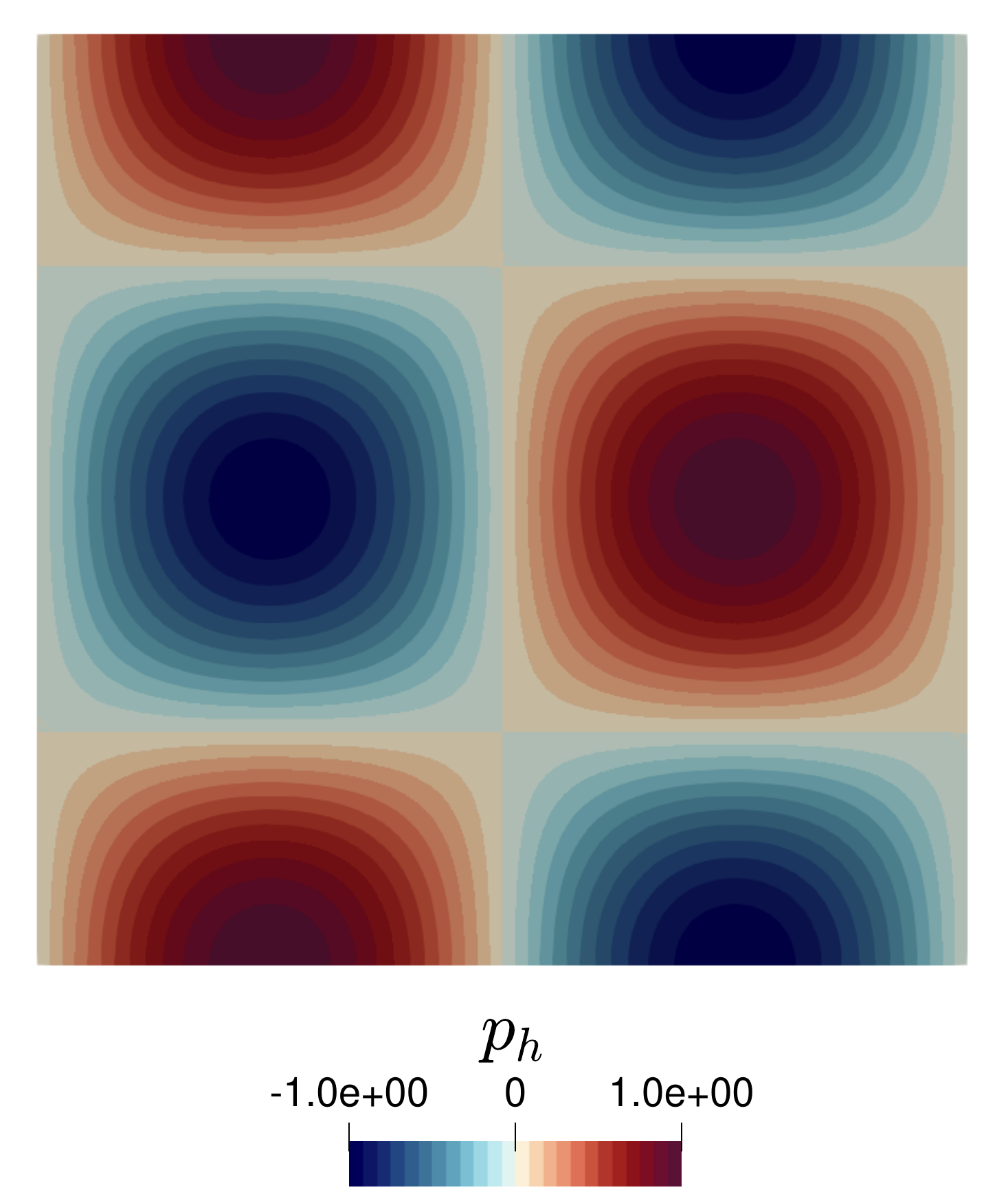}

\vspace{-4mm}
\caption{Example 1. Approximate solutions for the 2D manufactured problem with $\ell=2$: velocity line integral contours, strain rate magnitude, pseudostress magnitude, concentration profile, diffusive flux and total flux magnitudes,  and postprocessed pressure.}
\label{fig:convPlots}
\end{figure}

In Figure \ref{fig:convPlots}, we present the approximate solutions generated by our fully mixed technique on a barycentric--refined mesh with $N=1954946$ degrees of freedom.

\begin{figure}[t]
\centering
\includegraphics[width=0.245\textwidth]{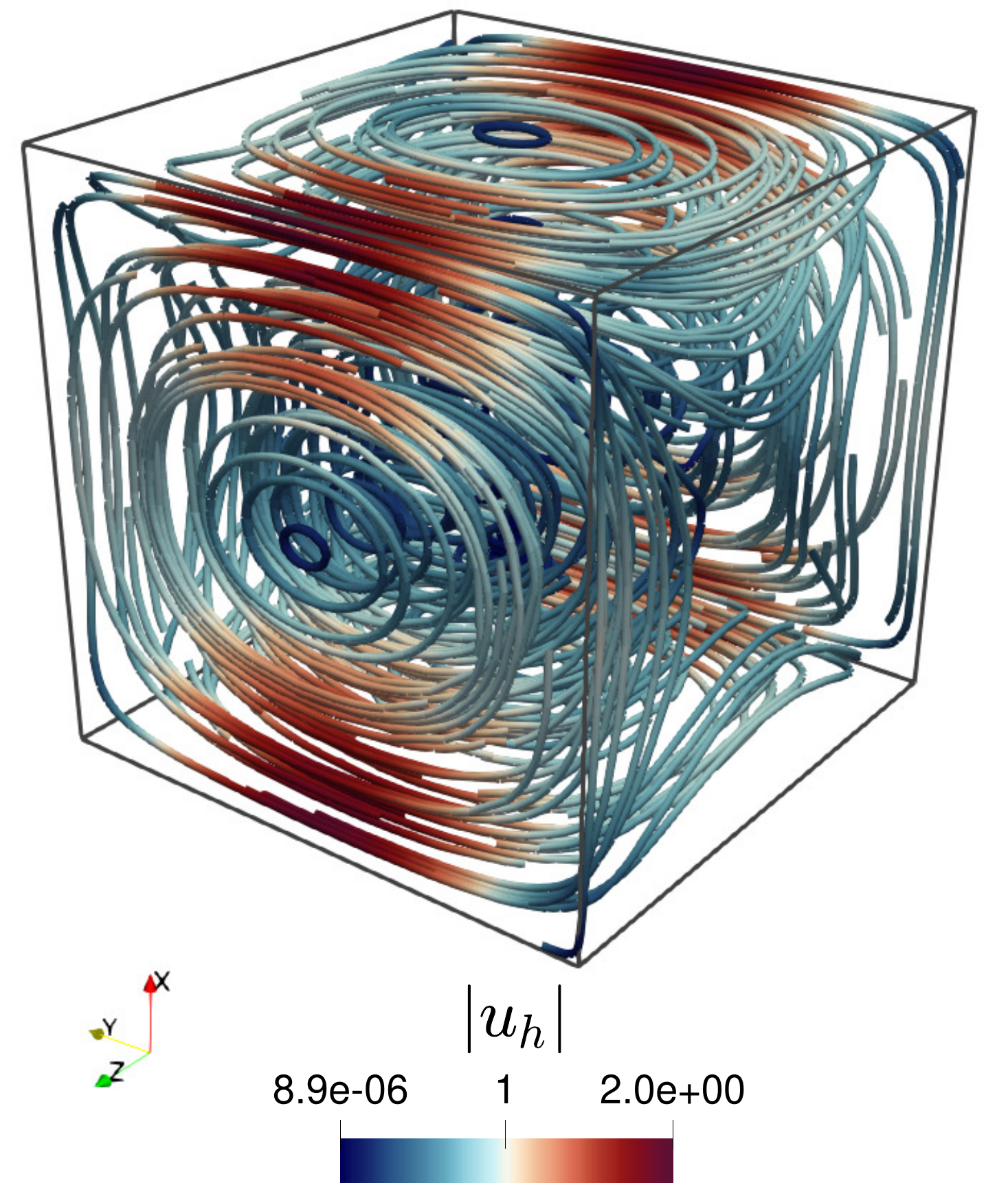}
\includegraphics[width=0.245\textwidth]{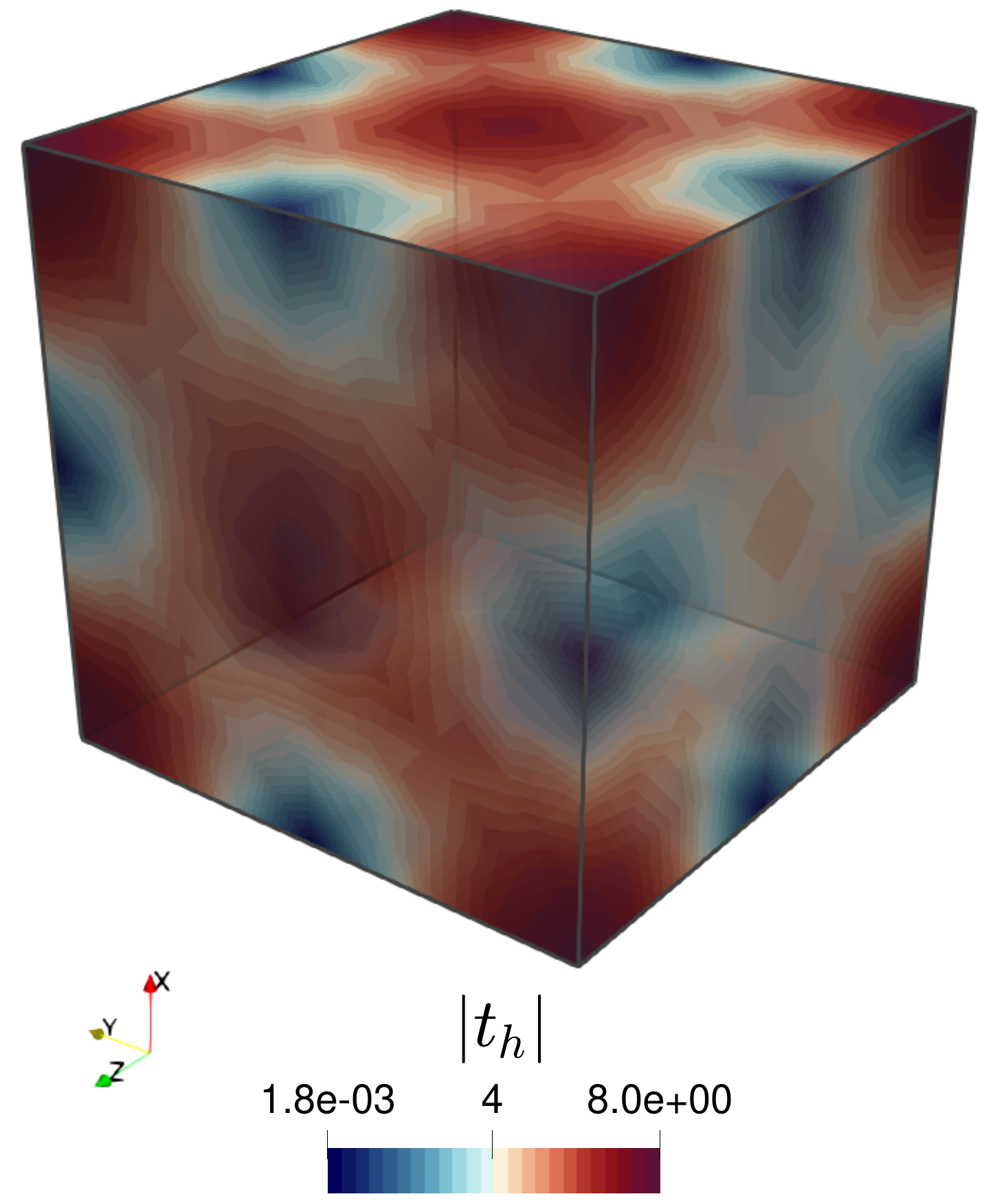}
\includegraphics[width=0.245\textwidth]{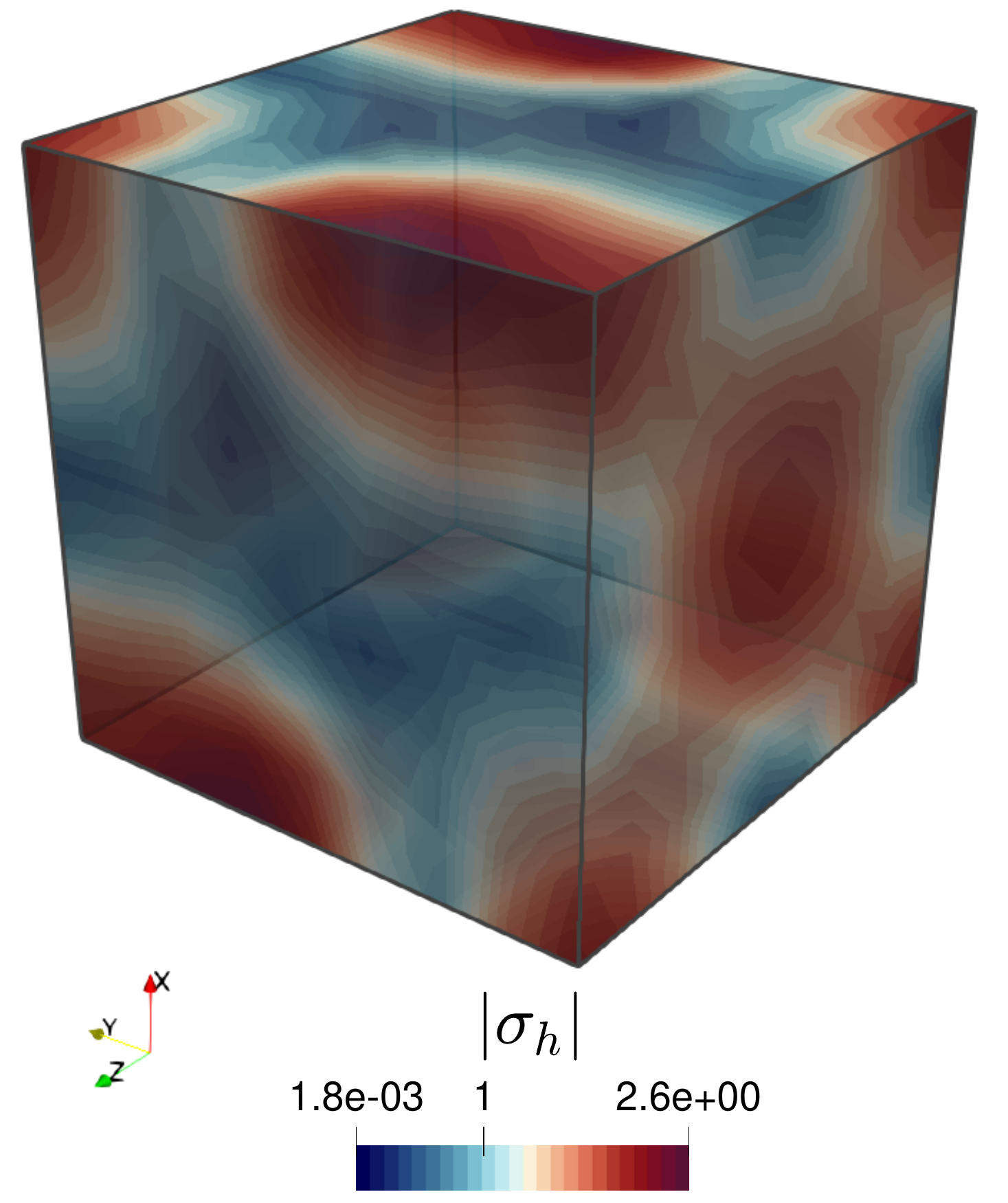}\\
\includegraphics[width=0.244\textwidth]{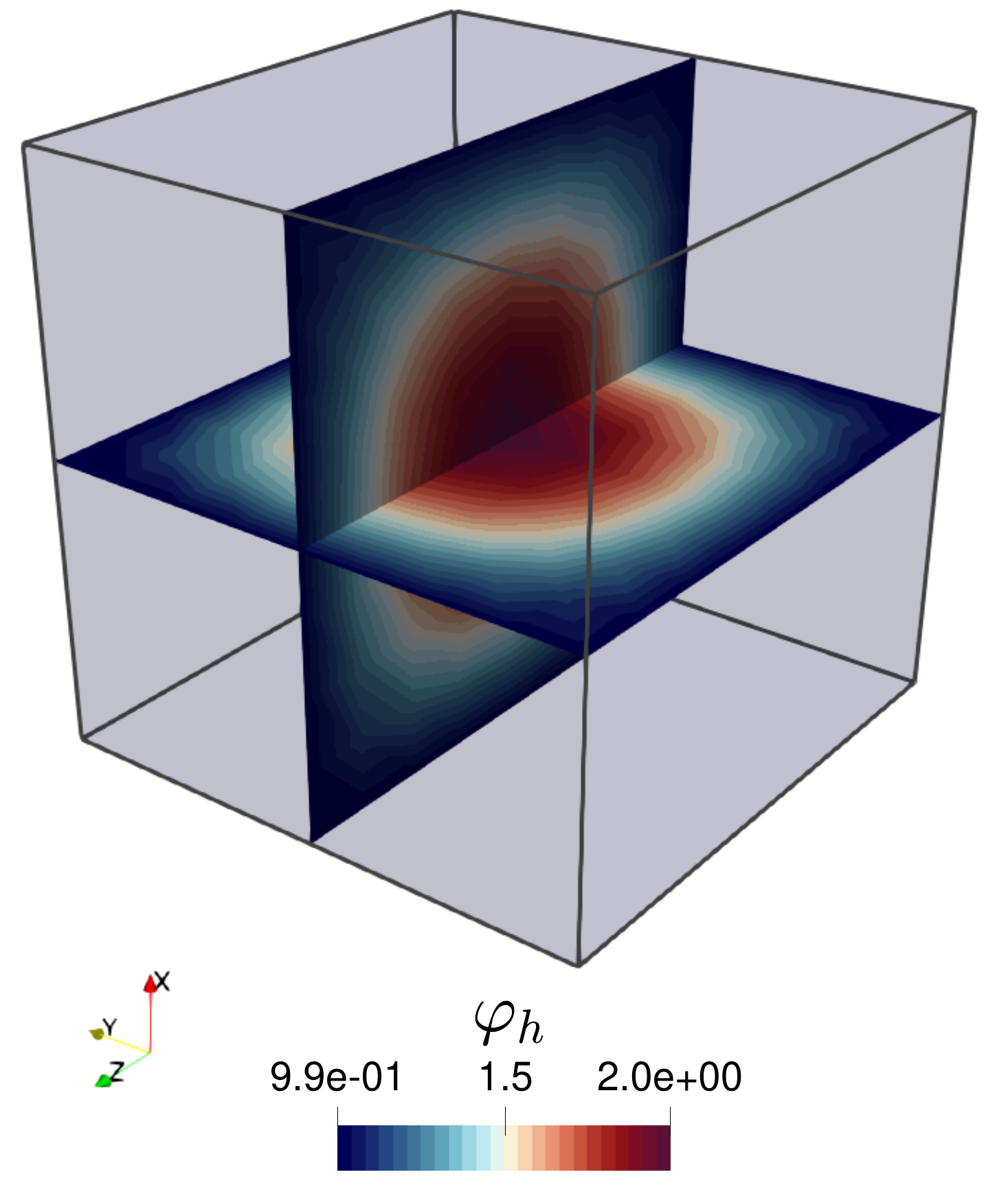}
\includegraphics[width=0.244\textwidth]{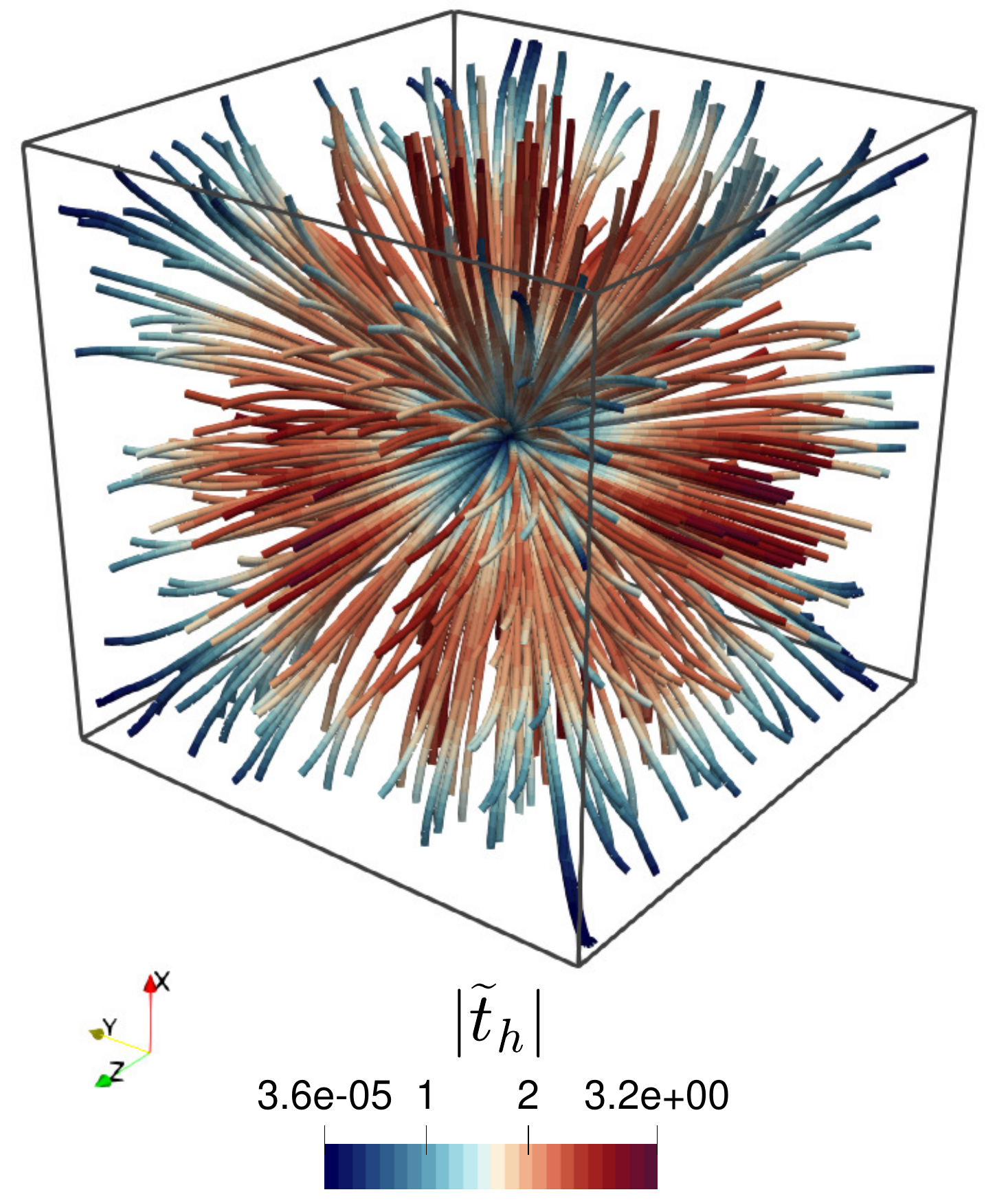}
\includegraphics[width=0.244\textwidth]{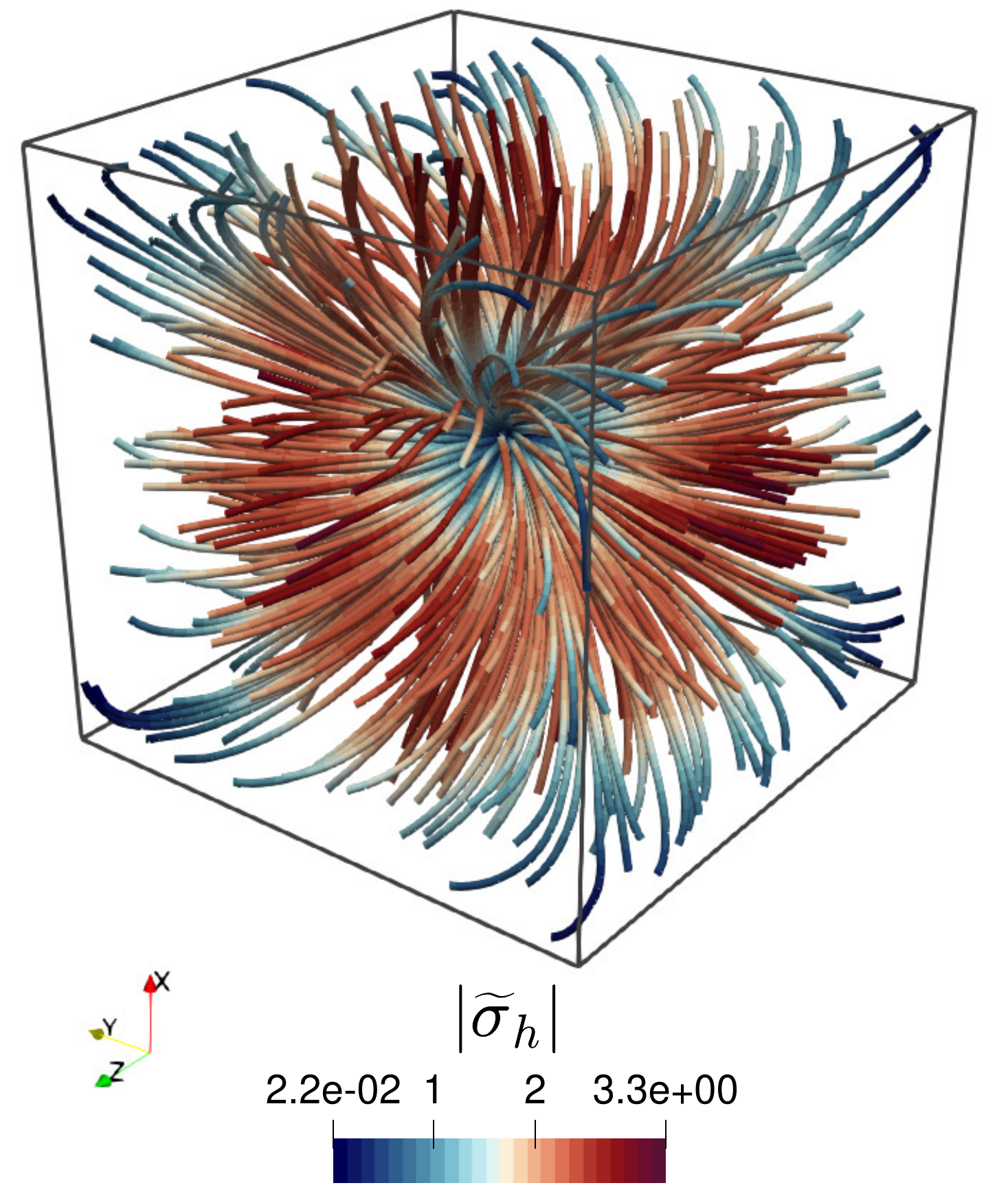}
\includegraphics[width=0.244\textwidth]{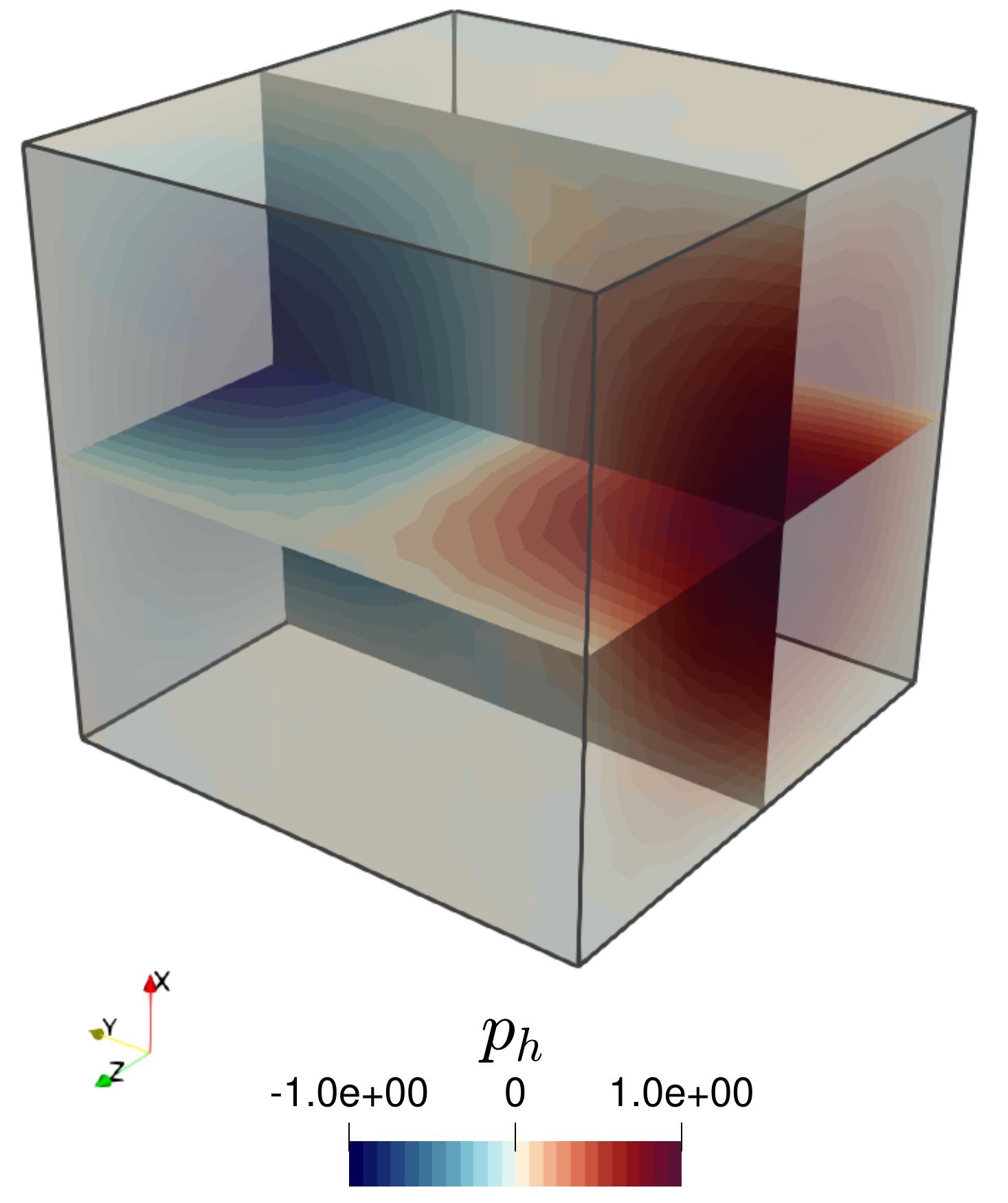}

\vspace{-4mm}
\caption{Example 1. Approximate solutions for the 3D manufactured problem with $\ell=2$: velocity streamlines, strain rate magnitude, pseudostress magnitude, slices of concentration profile, diffusive flux and total flux streamlines,  and slices of postprocessed pressure.}
\label{fig:convPlots3d}
\end{figure}

We further confirm the convergence properties in 3D by examining the approximate solutions of the bioconvection system with the following manufactured solutions defined on the unit cube domain $\Omega = (0,1)^3$:
\begin{gather*}
\bu(x,y,z) = \begin{pmatrix} 
\sin(\pi x)\cos(\pi y)\cos(\pi z)\\
 -2\cos(\pi x)\sin(\pi y)\cos(\pi z)\\
  \cos(\pi x)\cos(\pi y)\sin(\pi z)\end{pmatrix}, \quad 
p(x,y,z) = \sin(\pi x)\cos(\pi y)\sin(\pi z), \\ 
\varphi(x,y,z) = \sin(\pi x)\sin(\pi y)\sin(\pi z)+1.\end{gather*}

We use the polynomial degree $\ell=2$ (we recall that $\ell=1$ is not sufficient for the 3D case) and set the coefficients for this case as $\kappa  = g = 1$, $\gamma = \alpha = 0.5$, $U   =0.01$. The observed rates confirm the optimal approximation properties of the mixed finite element spaces under barycentric refinement. However it seems that the effectivity index decreases with $h$. The number of required nonlinear iterations to converge is 4 in all mesh refinements. We also note from Figure~\ref{fig:convPlots3d} that even for a relatively coarse mesh all fields are very accurate, since the number of degrees of freedom is quite large (more than 3M).

\begin{table}[t!]
\centering
\caption{Example 1. Convergence history for the 3D steady bioconvection test with polynomial degree $\ell=2$.}
\label{tab:conv3Dbioconvection}
\resizebox{\textwidth}{!}{
\rowcolors{2}{white}{lightgray}
\begin{tabular}{rccccccccccccccccc}
\toprule
\rowcolor{white}
$N$ & $h$ & $e(\bu)$ & $r(\bu)$ & $e(\bt)$ & $r(\bt)$ & $e(\bsi)$ & $r(\bsi)$ &
$e(\varphi)$ & $r(\varphi)$ & $e(\widetilde{\bt})$ & $r(\widetilde{\bt})$ & $e(\widetilde{\bsi})$ & $r(\widetilde{\bsi})$ &
$e(p)$ & $r(p)$ & $\mathrm{eff}$ & $\mathrm{it}$\\
\midrule
   6050   & 1.732 & 2.26e--01 & --    & 1.89e+00 & --    & 1.05e+00 & --    & 1.05e--01 & --    & 3.36e--01 & --    & 1.90e+00 & --    & 1.65e--01 & --    & 0.094 & 4\\
  47810   & 1.225 & 4.93e--02 & 4.40  & 5.74e--01 & 3.43  & 2.44e--01 & 4.22  & 1.88e--02 & 4.98  & 6.10e--02 & 4.93  & 2.11e--01 & 6.34  & 3.97e--02 & 4.11  & 0.048 & 4\\
 380162   & 0.612 & 7.81e--03 & 2.66  & 1.50e--01 & 1.94  & 5.38e--02 & 2.18  & 2.91e--03 & 2.69  & 8.79e--03 & 2.79  & 2.64e--02 & 3.00  & 7.97e--03 & 2.32  & 0.020 & 4\\
3032066   & 0.306 & 9.67e--04 & 3.01  & 2.65e--02 & 2.50  & 8.71e--03 & 2.63  & 3.84e--04 & 2.92  & 1.16e--03 & 2.93  & 3.37e--03 & 2.97  & 1.12e--03 & 2.83  & 0.0049 & 4\\
\bottomrule
\end{tabular}}
\end{table}

\subsection*{Example 2: convergence for non-smooth solutions under uniform and adaptive mesh refinement}

We perform convergence tests for the mixed finite element scheme for the bioconvection equations using the non-regular Verf\"urth manufactured solutions on an L-shaped domain
(see, e.g., \cite{carstensen2005}):
\begin{gather*}
    \bu(r,\theta) = r^\lambda \begin{pmatrix}(1+\lambda)\sin(\theta)\psi(\theta) + \cos(\theta) \psi'(\theta)\\ \sin(\theta)\psi'(\theta) - (1+\lambda)\cos(\theta)\psi(\theta)
            \end{pmatrix},   \quad    
    p(r,\theta) = -\nu \frac{r^{\lambda -1 }}{1-\lambda}((1+\lambda)^2\psi'(\theta) + \psi'''(\theta)), \\
    \varphi(r,\theta) = r^{\frac23}\sin(\frac23\theta),
\end{gather*}

in polar coordinates centered at the origin $(r,\theta) \in (0,\infty)\times (0,\frac{3\pi}{2})$, where 
\[ \psi(\theta) = \frac{\sin((1+\lambda)\theta)\cos(\lambda w)}{1+\lambda} - \cos ((1+\lambda) \theta) - \frac{\sin((1-\lambda)\theta)\cos(\lambda w)}{1-\lambda} + \cos ((1-\lambda) \theta).\]
Here $\lambda = \frac{856399}{1572864} \approx 0.5444837$ is the smallest positive solution of $\sin(\lambda w) + \lambda \sin (w) = 0$, and we take $w = \frac{3\pi}{2}$.  Second-order derivatives for velocity (and first order derivatives for pressure and vorticity) are not square integrable, and therefore these solutions do not have higher regularity. Nevertheless, the exact boundary velocity is zero on the reentrant edges (at $\theta = 0$ and $\theta = \frac{3\pi}{2}$) and so the boundary data oscillation can be considered of high order. The parameters are taken adimensional and fixed to 
$\kappa =   g = \gamma = \alpha = U = 1$.

\begin{figure}[t!]
\begin{center}
\includegraphics[width=0.99\textwidth]{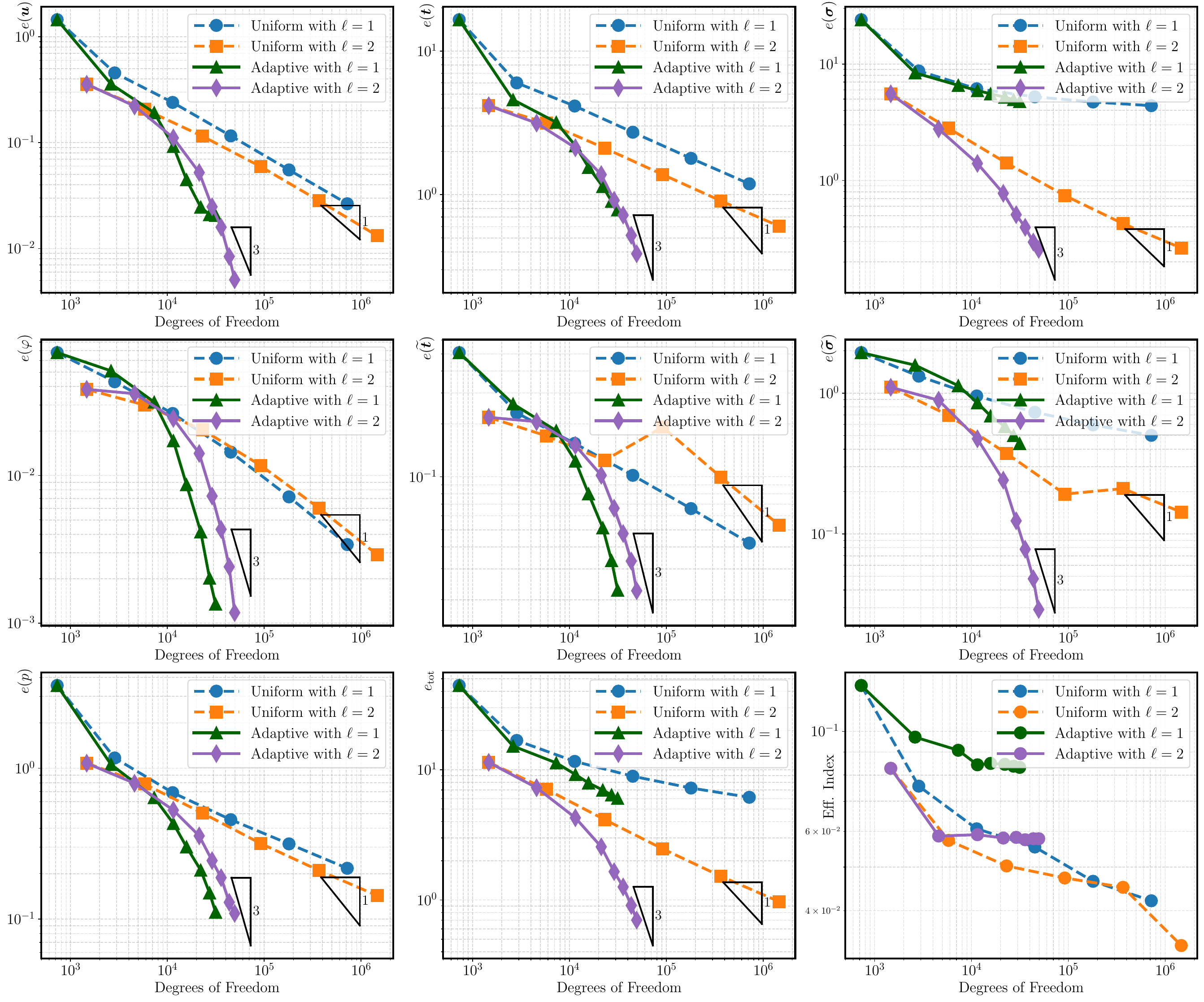}
\end{center}

\vspace{-4mm}
\caption{Example 2. Convergence for each field variable against non-smooth solutions on an L-shaped domain using uniform and adaptive mesh refinement with barycentrically refined meshes. The bottom right plot shows effectivity indexes for all cases. For reference we show triangles indicating linear and cubic slopes.}\label{fig:aposte}
\end{figure}

\begin{figure}[t!]
\centering
\includegraphics[width=0.245\textwidth]{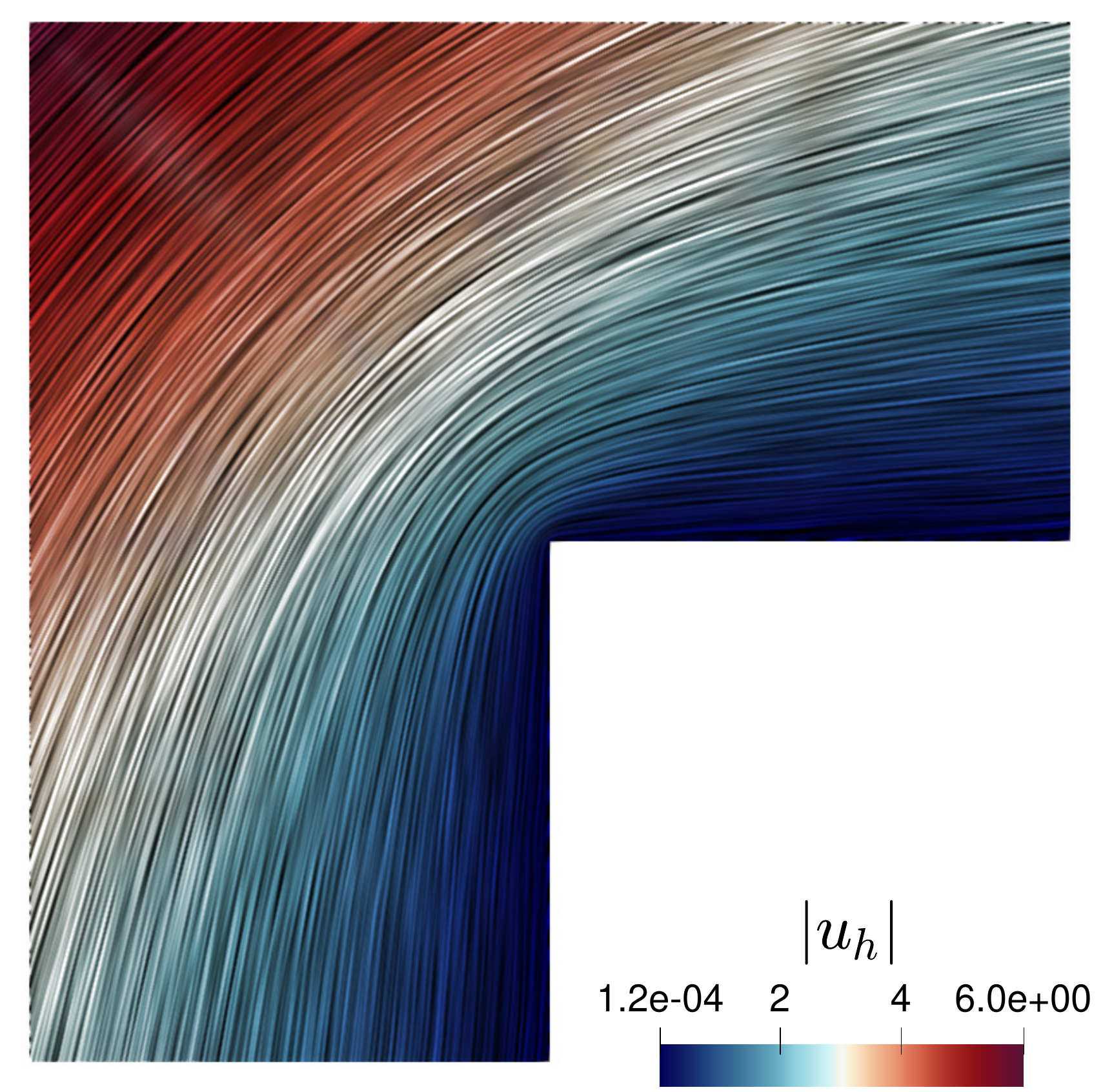}
\includegraphics[width=0.245\textwidth]{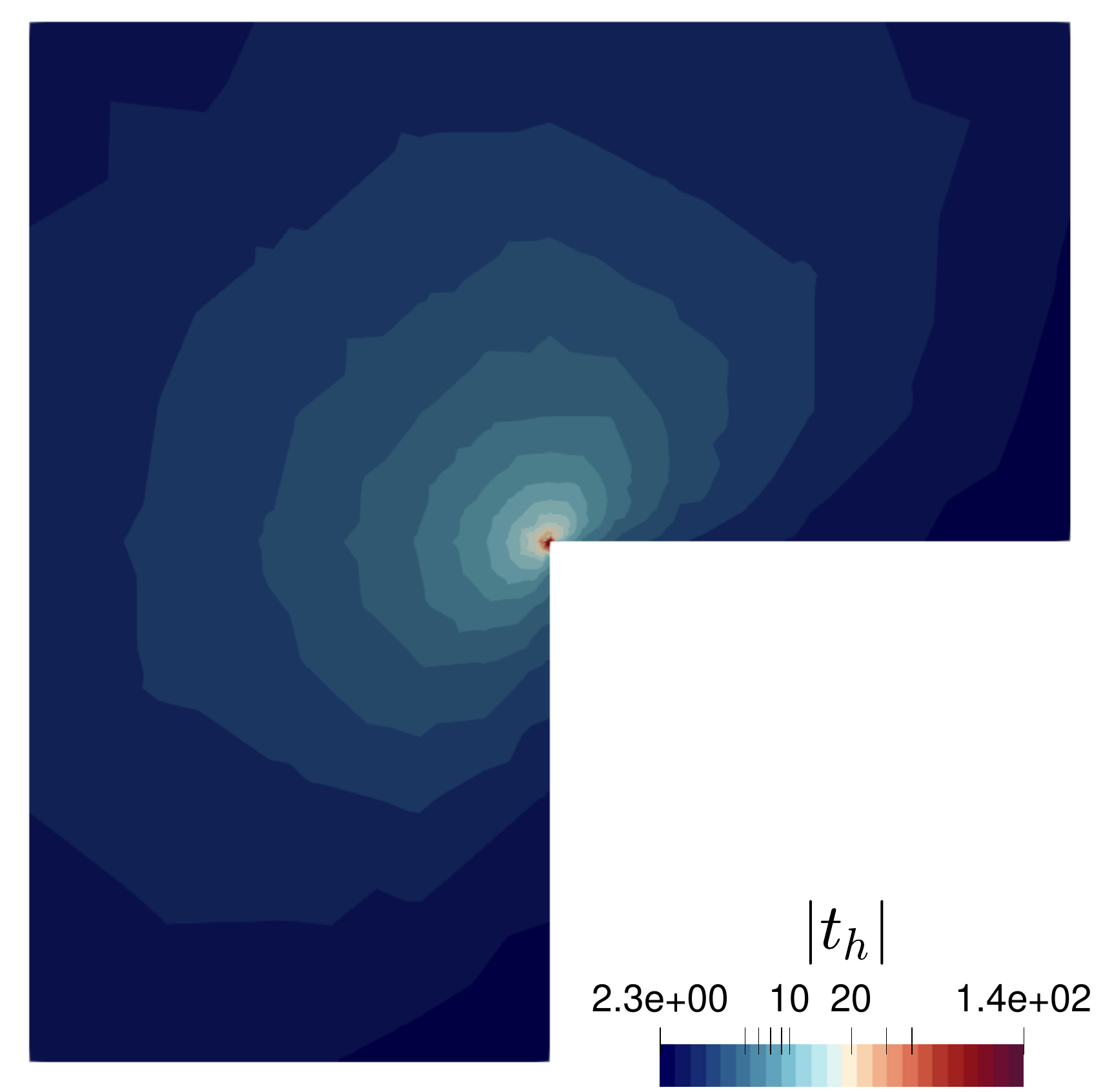}
\includegraphics[width=0.245\textwidth]{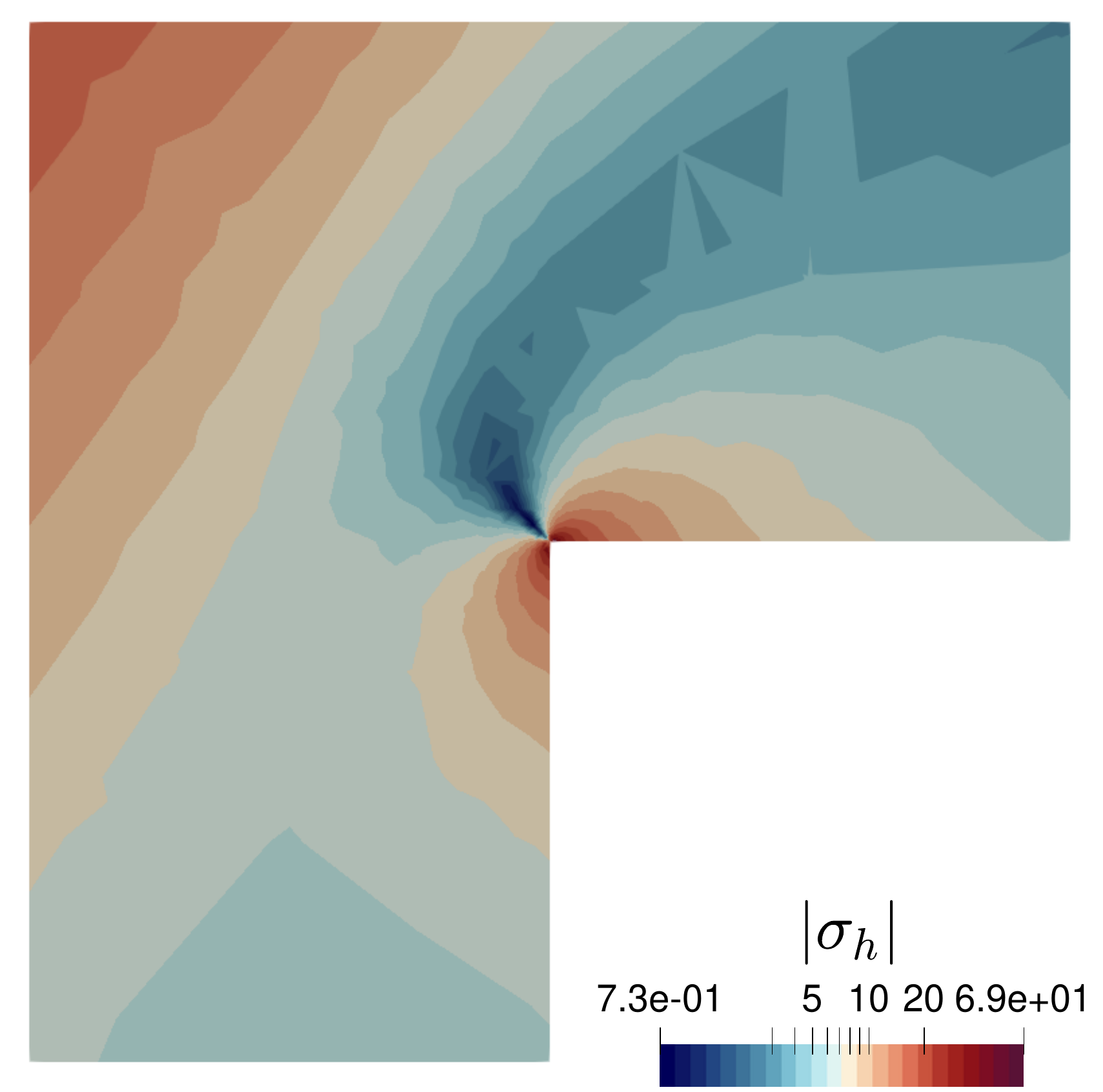}\\
\includegraphics[width=0.245\textwidth]{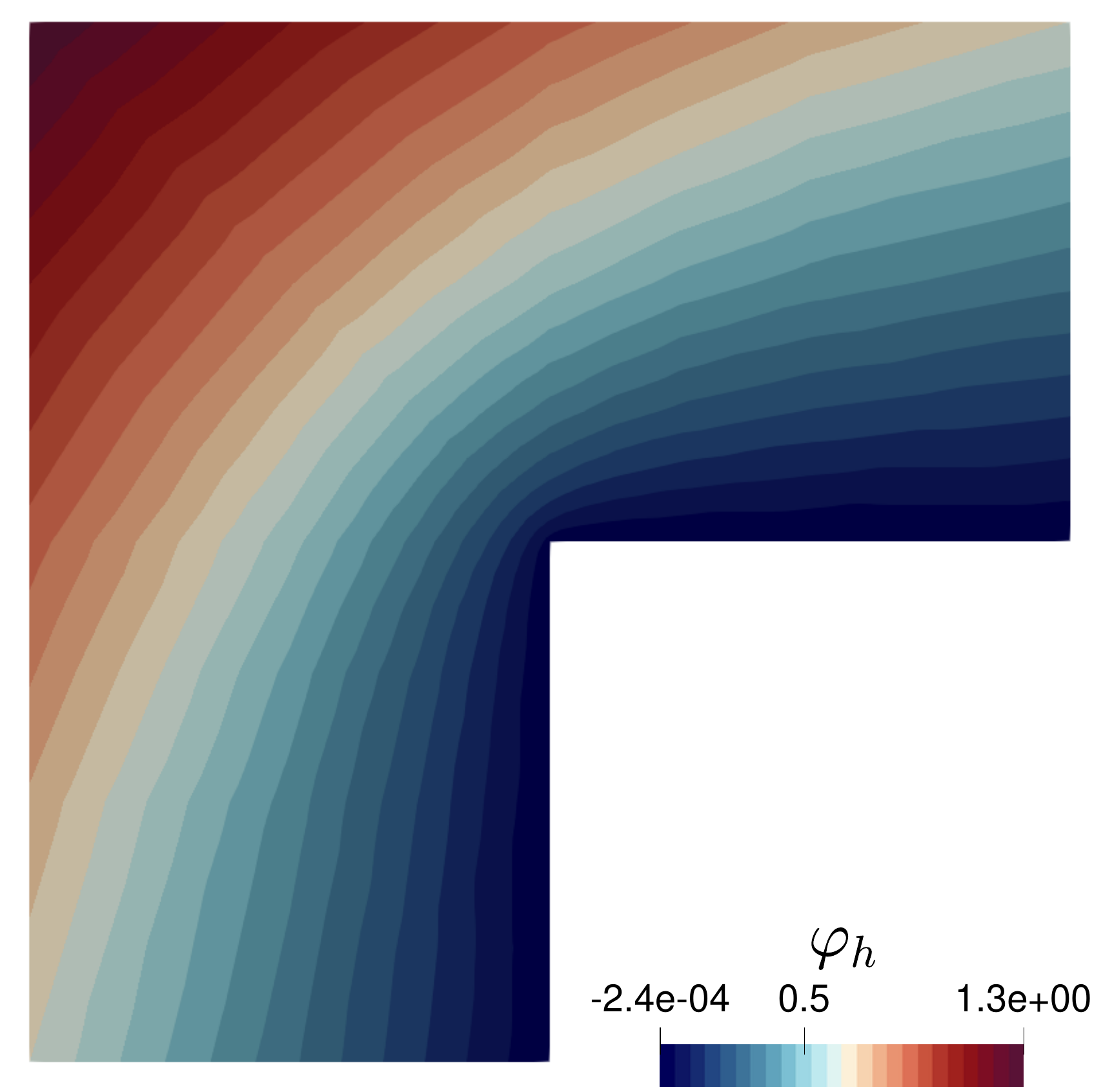}
\includegraphics[width=0.245\textwidth]{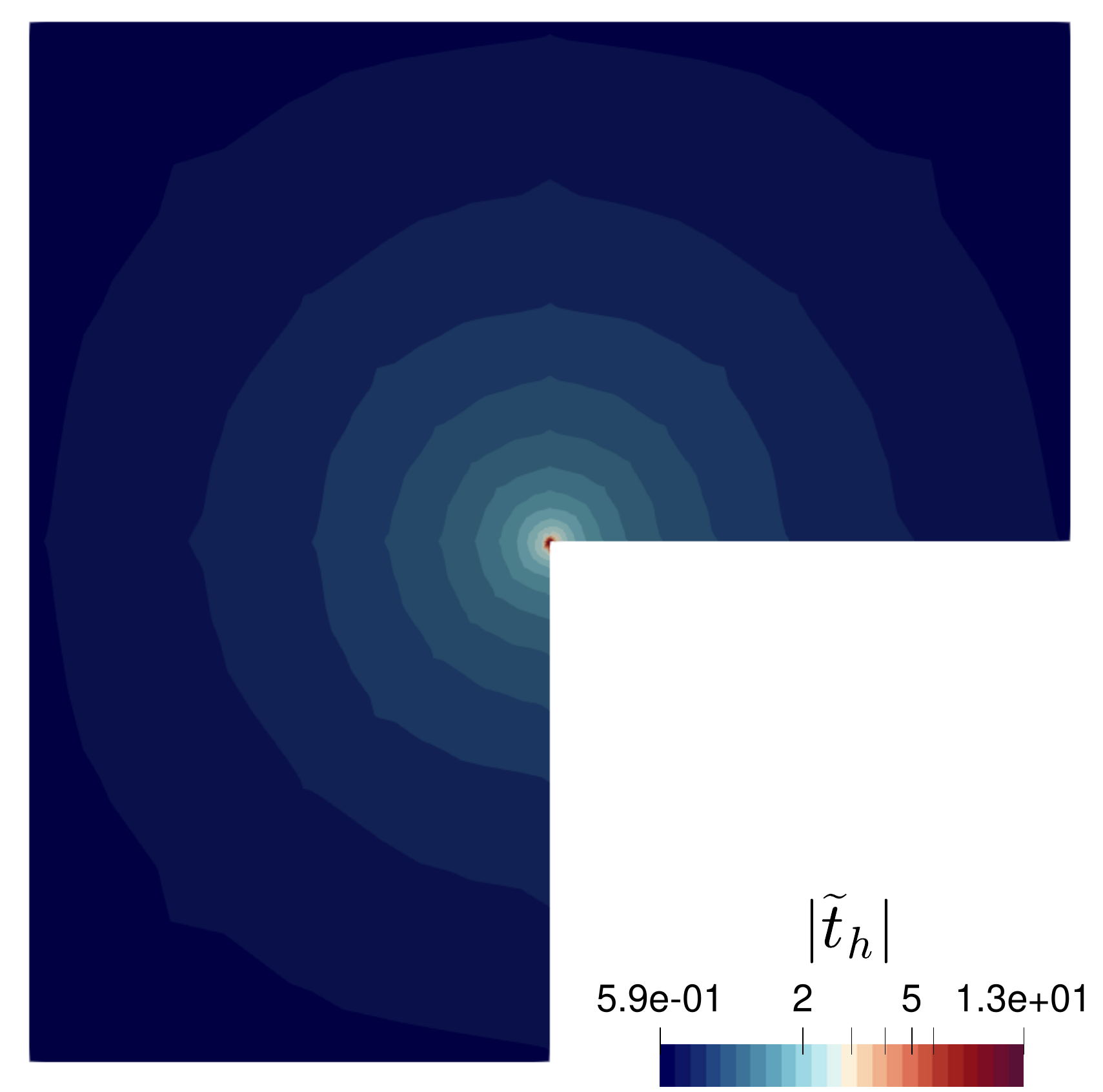}
\includegraphics[width=0.245\textwidth]{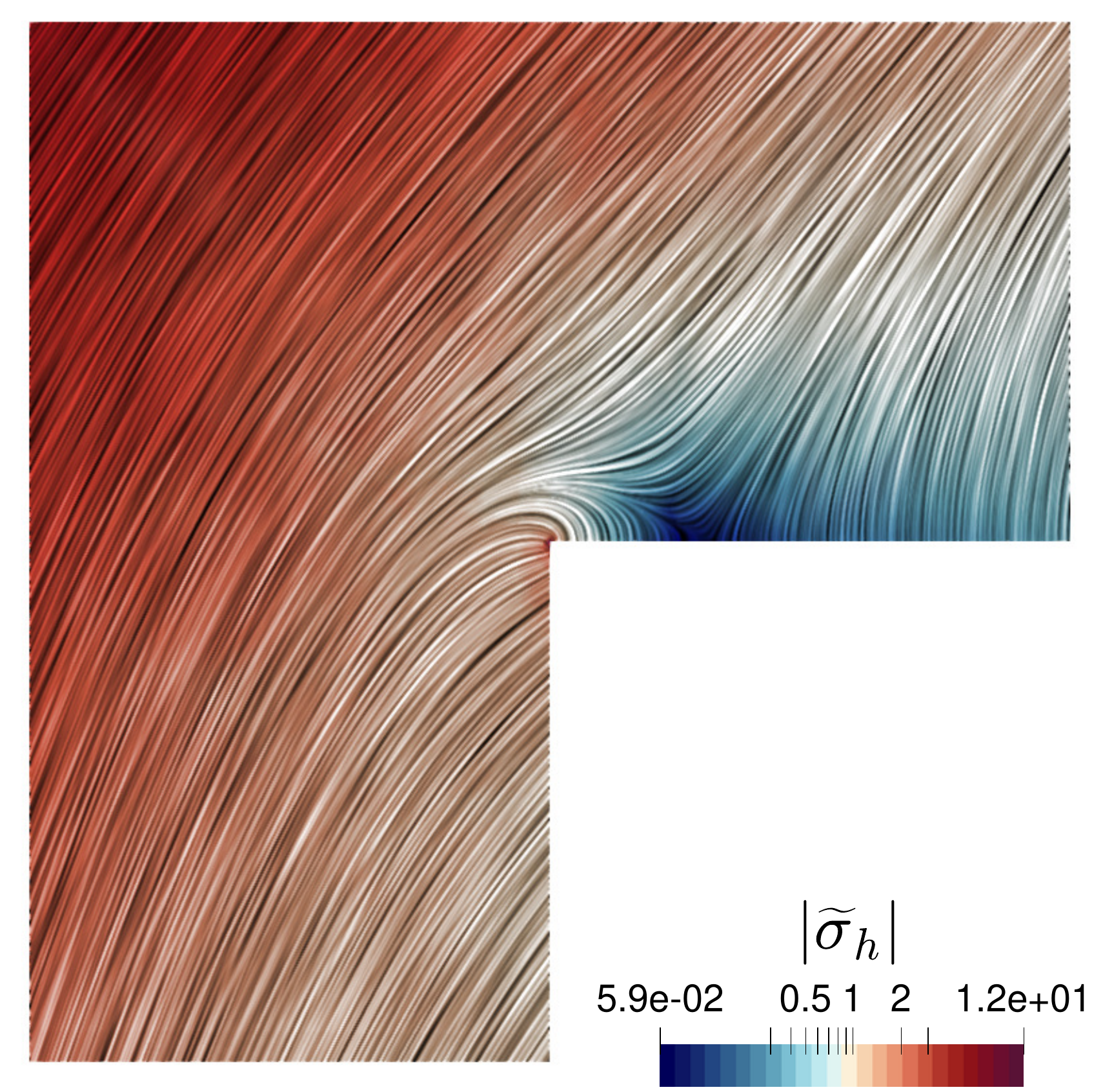}
\includegraphics[width=0.245\textwidth]{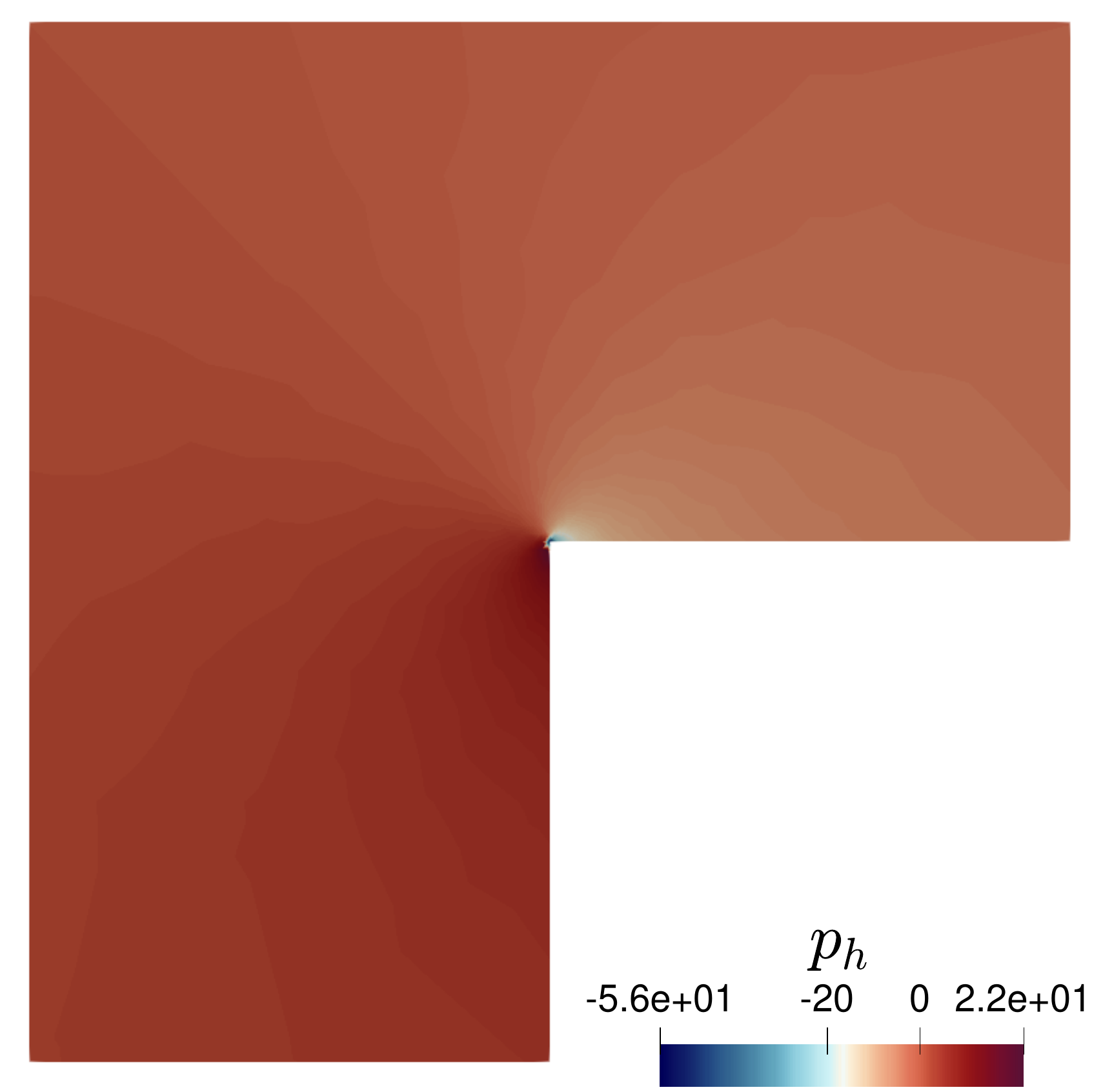}\\
\includegraphics[width=0.244\textwidth]{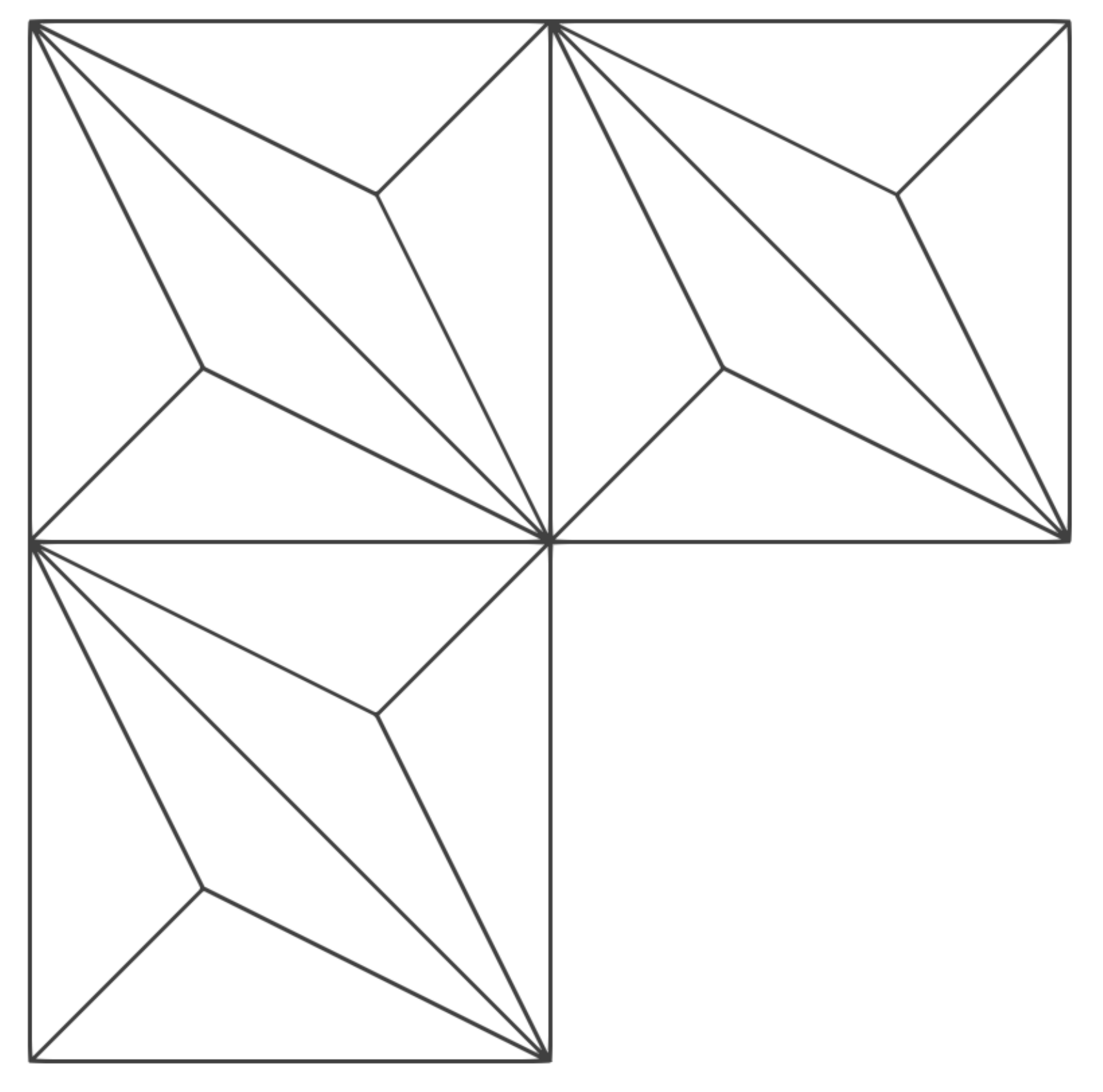}
\includegraphics[width=0.244\textwidth]{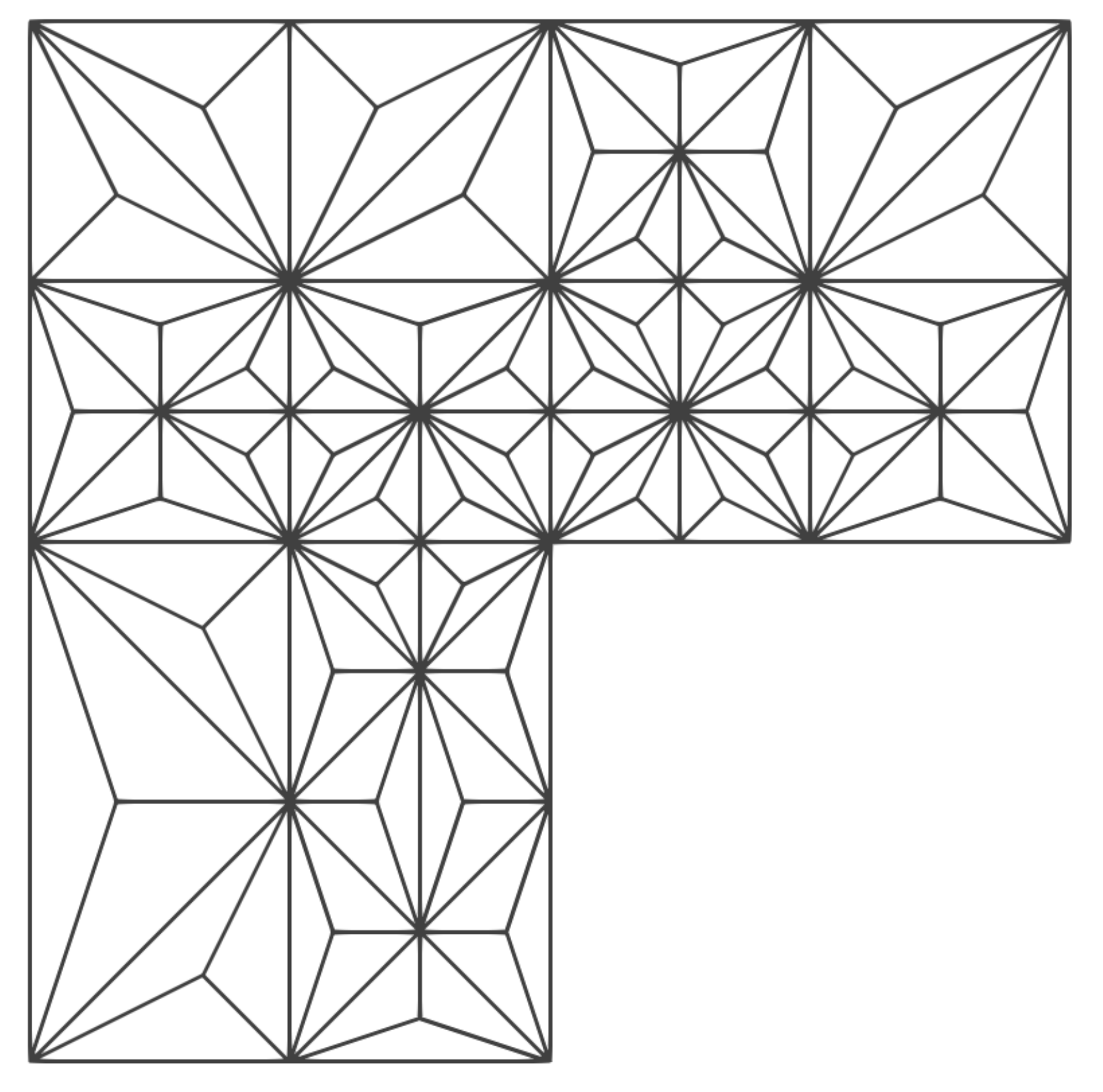}
\includegraphics[width=0.244\textwidth]{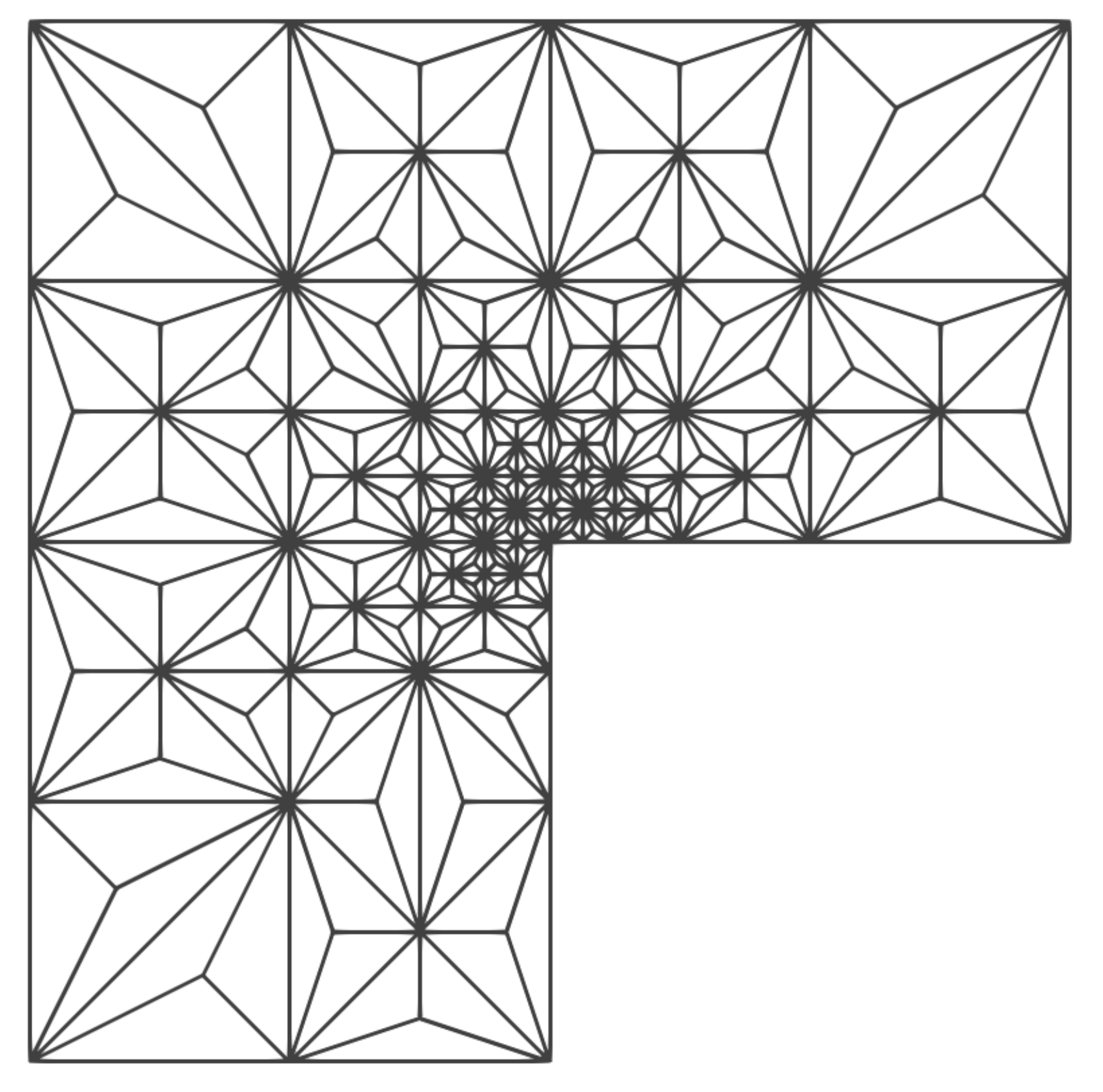}
\includegraphics[width=0.244\textwidth]{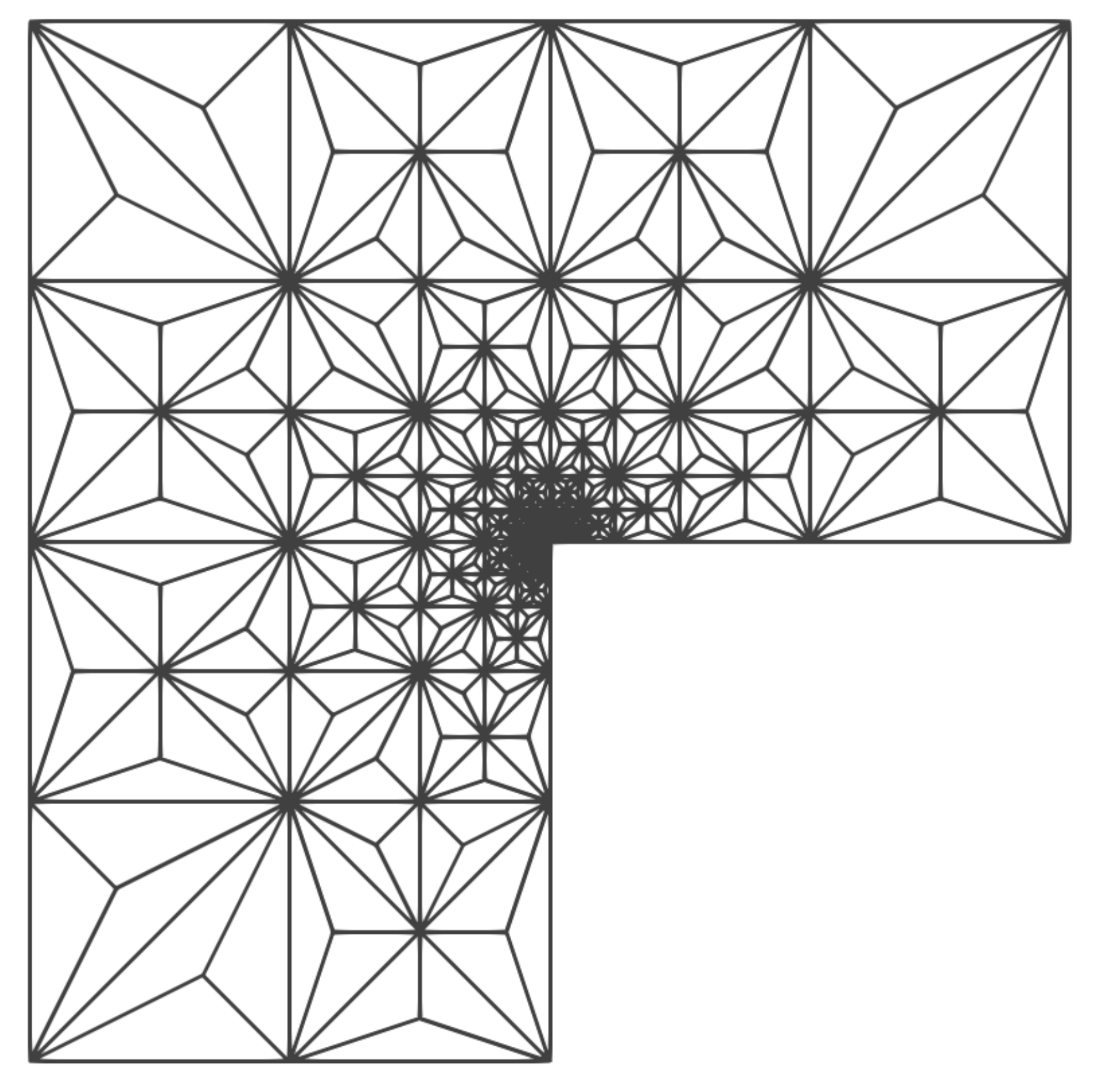}

\vspace{-4mm}
\caption{Example 2. Approximate solutions for the 2D manufactured problem on an L-shaped domain with $\ell=2$ and rendered on a coarse mesh: velocity line integral contours, strain rate magnitude, pseudostress magnitude, slices of concentration profile, diffusive flux and total flux line integral contours,  and  postprocessed pressure. The bottom row shows four adaptively refined meshes.}
\label{fig:convAdapt}
\end{figure}

 In the computation pipeline we follow \cite[Algorithm 1]{GIRS-2022}.  Adaptive refinement employs Dörfler marking with $\theta = 0.5$ and follows the classical solve–estimate–mark–refine cycle, applied to a two-level macro–bary\-centric mesh hierarchy. Finite element spaces are defined on the barycentrically refined mesh. Then we compute the approximate solution.  Local estimators are first computed on the barycentric mesh and projected onto the macro mesh to mark elements for refinement. Barycentric refinement then regenerates the next barycentric mesh for the next finite element spaces. Cells exceeding a fraction of the maximum indicator are refined iteratively until the prescribed tolerance is reached. 

Figure~\ref{fig:aposte} reports convergence histories for polynomial degrees $\ell=1$ and $\ell=2$, including errors in velocity $e(\mathbf{u})$, gradient $e(\mathbf{t})$, stress $e(\boldsymbol{\sigma})$, concentration $e(\varphi)$, auxiliary fluxes $e(\tilde{\mathbf{t}})$ and $e(\tilde{\boldsymbol{\sigma}})$, pressure $e(p)$, and total error $e_\mathrm{tot}$. The results demonstrate suboptimal convergence for all fields (even sub-linear) for the case of uniform refinement while optimal or superconvergent behaviour is restored under adaptive mesh refinement. Effectivity indices remain between 0.035 and 0.4, and we note that for the adaptive refinement case these numbers reach a plateau. Figure~\ref{fig:convAdapt} shows adaptively refined meshes and approximate solutions obtained with the scheme using $\ell = 2$.  Local contributions of $\Theta$ accurately identify regions of high residuals near the reentrant corner, where steep gradients in pressure, velocity, and microorganism concentration occur.

\subsection*{Example 3: adaptive mesh refinement in a square cavity with inclusions}
Next, we consider bioconvective flow in a square domain with two square inclusions. The geometry  follows the configuration used in \cite{izadpanah} for thermo-bioconvection of gyrotactic micro-organisms. The square has base 10\,cm with two inclusions of 2.5\,cm per side. While for the flow we consider no-slip boundary conditions everywhere, the boundary conditions for the concentration equations are different than those analyzed in the paper. We set $\varphi = \alpha$ on the outer left edge and $\varphi = 1 + \alpha$ on the outer right sub-boundary, and impose $\tilde{\bsi}\cdot \bn = 0$ on the outer top and bottom as well as on the boundaries of the inclusions. Then we no longer require to enforce the mean value of the concentration. 
The remaining parameters are as follows: 
\[ \alpha \in \{0,0.01,0.2,0.3\}, \quad \mu(\varphi) = \mu_0 e^{-\varphi - \alpha}, \quad g = 9.8, \quad \gamma = \kappa = 0.1, \quad U = 0.01, \quad \mu_0 = 1.\]
We generate a coarse macro mesh and its corresponding barycentric refinement, and we solve for each value of $\alpha$ and for seven steps of adaptive mesh refinement with the D\"orfler agglomeration coefficient set to $\theta = 0.2$ (for the previous examples we had taken 0.5). For the case of higher $\alpha$ the Newton--Raphson algorithm takes six iterations to reach the prescribed tolerance, irrespective of the mesh refinement level. For this example we are using the polynomial degree $ \ell =1$. We show in Figure~\ref{fig:inclusions} three samples of adapted grids and we also show  approximate solutions (we only include velocity, concentration and flux) for the case of $\alpha = 0.3$. The mesh plots indicate a much more marked refinement near the reentrant corners of the square inclusions as well as near the right edge of the domain -- which is where the high gradient of concentration is.  All fields show very well resolved profiles.

\begin{figure}[t!]
    \centering
 \includegraphics[width=0.325\linewidth]{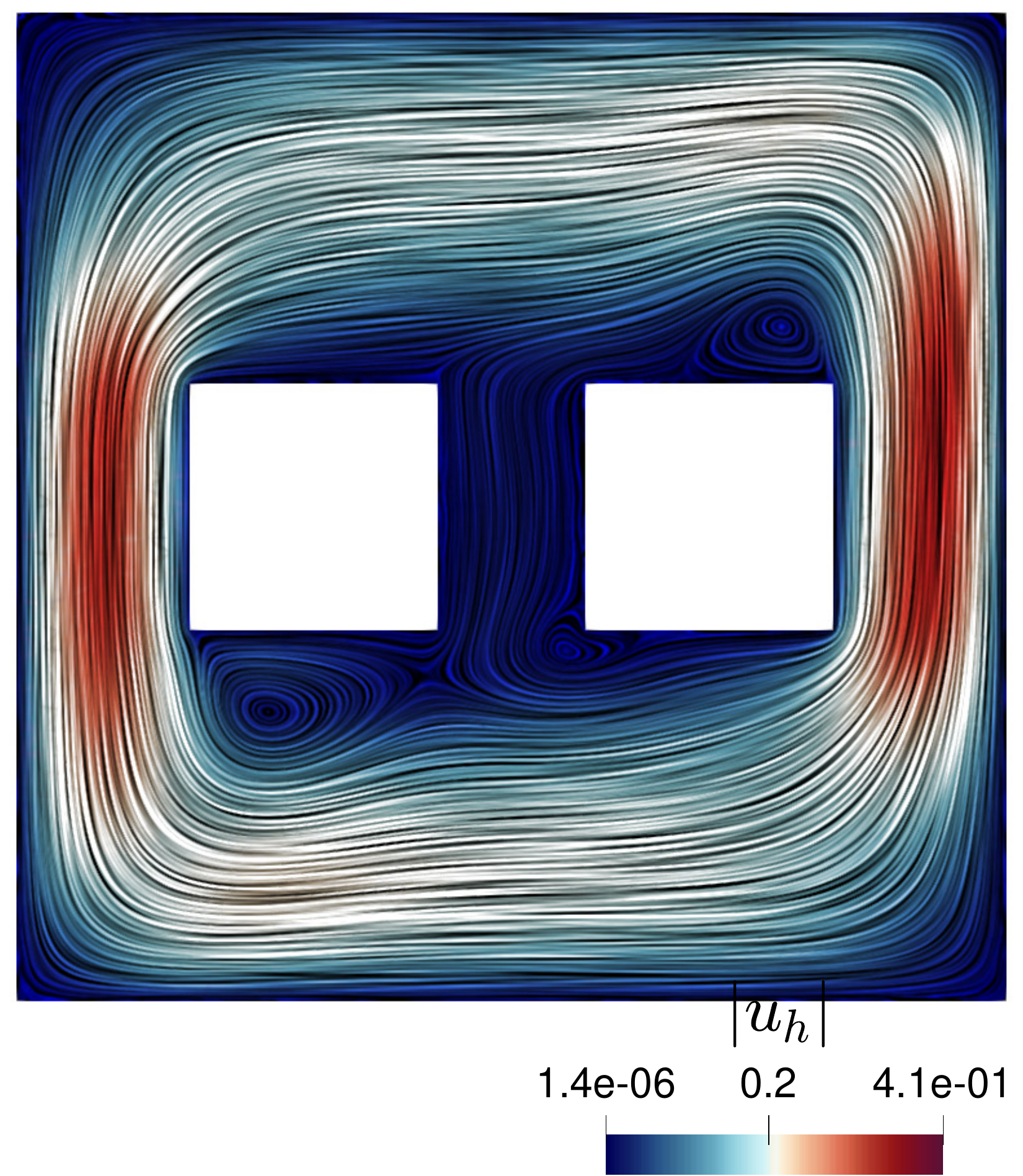}
  \includegraphics[width=0.325\linewidth]{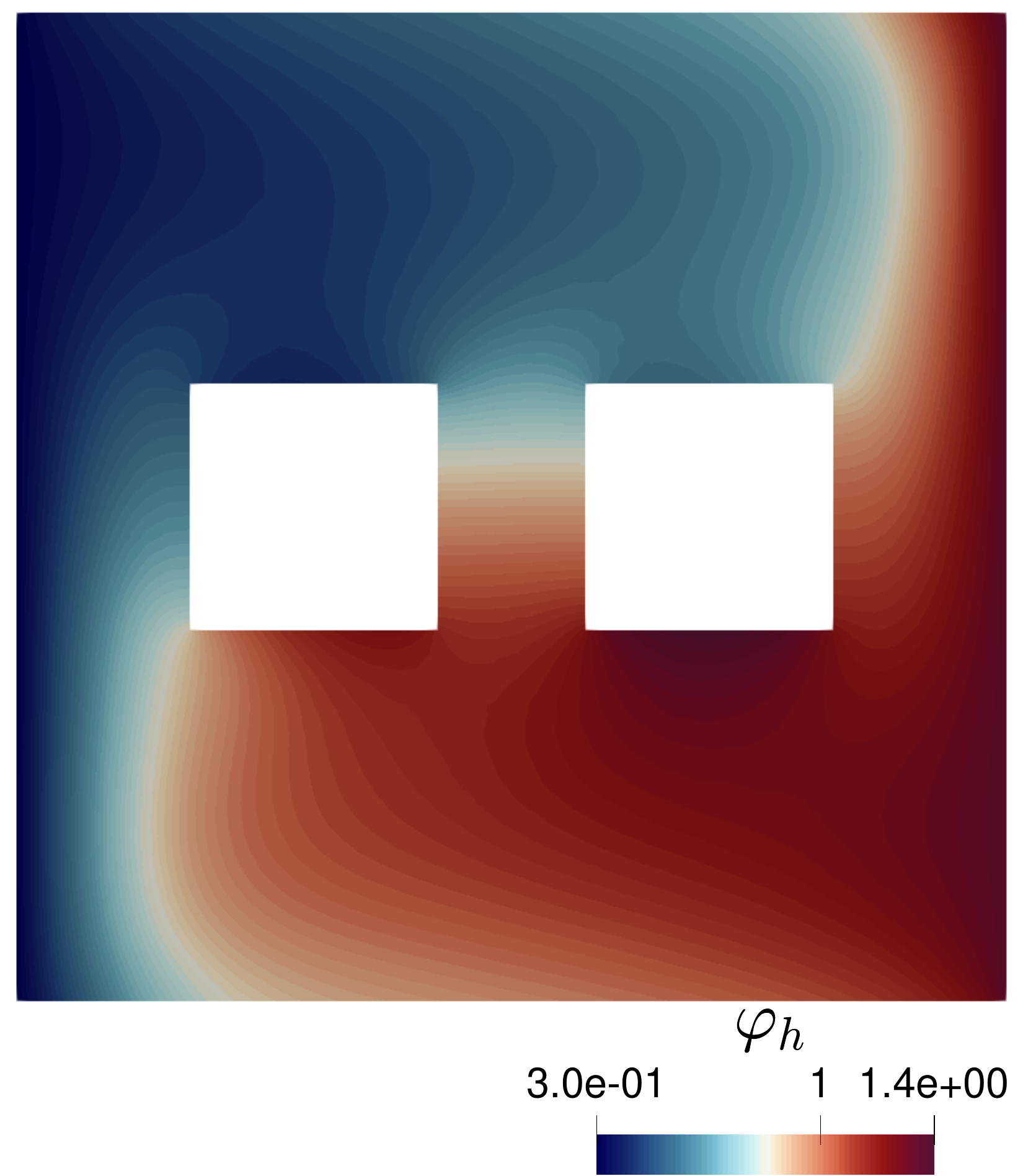}
   \includegraphics[width=0.325\linewidth]{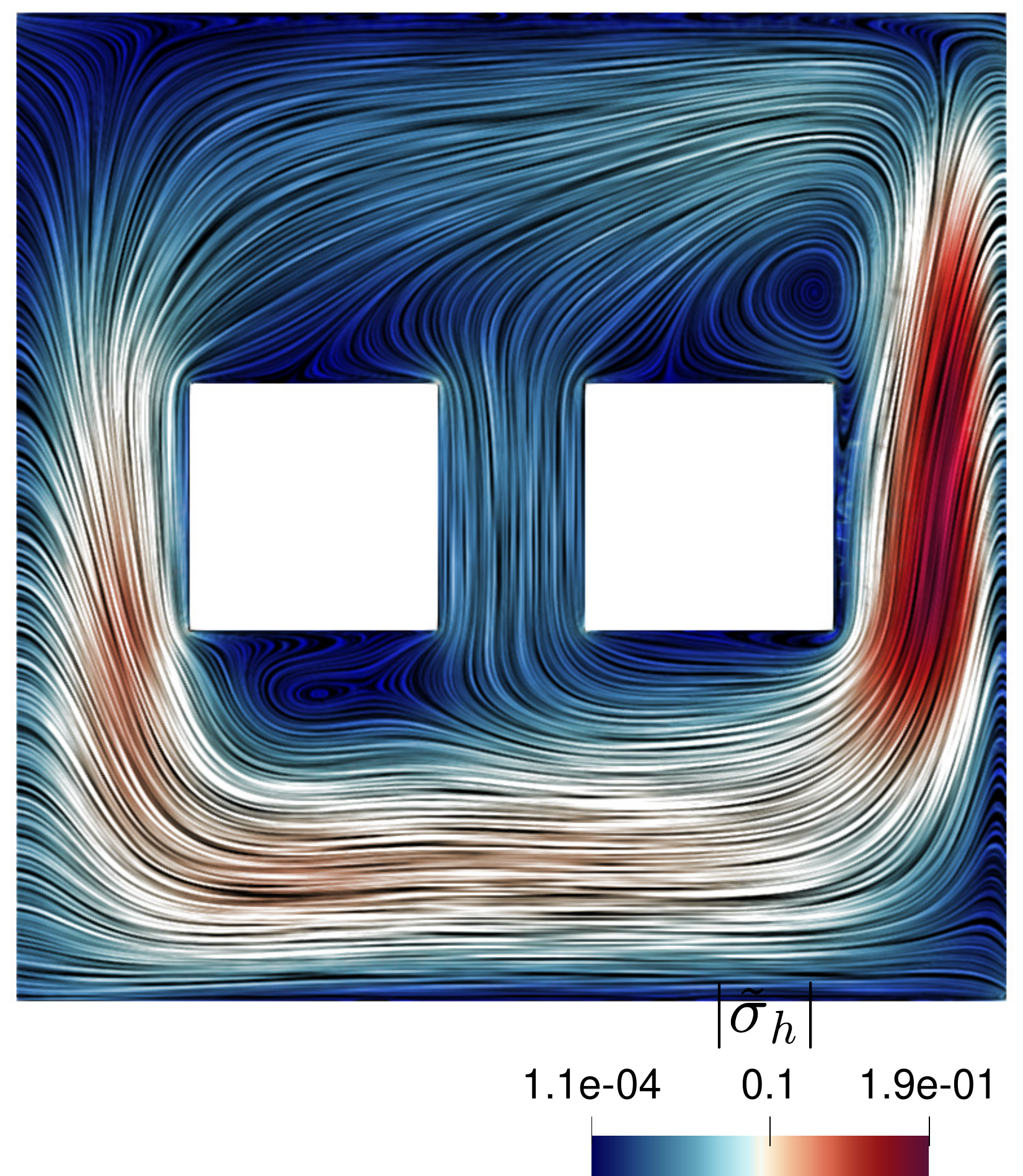}\\
 \includegraphics[width=0.325\linewidth]{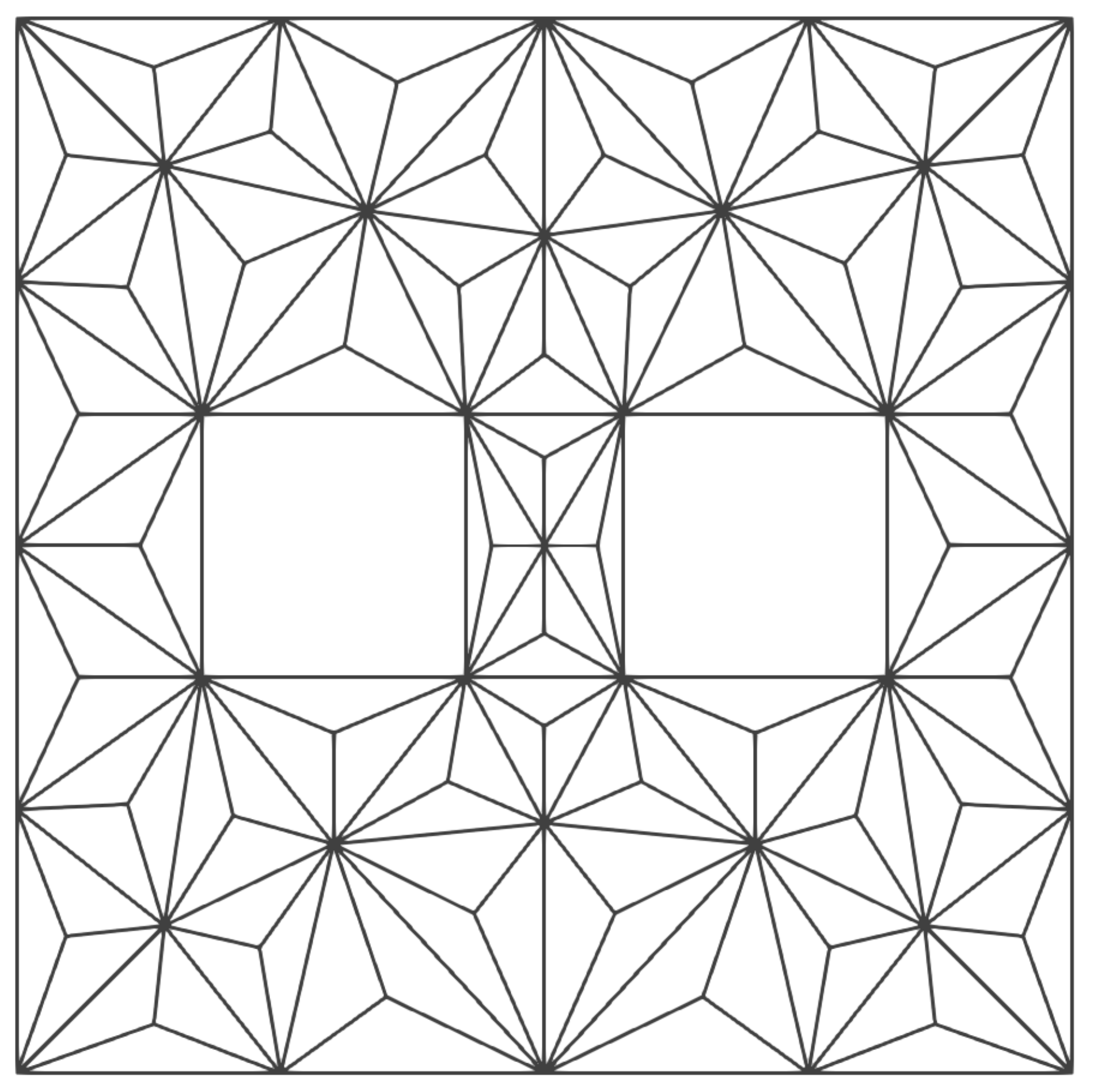}
    \includegraphics[width=0.325\linewidth]{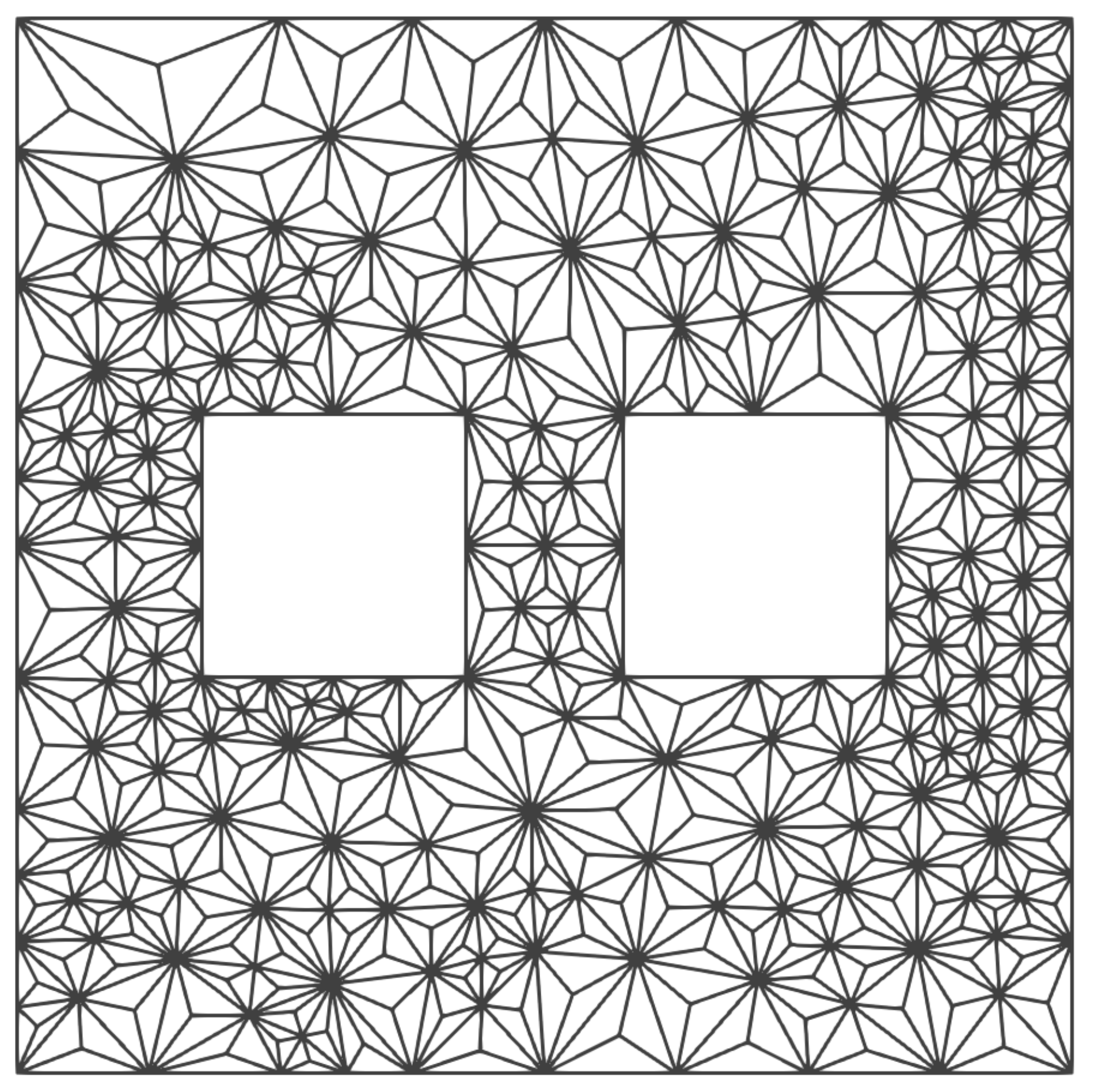}
    \includegraphics[width=0.325\linewidth]{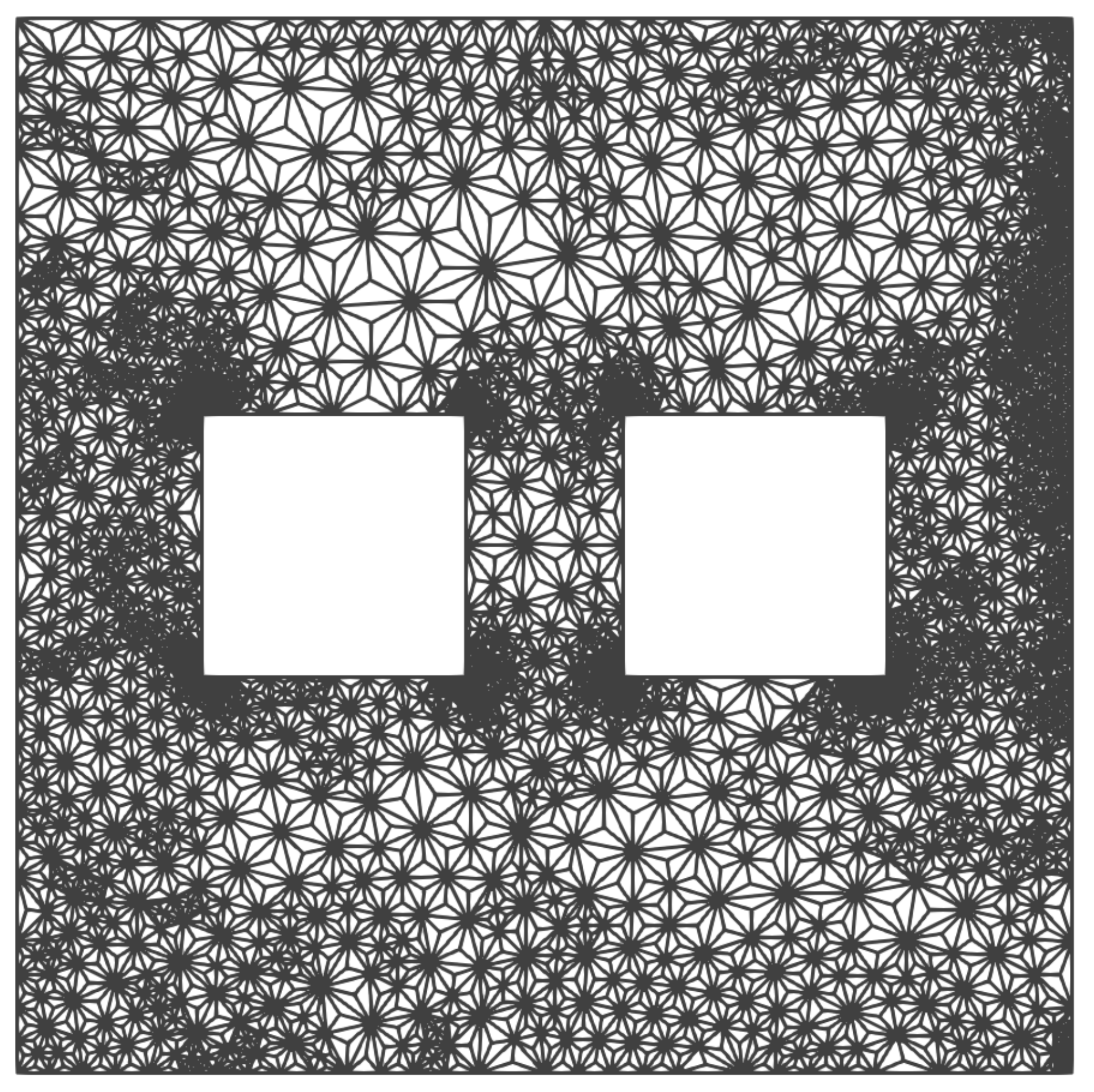}

    \vspace{-4mm}
    \caption{Example 3. Bioconvection of micro-organisms on a cavity with square inclusions and mixed boundary conditions. Velocity line integral contours, concentration profile, and total flux line integral contours for the case $\alpha = 0.3$. The bottom row shows three samples of adaptively refined meshes.}
    \label{fig:inclusions}
\end{figure}

\subsection*{Example 4: application to time-dependent bioconvective flows}
Finally, we consider a physically motivated benchmark that exhibits strong bioconvection coupling. The flow is driven by density differences arising from microorganism concentration gradients. We use the  Einstein/Batchelor motivated model as in \cite{cao-2013}
\[
\mu(\varphi) =
\begin{cases}
\mu_0, & \varphi_r \le 0, \\[6pt]
\mu_0 \bigl(1 + 2.5\,\varphi_r + 5.3\,\varphi_r^2 \bigr), & 0 < \varphi_r \le 0.10, \\[6pt]
\mu_0 \exp\!\left( \dfrac{2.5\,\varphi_r}{1 - 1.4\,\varphi_r} \right), & 0.10 < \varphi_r \le 0.60, \\[6pt]
\mu_0 \exp(9.375), & \varphi_r > 0.60,
\end{cases} \qquad \varphi_r = \varphi/\varphi_{\max},
\]
with \(\mu_0=0.01\)\,[cm$^2$/s] (reference viscosity), and $\varphi_{\max} = 7\cdot10^6$.  These choices follow the modelling and numerical examples in 
\cite{edwards14} and use the model parameters 
\[
\kappa = 0.01\ \text{[cm$^2$/s]},\quad U = 0.1\ \text{[cm/s]},\quad g=980.665\ \text{[cm$^2$/s]},\quad \gamma = 5\cdot 10^{-5} \ \text{[cm$^2$/cells]}.
\]
The computational domain is the rectangle $\Omega = (0,L)\times (0,H)$, with $L = 16$\,cm and $H = 2$\,cm. 
 We use a time-dependent variant of the model problem, including simply a backward Euler discretization of the acceleration and concentration rate terms with a constant time step of $\Delta t = 0.25$\,s. The boundary conditions for velocity are different than those analyzed in the paper: we impose no-slip conditions on the bottom and vertical walls, whereas we set a slip condition $\bu \cdot \bn = 0 $ and a zero shear traction on the top boundary. We run the simulations until the final time $T = 150$\,s. The results of the computations (using again $\ell =1$ in this case) are shown in Figure~\ref{fig:transient}. 
For earlier times the solution is essentially quiescent: \(\bu\approx \boldsymbol{0}\) except near the top boundary and \(\varphi\) varies only slowly in the vertical direction (no pattern). For larger times, we observe that bioconvective patterns develop: concentrated rising plumes or cellular recirculations appear, with strong vertical gradients of concentration and localised shear in velocity. This is in agreement with the simulations reported in \cite[Example 2]{edwards14}.

\begin{figure}[t!]
    \centering
\includegraphics[width=0.495\linewidth]{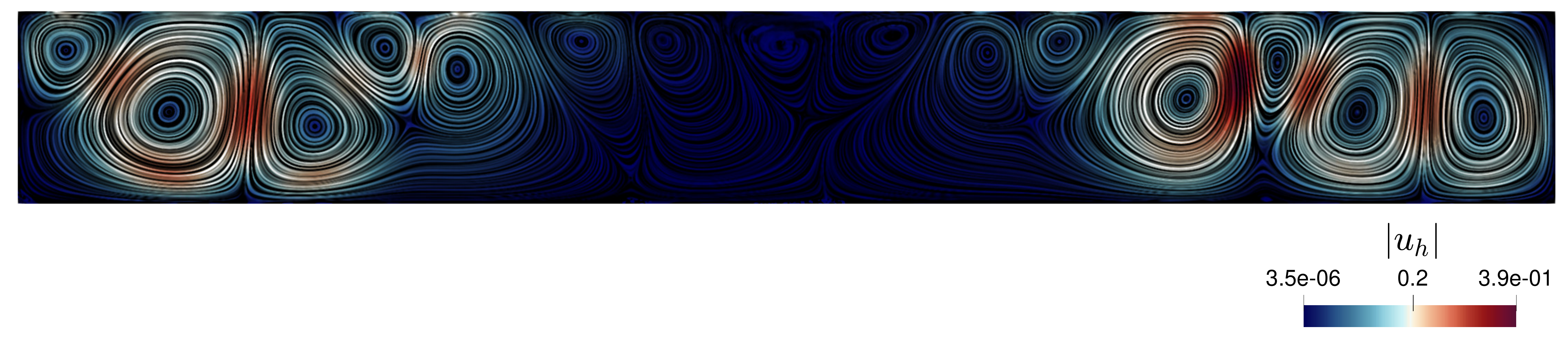}
\includegraphics[width=0.495\linewidth]{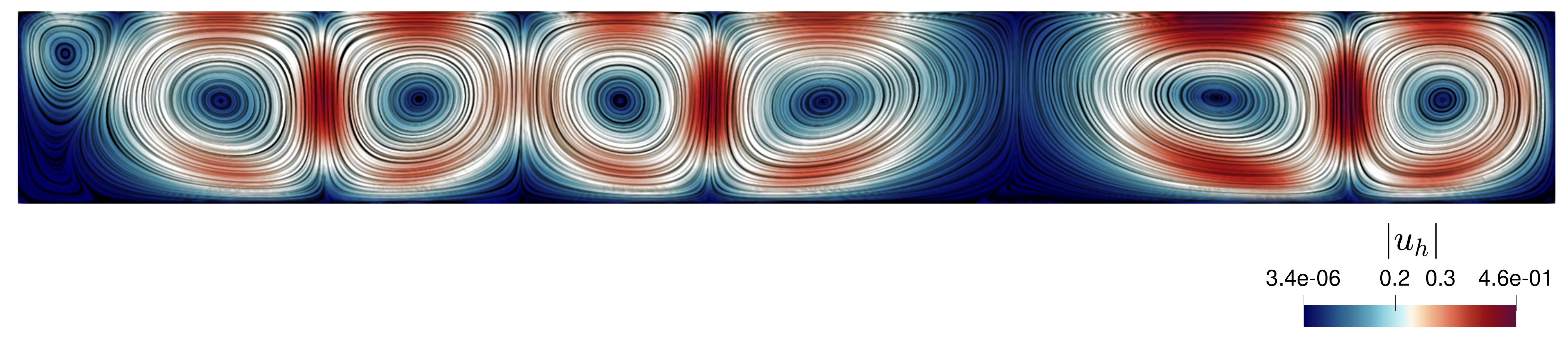}\\
\includegraphics[width=0.495\linewidth]{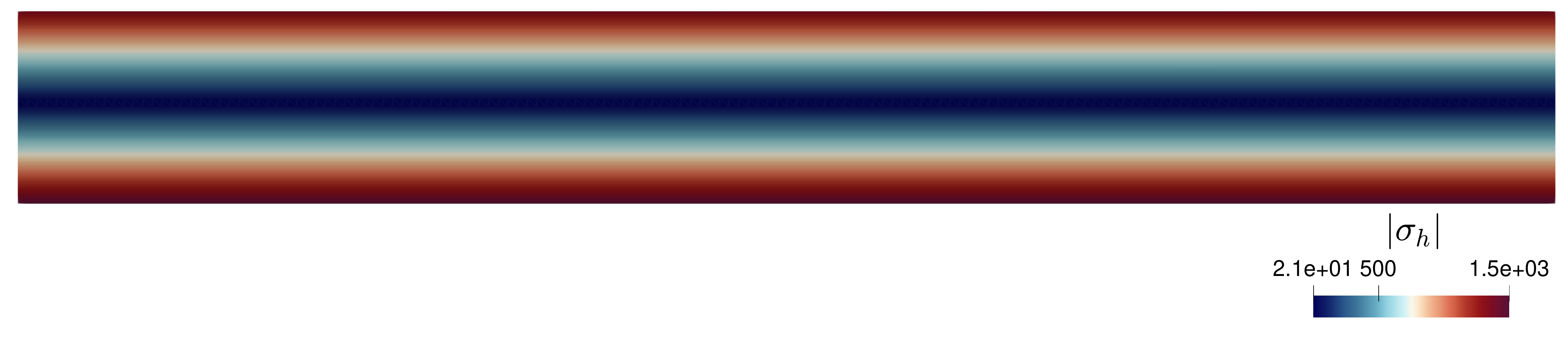}
\includegraphics[width=0.495\linewidth]{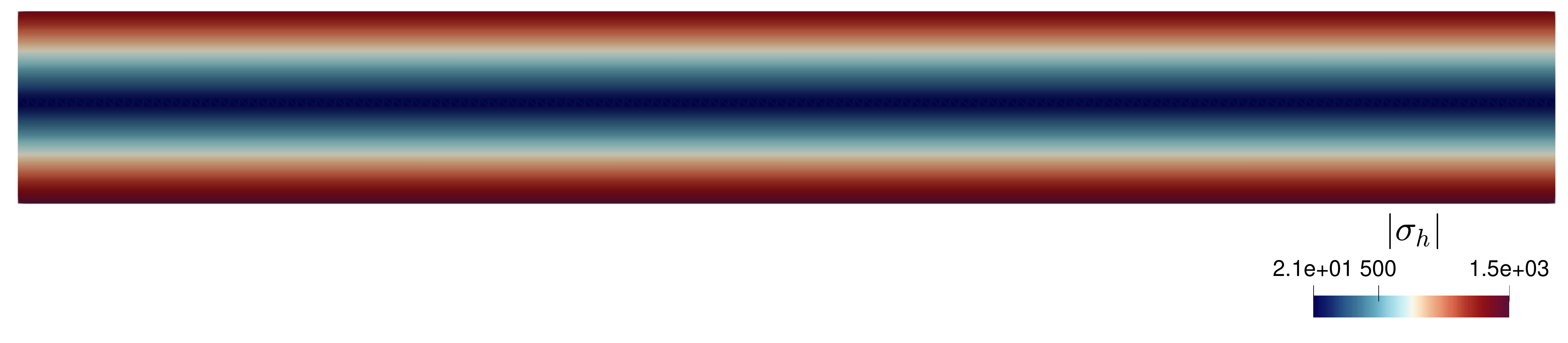}\\
\includegraphics[width=0.495\linewidth]{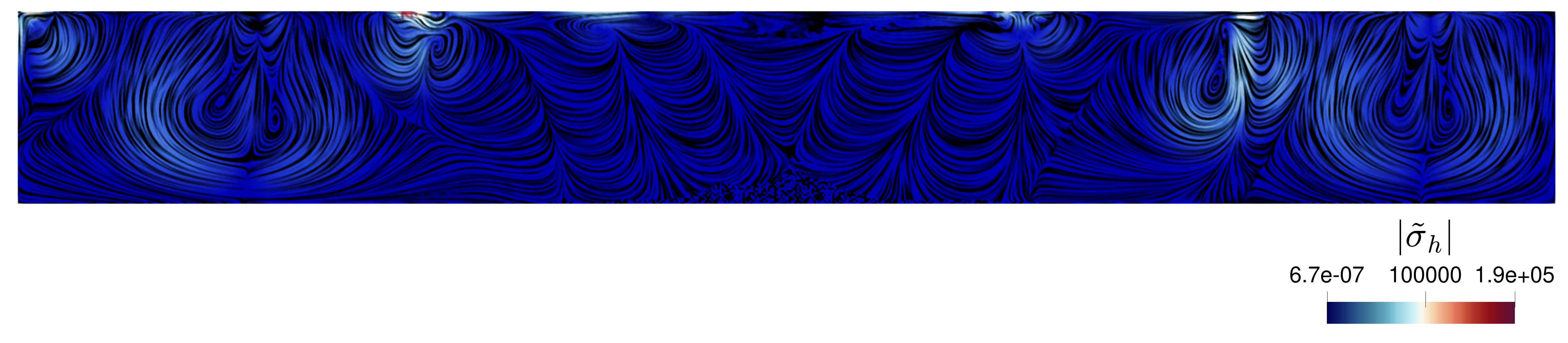}
\includegraphics[width=0.495\linewidth]{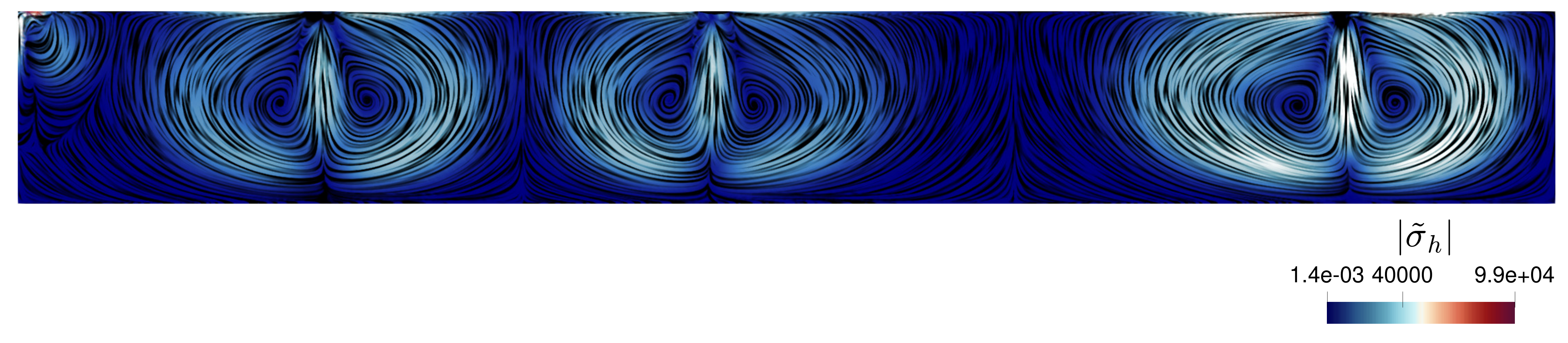}\\
\includegraphics[width=0.495\linewidth]{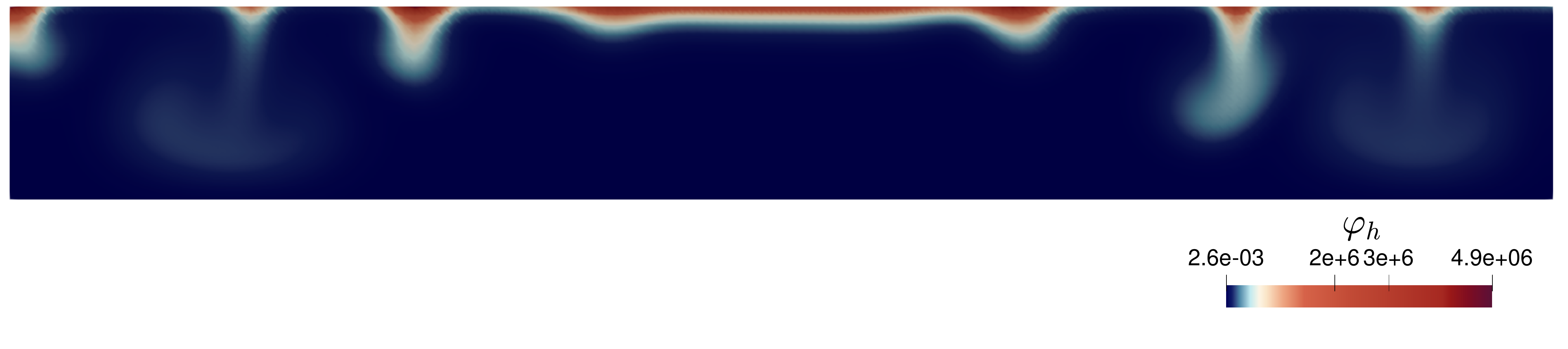}
\includegraphics[width=0.495\linewidth]{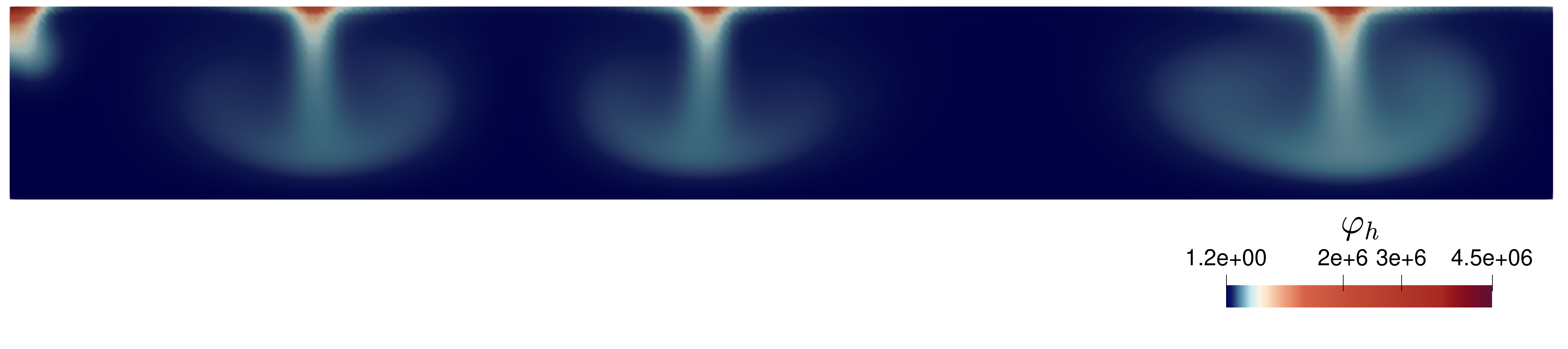}

\vspace{-4mm}
    \caption{Example 4. Snapshots of the numerical solution (showing here only velocity, stress magnitude, flux, and concentration) at $t = 75$\,s (left) and $t = 150$\,s (right).}
    \label{fig:transient}
\end{figure}


\end{document}